\documentclass[reqno,a4paper]{amsart}%
\usepackage{amsfonts}
\usepackage{amsmath}
\usepackage{amssymb}
\usepackage{graphicx}%
\providecommand{\U}[1]{\protect\rule{.1in}{.1in}}
\newtheorem{theorem}{Theorem}
\theoremstyle{plain}

\newtheorem{axiom}{Axiom}

\newtheorem{corollary}[theorem]{Corollary}

\newtheorem{example}[theorem]{Example}

\newtheorem{lemma}[theorem]{Lemma}

\newtheorem{problem}[theorem]{Problem}
\newtheorem{proposition}[theorem]{Proposition}
\newtheorem{remark}[theorem]{Remark}

\numberwithin{equation}{section}
\numberwithin{theorem}{section}
\begin{document}
\title[Topological radicals, V.]{Topological radicals, V. From Algebra to Spectral Theory}
\author{Victor S. Shulman}
\address{Vologda State Technical University\\
Lenin St. 15, Vologda 160000, Russia}
\email{shulman.victor80@gmail.com}
\author{Yurii V. Turovskii}
\address{Institute of Mathematics and Mechanics, National Academy of Sciences of Azerbaijan\\
F. Agaev St. 9, Baku AZ1141, Azerbaijan}
\email{yuri.turovskii@gmail.com}
\thanks{This paper is in final form and no version of it will be submitted for
publication elsewhere.}
\dedicatory{To the memory of Bill Arveson, a great mathematician and a great person}\date{}

\begin{abstract}
We introduce and study procedures and constructions of the theory of general
topological radicals that are related to the spectral theory --- the
centralization, primitivity and socle procedures, the scattered radical, the
radicals related to the continuity of the usual, joint and tensor radii. Among
other applications we find some sufficient conditions of continuity of the
spectrum and spectral radii of different types, and in particular prove that
in a GCR C*-algebra the joint spectral radius is continuous on precompact sets
and coincides with the Berger-Wang radius.

\end{abstract}
\maketitle
\tableofcontents

\section{Introduction and preliminaries}

\subsection{Introduction}

The general theory of radicals can be considered as a global structure theory
of algebras which is aimed at the study of ideals and quotients simultaneously
defined for a large class of algebras and related to some general properties
of algebras.

Briefly speaking, a map $P$ defined on some class $\mathfrak{U}$ of algebras
and sending each algebra $A\in\mathfrak{U}$ to its ideal $P(A)$ is called a
radical if it satisfies some natural conditions of covariance and stability
(see Section 2.1). The most popular example of a radical is the
\textit{Jacobson radical }$\operatorname{rad}$ defined on the class of all
algebras; its restriction to the class of all Banach algebras is denoted by
$\operatorname{Rad}$. The importance of this map for the theory of Banach
algebras cannot be overestimated.

The first significant achievements of the theory of radicals were related to
the nilpotency and some close properties, and obtained by Baer, Levitzki,
K\"{o}the, Amitsur, Kurosh and others prominent algebraists. The functional
analytic counterpart of this theory was initiated by Peter Dixon in his paper
\cite{D97}, which contained basic definitions and presented first
applications. Since then the theory was developed and applied to different
problems of operator theory and Banach algebras in \cite{TR0, TR1, TR2, TR3,
TR4}. The problems that turned out to be "solvable in radicals" had
their\textbf{ }origin in the theory of invariant subspaces, irreducible
representations, semigroups of operators, tensor products, linear operator
equations, joint spectral radius, Banach Lie algebras and other topics. They
lead to introducing and study of corresponding radicals -- the
(\textit{closed}-) \textit{hypofinite radical} $\mathcal{R}_{\mathrm{hf}}$,
the \textit{hypocompact radical} $\mathcal{R}_{\mathrm{hc}}$, the
\textit{tensor radical} $\mathcal{R}_{t}$, the \textit{compactly
quasinilpotent radical} $\mathcal{R}_{\mathrm{cq}}$ and others (see Section 2
for definitions and discussion of these and other examples). It can be said
that the essence of the radical approach to a problem (a property, a
phenomenon) is to single out an ideal that accumulates elements related to
this problem, and study the dependence of this ideal on the algebra. The
construction of the corresponding radical is a typical result in the theory,
and this is reasonable because all \textquotedblleft well
behaved\textquotedblright\ radicals in their time find applications.

For functional analysts, the radicals on Banach algebras are most interesting.
However, to have a more flexible technique, one has to consider topological
radicals defined on non-complete normed algebras and also algebraic radicals
defined on algebras without topology. For instance, the study of non-closed
ideals of a Banach algebra ideals which are complete in a larger norm (we call
them Banach ideals) turns out to be very useful. This is a reason to consider
radicals in the wider context. So we outline three levels of consideration of
radicals: algebraic, normed and Banach. The relations between these three
theories are quite complicated. For example, the restriction of
$\operatorname{rad}$ to the class of \textit{all} normed algebras cannot be
considered as a topological radical, because $\operatorname{rad}\left(
A\right)  $ can be non-closed in $A$ if $A$ is not complete.

In this paper our main objects are spectral characteristics --- spectrum,
primitive ideals, spectral radii of different nature, socle, nilpotency,
spectral continuity etc --- and we construct special radicals and general
procedures of the theory of radicals aiming at applications to the spectral
theory in its algebraic and functional analytic aspects. The interplay between
algebraic and functional analytic sides of the spectral theory forces us to
devote a large part of the paper to understanding the links between algebraic
and topological theories of radicals.

All main objects and constructions are related, to a greater or lesser extent,
to the \textit{scattered radical} $\mathcal{R}_{s}$.This radical associates
with any Banach algebra the largest ideal whose elements have at most
countable spectrum. From the viewpoint of the hierarchy of radicals the
formulas
\begin{align}
\mathcal{P}_{\beta}  &  =\mathcal{R}_{\mathrm{hf}}\wedge\operatorname{Rad}%
,\label{i1}\\
\mathcal{R}_{s}  &  =\mathcal{R}_{\mathrm{hf}}\vee\operatorname{Rad}
\label{i0}%
\end{align}
exactly determine the place of the signed radicals in Banach algebras:
(\ref{i1}) and (\ref{i0}) mean that the \textit{closed-Baer radical}
$\mathcal{P}_{\beta}$ is the infimum of $\mathcal{R}_{\mathrm{hf}}$ and
$\operatorname{Rad}$, and $\mathcal{R}_{s}$ is the supremum of the same pair
of radicals. All these radicals are topological, but $\mathcal{R}_{s}$ and
$\operatorname{Rad}$ are the restrictions of some algebraic radicals to Banach
algebras while $\mathcal{P}_{\beta}$ and $\mathcal{R}_{\mathrm{hf}}$ have
algebraic analogs.

The paper is organized as follows. In the first section we gather the
necessary preliminary information. The second section contains the basic
definitions of the radical theory and the discussion of additional properties
of radical-like maps. We present here also the basic examples of topological
radicals. We introduce also radicals defined on the class of C*-algebras.

We call by a \textit{procedure} any rule that transforms ideal maps. The
simplest one which we use by default is the \textit{restriction }of radicals
to a subclass of algebras. Section 3 is devoted to the study of several
important procedures. The first two of them (the \textit{convolution procedure
}$P\longmapsto P^{\ast}$ and \textit{superposition} \textit{procedure
}$P\longmapsto P^{\circ}$) were introduced by Baer in the purely algebraic
context, and by Dixon in the topological one. They produce a radical from an
ideal map that lacks some stability property --- an under radical and an over
radical; we complement Dixon's results for the case of normed algebras. The
third one (the \textit{closure}) is very simple -- from a map $A\longmapsto
P(A)$ with non-necessarily closed ideals $P(A)$ it makes the map
$A\longmapsto\overline{P(A)}$ which can be a topological radical or be
transformed into a topological radical by means of the convolution procedure.
Thus by means of the closure one can obtain a topological radical from an
algebraic one. The fourth procedure is the \textit{regularization
}$P\longmapsto P^{r}$; it allows one to extend a radical from the class of all
Banach algebras to a class of normed algebras by taking the completion.

In Section 4 we study some operations (multiplace procedures), that produce
new radicals or radical-like maps from given families of radicals: supremum,
infimum and two-place procedures --- the convolution operation $\ast$ and
superposition operation $\circ$. One of our aims here is to show that they all
are closely related. For instance, the convolution and superposition
procedures are the results of transfinite applications of the convolution and
superposition operations. Supremum is reached by the convolution procedure;
respectively, infimum is reached by the superposition procedure. We also study
conditions which imply that a class of algebras is radical or semisimple.

We consider the problem of heredity of a radical that is obtained by closure
and convolution procedures. A radical $P$ is called \textit{hereditary} if
$P(J)=J\cap P(A)$ for any appropriate ideal $J$ of an algebra. This property
is very convenient for the study and use of a radical. Since many (almost all)
most important radicals are constructed by means of the closure and
convolution, the heredity problem is one of the main inner problems of the
theory. We give a criterion of heredity for the resulting radical which works
for almost all known examples. In particular this approach allows us to answer
in affirmative the question of Dixon \cite{D97} about the heredity of the
closed-Baer radical $\mathcal{P}_{\beta}$. This radical is the smallest one
among topological radicals on Banach algebras for which all algebras with
trivial multiplication are radical.

In Section 5 we deal with the radical approach \ to the property of
commutativity modulo a radical. Namely, starting with a radical $P$ we define
an under radical $P^{a}$, such that $P^{a}$-radical algebras are precisely
those algebras that are commutative modulo $P$:
\[
\lbrack A,A]\subset P(A).
\]
We find the conditions on $P$ under which $P^{a}$ is a radical, and check that
they are fulfilled for our main examples. For instance, this is true for all
algebraic hereditary radicals and their topological analogs obtained by the
closure and convolution procedures. In particular, $\operatorname{rad}^{a}$,
$\mathcal{R}_{s}^{a}$, $\mathcal{P}_{\beta}^{a}$, $\mathcal{R}_{\mathrm{hf}%
}^{a}$, $\mathcal{R}_{\mathrm{cq}}^{a}$ and $\mathcal{R}_{t}^{a}$ are
radicals. While the first four radicals satisfy our criterion, the last two
need a separate consideration\textbf{ } (the result for $\mathcal{R}%
_{\mathrm{cq}}^{a}$ was proved in \cite{TR3}). This underlines the advantage
of the joint consideration of algebraic and topological radicals. As the
commutativity modulo $\operatorname{rad}$, $\mathcal{R}_{\mathrm{cq}}$ and
$\mathcal{R}_{t}$ is often used in the theory of spectral radii of different
types, some important applications to spectral theory are indicated at the end
of the section and also in Section 9. In particular, we consider the
sufficient conditions for a Banach algebra to be Engel.

The subject of Section 6 is the very popular in the theory of algebras and
Banach algebras notion of the socle of an algebra. We consider the ideal map
$\operatorname{soc}:A\longmapsto\mathrm{socle}(A)$ and define a new procedure
that transforms each radical $P$ into the convolution $\operatorname{soc}\ast
P$ (\textquotedblleft the socle modulo $P$\textquotedblright). By means of
this procedure we establish some relations between, for example, hypofinite
radical and Baer radical. Main applications of this construction will be given
in Section 8 where we introduce the scattered radical.

Our object in Section 7 is more spectral: we study procedures related to the
space $\operatorname{Prim}(A)$ of all primitive ideals of an algebra $A$. All
of them begin with a choice of a subset $\Omega(A)\subset\operatorname{Prim}%
(A)$, but such a choice (a \textit{primitive map}) must be done simultaneously
for all algebras $A$ in the given class, and subjected to some natural
restrictions. For example one can choose for $\Omega(A)$ all the space
$\operatorname{Prim}(A)$, or $\varnothing$, or the set of primitive ideals of
finite codimension. If a primitive map is fixed then the procedures we deal
with are

\begin{itemize}
\item The $\Omega$-\textit{hull-kernel closure} $P^{\mathrm{kh}_{\Omega}}$ of
a preradical $P$ which sends $A$ to the intersection of those ideals
$I\in\Omega(A)$ that contain $P(A)$;

\item The $\Omega$-\textit{primitivity extension }$P^{p_{\Omega}}$ of $P$ that
sends $A$ to the ideal of all elements $a\in A$ such that $a/I\in P(A/I)$, for
all $I\in\Omega(A)$.
\end{itemize}

\noindent\ Our main aims here are to study the interrelations between these
procedures and to find the conditions on $P$ under which the produced maps
have sufficiently convenient properties. Note that the primitivity extension
of the scattered radical take a central place in applications to the spectral continuity.

Section 8 which can be considered central in this paper, is devoted to the
theory of the scattered radical $\mathcal{R}_{s}$. Almost all constructions of
the previous sections find a reflection in or application to the theory of
$\mathcal{R}_{s}$. We introduce and study $\mathcal{R}_{s}$ in its algebraic
and Banach algebraic versions. The algebraic version of $\mathcal{R}_{s}$ is
defined via the socle and convolution procedure: it equals
$(\operatorname{soc}{\ast}\operatorname{rad})^{\ast}$. We show, using the
connection between both versions, that $\mathcal{R}_{s}$ is a hereditary
topological radical satisfying some additional flexibility\ conditions which
allow one to use its centralization and primitivity extensions.

Another important fact is that the structure space of an $\mathcal{R}_{s}%
$-radical algebra is dispersed and that (the classes of the equivalence of)
strictly irreducible representations of such algebras are uniquely determined
by their kernels (= primitive ideals). We show that in hereditarily semisimple
Banach algebras (in particular, in C*-algebras) $\mathcal{R}_{s}$ coincides
with the hypocompact radical and that scattered C*-algebras can be
characterized by many other equivalent conditions.

The subject of Section 9 is the continuity of spectral characteristics of an
element or a family of elements of a Banach algebra. We show that the map
$a\longmapsto\sigma(a)$ is continuous at points $a\in\mathcal{R}_{s}^{p}(A)$
where $\mathcal{R}_{s}^{p}$ is the primitivity extension of the scattered
radical, and that this property is not fulfilled for elements of
$\mathcal{R}_{s}^{p\ast}(A)$ and of $\mathcal{R}_{s}^{a}$. Among other results
we obtain that if $A$ is a C*-algebra then $\mathcal{R}_{s}^{p\ast}(A)$ is the
largest GCR-ideal of $A$ (so the radical $\mathcal{R}_{s}^{p\ast}$\textit{
}extends the GCR-radical from C*-algebras to Banach algebras).

We construct topological radicals $\mathcal{R}_{\overrightarrow{\rho}}$,
$\mathcal{R}_{\overrightarrow{\rho_{j}}}$ and $\mathcal{R}%
_{\overrightarrow{\rho_{t}}}$ that have the properties that elements, compact
subsets and summable families in $\mathcal{R}_{\overrightarrow{\rho}}(A)$,
$\mathcal{R}_{\overrightarrow{\rho_{j}}}(A)$ and $\mathcal{R}%
_{\overrightarrow{\rho_{t}}}(A)$, respectively, are the continuity points for
spectral radius, joint spectral radius and tensor spectral radius,
respectively. It is shown that $\mathcal{R}_{\overrightarrow{\rho}}%
\geq\mathcal{R}_{s}^{p\ast}$. Thus the spectral radius, unlike the spectrum,
is continuous at all elements of $\mathcal{R}_{s}^{p\ast}(A)$. In particular,
$\rho$ is continuous on any GCR C*-algebra.

In Section 10 we apply some of the listed results to the problem of recovering
the spectral characteristics of an element or a family of elements of an
algebra $A$ from the corresponding information about its images in some
quotients $A/I$, where $I$ belong to a fixed family $\mathcal{F}$ of ideals in
$A$ with trivial intersection. Our main interest is in the joint spectral
radius $\rho(M)$ of a precompact set $M\subset A$ --- a characteristic that
attracts much interest not only in operator theory but in such branches of
mathematics as the topological dynamics and fractal theory. We show that
\begin{equation}
\rho(M)=\sup_{I\in\mathcal{F}}\rho(M/I) \label{supId}%
\end{equation}
if $\mathcal{F}$ is finite, and extend this equality to the case of infinite
family $\mathcal{F}$ for algebras that have non-zero compact elements. An
especially interesting case is $\mathcal{F}=\operatorname{Prim}(A)$. We relate
(\ref{supId}) to the equality $\rho(M)=r(M)$, where $r(M)$ is a spectral
characteristic introduced by Berger and Wang \cite{BW92}; in the algebras
where this equality holds for all precompact subsets (the \textit{Berger-Wang
algebras}) the analysis of many spectral problem becomes much easier.

We show that in the $\mathcal{R}_{s}^{p\ast}$-radical Berger-Wang algebras the
joint spectral radius is continuous and find some criteria for (\ref{supId})
to hold. Then we study the case of C*-algebras and show that each GCR-algebra
is a Berger-Wang algebra and therefore (if one takes into account that
GCR-algebras are $\mathcal{R}_{s}^{p\ast}$-radical) the joint spectral radius
is continuous on GCR-algebras.

It is natural that simultaneous consideration of various kinds of radicals --
the algebraic ones, the radicals on Banach algebras, normed algebras,
$Q$-algebras, C*-algebras and so on -- can be difficult for the first
acquaintance with the topic. However, apart from the reasons above, it gives
the possibility not to repeat similar arguments many times. Anyway for the
first reading it seems to be reasonable to restrict oneself to a concrete type
of radicals and consider only Banach or C*-algebras. We can say that we did
our best to simplify the text and, in particular, following the advice of
Laurence Sterne, generously lighted up the dark places by stars.

\bigskip

\textbf{Acknowledgment.}\textit{ The authors are grateful to Matej Bre\v{s}ar
and Peng Cao for useful discussions and information, and to Edward Kissin for
his encouraging interest in this work. They express special thanks to the
referee for numerous useful remarks and suggestions.}

\subsection{Preliminaries}

\subsubsection{Spaces}

Let $X,U$ be linear spaces, and let $Y$ and $Z$ be subspaces of $X$. If
$Z\subset Y$ then $\left(  X/Z\right)  /\left(  Y/Z\right)  \cong X/Y$. If
$(Y_{\alpha})_{\alpha\in\Lambda}$ is a family of subspaces of $X$ and
$Z\subset\cap_{\alpha\in\Lambda}Y_{\alpha}$ then
\begin{equation}
(\cap_{\alpha\in\Lambda}Y_{\alpha})/Z=\cap_{\alpha\in\Lambda}(Y_{\alpha}/Z).
\label{pr}%
\end{equation}
If $X$ is normed, let $\overline{Y}$, or more exactly $\overline{Y}^{\left(
X\right)  }$, denote the closure of $Y$ in $X$, and $\widehat{X}$ denote the
completion of $X$. If $Z\subset Y$ is a closed subspace of $Y$ then $Y/Z\cong
q\left(  Y\right)  $, where $q:X\longrightarrow X/\overline{Z}^{\left(
X\right)  }$ is the standard quotient map $x\longmapsto$ the coset of $x$.

Let $f:X\longrightarrow U$ be a linear map. If $M\subset U$ then
\[
f^{-1}\left(  M\right)  :=\left\{  x\in X:f\left(  x\right)  \in M\right\}  .
\]
Clearly $N\subset f^{-1}\left(  f\left(  N\right)  \right)  $ for every
$N\subset X$; $N=f^{-1}\left(  f\left(  N\right)  \right)  $ $\Leftrightarrow$
$N=f^{-1}\left(  M\right)  $ for some $M\subset U$; in particular this holds
if $f$ is invertible. If $X,U$ are normed, and $f$ is open and continuous,
then%
\begin{equation}
f^{-1}(\overline{M})=\overline{f^{-1}(M)}. \label{gen0}%
\end{equation}

If $X$ is a linear space then $L\left(  X\right)  $ is the algebra of all
linear operators on $X$, $\mathcal{F}\left(  X\right)  $ the ideal of finite
rank operators. If $X$ is normed then $\mathcal{B}\left(  X\right)  $ is the
algebra of all bounded operators, $\mathcal{K}\left(  X\right)  $ is the ideal
of compact operators.

\subsubsection{Algebras}

Let $A$ be an associative complex algebra; then $A^{1}$ equals $A$ if $A$ is
unital, and is the algebra obtained from $A$ by adjoining the identity element
otherwise. In what follows, an \textit{ideal} of an algebra means a two-sided
ideal. If $I$ is an ideal of $A$ then $q_{I}$ denotes the standard quotient
map $A\longrightarrow A/I$ by default; we also write $a/I$ instead of
$q_{I}\left(  a\right)  $ for any $a\in A$. Respectively, if $M=\left(
a_{\alpha}\right)  _{\Lambda}$ is a family in $A$ then $M/I$ denotes the
family $\left(  a_{\alpha}/I\right)  _{\Lambda}$ in $A/I$. Define operators
\textrm{L}$_{a}$, $\mathrm{R}_{a}$ and \textrm{W}$_{a}$ on $A$\textrm{ by }%
\[
\mathrm{L}_{a}x=ax\text{, }\mathrm{R}_{a}x=xa\text{ and }\mathrm{W}%
_{a}=\mathrm{L}_{a}\mathrm{R}_{a};
\]
again $\mathrm{L}_{M}$ is a family $\left(  \mathrm{L}_{a_{\alpha}}\right)
_{\Lambda}$ of operators. Such rules act for sets by default.

\subsubsection{Representations\label{algebras}}

A representation $\pi$ of an algebra $A$ on a linear space $X$ is called
\textit{strictly irreducible} if $\pi(A)\xi:=\{\pi(a)\xi:a\in A\}=X$, for
every $\xi\neq0$. Representations $\pi$ and $\tau$, acting on the spaces $X$
and $Y$, respectively, are \textit{equivalent }(write $\pi$ $\sim$ $\tau$) if
there is a linear bijective operator $T:X\rightarrow Y$ satisfying the
condition $T\pi(a)=\tau(a)T$ for all $a\in A$. The direct sum $\oplus
_{i=1}^{n}\pi_{i}$ of representations $\pi_{1},...,\pi_{n}$, acting on the
spaces $X_{1},...,X_{n}$, is the representation $\pi$ on $X=\oplus_{i=1}%
^{n}X_{i}$ defined by the formula
\[
\pi(a)(\oplus_{i=1}^{n}\xi_{i})=\oplus_{i=1}^{n}\pi_{i}(a)\xi_{i}.
\]
It is well known (see for example \cite[Chapter 17]{Lang}) that if strictly
irreducible representations $\pi_{1},...,\pi_{n}$ are pairwise nonequivalent
and $0\neq\eta_{i}\in X_{i}$, $1\leq i\leq n$, then the vector $\eta
=\oplus_{i=1}^{n}\eta_{i}$ is cyclic for the representation $\pi$: $\pi
(A)\eta=X$.

\begin{lemma}
\label{reprfin} Let $\{\pi_{i}:1\leq i\leq n\}$ be mutually inequivalent
strictly irreducible representations. Then $\sum_{i=1}^{n}\mathrm{rank}({\pi
}_{i}(a))\leq\dim aAa$ for any $a\in A$.
\end{lemma}

\begin{proof}
Fix $a\in A$. We may suppose that $\dim aAa <\infty$ and that $\pi_{i}%
(a)\neq0$ for all $i$.

Let $\pi=\oplus_{i=1}^{n}\pi_{i}$ and $X=\oplus_{i=1}^{n}X_{i}$. Clearly
\[
\dim\pi(aAa)\xi\leq\dim aAa
\]
for each $\xi\in X$. Let us choose vectors $\xi_{i}\in X_{i}$ such that
$\eta_{i}:=\pi_{i}(a)\xi_{i}\neq0$, $1\leq i\leq n$. Let $\xi=\oplus_{i=1}%
^{n}\xi_{i}$; then the vector $\eta=\pi(a)\xi=\oplus_{i=1}^{n}\pi_{i}%
(a)\xi_{i}$ is cyclic for $\pi(A)$. Therefore $\pi(aAa)\xi=\pi(a)\pi
(A)\eta=\pi(a)X=\oplus_{i=1}^{n}\pi_{i}(a)X_{i}$ whence
\[
\dim\pi(aAa)\xi=\sum\dim\pi_{i}(a_{i})X_{i}=\sum_{i=1}^{n}\mathrm{rank}%
(\pi_{i}(a)).
\]

\end{proof}

Let $\operatorname{Irr}\left(  A\right)  $ be the set of all strictly
irreducible representations of $A$, $\operatorname{Prim}\left(  A\right)  $
the set of all primitive ideals, i.e., the kernels of all $\pi\in
\operatorname{Irr}\left(  A\right)  $. For any subset $E$ of $A$, let
$\mathrm{h}(E;A)$ be the set of all $I\in\operatorname{Prim}\left(  A\right)
$ with $E\subset I$ and $\mathrm{kh}\left(  E;A\right)  =\cap_{I\in
\mathrm{h}\left(  E;A\right)  }I$.

If $\pi\in\operatorname{Irr}\left(  A\right)  $ then let $X_{\pi}$ be the
space on which $\pi$ acts (the representation space). It is well known that if
$A$ is not unital then
\begin{align}
\operatorname{Irr}\left(  A^{1}\right)   &  =\left\{  \pi_{\mathrm{triv}%
}:A^{1}\longrightarrow A^{1}/A\cong\mathbb{C}\right\}  \cup\left\{  \pi
^{1}:\pi\in\operatorname{Irr}\left(  A\right)  \right\}  ,\label{p2}\\
\operatorname{Prim}\left(  A^{1}\right)   &  =\{A\}\cup\operatorname{Prim}%
\left(  A\right)  \text{ and }\operatorname{rad}\left(  A^{1}\right)
=\operatorname{rad}\left(  A\right)  \label{p1}%
\end{align}
where every $\pi\in\operatorname{Irr}\left(  A\right)  $ extends to
$\pi^{1}\in\operatorname{Irr}\left(  A^{1}\right)  $ by setting $\pi
^{1}\left(  1_{A^{1}}\right)  =1_{X_{\pi}}$.

If $J$ is a proper ideal of $A$ then
\begin{align}
\operatorname{Irr}\left(  J\right)   &  =\left\{  \pi|_{J}:\pi\in
\operatorname{Irr}\left(  A\right)  ,\pi\left(  J\right)  \neq0\right\}
,\label{p3}\\
\operatorname{Prim}\left(  J\right)   &  =\left\{  I\cap J:I\in
\operatorname{Prim}\left(  A\right)  ,I\cap J\neq J\right\}  . \label{p4}%
\end{align}
Every $\pi\in\operatorname{Irr}\left(  J\right)  $ can be uniquely extended to
a representation $\widetilde{\pi}\in\operatorname{Irr}\left(  A\right)  $. To
construct the extension it suffices to choose a non-zero vector $\zeta\in
X_{\pi}$ and, for each $a\in A$, to set $\widetilde{\pi}\left(  a\right)
\xi=\pi\left(  ab\right)  \zeta$ for every $b\in J$ such that $\xi=\pi\left(
b\right)  \zeta$.

Every $\pi\in\operatorname{Irr}\left(  A/J\right)  $ induces $\pi_{A}%
\in\operatorname{Irr}\left(  A\right)  $ by setting $\pi_{A}\left(  a\right)
=\pi\left(  a/J\right)  $ for every $a\in A$. Every $\pi\in\operatorname{Irr}%
\left(  A\right)  $ with $\pi\left(  J\right)  =0$ induces $\pi_{q_{J}}%
\in\operatorname{Irr}\left(  A/J\right)  $ by setting $\pi_{q_{J}}\left(
q_{J}\left(  a\right)  \right)  =\pi\left(  a\right)  $ for every $a\in A$.
Hence
\begin{equation}
\operatorname{Prim}\left(  A/J\right)  =\left\{  I/J:I\in\operatorname{Prim}%
\left(  A\right)  ,J\subset I\right\}  . \label{p5}%
\end{equation}
Note that $A/J$ can be unital for an non-unital algebra $A$.

If $f:A\longrightarrow B$ is a homomorphism onto then $\pi\in
\operatorname{Irr}\left(  B\right)  $ induces $\pi^{f}\in\operatorname{Irr}%
\left(  A\right)  $ by setting $\pi^{f}\left(  a\right)  =\pi\left(  f\left(
a\right)  \right)  $ for every $a\in A$. So
\begin{equation}
\left\{  f^{-1}\left(  I\right)  :I\in\operatorname{Prim}\left(  B\right)
\right\}  \subset\operatorname{Prim}\left(  A\right)  . \label{p6}%
\end{equation}
Let $A$ be a normed algebra, let $\operatorname{Irr}_{\mathrm{n}}\left(
A\right)  $ (respectively, $\operatorname{Irr}_{\mathrm{b}}\left(  A\right)
$) be the set of all continuous strictly irreducible representations of $A$ by
bounded operators on a normed (respectively, Banach) space, and let
$\operatorname{Prim}_{\mathrm{n}}\left(  A\right)  =\left\{  \ker\pi:\pi
\in\operatorname{Irr}_{\mathrm{n}}\left(  A\right)  \right\}  $ and
$\operatorname{Prim}_{\mathrm{b}}\left(  A\right)  =\left\{  \ker\pi:\pi
\in\operatorname{Irr}_{\mathrm{b}}\left(  A\right)  \right\}  $.

It should be noted that for a C*-algebra, a primitive ideal is defined as the
kernel of an irreducible *-representation on a Hilbert space. The fact that
this definition is equivalent to the general one, is proved in \cite[Corollary
2.9.6]{Dixmier}.

\subsubsection{ $Q$-algebras}

A normed algebra $A$ is called a $Q$\textit{-algebra} if the set of all
invertible elements of $A^{1}$ is open. A normed algebra $A$ is a $Q$-algebra
$\Leftrightarrow$ $\sum_{1}^{\infty}a^{n}$ converges for any $a\in A$ with
$\left\Vert a\right\Vert <1$ $\Leftrightarrow$ every strictly irreducible
representation of $A$ is equivalent to a continuous representation by bounded
operators on a normed space \cite[Theorem 2.1]{TR1}. A normed algebra $A$ is
called a $Q_{\mathrm{b}}$\textit{-algebra }if every strictly irreducible
representation of $A$ is equivalent to a continuous representation by bounded
operators on a Banach space.

Let $A$ be an algebra and $a\in A$. The \textit{spectrum} $\sigma\left(
a\right)  $, or more exactly $\sigma_{A}\left(  a\right)  $, is the set of all
$\lambda\in\mathbb{C}$ such that $a-\lambda$ is not invertible in $A^{1}$;
this is related to the definition in \cite{BD73}. So $\sigma_{A}\left(
a\right)  $ and $\sigma_{A^{1}}\left(  a\right)  $ determine the same set, but
sometimes we prefer to write $\sigma_{A^{1}}\left(  a\right)  $ for
exactness.\textit{ }

Let $B$ is a subalgebra of $A$; it is called a \textit{spectral subalgebra} of
$A$ if $\sigma_{B}\left(  a\right)  \backslash\left\{  0\right\}  =$
$\sigma_{A}\left(  a\right)  \backslash\left\{  0\right\}  $ for every $a\in
B$. Every ideal of $A$ is a spectral subalgebra of $A$; a normed algebra $A$
is a $Q$-algebra $\Leftrightarrow$ $A$ is a spectral subalgebra of
$\widehat{A}$, see \cite[Theorem 4.2.10]{P94} and \cite[Lemma 20.9]{KS97}.

\subsubsection{The joint spectrum\label{spectrum}}

For the joint spectral theory in Banach algebras we refer to \cite{H88, M07}.
If $M=\left(  a_{\alpha}\right)  _{\Lambda}\subset A$ is a family in $A$ then
the \textit{left spectrum} $\sigma^{l}\left(  M\right)  $ is the set of all
families $\lambda=\left(  \lambda_{\alpha}\right)  _{\Lambda}\subset
\mathbb{C}$ such that the family $M-\lambda:=\left(  a_{\alpha}-\lambda
_{\alpha}\right)  _{\Lambda}$ generates the proper left ideal of $A^{1}$. The
\textit{right spectrum} $\sigma^{r}\left(  M\right)  $ is defined similarly by
replacing `left' by `right'; $\sigma\left(  M\right)  =\sigma^{l}\left(
M\right)  \cup\sigma^{r}\left(  M\right)  $ is called the \textit{Harte
spectrum, }or simply the \textit{spectrum}. We write $\sigma^{l}\left(
a\right)  $ and $\sigma^{r}\left(  a\right)  $ if $M=\left\{  a\right\}  $.
Let $A$ be unital. A subalgebra $B$ of $A$ is called \textit{unital} if it
contains the identity element of $A$, and \textit{inverse-closed }if $B$
contains $x^{-1}$ for every $x\in B$ which is invertible in $A$. If $B$ is
inverse-closed then $\sigma_{B}\left(  a\right)  =\sigma_{A}\left(  a\right)
$ for every $a\in B$.

The following folklore lemma determines the operational possibilities of the
joint spectra.

\begin{lemma}
\label{sp0}Let $A$ be an algebra and $M=\left(  a_{\alpha}\right)  _{\Lambda}$
be a family in $A$. Then

\begin{enumerate}
\item $\lambda\notin\sigma_{A}^{l}\left(  M\right)  $ if and only if
$\lambda|_{N}\notin\sigma_{A}^{l}\left(  N\right)  $ for some finite subfamily
$N\subset M$;

\item If $B$ is a unital subalgebra of $A^{1}$ and $M\subset B$ then
$\sigma_{A}^{l}\left(  M\right)  \subset\sigma_{B}^{l}\left(  M\right)  $;

\item If $f:A\longrightarrow C$ is a surjective homomorphism then $\sigma
_{C}^{l}\left(  f\left(  M\right)  \right)  \subset\sigma_{A}^{l}\left(
M\right)  $;

\item $\sigma^{l}\left(  M\right)  =\sigma^{l}\left(  M/\operatorname{rad}%
\left(  A\right)  \right)  $.
\end{enumerate}
\end{lemma}

\begin{proof}
$\left(  1\right)  $, $\left(  2\right)  $ and $\left(  4\right)  $ are
standard and trivial.

$\left(  3\right)  $ Consider only the case when $C$ is unital and $A$ is not.
One can extend $f$ up to $f^{\prime}:A^{1}\longrightarrow C$ by setting
$f^{\prime}\left(  1_{A}\right)  =1_{C}$, whence $\sigma_{C}^{l}\left(
f^{\prime}\left(  M\right)  \right)  \subset\sigma_{A}^{l}\left(  M\right)  $,
but $f^{\prime}\left(  M\right)  =f\left(  M\right)  $.
\end{proof}

If $N=\left(  T_{\beta}\right)  $ is a family in $L\left(  X\right)  $ then
$\sigma_{X}^{p}\left(  N\right)  $, or simply $\sigma^{p}\left(  N\right)  $,
is a \textit{point spectrum} of $N$, i.e., the set of all $\lambda=\left(
\lambda_{\beta}\right)  \subset\mathbb{C}$ such that there is a non-zero
vector $\zeta\in X$ and $T_{\beta}\zeta=\lambda_{\beta}\zeta$; $\lambda$ is
called an \textit{eigenvalue} of $N$ corresponding to \textit{eigenvector}
$\zeta$. If $X$ is normed then it is useful to use an \textit{approximate
point spectrum} $\sigma_{X}^{a}\left(  N\right)  $, or simply $\sigma
^{a}\left(  N\right)  $; it is the set of all $\lambda=\left(  \lambda_{\beta
}\right)  \subset\mathbb{C}$ such that there is a net $\zeta=\left(
\zeta_{\gamma}\right)  \subset X$ with $\lim_{\gamma}\left\Vert \zeta_{\gamma
}\right\Vert >0$ and $\lim_{\gamma}\left\Vert \left(  T_{\beta}-\lambda
_{\beta}\right)  \zeta_{\gamma}\right\Vert =0$ for every $\beta$; $\lambda$ is
called an \textit{approximate} \textit{eigenvalue} of $N$ corresponding to
\textit{approximate} \textit{eigenvector} $\zeta$. If $A$ is a normed algebra
and $M$ is a family in $A$ then set $\sigma_{A}^{a}\left(  M\right)
=\sigma_{A}^{a}\left(  \mathrm{L}_{M}\right)  $.

If $T$ $\in L\left(  X\right)  $ then put $\sigma_{X}\left(  T\right)
:=\sigma_{L\left(  X\right)  }\left(  T\right)  $. If $T$ is one-to-one and
onto then $T$ is invertible in $L\left(  X\right)  $ whence $\sigma_{X}\left(
T\right)  =\sigma_{X}^{p}\left(  T\right)  \cup\sigma_{L\left(  X\right)
}^{r}\left(  T\right)  $. If $X$ is a Banach space, $T$ is bounded and the
kernel and the image of $T$ satisfy the mentioned conditions then $T$ is
invertible in $\mathcal{B}\left(  X\right)  $ by Banach's theorem, whence
\begin{equation}
\sigma_{\mathcal{B}\left(  X\right)  }\left(  T\right)  =\sigma_{X}^{p}\left(
T\right)  \cup\sigma_{\mathcal{B}\left(  X\right)  }^{r}\left(  T\right)
=\sigma_{X}^{p}\left(  T\right)  \cup\sigma_{L\left(  X\right)  }^{r}\left(
T\right)  =\sigma_{X}\left(  T\right)  \label{sf3}%
\end{equation}

\begin{theorem}
\label{sp1}Let $A$ be an algebra and $M$ be a family in $A$. Then

\begin{enumerate}
\item $\sigma_{A}^{l}\left(  M\right)  =\cup_{\pi\in\operatorname{Irr}\left(
A^{1}\right)  }\sigma_{X_{\pi}}^{p}\left(  \pi\left(  M\right)  \right)  $;

\item If $\pi$ is a representation of $A^{1}$ by bounded operators on a normed
space $X$ then $\sigma_{X}^{a}\left(  \pi\left(  M\right)  \right)
\subset\sigma_{A}^{l}\left(  M\right)  ;$

\item If $A$ is a $Q$-algebra then $\sigma_{A}^{l}\left(  M\right)  =\cup
_{\pi\in\operatorname{Irr}_{\mathrm{n}}\left(  A^{1}\right)  }\sigma_{X_{\pi}%
}^{a}\left(  \pi\left(  M\right)  \right)  $.
\end{enumerate}
\end{theorem}

\begin{proof}
$\left(  1\right)  $ For $\lambda\in\sigma_{A}^{l}\left(  M\right)  $, there
is a maximal left ideal $J$ of $A^{1}$ such that $M-\lambda=\left(  a_{\alpha
}-\lambda_{\alpha}\right)  \subset J\neq A^{1}$. The representation $\pi$ of
$A^{1}$ defined by $\pi\left(  b\right)  =\mathrm{L}_{b}|_{X}$ is strictly
irreducible where $X=A^{1}/J$. For $\xi=1/J\in X$ we obtain that $\pi\left(
a_{\alpha}\right)  \xi=\lambda_{\alpha}\xi$ for every $\alpha$. Therefore
$\lambda\in\sigma_{X}^{p}\left(  \pi\left(  M\right)  \right)  $. We proved
the inclusion $\subset$.

Conversely, let $\lambda\in\sigma_{X_{\pi}}^{p}\left(  \pi\left(  M\right)
\right)  $ for some $\pi\in\operatorname{Irr}\left(  A^{1}\right)  $. Then
$\pi\left(  a_{\alpha}-\lambda_{\alpha}\right)  \xi=0$ for some non-zero
vector $\xi\in X_{\pi}$ and each $\alpha$. But $\left\{  x\in A^{1}:\pi\left(
x\right)  \xi=0\right\}  $ is a proper left ideal of $A^{1}$. Therefore
$\lambda\in\sigma_{A}^{l}\left(  M\right)  $.

$\left(  2\right)  $ Let $0\in\sigma_{X_{\pi}}^{a}\left(  \pi\left(  M\right)
\right)  \backslash\sigma_{A}^{l}\left(  M\right)  $. By Lemma \ref{sp0},
there is a finite subfamily $N=\left(  a_{\alpha}\right)  _{\Lambda^{\prime}%
}\subset M$ with $0\notin\sigma_{A}^{l}\left(  N\right)  $. So there are
$b_{\alpha}\in A^{1}$ for $\alpha\in\Lambda^{\prime}$ such that $\sum
_{\Lambda^{\prime}}b_{\alpha}a_{\alpha}=1$. If $\zeta=\left(  \zeta_{\gamma
}\right)  \subset X$ is an approximate eigenvector related to $0$ then
\[
0<\lim_{\gamma}\left\Vert \zeta_{\gamma}\right\Vert \leq\lim_{\gamma}%
\sum_{\alpha\in\Lambda^{\prime}}\left\Vert \pi\left(  b_{\alpha}a_{\alpha
}\right)  \zeta_{\gamma}\right\Vert \leq\sum_{\alpha\in\Lambda^{\prime}%
}\left\Vert \pi\left(  b_{\alpha}\right)  \right\Vert \lim_{\gamma}\left\Vert
\pi\left(  a_{\alpha}\right)  \zeta_{\gamma}\right\Vert =0,
\]
a contradiction. Therefore $\sigma_{X_{\pi}}^{a}\left(  \pi\left(  M\right)
\right)  \subset\sigma_{A}^{l}\left(  M\right)  $.

$\left(  3\right)  $ follows from $\left(  1\right)  $ and $\left(  2\right)
$.
\end{proof}

One can obtain the related statements for $\sigma_{A}^{r}\left(  M\right)  $
if pass to the opposite algebra.

\subsubsection{Spectrum and primitive ideals}

If an algebra $A$ is normed then $\sigma\left(  a\right)  $ is not empty by
the Gelfand-Mazur theorem; if $A$ is a $Q$-algebra then $\sigma\left(
a\right)  $ is a compact subset of $\mathbb{C}$.

\begin{theorem}
\label{sp11}Let $A$ be an algebra and $a\in A$. Then

\begin{enumerate}
\item $\sigma_{A}\left(  a\right)  \backslash\sigma_{A}^{l}\left(  a\right)
\subset\cup_{\pi\in\operatorname{Irr}\left(  A^{1}\right)  }\sigma_{L\left(
X_{\pi}\right)  }^{r}\left(  \pi\left(  a\right)  \right)  $;

\item The following chain of equalities holds:%
\begin{align*}
\sigma_{A}\left(  a\right)   &  =\cup_{I\in\operatorname{Prim}\left(
A^{1}\right)  }\sigma_{A^{1}/I}\left(  a/I\right)  =\cup_{\pi\in
\operatorname{Irr}\left(  A^{1}\right)  }\sigma_{\pi\left(  A^{1}\right)
}\left(  \pi\left(  a\right)  \right)  =\cup_{\pi\in\operatorname{Irr}\left(
A^{1}\right)  }\sigma_{X_{\pi}}\left(  \pi\left(  a\right)  \right) \\
&  =\cup_{\pi\in\operatorname{Irr}\left(  A^{1}\right)  }\left(
\sigma_{X_{\pi}}^{p}\left(  \pi\left(  a\right)  \right)  \cup\sigma_{X_{\pi}%
}^{r}\left(  \pi\left(  a\right)  \right)  \right)  ;
\end{align*}

\item If $A$ is a $Q$-algebra then
\[
\sigma_{A}\left(  a\right)  =\cup_{\pi\in\operatorname{Irr}_{\mathrm{n}%
}\left(  A^{1}\right)  }\sigma_{\overline{\pi\left(  A\right)  }}\left(
\pi\left(  a\right)  \right)  =\cup_{\pi\in\operatorname{Irr}_{\mathrm{n}%
}\left(  A^{1}\right)  }\sigma_{\mathcal{B}\left(  X_{\pi}\right)  }\left(
\pi\left(  a\right)  \right)  ;
\]

\item If $A$ is a $Q_{\mathrm{b}}$-algebra then%
\[
\sigma_{A}\left(  a\right)  =\cup_{\pi\in\operatorname{Irr}_{\mathrm{b}%
}\left(  A^{1}\right)  }\left(  \sigma_{X_{\pi}}^{p}\left(  \pi\left(
a\right)  \right)  \cup\sigma_{\mathcal{B}\left(  X_{\pi}\right)  }^{r}\left(
\pi\left(  a\right)  \right)  \right)  .
\]

\end{enumerate}
\end{theorem}

\begin{proof}
$\left(  1\right)  $ Let $\lambda\in$ $\sigma_{A}\left(  a\right)
\backslash\sigma_{A}^{l}\left(  a\right)  $; then there is $b\in A^{1}$ such
that $b\left(  a-\lambda\right)  =1$. Assume, to the contrary, that
$\lambda\notin\sigma_{L\left(  X_{\pi}\right)  }^{r}\left(  \pi\left(
a\right)  \right)  $ for every $\pi\in\operatorname{Irr}\left(  A\right)  $:
there is an operator $T_{\pi}$ on $X_{\pi}$ such that $\left(  \pi\left(
a\right)  -\lambda\right)  T_{\pi}=1$. As $\pi\left(  b\right)  \left(
\pi\left(  a\right)  -\lambda\right)  =1$ then $T_{\pi}=\pi\left(  b\right)  $
for every $\pi\in\operatorname{Irr}\left(  A^{1}\right)  $. Hence $\left(
a-\lambda\right)  b-1\in\cap_{\pi\in\operatorname{Irr}\left(  A^{1}\right)
}\ker\pi=\operatorname{rad}\left(  A^{1}\right)  $ and $\left(  a-\lambda
\right)  b=1+c$ for some $c\in\operatorname{rad}\left(  A^{1}\right)  $. So
$\left(  a-\lambda\right)  b\left(  1+c\right)  ^{-1}=1$ and $\lambda
\notin\sigma_{A}\left(  a\right)  $, a contradiction.

$\left(  2\right)  $ By Lemma \ref{sp0}$\left(  2\text{-}3\right)  $,
\begin{align*}
\sigma_{A}\left(  a\right)   &  \supset\sigma_{A^{1}/I}\left(  a/I\right)
=\sigma_{\pi\left(  A^{1}\right)  }\left(  \pi\left(  a\right)  \right)
\supset\sigma_{L\left(  X_{\pi}\right)  }\left(  \pi\left(  a\right)  \right)
\\
&  =\sigma_{X_{\pi}}^{p}\left(  \pi\left(  a\right)  \right)  \cup
\sigma_{L\left(  X_{\pi}\right)  }^{r}\left(  \pi\left(  a\right)  \right)
\end{align*}
where $\pi\in\operatorname{Irr}\left(  A^{1}\right)  $ and $I=\ker\pi$. The
result follows from $\left(  1\right)  $ and Theorem \ref{sp1}$\left(
1\right)  $.

$\left(  3\right)  $ If $A$ is a $Q$-algebra then every $\pi\in
\operatorname{Irr}\left(  A^{1}\right)  $ is equivalent to a (bounded)
strictly irreducible representation by bounded operators on a normed space
$X_{\pi}$. As $\pi\left(  A\right)  \subset\overline{\pi\left(  A\right)
}\subset\mathcal{B}\left(  X_{\pi}\right)  \subset L\left(  X_{\pi}\right)  $
then
\[
\sigma_{L\left(  X_{\pi}\right)  }\left(  \pi\left(  a\right)  \right)
=\sigma_{\mathcal{B}\left(  X_{\pi}\right)  }\left(  \pi\left(  a\right)
\right)  \subset\sigma_{\overline{\pi\left(  A\right)  }}\left(  \pi\left(
a\right)  \right)  \subset\sigma_{\pi\left(  A\right)  }\left(  \pi\left(
a\right)  \right)
\]
for every $\pi\in\operatorname{Irr}\left(  A^{1}\right)  $, and the result
follows from $\left(  2\right)  $.

$\left(  4\right)  $ If $A$ is a $Q_{\mathrm{b}}$-algebra then every $\pi
\in\operatorname{Irr}\left(  A^{1}\right)  $ is equivalent to a (bounded)
strictly irreducible representation by bounded operators on a Banach space
$X_{\pi}$. Then the result follows from $\left(  3\right)  $ and (\ref{sf3}).
\end{proof}

The first equality in Theorem \ref{sp11}$\left(  2\right)  $ was proved in
\cite[Proposition 1]{Z80} (for Banach algebras).

\begin{corollary}
\label{sinq}Let $A$ be an algebra, and let $I$ be an ideal of $A$ and $a\in
A$. Then

\begin{enumerate}
\item If $\left(  A/I\right)  ^{1}\cong A^{1}/I$ then
\[
\sigma_{A/I}\left(  a/I\right)  =\sigma_{A^{1}/\mathrm{kh}\left(
I;A^{1}\right)  }\left(  a/\mathrm{kh}\left(  I;A^{1}\right)  \right)
=\cup_{J\in\mathrm{h}\left(  I;A^{1}\right)  }\sigma_{A^{1}/J}\left(
a/J\right)  ;
\]

\item If $\left(  A/I\right)  ^{1}\ncong A^{1}/I$ then
\[
\sigma_{A/I}\left(  a/I\right)  =\sigma_{A/\mathrm{kh}\left(  I;A\right)
}\left(  a/\mathrm{kh}\left(  I;A\right)  \right)  =\cup_{J\in\mathrm{h}%
\left(  I;A\right)  }\sigma_{A/J}\left(  a/J\right)  ;
\]

\item If $A$ is a $Q$-algebra then $\sigma_{A/I}\left(  a/I\right)
=\sigma_{A/\overline{I}}\left(  a/\overline{I}\right)  $.
\end{enumerate}
\end{corollary}

\begin{proof}
$\left(  1\right)  $ As $\mathrm{kh}\left(  I;A^{1}\right)  =q_{I}^{-1}\left(
\operatorname{rad}\left(  A^{1}/I\right)  \right)  $ then
\begin{equation}
A^{1}/\mathrm{kh}\left(  I;A^{1}\right)  \cong\left(  A^{1}/I\right)
/\operatorname{rad}\left(  A^{1}/I\right)  \cong\left(  \left(  A/I\right)
/\operatorname{rad}\left(  A/I\right)  \right)  ^{1} \label{a1}%
\end{equation}
The first equality follows from Lemma \ref{sp0}$\left(  4\right)  $. As
\[
\operatorname{Prim}\left(  A^{1}/I\right)  =\left\{  J/I:J\in\mathrm{h}\left(
I;A^{1}\right)  \right\}
\]
by $\left(  \ref{p5}\right)  $ and $\left(  A^{1}/I\right)  /\left(
J/I\right)  \cong A^{1}/J$ then, by Theorem \ref{sp11}$\left(  2\right)  ,$
\[
\sigma_{A/I}\left(  a/I\right)  =\cup_{J\in\mathrm{h}\left(  I;A^{1}\right)
}\sigma_{\left(  A^{1}/I\right)  /\left(  J/I\right)  }\left(  a/\left(
J/I\right)  \right)  =\cup_{J\in\mathrm{h}\left(  I;A^{1}\right)  }%
\sigma_{A^{1}/J}\left(  a/J\right)  .
\]

$\left(  2\right)  $ This is a case when $A/I$ is unital and $A$ is not
unital. We see as in $\left(  \ref{a1}\right)  $ that $A/\mathrm{kh}\left(
I;A\right)  =\left(  \left(  A/I\right)  /\operatorname{rad}\left(
A/I\right)  \right)  ^{1}$ is unital and the first equality follows.
Furthermore, $\operatorname{Prim}\left(  A/I\right)  =\left\{  J/I:J\in
\mathrm{h}\left(  I;A\right)  \right\}  $ by $\left(  \ref{p5}\right)  $. If
$J\in\mathrm{h}\left(  I;A\right)  $ then $A/J$ is also unital, whence as in
$\left(  1\right)  $ we have $\sigma_{A/I}\left(  a/I\right)  =\cup
_{J\in\mathrm{h}\left(  I;A\right)  }\sigma_{A/J}\left(  a/J\right)  .$

$\left(  3\right)  $ Consider only the case $\left(  A/I\right)  ^{1}\cong
A^{1}/I$ because the other case is similar. As primitive ideals are closed,
$\overline{I}\subset J$ for every $J\in\operatorname{Prim}\left(
A^{1}\right)  $ with $I\subset J$. Then $I\subset\overline{I}\subset
\mathrm{kh}\left(  I;A^{1}\right)  $ implies $\sigma_{A^{1}/\mathrm{kh}\left(
I;A^{1}\right)  }\left(  a/\mathrm{kh}\left(  I;A^{1}\right)  \right)
\subset\sigma_{A/\overline{I}}\left(  a/\overline{I}\right)  \subset
\sigma_{A/I}\left(  a/I\right)  $ and the result follows from $\left(
1\right)  $.
\end{proof}

\begin{remark}
As a consequence of Corollary $\ref{sinq}$, the spectrum of $a/I$ in the
quotient $A/I$ of a $Q$-algebra $A$ by a possibly unclosed ideal $I$ is a
compact set in $\mathbb{C}$.
\end{remark}

\subsubsection{Banach ideals\label{banach}}

Let $A$ be a normed algebra with norm $\left\Vert \cdot\right\Vert $; we write
also $\left(  A;\left\Vert \cdot\right\Vert \right)  $. An ideal $I$ of $A$ is
called \textit{normed }if there are a norm $\left\Vert \cdot\right\Vert _{I}$
on $I$ and $t>0$ such that
\[
\left\Vert x\right\Vert \leq t\left\Vert x\right\Vert _{I}%
\]
for every $x\in I$; the norm $\left\Vert \cdot\right\Vert _{I}$ is called
\textit{flexible }if $t=1$. A normed ideal $I$ with a flexible norm is also
called a flexible ideal; $I$ is called a Banach ideal if it is complete with
respect to $\left\Vert \cdot\right\Vert _{I}$.

An ideal of $A$ with the norm inherited from $A$ is of course flexible. We
will describe now a more interesting class of examples.

It is well known that the sum of two closed ideals $I$ and $J$ of a Banach
algebra $A$ may be non-closed, see an excellent discussion in \cite{D00}. One
can consider this sum as a Banach ideal with the norm
\[
\left\Vert a\right\Vert _{I+J}=\inf\left\{  \left\Vert x\right\Vert
+\left\Vert y\right\Vert :a=x+y\text{, }x\in J\text{, }y\in I\right\}
\]
for every $a\in J+I$. If $J\cap I=0$ then there is only one pair $x\in J$,
$y\in I$ with $a=x+y$ and $\left\Vert a\right\Vert _{I+J}=\left\Vert
x\right\Vert +\left\Vert y\right\Vert $. So we have in general for $x\in J$,
$y\in I$ that
\[
\left\Vert x+y\right\Vert _{I+J}=\inf\left\{  \left\Vert x-z\right\Vert
+\left\Vert y+z\right\Vert :z\in J\cap I\right\}  .
\]
It is clear that if $a$ lies in $I$ or $J$ then $\left\Vert a\right\Vert
=\left\Vert a\right\Vert _{I+J}$. Therefore $J\cap I$ is a closed ideal of
$K:=\left(  J+I;\left\Vert \cdot\right\Vert _{I+J}\right)  $. Then $K/\left(
J\cap I\right)  $ is a Banach algebra with the norm
\[
\left\Vert a/\left(  J\cap I\right)  \right\Vert _{K/\left(  J\cap I\right)
}=\left\Vert x/\left(  J\cap I\right)  \right\Vert +\left\Vert y/\left(  J\cap
I\right)  \right\Vert
\]
for $a/\left(  J\cap I\right)  =x/\left(  J\cap I\right)  +y/\left(  J\cap
I\right)  $ with $x\in J$, $y\in I$. Also, $I/\left(  J\cap I\right)  $ is a
closed ideal of $K/\left(  J\cap I\right)  $, whence clearly
\begin{align}
\left\Vert a/I\right\Vert _{K/I}  &  =\left\Vert \left(  a/\left(  J\cap
I\right)  \right)  /\left(  I/\left(  J\cap I\right)  \right)  \right\Vert
_{\left(  K/\left(  J\cap I\right)  \right)  /\left(  I/\left(  J\cap
I\right)  \right)  }\nonumber\\
&  =\inf\left\{  \left\Vert x/\left(  J\cap I\right)  \right\Vert +\left\Vert
\left(  y-z\right)  /\left(  J\cap I\right)  \right\Vert :z\in I\right\}
\nonumber\\
&  =\left\Vert x/\left(  J\cap I\right)  \right\Vert . \label{ai}%
\end{align}
On the other hand, as $I$ is a closed ideal of $K$ then%
\begin{align*}
\left\Vert a/I\right\Vert _{K/I}  &  =\inf\left\{  \left\Vert a-z\right\Vert
_{I+J}:z\in I\right\} \\
&  =\inf\left\{  \left\Vert x\right\Vert +\left\Vert y\right\Vert
:a=x+y+z\text{, }x\in J\text{, }y,z\in I\right\}  .
\end{align*}
So we obtain the following

\begin{lemma}
\label{bi}Let $(A,\|\cdot\|)$ be a normed algebra, and let $I,J$ be closed
ideals of $A$. Then

\begin{enumerate}
\item $J/\left(  J\cap I\right)  $ is isometrically isomorphic to $\left(
J+I\right)  /I$ with the norm
\[
\left\Vert a/I\right\Vert _{\left(  J+I\right)  /I}=\inf\left\{  \left\Vert
x\right\Vert +\left\Vert y\right\Vert :a=x+y+z\text{, }x\in J\text{, }y,z\in
I\right\}  .
\]

\item $\left(  J+I\right)  /I$ with norm $\left\Vert \cdot\right\Vert
_{\left(  J+I\right)  /I}$ is a flexible ideal of $A/I$.
\end{enumerate}
\end{lemma}

\begin{proof}
$\left(  1\right)  $ There is a bounded isomorphism $\phi$ from $J/\left(
J\cap I\right)  $ onto $\left(  J+I\right)  /I$, so for every $x/\left(  J\cap
I\right)  $ with $x\in J$ there is only one $a/I$ with $a\in J+I$. The above
argument clearly works if $A$ is a normed algebra, and $\left(  \ref{ai}%
\right)  $ shows that $\phi$ is an isometry.

$\left(  2\right)  $ is obvious.
\end{proof}

\begin{remark}
It follows from the lemma that if $A$ is a Banach algebra then the ideal
$\left(  J+I\right)  /I$ with norm $\left\Vert \cdot\right\Vert _{\left(
J+I\right)  /I}$ is a Banach ideal of $A/I$. Indeed, $J/\left(  J\cap
I\right)  $ is complete whence $\left(  J+I\right)  /I$ with norm $\left\Vert
\cdot\right\Vert _{\left(  J+I\right)  /I}$ is complete.
\end{remark}

\section{Radicals and other ideal maps}

\subsection{Definitions\label{definitions}}

Radicals are defined on classes of algebras satisfying some natural conditions
(see \cite{TR1}). For our aims it is sufficient to assume as a rule that each
of these classes, $\mathfrak{U}$, is either the class $\mathfrak{U}%
_{\mathrm{a}}$ of all associative complex algebras for the algebraic case or
one of the classes $\mathfrak{U}_{\mathrm{b}}$ and $\mathfrak{U}_{\mathrm{n}}$
of all Banach and all normed associative complex algebras, respectively, for
the topological case. We consider also radicals defined on subclasses of these
classes, for example on the subclass of $\mathfrak{U}_{\mathrm{n}}$ consisting
of all $Q$-algebras and the subclass of $\mathfrak{U}_{\mathrm{b}}$ consisting
of all C*-algebras.

We use the term \textit{morphism} for a surjective homomorphism in the
algebraic case and a open continuous surjective homomorphism in the
topological case. In what follows, unless necessity for clarity we omit words
`algebraic' or `topological' for notions of the radical theory in statements
valid in both the algebraic and topological contexts. Sometimes we add
necessary topological or algebraic specifications in square brackets. For
instance, the expression \textquotedblleft a [closed] ideal\textquotedblright%
\ means that the ideal in the proposition must be closed when we consider the
proposition in the topological context.

A map $P$ on $\mathfrak{U}$ is called a [\textit{closed}] \textit{ideal map}
if $P\left(  A\right)  $ is a [closed] ideal of $A$ for each [normed]
$A\in\mathfrak{U}$. An \textit{algebraic }(respectively, \textit{topological})
\textit{radical }is a (respectively, closed) ideal map $P$ that satisfies the
following axioms:

\begin{axiom}
\label{1}$f\left(  P\left(  A\right)  \right)  \subset P\left(  B\right)  $
for every morphism $f:A\longrightarrow B$ of algebras from $\mathfrak{U}$;
\end{axiom}

\begin{axiom}
\label{2}$P\left(  A/P\left(  A\right)  \right)  =0$;
\end{axiom}

\begin{axiom}
\label{3}$P\left(  P\left(  A\right)  \right)  =P\left(  A\right)  $;
\end{axiom}

\begin{axiom}
\label{4}$P\left(  I\right)  $ is an ideal of $A$ contained in $P\left(
A\right)  $, for each ideal $I\in\mathfrak{U}$ of $A$.
\end{axiom}

It is assumed in Axiom $4$ that if $A$ is normed then the norm on $I$ is
inherited from $A$. If a (closed) ideal map $P$ satisfies Axiom $1$ on
$\mathfrak{U}$ then $P$ is called a \textit{preradical} on $\mathfrak{U}$;
$P$ is called \textit{pliant }if Axiom $1$ on $\mathfrak{U}$ holds with
algebraic morphisms. A (non-necessarily closed) ideal map $P$ is a \textit{preradical with topological morphisms} if
it defined on a class of normed algebras and satisfies Axiom $1$ with topological
morphisms (this class of maps contains the
restrictions of arbitrary algebraic preradicals to classes of normed algebras). Algebraic preradicals are always pliant.

\begin{lemma}
\label{ap0}Let $P$ be a pliant preradical, and let $A$ be an algebra with two
norms $\left\Vert \cdot\right\Vert _{1}$ and $\left\Vert \cdot\right\Vert
_{2}$. Then $P\left(  A;\left\Vert \cdot\right\Vert _{1}\right)  =P\left(
A;\left\Vert \cdot\right\Vert _{2}\right)  $.
\end{lemma}

\begin{proof}
Indeed, the identity map $\left(  A;\left\Vert \cdot\right\Vert _{1}\right)
\longrightarrow\left(  A;\left\Vert \cdot\right\Vert _{2}\right)  $ is an
algebraic isomorphism.
\end{proof}

Sometimes the intermediate definitions are useful: a preradical $P$ on
$\mathfrak{U}\subset\mathfrak{U}_{n}$ is called \textit{strict} if
\begin{equation}
\overline{f\left(  P\left(  A\right)  \right)  }=P\left(  B\right)
\label{PRER}%
\end{equation}
for any continuous isomorphism $f:A\longrightarrow B$ of algebras from
$\mathfrak{U}$, and called \textit{strong }if $P$ satisfies Axiom $1$ with
respect to the continuous surjective homomorphisms of algebras from
$\mathfrak{U}$. It is clear that every strict preradical is strong.

\begin{remark}
For radicals on $\mathfrak{U}_{\mathrm{b}}$ or $\mathfrak{U}_{\mathrm{a}}$ our
definition coincides with Dixon's one, but for normed algebras the definitions differ:

\begin{enumerate}
\item We do not assume that ideals in Axiom $4$ are closed;

\item In \cite{D97} the morphisms are not necessarily open.
\end{enumerate}
\end{remark}

A preradical $P$ is called \textit{hereditary} if
\[
P\left(  I\right)  =I\cap P\left(  A\right)
\]
for every ideal $I$ of an algebra $A\in\mathfrak{U}$. The class of hereditary
radicals is especially important; the heredity of a preradical $P$ implies the
fulfillment of Axioms $3$ and $4$ for $P$. So one can define\textit{
hereditary radicals} as hereditary preradicals satisfying Axiom $2$.

In several cases we will impose on a preradical $P$, defined on $\mathfrak{U}%
_{\mathrm{b}}$, a more strong condition -- the \textit{condition of Banach
heredity}:
\begin{equation}
P\left(  L,\left\Vert \cdot\right\Vert _{L}\right)  =L\cap P\left(  A\right)
\text{ for any Banach ideal }\left(  L,\left\Vert \cdot\right\Vert
_{L}\right)  \subset A, \label{BanRad}%
\end{equation}
where $A$ is a Banach algebra. One can introduce similarly the
\textit{condition of flexible heredity}, for normed algebras.

An algebra $A\in\mathfrak{U}$ is called $P$\textit{-radical} if $A=P\left(
A\right)  $ and $P$\textit{-semisimple} if $P\left(  A\right)  =0$. Let
$\mathbf{Rad}\left(  P\right)  $ denote the class of all $P$-radical algebras
and $\mathbf{Sem}\left(  P\right)  $ the class of all $P$-semisimple algebras.
A preradical $P$ is called \textit{uniform} if every subalgebra $B\in
\mathfrak{U}$ of a $P$-radical algebra is $P$-radical. It is clear that
uniform radicals are hereditary.

By definition \cite{D97}, an \textit{under} \textit{radical} is a preradical
which satisfies all axioms besides, possibly, Axiom $2$. For instance, any
hereditary preradical is an under radical; moreover, restrictions of algebraic
radicals to appropriate classes of normed algebras are algebraic under
radicals in general. The dual notion for under radicals is a notion of an over
radical. By definition \cite{D97}, an \textit{over} \textit{radical} is a
preradical which satisfies all axioms besides, possibly, Axiom $3$.

In the class of preradicals on $\mathfrak{U}$ there is a natural order:
$P_{1}\leq P_{2}$ if $P_{1}\left(  A\right)  \subset P_{2}\left(  A\right)  $
for each $A\in\mathfrak{U}$. We write $P_{1}<P_{2}$ if $P_{1}\leq P_{2}$ and
$P_{1}\neq P_{2}$. It is easy to see that
\begin{equation}
P_{1}\leq P_{2}\text{ }\Longrightarrow\text{ }\mathbf{\mathbf{Rad}}\left(
P_{1}\right)  \subset\mathbf{\mathbf{Rad}}\left(  P_{2}\right)  \text{ and
}\mathbf{Sem}\left(  P_{2}\right)  \subset\mathbf{Sem}\left(  P_{1}\right)  .
\label{pprs}%
\end{equation}
The following theorem is useful for comparing radicals.

\begin{theorem}
\label{equality}Let $P_{1}$ and $P_{2}$ be over radicals on $\mathfrak{U}$,
and let $R_{1}$ and $R_{2}$ be under radicals on $\mathfrak{U}$. Then

\begin{enumerate}
\item $\mathbf{Sem}\left(  P_{2}\right)  \subset\mathbf{Sem}\left(
P_{1}\right)  $ if and only if $P_{1}\leq P_{2}$;

\item $\mathbf{\mathbf{Rad}}\left(  R_{1}\right)  \subset\mathbf{\mathbf{Rad}%
}\left(  R_{2}\right)  $ if and only if $R_{1}\leq R_{2}$.
\end{enumerate}
\end{theorem}

\begin{proof}
Let $A$ be an algebra. Then $R_{1}\left(  A\right)  \in\mathbf{\mathbf{Rad}%
}\left(  R_{1}\right)  $ by Axiom $3$ and $A/P_{2}\left(  A\right)
\in\mathbf{Sem}\left(  P_{2}\right)  $ by Axiom $2$.

$\left(  1\right)  $ If $\mathbf{Sem}\left(  P_{2}\right)  \subset
\mathbf{Sem}\left(  P_{1}\right)  $ then $A/P_{2}\left(  A\right)
\in\mathbf{Sem}\left(  P_{1}\right)  $. As $P_{1}$ is a preradical,
\[
q\left(  P_{1}\left(  A\right)  \right)  \subset P_{1}\left(  A/P_{2}\left(
A\right)  \right)  =0
\]
for the standard quotient map $q:A\longrightarrow A/P_{2}\left(  A\right)  $.
Therefore $P_{1}\left(  A\right)  \subset P_{2}\left(  A\right)  $, i.e.,
$P_{1}\leq P_{2}$, and $\left(  \ref{pprs}\right)  $ completes the proof.

$\left(  2\right)  $ If $\mathbf{\mathbf{Rad}}\left(  R_{1}\right)
\subset\mathbf{\mathbf{Rad}}\left(  R_{2}\right)  $ then $R_{1}\left(
A\right)  \in\mathbf{\mathbf{Rad}}\left(  R_{2}\right)  $. As $R_{1}\left(
A\right)  $ is an ideal of $A$,
\[
R_{1}\left(  A\right)  =R_{2}\left(  R_{1}\left(  A\right)  \right)  \subset
R_{2}\left(  A\right)
\]
by Axiom $4$, i.e., $R_{1}\leq R_{2}$, and the converse follows by $\left(
\ref{pprs}\right)  $.
\end{proof}

\begin{corollary}
\label{ineq}Let $P_{1}$ and $P_{2}$ be radicals on $\mathfrak{U}$. If
$\mathbf{Sem}\left(  P_{1}\right)  =\mathbf{Sem}\left(  P_{2}\right)  $ or
$\mathbf{\mathbf{Rad}}\left(  P_{1}\right)  =\mathbf{\mathbf{Rad}}\left(
P_{2}\right)  $ then $P_{1}=P_{2}$.
\end{corollary}

\subsection{\label{uni}Classes of algebras}

\subsubsection{Base classes of algebras\label{remarks}}

Let us say more about classes of algebras on which the radicals are defined.
It is convenient to assume by default that a class of algebras contains all
appropriate [topologically] isomorphic images of its elements.

A class $\mathfrak{U}$ of algebras is called \textit{algebraically universal}
if it contains quotients and ideals of algebras from $\mathfrak{U}$. A class
$\mathfrak{U}$ of normed algebras is called \textit{universal} if it contains
quotients by closed ideals and ideals of algebras from $\mathfrak{U}$, and
\textit{ground} if it contains quotients by closed ideals and closed ideals of
algebras from $\mathfrak{U}$.

For instance, $\mathfrak{U}_{\mathrm{a}}$ is algebraically universal,
$\mathfrak{U}_{\mathrm{n}}$ is universal, and $\mathfrak{U}_{\mathrm{b}}$ is
ground. These classes are main in this paper. We also use the following
classes of algebras:

\begin{enumerate}
\item $\mathfrak{U}_{\mathrm{q}}$ and $\mathfrak{U}_{\mathrm{q}_{b}}$ are the
class of all $Q$-algebras and $Q_{\mathrm{b}}$-algebras, respectively. These
classes are universal \cite{TR1}.

\item $\mathfrak{U}_{\mathrm{b}}^{u}$ is the smallest universal class
containing all Banach algebras. An algebra $A$ is called a \textit{subideal}
of an algebra $B$ if there is a finite series of algebras $A=I_{0}%
\subset\cdots\subset I_{n}=B$ such that $I_{i-1}$ is an ideal of $I_{i}$ for
$i=1,\ldots,n$; in such a case $A$ is called an $n$\textit{-subideal} of $B$.
By \cite[Theorem 2.24]{TR1}, $\mathfrak{U}_{\mathrm{b}}^{u}$ is the class of
all subideals of Banach algebras.

\item $\mathfrak{U}_{\mathrm{c}^{\ast}}$ is the class of all C*-algebras.
Recall that every closed ideal of a C*-algebra is automatically self-adjoint
\cite[Proposition 1.8.2]{Dixmier}. So this class is ground.
\end{enumerate}

\subsubsection{About the definition of radicals in C*-algebras}

We will also consider the restrictions of radicals of Banach algebras to
$\mathfrak{U}_{\mathrm{c}^{\ast}}$. The work with C*-algebra has several advantages:

\begin{enumerate}
\item The sum of two closed ideals of a C*-algebra is closed \cite[Corollary
1.5.6]{Dav};

\item Each irreducible *-representation of a C*-algebra is strictly
irreducible \cite[Theorem 2.8.3]{Dixmier};

\item A closed ideal of a closed ideal of a C*-algebra is an ideal of the
algebra \cite[Proposition 1.8.5]{Dixmier}.
\end{enumerate}

So Axiom $4$ for radicals on C*-algebras is equivalent to the following:

\bigskip

\noindent\textbf{Axiom 4 for }$\mathfrak{U}_{\mathrm{c}^{\ast}}$\textbf{.
}$P(I)\subset P(A)$ for any closed ideal $I$ of $A$.

\bigskip

As C*-algebras are semisimple, every morphism in $\mathfrak{U}_{\mathrm{c}%
^{\ast}}$ is automatically continuous by Johnson's theorem \cite{J67}. It
seems to be natural to consider all *-epimorphisms (i.e. *-morphisms) as
morphisms in $\mathfrak{U}_{\mathrm{c}^{\ast}}$. The following result shows
that this does not change the class of preradicals.

\begin{theorem}
\label{star-mor} Let $P$ be a closed ideal map on $\mathfrak{U}_{\mathrm{c}%
^{\ast}}$. If $f(P(A))\subset P(B)$ for each *-morphism $f:A\longrightarrow B$
then the same is true for all morphisms.
\end{theorem}

\begin{proof}
Let $f:A\longrightarrow B$ be an epimorphism, $J=\ker f$ and $C=A/J$. Then
$f=g\circ q_{J}$, where $g$ is an isomorphism of $C$ onto $B$. So it suffices
to show that $q_{J}(P(A))\subset P(C)$ and $g(P(C))\subset P(B)$. The first
inclusion is evident because $q_{J}$ is a *-epimorphism.

One can consider $B$ as an operator C*-algebra, that is, a closed *-subalgebra
of $\mathcal{B}(H)$ where $H$ is a Hilbert space. It was proved by T. Gardner
\cite{Gard} that $g=g_{1}\circ g_{2}$ where $g_{2}$ is a *-isomorphism of $C$
onto $B$, and $g_{1}:B\longrightarrow B$ acts by the rule
\[
g_{1}(x)=vxv^{-1}%
\]
where $v$ is an invertible positive operator on $H$. So, to prove that
$g(P(C))\subset P(B)$, it suffices to show that $g_{1}$ preserves closed
ideals of $B$.

Let $S$ be the operator on $\mathcal{B}(H)$ defined by $Sx=vxv^{-1}$ for all
$x\in\mathcal{B}(H)$. By the assumption, $S$ preserves $B$ and $g_{1}=S|_{B}$.
Clearly $S$ is a product of two commuting operators: $S=\mathrm{L}%
_{v}\mathrm{R}_{v^{-1}}$. Since $\mathrm{L}_{v}$ and $\mathrm{R}_{v^{-1}}$
have positive spectra on $\mathcal{B}(H)$, the same is true for $S$ on
$\mathcal{B}(H)$.

Let us denote by $\log$ the holomorphic extension of the function
\[
\phi(z)=\sum_{k=1}^{\infty}\frac{(-1)^{k-1}}{k}(z-1)^{k}%
\]
from the disk $D=\{z:|z-1|<1$ to $\mathbb{C}\setminus(-\infty,0]$. By the
functional calculus, this function can be applied to every element of a Banach
algebra whose spectrum doesn't intersect $(-\infty,0]\subset\mathbb{R}$. As
$E:=\left\{  K\in\mathcal{B}(\mathcal{B}(H))):KB\subset B\right\}  $ is a
Banach algebra and $S\in E$ has spectrum outside of $(-\infty,0]\subset
\mathbb{R}$ then the operator $T=\log S$ preserves $B$.

Moreover, $T=\log\mathrm{L}_{v}+\log\mathrm{R}_{v^{-1}}$. Indeed, to prove
that
\begin{align*}
\log\left(  1+\lambda\left(  \mathrm{L}_{v}-1\right)  \right)  \left(
1+\lambda\left(  \mathrm{R}_{v^{-1}}-1\right)  \right)   &  =\log\left(
1+\lambda\left(  \mathrm{L}_{v}-1\right)  \right) \\
&  +\log\left(  1+\lambda\left(  \mathrm{R}_{v^{-1}}-1\right)  \right)
\end{align*}
for $\lambda=1$, it suffices to show that the equality holds for sufficiently
small $|\lambda|$. Clearly the equality $\phi(KL)=\phi(K)+\phi(L)$ is checked
for commuting operators $K,L$ having spectra in $D$ by the calculation as in
the case of numerical series.

Let $u=\log v$. As $0=\log1=\log v+\log v^{-1}$ and $\log\mathrm{L}%
_{v}=\mathrm{L}_{\log v}$ then $\log\mathrm{L}_{v}=\mathrm{L}_{u}$ and
$\log\mathrm{R}_{v^{-1}}=-\mathrm{R}_{u}$. We proved that the operator $T$
acts by the formula
\[
Tx=ux-xu
\]
for $x\in\mathcal{B}(H)$ and preserves $B$.

The restriction of $T$ to $B$ is a derivation of $B$ and therefore preserves
closed ideals of $B$. Indeed, if $I$ is a closed ideal of $B$ then it is a
C*-algebra, so it has a bounded approximate identity whence, by \cite[Theorem
11.10]{BD73}, each $a\in J$ can be written in the form $a=bc$ for $a,b\in J$,
and
\[
T(a)=bT(c)+T(b)c\in I.
\]
Therefore $S=\exp(T)$ also preserves all closed ideals of $B$.
\end{proof}

As a result, Axiom 1 for radicals on C*-algebras is equivalent to the following:

\bigskip

\noindent\textbf{Axiom 1 for }$\mathfrak{U}_{\mathrm{c}^{\ast}}$\textbf{.
}$f\left(  P\left(  A\right)  \right)  \subset P\left(  B\right)  $ for any
*-morphism $f:A\longrightarrow B$.

\subsection{Some important examples}

\subsubsection{\label{sjac}The Jacobson radical}

The most famous radical on $\mathfrak{U}_{\mathrm{a}}$ is the \textit{Jacobson
radical} $\operatorname{rad}$ which is defined by
\[
\operatorname{rad}(A)=\ker\operatorname{Prim}(A^{1}):=\cap\{I:I\in
\operatorname{Prim}(A^{1})\}
\]
for every algebra $A$; $\operatorname{rad}$ is an algebraic hereditary
radical. As usual, $A$ is called \textit{radical} if $A=\operatorname{rad}%
\left(  A\right)  $, and \textit{semisimple} if $\operatorname{rad}\left(
A\right)  =0$. By Johnson's theorem \cite{J67}, the topology of a complete
norm in a semisimple Banach algebra is unique.

\begin{theorem}
\label{am}Let $A,B$ be Banach algebras, and let $P$ be a topological over
radical such that $\operatorname{rad}\leq P$ on Banach algebras. Then

\begin{enumerate}
\item If $f:A\longrightarrow B$ is an algebraic morphism then $f\left(
P\left(  A\right)  \right)  \subset P\left(  B\right)  $;

\item If $f:A\longrightarrow B$ is an algebraic isomorphism then $f\left(
P\left(  A\right)  \right)  =P\left(  B\right)  $.
\end{enumerate}
\end{theorem}

\begin{proof}
$\left(  1\right)  $ Let $B^{\prime}=B/\operatorname{rad}\left(  B\right)  $
and $q:B\longrightarrow B^{\prime}$ be the standard quotient map. Then $q\circ
f:A\longrightarrow B^{\prime}$ is a topological morphism by Johnson's result
\cite[Corollary 5.5.3]{A91}. Therefore $\left(  q\circ f\right)  \left(
P\left(  A\right)  \right)  \subset P\left(  B^{\prime}\right)  $, whence
$f\left(  P\left(  A\right)  \right)  \subset q^{-1}\left(  P\left(
B^{\prime}\right)  \right)  $.

Let $g:B/\operatorname{rad}(B)\rightarrow B/P(B)$ be the natural morphism,
then
\[
g(P(B/\operatorname{rad}(B)))\subset P(B/P(B))=0
\]
whence $P\left(  B^{\prime}\right)  =P(B/\operatorname{rad}(B))\subset\ker
g=P(B)/\operatorname{rad}(B)=q(P(B))$. Thus
\[
f(P(A))\subset q^{-1}\left(  q\left(  P(B)\right)  \right)  =P(B).
\]

$\left(  2\right)  $ follows from $\left(  1\right)  $ since $f^{-1}%
:B\longrightarrow A$ is also an algebraic morphism.
\end{proof}

As a consequence, topological radicals larger than (or equal to)
$\operatorname{rad}$ on Banach algebras don't depend on the complete norm topology.

Let $\operatorname{Rad}$ be the restriction of $\operatorname{rad}$ to Banach
algebras: $\operatorname{Rad}=\operatorname{rad}|_{\mathfrak{U}_{\mathrm{b}}}%
$. Then $\operatorname{Rad}$ is a pliant hereditary topological radical, while
the restriction $\operatorname{rad}|_{\mathfrak{U}_{\mathrm{n}}}$ is not even
a topological radical \cite{D97}. The reason was already discussed: a normed
algebra can have nonclosed primitive ideals. It is clear that
$\operatorname{Rad}$ has a unique hereditary extension to $\mathfrak{U}%
_{\mathrm{b}}^{u}$, namely $\operatorname{rad}|_{\mathfrak{U}_{\mathrm{b}}%
^{u}}$. The hereditary extension of $\operatorname{Rad}$ to normed algebras is
given by the \textit{regular Jacobson radical} $\operatorname{Rad}^{r}$ (see
Section \ref{regular}):
\begin{equation}
\operatorname{Rad}^{r}\left(  A\right)  =\left\{  a\in A:\rho\left(
ab\right)  =0\text{ }\forall b\in\widehat{A}\right\}  \label{rjr}%
\end{equation}
where $\rho\left(  a\right)  $ is a (geometric) spectral radius $\inf
_{n}\left\Vert a^{n}\right\Vert ^{1/n}$. Below we consider the other
topological hereditary extensions of $\operatorname{Rad}$ to $\mathfrak{U}%
_{\mathrm{n}}$.

\subsubsection{Primitive maps and related radicals\label{primitive}}

Let $\mathfrak{U}$ be a base class of algebras. A rule that indicates, for
each algebra $A\in\mathfrak{U}$, a subset $\Omega\left(  A\right)  $ of
$\operatorname{Prim}\left(  A\right)  $ is called a \textit{ primitive map }
on $\mathfrak{U}$ if

\begin{enumerate}
\item[(1$_{pm}$)] $\Omega\left(  B\right)  =\left\{  f\left(  I\right)
:I\in\Omega\left(  A\right)  \right\}  $ for an injective morphism
$f:A\longrightarrow B$ of algebras from $\mathfrak{U}$;

\item[(2$_{pm}$)] $\Omega\left(  J\right)  =\left\{  J\cap I:I\in\Omega\left(
A\right)  ,J\cap I\neq J\right\}  $ for every ideal $J\in\mathfrak{U}$ of $A$;

\item[(3$_{pm}$)] $\Omega\left(  A/J\right)  =\left\{  J/I:I\in\Omega\left(
A\right)  ,J\subset I\right\}  $ for every [closed] ideal $J\in\mathfrak{U}$
of $A$.
\end{enumerate}

If $\left(  1_{pm}\right)  $ holds for arbitrary isomorphisms of algebras
$A,B\in\mathfrak{U}$, then we say that $\Omega$ is \textit{pliant}; in other
words, a primitive map $\Omega$ on $\mathfrak{U}$ is pliant if the following
condition holds:

\begin{enumerate}
\item[(4$_{pm}$)] $\Omega\left(  A\right)  $ doesn't depend on the choice of a
norm in $A$, for every $A$ from $\mathfrak{U}$ (in short: on the choice of
$\mathfrak{U}$-norm).
\end{enumerate}

Clearly all primitive maps on $\mathfrak{U}_{\mathrm{a}}$ and on
$\mathfrak{U}_{\mathrm{c}^{\ast}}$ are pliant.

\begin{proposition}
All primitive maps defined on Banach algebras are pliant.
\end{proposition}

\begin{proof}
Let $A$ be a Banach algebra. It follows from (3$_{pm}$) that $\Omega\left(
A\right)  $ is uniquely determined by $\Omega\left(  A/\operatorname{Rad}%
\left(  A\right)  \right)  $, but $A/\operatorname{Rad}\left(  A\right)  $ has
only one complete norm topology by Johnson's theorem.
\end{proof}

The following proposition is straightforward.

\begin{proposition}
A primitive map defined on some class $\mathfrak{U}$ of normed algebras is
pliant if and only if it does not depend on the choice of a $\mathfrak{U}%
$-norm on semisimple algebras from $\mathfrak{U}$.
\end{proposition}

It is easy to see from the results of Section \ref{algebras} that setting
$\Omega(A)=\operatorname{Prim}(A)$ for each $A$, we obtain a primitive map
on\textbf{ $\mathfrak{U}_{\mathrm{a}}$; }it will be denoted by
$\operatorname{Prim}$.

\begin{theorem}
\label{pm}Let $\mathcal{F}=\left(  \Omega_{\alpha}\right)  $ be a family of
(pliant)\textit{ primitive maps. Then }

\begin{enumerate}
\item $\Omega_{\cap\mathcal{F}}:A\longrightarrow\cap_{\alpha}\Omega_{\alpha
}\left(  A\right)  $ and $\Omega_{\cup\mathcal{F}}:A\longrightarrow
\cup_{\alpha}\Omega_{\alpha}\left(  A\right)  $ are (pliant) \textit{
primitive maps;}

\item If $\Omega_{i}$ is a (pliant)\textit{ primitive map for }$i=1,2$\textit{
then so is }$\Omega_{1}\backslash\Omega_{2}:A\longmapsto\Omega_{1}\left(
A\right)  \backslash\Omega_{2}\left(  A\right)  $.
\end{enumerate}
\end{theorem}

\begin{proof}
$\left(  1\right)  $ Since
\begin{align*}
\cup_{\alpha}\Omega_{\alpha}\left(  J\right)   &  =\cup_{\alpha}\left\{  J\cap
I:I\in\Omega_{\alpha}\left(  A\right)  ,J\cap I\neq J\right\} \\
&  =\left\{  J\cap I:I\in\cup_{\alpha}\Omega_{\alpha}\left(  A\right)  ,J\cap
I\neq J\right\}
\end{align*}
we obtain (2$_{pm}$) for $\Omega_{\cup\mathcal{F}}$. The remaining assertions
can be proved similarly.

$\left(  2\right)  $. Since
\begin{align*}
\left(  \Omega_{1}\backslash\Omega_{2}\right)  \left(  A/J\right)   &
=\left\{  J/I:I\in\Omega_{1}\left(  A\right)  ,J\subset I\right\}
\backslash\left\{  J/I:I\in\Omega_{2}\left(  A\right)  ,J\subset I\right\} \\
&  =\left\{  J/I:I\in\Omega_{1}\left(  A\right)  \backslash\Omega_{2}\left(
A\right)  ,J\subset I\right\}
\end{align*}
we get (3$_{pm}$) for $\Omega_{1}\backslash\Omega_{2}$. The other conditions
can be checked in a similar way.
\end{proof}

As usual for subsets of $\operatorname{Prim}\left(  A\right)  $, $\ker
\Omega\left(  A\right)  $ is simply $\cap\left\{  I:I\in\Omega\left(
A\right)  \right\}  $. We define an ideal map $\Pi_{\Omega}$ by
\[
\Pi_{\Omega}\left(  A\right)  =\ker\Omega\left(  A\right)
\]
for every $A\in\mathfrak{U}$, meaning that $\Pi_{\Omega}\left(  A\right)  =A$
if $\Omega\left(  A\right)  =\varnothing$, or, more formally, by $\Pi_{\Omega
}\left(  A\right)  =A\cap\ker\Omega\left(  A\right)  $. Not all primitive maps
determine radicals (see Example \ref{pm5}). The obstacle is that the ideals
$\Pi_{\Omega}\left(  A\right)  $ can be non-closed.

\begin{theorem}
\label{pm1}Let $\Omega$ be a\textit{ }primitive map on $\mathfrak{U}$. Then

\begin{enumerate}
\item If ($\Omega$ is pliant and) $\mathfrak{U}\subset\mathfrak{U}%
_{\mathrm{q}}$ then $\Pi_{\Omega}$ is a (pliant) topological hereditary radical;

\item If $\Omega$ is pliant and $\mathfrak{U}_{\mathrm{q}}\subset
\mathfrak{U}\subset\mathfrak{U}_{\mathrm{n}}$ then $\Pi_{\Omega}$ is a pliant
hereditary preradical;

\item If $\mathfrak{U}=\mathfrak{U}_{\mathrm{a}}$ then $\Pi_{\Omega}$ is an
algebraic hereditary radical.
\end{enumerate}
\end{theorem}

\begin{proof}
$\left(  3\right)  $ Let $A\in\mathfrak{U}$ be an algebra, and let
$J\in\mathfrak{U}$ be an ideal of $A$. It follows from (2$_{pm}$) that
$\Pi_{\Omega}\left(  J\right)  =J\cap\Pi_{\Omega}\left(  A\right)  $. So
$\Pi_{\Omega}$ is hereditary.

Let $J=\Pi_{\Omega}\left(  A\right)  $. Then $J\subset I$ for every
$I\in\Omega\left(  A\right)  $. By $\left(  \ref{pr}\right)  $,
\[
\Pi_{\Omega}\left(  A/J\right)  =\ker\Omega\left(  A\right)  /J=\Pi_{\Omega
}\left(  A\right)  /J=J/J=0.
\]

Since every morphism is represented as the superposition of a standard quotient
map and an injective morphism, (1$_{pm}$) and (3$_{pm}$) imply that
$\Pi_{\Omega}$ is a pliant radical on $\mathfrak{U}$.

$\left(  1\right)  $ Indeed, every strictly irreducible representation of a
normed $Q$-algebra $A$ is equivalent to a bounded representation by bounded
operators on a normed space \cite[Theorem 2.1]{TR1}. So $\operatorname{Prim}%
\left(  A\right)  $ consists of closed ideals and $\Pi_{\Omega}$ is a closed
ideal map on normed $Q$-algebras. As in $\left(  3\right)  $, $\Pi_{\Omega}$
is a hereditary radical.

$\left(  2\right)  $ is evident.
\end{proof}

\begin{remark}
\label{pm3}An algebra $A\in\mathfrak{U}$ is $\Pi_{\Omega}$-radical if and only
if $\Omega\left(  A\right)  =\varnothing$, and $A$ is $\Pi_{\Omega}%
$-semisimple if and only if\quad$\cap\left\{  I:I\in\Omega\left(  A\right)
\right\}  =0$; it is clear that every $\Pi_{\Omega}$-semi%
\-%
sim%
\-%
ple algebra is $\operatorname{rad}$-semisimple, i.e. $\operatorname{rad}%
\leq\Pi_{\Omega}$ on $\mathfrak{U}$. In particular, Theorem $\ref{am}$ holds
for $P=\Pi_{\Omega}$.
\end{remark}

The following statement is obvious.

\begin{corollary}
\label{pm2} Let $\Omega$ be a pliant primitive map on a universal class
$\mathfrak{U}$ of normed algebras. Then $\Pi_{\Omega}$ satisfies the condition
of flexible heredity (in particular, the condition of Banach heredity) and
$f\left(  P\left(  A\right)  \right)  =P\left(  B\right)  $ for an algebraic
injective morphism $f:A\longrightarrow B$ of any algebras $A,B\in\mathfrak{U}$.
\end{corollary}

See the definitions of $\operatorname{Prim}_{\mathrm{n}}\left(  A\right)  $
and $\operatorname{Prim}_{\mathrm{b}}\left(  A\right)  $ for a normed algebra
$A$ in Section \ref{algebras}.

\begin{example}
\label{pm4}$\operatorname{Prim}_{\mathrm{n}}$ and $\operatorname{Prim}%
_{\mathrm{b}}$ are primitive maps on $\mathfrak{U}_{\mathrm{n}}$.
\end{example}

\begin{proof}
In virtue of Section \ref{algebras}, it is sufficient to check that the
declared properties of representations from $\operatorname{Irr}_{\mathrm{n}%
}\left(  A\right)  $ or $\operatorname{Irr}_{\mathrm{b}}\left(  A\right)  $
are preserved under passing to an ideals, quotients, actions of isomorphisms
and taking the converse manipulations. But this is straightforward in virtue
of \cite[Lemmas 2.3 and 2.4]{TR1}.
\end{proof}

The topological radicals $\Pi_{\operatorname{Prim}_{\mathrm{n}}}$ and
$\Pi_{\operatorname{Prim}_{\mathrm{b}}}$ coincide with the introduced in
\cite{D97, TR1} extensions $\operatorname{rad}_{\mathrm{b}}$ and
$\operatorname{rad}_{\mathrm{b}}$ of $\operatorname{Rad}$ to normed algebras.
On the base of $\operatorname{Prim}_{\mathrm{n}}$ and $\operatorname{Prim}%
_{\mathrm{b}}$ one can construct other useful primitive maps by imposing
additional requirements on representations $\pi$ from $\operatorname{Irr}%
_{\mathrm{n}}\left(  A\right)  $ or $\operatorname{Irr}_{\mathrm{b}}\left(
A\right)  $.

\begin{example}
\label{pm5}By Theorem $\ref{pm}$, $\Omega_{\mathrm{n}}:=$ $\operatorname{Prim}%
\backslash\operatorname{Prim}_{\mathrm{n}}$ and $\Omega_{\mathrm{b}}:=$
$\operatorname{Prim}\backslash\operatorname{Prim}_{\mathrm{b}}$ are primitive
maps on $\mathfrak{U}_{\mathrm{n}}$, but $\Pi_{\Omega_{\mathrm{n}}}$ and
$\Pi_{\Omega_{\mathrm{b}}}$ are not closed ideal maps on $\mathfrak{U}%
_{\mathrm{n}}$; $Q$-algebras are $\Pi_{\Omega_{\mathrm{n}}}$-radical, and
$\Pi_{\Omega_{\mathrm{n}}}\left(  A\right)  $ is a $Q$-algebra for every
normed algebra $A$.
\end{example}

The simplest characteristic of a representation is its dimension. Recall that
the \textit{dimension} of $\pi\in\operatorname{Irr}\left(  A\right)  $ is the
dimension of $X_{\pi}$. Let $N$ be a subset of $\mathbb{N\;\cup}\left\{
\infty\right\}  $, and let $\operatorname{Irr}^{\dim\in N}\left(  A\right)  $
be the set of all representations $\pi\in\operatorname{Irr}\left(  A\right)  $
such that $\dim X_{\pi}\in N$. Respectively, let $\operatorname{Prim}^{\dim\in
N}\left(  A\right)  =\left\{  \ker\pi:\pi\in\operatorname{Irr}^{\dim\in
N}\left(  A\right)  \right\}  $. The simplest special cases of the relation
$\dim\in N$ are $\dim=1$, $\dim>1$, $\dim<\infty$.

\begin{theorem}
\label{pm6}Let $N\subset\mathbb{N\;\cup}\left\{  \infty\right\}  $. Then
$\operatorname{Prim}^{\dim\in N}$ is a pliant hereditary primitive map$.$
\end{theorem}

\begin{proof}
It suffices to note that the representation spaces are not changed when we
extend a representation from an ideal to the algebra, restrict to an ideal or
induce the strictly irreducible representations from a quotient.
\end{proof}

\begin{corollary}
\label{dim}Let $N\subset\mathbb{N\;\cup}\left\{  \infty\right\}  $. Then the
map
\[
\operatorname{rad}^{\dim\in N}:A\longmapsto\Pi_{\operatorname{Prim}^{\dim\in
N}}\left(  A\right)
\]
is a pliant hereditary radical$.$
\end{corollary}

Note that $\operatorname{rad}^{\dim\in N}$ is a pliant hereditary topological
radical on $Q$-algebras.

\begin{example}
An algebra $A$ is $\operatorname{rad}^{\dim>1}$-radical $\Leftrightarrow$ $A$
admits only one-dimensional strictly irreducible representations, and $A$ is
$\operatorname{rad}^{\dim=1}$-semisimple $\Leftrightarrow$ the intersection of
kernels of one-dimensional strictly irreducible representations is trivial.

Respectively, $A$ is $\operatorname{rad}^{\dim=\infty}$-radical
$\Leftrightarrow$ $A$ admits only finite-dimensional strictly irreducible
representations, and $A$ is $\operatorname{rad}^{\dim<\infty}$-semisimple
$\Leftrightarrow$ the intersection of kernels of finite-dimensional strictly
irreducible representations is trivial.
\end{example}

Let us finish by showing that in terms of primitive maps one can define one of
the most important objects of the theory of C*-algebras --- the largest GCR-ideal.

Recall that a C*-algebra $A$ is called \textit{elementary} if it is isomorphic
to the algebra $\mathcal{K}(H)$ of all compact operators on some Hilbert space
$H$. Furthermore, $A$ is a \textit{CCR-algebra} (a \textit{GCR-algebra}) if
for every irreducible representation $\pi$ of $A$, the image, $\pi(A)$,
consists of compact operators (respectively, contains a non-zero compact
operator, or, equivalently, all compact operators). A simplest example of a
CCR-algebra is $\mathcal{K}(H)$.

For $A\in\mathfrak{U}_{\mathrm{c}^{\ast}}$, let $\Omega_{\mathfrak{gcr}}(A)$
denote the set of all $I\in\operatorname{Prim}(A)$ such that $A/I$ has no
elementary ideals.

\begin{theorem}
$\Pi_{\Omega_{\mathfrak{gcr}}}$ is a topological radical on $\mathfrak{U}%
_{\mathrm{c}^{\ast}}$; $\Pi_{\Omega_{\mathfrak{gcr}}}\left(  A\right)  $ is
the largest GCR-ideal for every $A\in\mathfrak{U}_{\mathrm{c}^{\ast}}$.
\end{theorem}

\begin{proof}
It is easy to check that $\Omega_{\mathfrak{gcr}}$ is a primitive map whence
$\Pi_{\Omega_{\mathfrak{gcr}}}$ is a topological radical on $\mathfrak{U}%
_{\mathrm{c}^{\ast}}$.

Let $A$ be a C*-algebra and $J_{0}=\Pi_{\Omega_{\mathfrak{gcr}}}(A)=\cap
\{I\in\Omega_{\mathfrak{gcr}}(A)\}$. To show that $J_{0}$ is a GCR-algebra,
take an irreducible representation $\pi$ of $J_{0}$ on a Hilbert space $H$ and
extend it to an irreducible representation (again $\pi$) of $A$ on $H$. Then
$\ker\pi\notin\Omega_{\mathfrak{gcr}}(A)$, so $\pi(A)$ contains an elementary
ideal $K$. Since $K$ is irreducible and $\mathcal{K}(H)$ is a CCR-algebra,
$K=\mathcal{K}(H)$. Then
\[
0\neq\pi(J_{0})K(H)\subset\pi(J_{0})\pi(A)\subset\pi(J_{0}),
\]
so $\pi(J_{0})$ contains a non-zero compact operator. This means that $J_{0}$
is a GCR-algebra.

Let now $J$ be a GCR-ideal of $A$. If $I\in\Omega_{\mathfrak{gcr}}(A)$ then
$I=\ker\pi$ for an irreducible representation $\pi$ of $A$ such that $\pi(A)$
has no elementary ideals. But the restriction of $\pi$ to $J$ is also
irreducible (if non-zero), so $\pi(J)$ must contain $\mathcal{K}(H)$.
Therefore $\mathcal{K}(H)\subset\pi(A)$, a contradiction. Thus $\pi(J)=0$ and
then $J\subset I$. So
\[
J\subset\cap\{I:I\in\Omega_{\mathfrak{gcr}}(A)\}=J_{0}.
\]

\end{proof}

Another way to establish that the map sending each C*-algebra $A$ to its
largest GCR-ideal is a hereditary topological radical, will be considered in
Section \ref{c-algebras}.

\subsubsection{\label{stensor}The tensor Jacobson radical $\mathcal{R}_{t}$}

It is defined via the tensor spectral radius. Let $A$ be a normed algebra,
$M=\left(  a_{i}\right)  _{1}^{\infty}$ and $N=\left(  b_{j}\right)
_{1}^{\infty}$ be \textit{summable} families in $A$ , i.e. $\left\Vert
M\right\Vert _{+}=\sum_{i}\left\Vert a_{i}\right\Vert <\infty$ and $\left\Vert
N\right\Vert _{+}=\sum_{i}\left\Vert b_{i}\right\Vert <\infty$. The norm of a
family does not depend on the order, so one can consider the corresponding
equivalence relation $\simeq$ between families (see details in \cite{TR2}).
Define the product $MN$ within up to $\simeq$ as a family $H=\left(
c_{k}\right)  _{1}^{\infty}$ where $c_{k}=a_{i}b_{j}$ for $k=\phi(i,j)$ and
$\phi$ is an arbitrary bijection of $\mathbb{N\times}\mathbb{N}$ onto
$\mathbb{N}$. Then the power $M^{n+1}$ is defined by $M^{n+1}\simeq MM^{n}$
for every $n>0$ and the \textit{tensor (spectral) radius} $\rho_{t}\left(
M\right)  $ is defined by
\[
\rho_{t}\left(  M\right)  =\inf_{n}\left\Vert M^{n}\right\Vert _{+}^{1/n}%
=\lim_{n\rightarrow\infty}\left\Vert M^{n}\right\Vert _{+}^{1/n}.
\]
Let us define by $N\sqcup M$ the disjunct union of families $M=\left(
a_{i}\right)  $ and $N=\left(  b_{j}\right)  $ (that is a family equivalent to
$\left(  b_{1},a_{1},b_{2},a_{2},\ldots\right)  $); then $\mathcal{R}%
_{t}\left(  A\right)  $ is defined as the set of all $a\in A$ such that
$\rho_{t}\left(  \left\{  a\right\}  \sqcup M\right)  =\rho_{t}\left(
M\right)  $ for every summable family $M$ in $A$; $\mathcal{R}_{t}$ is a
uniform topological radical \cite{TR1, TR2}.

The term\textit{ tensor} is justified by the fact that $a\in\mathcal{R}%
_{t}\left(  A\right)  $ if and only if $a\otimes b\in\operatorname{Rad}\left(
A\widehat{\otimes}B\right)  $ for every normed algebra $B$ and $b\in B$
\cite[Theorem 3.36]{TR2}. By \cite[Theorem 3.29]{TR2},
\begin{equation}
\rho_{t}\left(  M\right)  =\rho_{t}\left(  M/\mathcal{R}_{t}\left(  A\right)
\right)  . \label{tr}%
\end{equation}

There is a problem whether $\mathcal{R}_{t}=\operatorname{Rad}$ on Banach
algebras. In the algebraic case it is known that the tensor product of a
radical algebra and the other algebra can be not radical.

\begin{problem}
Is there an algebraic radical $P$ on $\mathfrak{U}_{a}$ such that
$P=\mathcal{R}_{t}$ on $\mathfrak{U}_{\mathrm{n}}$?
\end{problem}

\subsubsection{\label{2.3.4}The compactly quasinilpotent radical
$\mathcal{R}_{\mathrm{cq}}$}

Let $A$ be a normed algebra, and let $M\subset A$ be bounded. Define the
\textit{norm }$\left\Vert M\right\Vert =\sup_{a\in M}\left\Vert a\right\Vert $
and the \textit{joint spectral radius}
\[
\rho\left(  M\right)  =\inf_{n}\left\Vert M^{n}\right\Vert ^{1/n}%
=\lim_{n\rightarrow\infty}\left\Vert M^{n}\right\Vert ^{1/n}%
\]
where $M^{n}=\left\{  a_{1}\cdots a_{n}:a_{i}\in M\right\}  $ \cite{RS60}. Let
$\mathfrak{k}\left(  A\right)  $ be the set of all precompact subsets of $A$.
Then $\mathcal{R}_{\mathrm{cq}}\left(  A\right)  $ is the set of all $a\in A$
such that $\rho\left(  \left\{  a\right\}  \cup M\right)  =\rho\left(
M\right)  $ for every $M\in\mathfrak{k}\left(  A\right)  $; $\mathcal{R}%
_{\mathrm{cq}}$ is a uniform topological radical \cite{TR1}.

A normed algebra $A$ is called \textit{compactly quasinilpotent} if
$\rho\left(  M\right)  =0$ for every $M\in\mathfrak{k}\left(  A\right)  $; the
completion of compactly quasinilpotent algebra is again compactly
quasinilpotent \cite[Lemma 4.11]{TR1}. Note that $\mathcal{R}_{\mathrm{cq}%
}\left(  A\right)  $ is the largest compactly quasinilpotent ideal
\cite[Corollary 4.21]{TR1}.

An important fact related to this radical is that every compact operator $a$
in the Jacobson radical of a closed operator algebra $A$ generates a compactly
quasinilpotent ideal of $A$; a close argument was a key in the proof that $A$
has a hyperinvariant subspace if $a\neq0$ \cite{S84}. The other invariant
subspace results \cite{T99, ST00} are connected with the calculation of the
joint spectral radius. It was established in \cite{ST00} that for a precompact
set $M$ of compact operators on a Banach space $X$
\begin{equation}
\rho\left(  M\right)  =r\left(  M\right)  \label{bwf}%
\end{equation}
where the \textit{Berger-Wang radius} $r\left(  M\right)  $ is defined by
\cite{BW92}:
\[
r\left(  M\right)  =\underset{n\rightarrow\infty}{\lim\sup}\sup\left\{
\rho\left(  a\right)  :a\in M^{n}\right\}  ^{1/n}.
\]
We call $\left(  \ref{bwf}\right)  $ the \textit{Berger-Wang formula} because
the matrix version of $\left(  \ref{bwf}\right)  $, for $\dim X<\infty$, was
established in \cite{BW92}. A useful contribution in the calculation of $\rho$
is the equality%
\begin{equation}
\rho\left(  M\right)  =\rho\left(  M/\mathcal{R}_{\mathrm{cq}}\left(
A\right)  \right)  \label{cq}%
\end{equation}
for every $M\in\mathfrak{k}\left(  A\right)  $ \cite[Theorem 4.18]{TR1}.

There is a version of the joint spectral radius for bounded countable families
$M=\left(  a_{n}\right)  _{1}^{\infty}$ in $A$, where $M^{n}$ is calculated by
rules of families \cite{TR2}.

\subsubsection{The hypocompact radical $\mathcal{R}_{\mathrm{hc}}$}

Let $A$ be a normed algebra. An element $a\in A$ is called \textit{compact }if
$\mathrm{W}_{a}:=\mathrm{L}_{a}\mathrm{R}_{a}$ (see Section \ref{algebras}) is
a compact operator on $A$; $A$ is \textit{compact} if it consists of compact
elements, \textit{bicompact} if $\mathrm{L}_{a}\mathrm{R}_{b}$ is a compact
operator on $A$ for every $a,b\in A$, and \textit{hypocompact} if every
non-zero quotient of $A$ by a closed ideal has a non-zero compact element.
Then $\mathcal{R}_{\mathrm{hc}}\left(  A\right)  $ is defined as the largest
hypocompact ideal of $A$ (equivalently, as the smallest closed ideal $J$ of
$A$ such that $A/J$ has no non-zero compact elements); $\mathcal{R}%
_{\mathrm{hc}}$ is a hereditary topological radical \cite{TR3}.

This radical plays an important role in the theory of the joint spectral
radius. It was proved in \cite{TR3} that the following equality (an
\textit{algebra version of the joint spectral radius formula}):\textit{ }%
\begin{equation}
\rho\left(  M\right)  =\max\left\{  \rho\left(  M/\mathcal{R}_{\mathrm{hc}%
}\left(  A\right)  \right)  ,r\left(  M\right)  \right\}  \label{af}%
\end{equation}
holds for every normed algebra $A$ and $M\in\mathfrak{k}\left(  A\right)  $.
In particular, it follows from $\left(  \ref{af}\right)  $ that
\begin{equation}
\mathcal{R}_{\mathrm{hc}}\left(  A\right)  \cap\operatorname{Rad}^{r}\left(
A\right)  \subset\mathcal{R}_{\mathrm{cq}}\left(  A\right)  \label{aff}%
\end{equation}
for every normed algebra $A$. It should be noted that $A\longmapsto
\mathcal{R}_{\mathrm{hc}}\left(  A\right)  \cap\operatorname{Rad}^{r}\left(
A\right)  $ is also a hereditary topological radical on normed algebras. This
radical is called the \textit{Jacobson hypocompact radical} and is denoted by
$\mathcal{R}_{\mathrm{jhc}}$.

If $A=\mathcal{B}\left(  X\right)  $ for a Banach space $X$, then
$\mathcal{R}_{\mathrm{hc}}\left(  A\right)  \supset\mathcal{K}\left(
X\right)  $, so that $\left(  \ref{af}\right)  $ implies $\left(
\ref{bwf}\right)  $ and a stronger result, the \textit{operator version of the
joint spectral radius formula}:
\begin{equation}
\rho\left(  M\right)  =\max\left\{  \rho_{e}\left(  M\right)  ,r\left(
M\right)  \right\}  \label{afo}%
\end{equation}
that holds for every $M\in\mathfrak{k}\left(  \mathcal{B}\left(  X\right)
\right)  $, where the \textit{essential spectral radius} $\rho_{e}\left(
M\right)  $ is defined as the joint spectral radius $\rho\left(
M/\mathcal{K}\left(  X\right)  \right)  $ in the Calkin algebra $\mathcal{B}%
\left(  X\right)  /\mathcal{K}\left(  X\right)  $ (see details in \cite{TR3}).

There is the largest of topological radicals $P$ which can change
$\mathcal{R}_{\mathrm{hc}}$ in $\left(  \ref{af}\right)  $ \cite{TR3}; it is
denoted by $\mathcal{R}_{\mathrm{bw}}$ and called the \textit{Berger-Wang
radical}.

A normed algebra $A$ is called a \textit{Berger-Wang algebra} if $\left(
\ref{bwf}\right)  $ holds for every $M\in\mathfrak{k}\left(  A\right)  $.

\subsubsection{The hypofinite radicals $\mathcal{R}_{\mathrm{hf}}$ and
$\mathfrak{R}_{\mathrm{hf}}$}

We begin with the algebraic hypofinite radical $\mathfrak{R}_{\mathrm{hf}}$
and describe its construction more transparently because of the lack of
references. An element $a$ of an algebra $A$ is called a \textit{finite rank
element} if $\dim aAa<\infty$. An algebra $A$ is called \textit{finite} if it
consists of finite rank elements, \textit{bifinite} if $\dim aAb<\infty$ for
every $a,b\in A$, and \textit{hypofinite} if $A/I$ contains a non-zero finite
rank element for every ideal $I\neq A$. We transfer these notions to ideals.
Let $\digamma\!\left(  A\right)  $ be the set of all finite rank elements of
$A$.

\begin{theorem}
\label{hf}Let $A$ be an algebra, and let $J$ be an ideal of $A$. Then

\begin{enumerate}
\item $f\left(  \digamma\!\left(  A\right)  \right)  \subset\digamma\!\left(
B\right)  $ for every morphism $f:A\longrightarrow B$;

\item If $\digamma\!\left(  J\right)  \neq0$ then $J\cap\digamma\!\left(
A\right)  \neq0$;

\item If $x\in\digamma\!\left(  A\right)  $ and $I$ is an ideal generated by
$x$ then $\mathrm{L}_{a}\mathrm{R}_{b}$ is a finite rank operator on $A$ for
every $a,b\in I$;

\item If $A$ is hypofinite and $I$ a non-zero (one-sided) ideal then
$I\cap\digamma\!\left(  A\right)  \neq0$;

\item The following conditions are equivalent:

\begin{enumerate}
\item $J$ is a hypofinite ideal of $A$;

\item for every morphism $f:A\longrightarrow B$, either $f\left(  J\right)
=0$ or $f\left(  J\right)  \cap\digamma\!\left(  B\right)  \neq0$;

\item there is an increasing transfinite sequence $\left(  J_{a}\right)
_{\alpha\leq\gamma}$ of ideals of $A$ such that $J_{0}=0$, $J_{\gamma}=J$,
\ $J_{\alpha}=\cup_{\alpha^{\prime}\,<\alpha}J_{\alpha^{\prime}}$ for every
limit ordinal $\alpha$, and all $J_{\alpha+1}/J_{\alpha}$ are bifinite.
\end{enumerate}

\item The following conditions are equivalent:

\begin{enumerate}
\item $A$ is hypofinite;

\item all ideals and quotients of $A$ are hypofinite;

\item $A/I$ and $I$ are hypofinite for some ideal $I$ of $A$.
\end{enumerate}

\item There is a largest hypofinite ideal of $A$ (which is denoted by
$\mathfrak{R}_{\mathrm{hf}}\left(  A\right)  $);

\item $\mathfrak{R}_{\mathrm{hf}}\left(  J\right)  =J\cap\mathfrak{R}%
_{\mathrm{hf}}\left(  A\right)  $;

\item $A/\mathfrak{R}_{\mathrm{hf}}\left(  A\right)  $ has no non-zero
hypofinite ideals and finite rank elements.
\end{enumerate}
\end{theorem}

\begin{proof}
$\left(  1\right)  $ $\mathrm{W}_{f\left(  a\right)  }B=f\left(
\mathrm{W}_{a}A\right)  $ is finite-dimensional if $a\in\digamma\!\left(
A\right)  $.

$\left(  2\right)  $ It is clear that
\begin{equation}
\mathrm{W}_{ba}=\mathrm{L}_{b}\mathrm{W}_{a}\mathrm{R}_{b}=\mathrm{R}%
_{a}\mathrm{W}_{b}\mathrm{L}_{a} \label{wba}%
\end{equation}
for every $a,b\in A$. If $a\in\digamma\!\left(  J\right)  $ then
$\mathrm{W}_{ba}\in\digamma\!\left(  A\right)  $, whence $J\digamma\!\left(
J\right)  \subset J\cap\digamma\!\left(  A\right)  $. If $J\digamma\!\left(
J\right)  =0$ then clearly $\digamma\!\left(  J\right)  \subset J\cap
\digamma\!\left(  A\right)  $.

$\left(  3\right)  $ follows from $\left(  \ref{wba}\right)  $.

$\left(  4\right)  $ Assume that $I$ is a left ideal. Let $K=\left\{  a\in
A:aI=0\right\}  $; $K$ is an ideal of $A$. If $K=A$ then $I\subset
\digamma\!\left(  A\right)  $, otherwise $A/K$ has a non-zero finite rank
element $b$. Let $q_{K}:A\longrightarrow A/K$ be the standard quotient map and
$q_{K}\left(  a\right)  =b$. Clearly there is $x\in J$ such that $ax\neq0$, it
is easy to check that $ax\in\digamma\!\left(  A\right)  $.

$\left(  5\text{a}\right)  \Longrightarrow\left(  5\text{b}\right)  $ Let
$I=\ker f$ and $K=I\cap J$. Assume that $f\left(  J\right)  \neq0$; then there
is a non-zero $a\in J/K\cap\digamma\!\left(  A/K\right)  $ by $\left(
2\right)  $. Let $b\in A$ be such that $a=q_{K}\left(  b\right)  $; then
$0\neq f\left(  b\right)  \in\digamma\!\left(  B\right)  $.

$\left(  5\text{b}\right)  \Longrightarrow\left(  5\text{c}\right)  $ Assume
that we already have built $J_{\beta}$ by transfinite induction. If $J\neq
J_{\beta}$ then $J/J_{\beta}$ has a non-zero finite rank element by
definition, and contains a non-zero finite rank element $a$ of $A$ by $\left(
2\right)  $. Then $J/J_{\beta}$ contains a non-zero bifinite ideal $I$ of $A$
by $\left(  3\right)  $. Take the preimage of $I$ in $A$ as $J_{\beta+1}$.
This proves the implication.

$\left(  5\text{c}\right)  \Longrightarrow\left(  5\text{a}\right)  $ Let $I$
be an ideal of $J$ and $I\neq J$. Then there is the first ordinal $\alpha$
such that $J_{\alpha}$ doesn't lie in $I$. Then $q_{I}\left(  J_{\alpha
}\right)  \subset\digamma\!\left(  J/I\right)  $.

$\left(  6\text{a}\right)  \Longrightarrow\left(  6\text{b}\right)  $ It is
clear that the quotients of a hypofinite algebra are hypofinite. Let $K$ be an
ideal of $A$, $f:A\longrightarrow B$ a morphism and $I=\ker f$. Assume that
$f\left(  K\right)  \neq0$; then $q_{I}\left(  K\right)  \neq0$. As $A$ is
hypofinite then there is $a\in K$ such that $0\neq q_{I}\left(  a\right)
\in\digamma\!\left(  A/I\right)  $ by $\left(  4\right)  $. Then $0\neq
f\left(  a\right)  \in\digamma\!\left(  B\right)  $ by $\left(  1\right)  $.
Therefore $K$ is hypofinite by $\left(  5\right)  $.

$\left(  6\text{b}\right)  \Longrightarrow\left(  6\text{c}\right)  $ is
obvious, and $\left(  6\text{c}\right)  \Longrightarrow\left(  6\text{a}%
\right)  $ follows by $\left(  5\right)  $: one can build a sequence of ideals
for $A$ as in $\left(  5\text{c}\right)  $ having such sequences for $I$ and
$A/I$.

$\left(  7\right)  $ Let $K$ be the sum of all hypofinite ideals of $A$. Then
an arbitrary morphism $f:A\longrightarrow B$ vanishes on $K$ if and only if it
vanishes on each hypofinite ideal of $A$. Therefore $K$ is hypofinite by
$\left(  5\right)  $.

$\left(  8\right)  $ As $J\cap\mathfrak{R}_{\mathrm{hf}}\left(  A\right)  $ is
a hypofinite ideal of $J$, we obtain that $J\cap\mathfrak{R}_{\mathrm{hf}%
}\left(  A\right)  \subset\mathfrak{R}_{\mathrm{hf}}\left(  J\right)  $. Let
$I$ be the ideal of $A$ generated by $\mathfrak{R}_{\mathrm{hf}}\left(
J\right)  $. Then clearly $I^{3}\subset\mathfrak{R}_{\mathrm{hf}}\left(
J\right)  $ is hypofinite and $I/I^{3}$ is bifinite, whence $I$ is hypofinite
by $\left(  6\right)  $. Therefore $\mathfrak{R}_{\mathrm{hf}}\left(
J\right)  \subset I\subset\mathfrak{R}_{\mathrm{hf}}\left(  A\right)  $.

$\left(  9\right)  $ If $I$ is a non-zero hypofinite ideal of $A/\mathfrak{R}%
_{\mathrm{hf}}\left(  A\right)  $ then its preimage $K$ in $A$ is hypofinite
by $\left(  6\right)  $, whence $\mathfrak{R}_{\mathrm{hf}}\left(  A\right)
\subsetneqq I$, a contradiction. Therefore $A/\mathfrak{R}_{\mathrm{hf}%
}\left(  A\right)  $ has no non-zero finite rank elements.
\end{proof}

\begin{corollary}
\label{hfr} $\mathfrak{R}_{\mathrm{hf}}$ is a hereditary radical.
\end{corollary}

\begin{proof}
By Theorem \ref{hf}(8-9), it suffices to show that $\mathfrak{R}_{\mathrm{hf}%
}$ is an algebraic preradical. Let $f:A\longrightarrow B$ be a morphism, and
let $q:B\longrightarrow B/\mathfrak{R}_{\mathrm{hf}}\left(  B\right)  $ be the
standard quotient map. Then $q\circ f$ is a morphism $A\longrightarrow
B/\mathfrak{R}_{\mathrm{hf}}\left(  B\right)  $ and it follows from Theorem
\ref{hf}$\left(  5\text{b}\right)  $ that either $\left(  q\circ f\right)
\left(  \mathfrak{R}_{\mathrm{hf}}\left(  A\right)  \right)  =0$ or
$B/\mathfrak{R}_{\mathrm{hf}}\left(  B\right)  $ has a non-zero finite rank
element which is impossible by Theorem \ref{hf}$\left(  9\right)  $. Therefore
$\left(  q\circ f\right)  \left(  \mathfrak{R}_{\mathrm{hf}}\left(  A\right)
\right)  =0$ and $f\left(  \mathfrak{R}_{\mathrm{hf}}\left(  A\right)
\right)  \subset\mathfrak{R}_{\mathrm{hf}}\left(  B\right)  $.
\end{proof}

The radical $\mathfrak{R}_{\mathrm{hf}}$ is called an (\textit{algebraic})
\textit{hypofinite radical}.

The radical $\mathcal{R}_{\mathrm{hf}}$ is a topological hereditary radical
and can be defined as the largest closed-hypofinite ideal: a normed algebra
(or ideal) $A$ is called (\textit{closed}-) \textit{hypofinite} if every
non-zero quotient of $A$ by a \textit{closed} ideal has a non-zero finite rank
element of $A$, and \textit{approximable} \cite{TR2} if $A=\overline
{\digamma\!\left(  A\right)  }$. Elements of $\overline{\digamma\!\left(
A\right)  }$ are called \textit{approximable elements} of $A$. Theorem
\ref{hf} and Corollary \ref{hfr} hold for $\mathcal{R}_{\mathrm{hf}}$ if one
replaces the terms for $\mathfrak{R}_{\mathrm{hf}}$ by appropriate terms for
$\mathcal{R}_{\mathrm{hf}}$.

\subsubsection{ Radicals on C*-algebras\label{c-algebras}}

Let $P$ be a radical. It is known (\cite[Theorems 2.9 and 2.10]{TR1}) that the
class $\mathbf{{Rad}}$$(P)$ is stable under extensions and quotients, and that
the [closure of] the union of an up-directed net of ideals $I_{\lambda}%
\in\mathbf{{Rad}}$$(P)$ of some algebra also belongs to $\mathbf{{Rad}}$$(P)$.
Moreover, if $P$ is hereditary then each ideal of $A\in\mathbf{{Rad}}$$(P)$
clearly belongs to $\mathbf{{Rad}}$$(P)$. The following result shows that the
converse holds for C*-algebras: a class of C*-algebras with the listed above
properties is $\mathbf{{Rad}}$$(P)$ for some (hereditary) radical $P$.

\begin{theorem}
\label{clasc} Let $\mathfrak{C}$ be a class of C*-algebras containing
*-isomorphic images of its elements and having the following properties:

\begin{enumerate}
\item[ ]

\begin{enumerate}
\item If $A\in\mathfrak{C}$ then $A/J\in\mathfrak{C}$ for every closed ideal
$J$ of $A$;

\item If $J\in\mathfrak{C}$ is an ideal of a C*-algebra $A$ and $A/J\in
\mathfrak{C}$ then $A\in\mathfrak{C}$;

\item If $\left(  J_{\alpha}\right)  _{{\alpha}}\subset\mathfrak{C}$ is an
up-directed net of ideals of a C*-algebra $A$ then $\overline{\cup_{\alpha
}J_{\alpha}}\in\mathfrak{C}$.
\end{enumerate}
\end{enumerate}

\noindent Then

\begin{enumerate}
\item Each C*-algebra $A$ has a largest ideal $R_{\mathfrak{C}}(A)\in$
$\mathfrak{C}$;

\item $R_{\mathfrak{C}}$ is a radical on $\mathfrak{U}_{c^{\ast}}$ such that
$\mathbf{{Rad}}$$(R_{\mathfrak{C}})=\mathfrak{C}$;

\item The radical $R_{\mathfrak{C}}$ is hereditary if and only if
$\mathfrak{C}$ satisfies the condition

\begin{enumerate}
\item[(d)] Any closed ideal of an algebra $A\in\mathfrak{C}$ belongs to
$\mathfrak{C}$;
\end{enumerate}

\item $R_{\mathfrak{C}}$ is uniform if and only if $B\in\mathfrak{C}$ for any
C*-subalgebra $B$ of every $A\in\mathfrak{C}$.
\end{enumerate}
\end{theorem}

\begin{proof}
(1) If ideals $I,J$ of $A$ belong to $\mathfrak{C}$ then
\[
(I+J)/I=J/(I\cap J)\in\mathfrak{C}%
\]
whence $I+J\in\mathfrak{C}$ by (a).

Let ${\Lambda}$ be the set of all finite families of ideals of $A$ that belong
to $\mathfrak{C}$; it is ordered by the inclusion. For $\omega\in{\Lambda}$,
let $J_{\omega}$ be the sum of all ideals in $\omega$. Then $\left(
I_{\omega}\right)  _{\omega\in{\Lambda}}\subset\mathfrak{C}$ is an up-directed
net of ideals of $A$, whence the closure of their sum is also in
$\mathfrak{C}$ by (c). Clearly it is a largest ideal in $A$ that belongs to
$\mathfrak{C}$.

(2) If $f:A\rightarrow B$ is a *-morphism and $J\in\mathfrak{C}$ is a ideal of
$A$ then $f(J)$ is an ideal of $B$ *-isomorphic to a quotient of $J$. Hence
$f(J)\in\mathfrak{C}$ is an ideal of $B$ and therefore $f(J)\subset
R_{\mathfrak{C}}(B)$.

In particular, $f(R_{\mathfrak{C}}(A))\subset R_{\mathfrak{C}}(B)$. We proved
that the map $A\rightarrow R_{\mathfrak{C}}(A)$ is a preradical.

As $R_{\mathfrak{C}}(A)\in\mathfrak{C}$ then $R_{\mathfrak{C}}(R_{\mathfrak{C}%
}(A))=R_{\mathfrak{C}}(A)$.

If $I$ is a closed ideal of $A$ then $R_{\mathfrak{C}}(I)$ belongs to
$\mathfrak{C}$, whence $R_{\mathfrak{C}}(A)\subset I\cap R_{\mathfrak{C}}(A)$.

Let $N=R_{\mathfrak{C}}(A/R_{\mathfrak{C}}(A))$ and $q=q_{R_{\mathfrak{C}}%
(A)}$. Setting $J=q^{-1}(N)$, we have that $J$ is a closed ideal of $A$ and
$J/R_{\mathfrak{C}}(A)=q(J)=N$ whence $J\in\mathfrak{C}$ by (b). Therefore
$J\subset R_{\mathfrak{C}}(A)$ and $N=0$.

We proved that the map $R_{\mathfrak{C}}$ is a radical on $\mathfrak{U}%
_{c^{\ast}}$; the equality $\mathbf{{Rad}}$$(R_{\mathfrak{C}})=\mathfrak{C}$
is obvious.

(3) If $R_{\mathfrak{C}}$ is hereditary then for each closed ideal $I$ of an
algebra $A\in\mathfrak{C}$, we have that $R_{\mathfrak{C}}(I)=I\cap
R_{\mathfrak{C}}(A)=I\cap A=I$, whence $I\in\mathfrak{C}$.

Conversely, let $\mathfrak{C}$ contain all closed ideals of all algebras in
$\mathfrak{C}$. Since $R_{\mathfrak{C}}(A)\in\mathfrak{C}$ then the ideal
$I\cap R_{\mathfrak{C}}(A)$ of $R_{\mathfrak{C}}(A)$ belongs to $\mathfrak{C}$
whence $I\cap R_{\mathfrak{C}}(A)\subset R_{\mathfrak{C}}(I)$. We proved that
$R_{\mathfrak{C}}$ is hereditary.

(4) obviously follows from (2).
\end{proof}

The above theorem permits to construct important examples of C*-radicals.
Recall that a C*-algebra $A$ is called \textit{a GCR-algebra} (or C*-algebra
of type I) if $\pi(A)$ contains a non-zero compact operator for any
irreducible representation $\pi$ of $A$. Furthermore $A$ is \textit{nuclear}
if, for each C*-algebra $B$, there is only one C*-norm on the algebraic tensor
product $A\otimes B$. If $A$ contains an up-directed net of finite-dimensional
*-subalgebras with dense union then $A$ is called an\textit{ AF-algebra}.

\begin{corollary}
\label{special} Each C*-algebra $A$ has the largest GCR-ideal $\mathcal{R}%
_{\mathfrak{gcr}}(A)$, the largest nuclear ideal $\mathcal{R}_{\mathfrak{nc}%
}(A)$ and the largest AF-ideal $\mathcal{R}_{\mathfrak{af}}(A)$. The maps
$\mathcal{R}_{\mathfrak{gcr}}$, $\mathcal{R}_{\mathfrak{nc}}$ and
$\mathcal{R}_{\mathfrak{af}}$ are hereditary topological radicals on
$\mathfrak{U}_{c^{\ast}}$.
\end{corollary}

\begin{proof}
We should check the properties (a)-(d) for the corresponding classes. For
GCR-algebras they are well known, as well as the existence of the largest
GCR-ideal (see \cite[Section 4.3]{Dixmier}).

Clearly a quotient of an AF-algebra is an AF-algebra. The proof of the fact
that a closed ideal of an AF-algebra is an AF-algebra see in \cite[Lemma
III.4.1]{Dav}. The property (b) for AF-algebras was established by Brown
\cite{Brown} . The property (c) for AF-algebras is obvious.

For nuclear algebras, the proof of (a), (b) and (d) can be found in
\cite[Corollary 15.3.4]{Tak}. The property (c) follows immediately from the
fact that an inductive limit of nuclear C*-algebras is nuclear; it was
established in the first paper on the subject \cite{Tak1}.
\end{proof}

The radical $\mathcal{R}_{\mathfrak{gcr}}$ is uniform, this follows from
\cite[Proposition 4.3.5]{Dixmier}. The radicals $\mathcal{R}_{\mathfrak{nc}}$
and $\mathcal{R}_{\mathfrak{af}}$ are not uniform; the construction of a
non-nuclear C*-subalgebra of a nuclear C*-algebra can be found in
\cite[Example 2.4]{Wass}; for an example of a non-AF C*-subalgebra of
AF-algebra one can take an irrational rotation algebra (see \cite[Theorem
6.5.2]{Dav}).

\section{Procedures}

We understand procedures in the theory of radicals as transforms of ideal
maps. In other words, a \textit{procedure} is a mapping $P\longmapsto
P^{\prime}$ from a class of ideal maps that act on a class $\mathfrak{U}$ of
algebras, into another class of ideal maps acting on a possibly different
class $\mathfrak{U}^{\prime}$ of algebras.

\subsection{The convolution and superposition procedures}

Let $A$ be a [normed] algebra. Increasing and decreasing transfinite sequences
$\left(  I_{\alpha}\right)  $ and $\left(  J_{\alpha}\right)  $, respectively,
of [closed] ideals of $A$ are called\textit{ transfinite chains }if
\begin{align}
I_{\alpha}  &  =\cup_{\alpha^{\prime}<\alpha}I_{\alpha^{\prime}}\text{
\quad\lbrack}I_{\alpha}=\overline{\cup_{\alpha^{\prime}<\alpha}I_{\alpha}%
}\text{],}\label{bb}\\
J_{\alpha}  &  =\cap_{\alpha^{\prime}<\alpha}J_{\alpha^{\prime}}\nonumber
\end{align}
for every limit ordinal $\alpha$. The quotients $I_{\alpha+1}/I_{\alpha}$ and
$J_{\alpha}/J_{\alpha+1}$ are called the \textit{gap-quotients} of $\left(
I_{\alpha}\right)  $ and $\left(  J_{\alpha}\right)  $, respectively.

Let $P$ and $R$ be preradicals satisfying Axiom $4$. For each $A\in
$\textbf{$\mathfrak{U}$, }we define transfinite chains of ideals $\left(
I_{\alpha}\right)  $ and $\left(  J_{\alpha}\right)  $ by the conditions
\begin{align*}
I_{0}  &  =0,\text{ \quad}I_{\alpha+1}=q_{I_{\alpha}}^{-1}\left(  R\left(
A/I_{\alpha}\right)  \right)  ,\\
J_{0}  &  =A,\text{ \quad}J_{\alpha+1}=P\left(  J_{\alpha}\right)  ,
\end{align*}
respectively, for every ordinal $\alpha$, where $q_{I_{\alpha}}%
:A\longrightarrow A/I_{\alpha}$ is the standard quotient map.

Let $R_{\alpha}\left(  A\right)  =I_{\alpha}$ and $P_{\alpha}\left(  A\right)
=J_{\alpha}$ for every $\alpha$. Transfinite sequences $\left(  R_{\alpha
}\right)  $ and $\left(  P_{\alpha}\right)  $ of ideal maps are called the
\textit{convolution} and \textit{superposition chains } for $R$ and $P$,
respectively. The transfinite chains of ideals stabilize at some steps
$\delta=\delta\left(  A\right)  $ and $\gamma=\gamma(A)$:%
\[
I_{\delta}=I_{\delta+1}\text{ \ and\ }J_{\gamma}=J_{\gamma+1}.
\]
It was proved in \cite[Theorem 6.6 and Theorem 6.10]{D97} that the map
$R^{\ast}$ defined by
\[
R^{\ast}\left(  A\right)  :=I_{\delta}%
\]
satisfies Axiom $2$, and the map $P^{\circ}$ defined by
\[
P^{\circ}\left(  A\right)  :=J_{\gamma}%
\]
satisfies Axiom $3$.

The maps $R\longmapsto R^{\ast}$ and $P\longmapsto P^{\circ}$ are called the
\textit{convolution} and \textit{superposition procedures}, respectively. We
underline that the convolution procedure for algebraic preradicals on
$\mathfrak{U}_{\mathrm{a}}$ differs from the one for topological preradicals
on $\mathfrak{U}_{\mathrm{n}}$ because the latter includes the operation of
the closure. So we have two different --- \textit{ algebraic} and
\textit{topological --- }convolution procedures.

The following theorem was proved by Dixon \cite[Corollary 6.12]{D97} for
algebraic (over/under) radicals and for topological (over/under) radicals on
normed algebras in assumption that ideals in Axiom $4$ are closed, in
particular for (over/under) radicals on Banach algebras. We extend the Dixon
results to preradicals of normed algebras satisfying Axiom $4$ without any restrictions.

\begin{theorem}
\label{ovunt}Let $R$, $P$ be preradicals satisfying Axiom $4$. Then

\begin{enumerate}
\item $R^{\ast}$ is \textit{a smallest over} \textit{radical} larger than or
equal to $R$.

If $R$ is an under radical then $R^{\ast}$ is a radical;

\item $P^{\circ}$ is a largest under radical smaller than or equal to $P$.

If $P$ is an over radical then $P^{\circ}$ is a radical.
\end{enumerate}
\end{theorem}

\begin{proof}
Using the transfinite induction, Dixon proved \cite{D97} that if $R$ is an
under radical and $P$ is an over radical then all $R_{\alpha}$ and $P_{\alpha
}$ are under and over radicals, respectively, under his assumption on ideals.
It is easy to see from Dixon's proof that if $R$ and $P$ are preradicals
satisfying Axiom $4$ restricted to closed ideals then so are $R_{\alpha}$ and
$P_{\alpha}$ for every $\alpha$.

Now let $R$ and $P$ satisfy Axiom $4$. We have to check Axiom $4$ for all maps
from $\left(  R_{\alpha}\right)  $ and $\left(  P_{\alpha}\right)  $; then
$R^{\ast}$ and $P^{\circ}$ will also satisfy Axiom $4$. For other assertions
of this theorem we refer to \cite{D97}.

Let $A$ be a normed algebra, and let $K$ be an ideal of $A$. Using the
transfinite induction, we fix an ordinal $\alpha$ and assume that $\left(
R_{\alpha^{\prime}}\right)  _{\alpha^{\prime}<\alpha}$ and $\left(
P_{\alpha^{\prime}}\right)  _{\alpha^{\prime}<\alpha}$ consist of preradicals
satisfying Axiom $4$.

If $\alpha$ is a limit ordinal then all $R_{\alpha^{\prime}}\left(  K\right)
\subset R_{\alpha^{\prime}}\left(  A\right)  $ and $P_{\alpha^{\prime}}\left(
K\right)  \subset P_{\alpha^{\prime}}\left(  A\right)  $ are ideals of $A$ for
$\alpha^{\prime}<\alpha$, whence
\begin{align}
R_{\alpha}\left(  K\right)   &  =\overline{\cup_{\alpha^{\prime}<\alpha
}R_{\alpha^{\prime}}\left(  K\right)  }^{\left(  K\right)  }\subset
\overline{\cup_{\alpha^{\prime}<\alpha}R_{\alpha^{\prime}}\left(  A\right)
}^{\left(  A\right)  }=R_{\alpha}\left(  A\right)  ,\label{rk}\\
P_{\alpha}\left(  K\right)   &  =\cap_{\alpha^{\prime}<\alpha}P_{\alpha
^{\prime}}\left(  K\right)  \subset\cap_{\alpha^{\prime}<\alpha}%
P_{\alpha^{\prime}}\left(  A\right)  =P_{\alpha}\left(  A\right) \nonumber
\end{align}
are ideals of $A$. Indeed, to check this for $R_{\alpha}\left(  K\right)  $
note that $I:=\cup_{\alpha^{\prime}<\alpha}R_{\alpha^{\prime}}\left(
K\right)  $ is an ideal of $A$, and for every $x\in R_{\alpha}\left(
K\right)  $ there is a sequence $\left(  y_{n}\right)  \subset I$ such that
$y_{n}\rightarrow x$ as $n\rightarrow\infty$. For every $a\in A$, one has
$ay_{n},y_{n}a\in I$ for each $n$, whence
\[
ax,xa\in K\cap\overline{I}^{\left(  A\right)  }=\overline{I}^{\left(
K\right)  }=R_{\alpha}\left(  K\right)  .
\]
Therefore $R_{\alpha}\left(  K\right)  $ is an ideal of $A$; the same
assertion for $P_{\alpha}$ is trivial.

Let now $\alpha=\alpha^{\prime}+1$, for some ordinal $\alpha^{\prime}$. Since
$P_{\alpha^{\prime}}\left(  K\right)  \subset P_{\alpha^{\prime}}\left(
A\right)  $ is an ideal of $A$ then $P_{\alpha}\left(  K\right)  =P\left(
P_{\alpha^{\prime}}\left(  K\right)  \right)  $ is an ideal of $A$. As
$P_{\alpha^{\prime}}\left(  K\right)  $ is an ideal of $P_{\alpha^{\prime}%
}\left(  A\right)  $, we obtain that
\[
P_{\alpha}\left(  K\right)  =P\left(  P_{\alpha^{\prime}}\left(  K\right)
\right)  \subset P\left(  P_{\alpha^{\prime}}\left(  A\right)  \right)
=P_{\alpha}\left(  A\right)  .
\]

By the hypothesis of induction (valid for any $A$), $J:=R_{\alpha^{\prime}%
}\left(  K\right)  \subset R_{\alpha^{\prime}}\left(  A\right)  $ is an ideal
of $A$. Then, as $K$ is an ideal of $\overline{K}$,
\begin{equation}
J\subset R_{\alpha^{\prime}}\left(  \overline{K}\right)  \label{jin}%
\end{equation}
is an ideal of $\overline{K}$. Since $J=K\cap\overline{J}$, the map
$x/J\longmapsto q_{\overline{J}}\left(  x\right)  $ is an injective map from
$K/J$ onto the ideal $q_{\overline{J}}\left(  K\right)  $ of $A/\overline{J}$.
So
\begin{equation}
R_{\alpha}\left(  K\right)  =\left\{  x\in K:q_{\overline{J}}\left(  x\right)
\in R\left(  q_{\overline{J}}\left(  K\right)  \right)  \right\}  \label{j1}%
\end{equation}
whence $R_{\alpha}\left(  K\right)  $ is an ideal of $A$. As $q_{\overline{J}%
}\left(  K\right)  $ is an ideal of $\overline{K}/\overline{J}$, we obtain
that
\begin{equation}
R\left(  q_{\overline{J}}\left(  K\right)  \right)  \subset R\left(
\overline{K}/\overline{J}\right)  . \label{j2}%
\end{equation}
Since $\overline{J}\subset R_{\alpha^{\prime}}\left(  \overline{K}\right)  $
is an ideal of $\overline{K}$ by $\left(  \ref{jin}\right)  $, it follows
that
\begin{equation}
q\left(  R\left(  \overline{K}/\overline{J}\right)  \right)  \subset R\left(
\overline{K}/R_{\alpha^{\prime}}\left(  \overline{K}\right)  \right)
=R_{\alpha}\left(  \overline{K}\right)  /R_{\alpha^{\prime}}\left(
\overline{K}\right)  . \label{j3}%
\end{equation}
for the natural map $q:\overline{K}/\overline{J}\longrightarrow\overline
{K}/R_{\alpha^{\prime}}\left(  \overline{K}\right)  $. It follows from
$\left(  \ref{j1}\right)  $, $\left(  \ref{j2}\right)  $ and $\left(
\ref{j3}\right)  $ that $R_{\alpha}\left(  K\right)  \subset R_{\alpha}\left(
\overline{K}\right)  $. As $\overline{K}$ is a closed ideal of $A$ then
$R_{\alpha}\left(  \overline{K}\right)  \subset R_{\alpha}\left(  A\right)  $
by \cite[Theorem 6.6]{D97}. Therefore $R_{\alpha}\left(  K\right)  \subset
R_{\alpha}\left(  A\right)  $. This completes the proof.
\end{proof}

We outline a part of the proof of Theorem \ref{ovunt} in the form of a
separate statement.

\begin{proposition}
\label{ax4}Let $R,R_{1}$, $P,P_{1}$ be preradicals satisfying Axiom $4$. Then

\begin{enumerate}
\item If $R_{2}$ is defined on $\mathfrak{U}$ by $R_{2}\left(  A\right)
=q_{R_{1}\left(  A\right)  }^{-1}\left(  R\left(  A/R_{1}\left(  A\right)
\right)  \right)  $ then $R_{2}$ is a preradical satisfying Axiom $4$;

\begin{enumerate}
\item If $R,R_{1}$ are under radicals then $R_{2}$ is an under radical;
\end{enumerate}

\item If $P_{2}$ is defined on $\mathfrak{U}$ by $P_{2}\left(  A\right)
=P\left(  P_{1}\left(  A\right)  \right)  $ then $P_{2}$ is a preradical
satisfying Axiom $4$;

\begin{enumerate}
\item If $P,P_{1}$ are over radicals then $P_{2}$ is an over radical.
\end{enumerate}
\end{enumerate}
\end{proposition}

\begin{proof}
This is the step $\alpha\mapsto\alpha+1$ of the transfinite induction applied
in Theorem \ref{ovunt}.
\end{proof}

\begin{theorem}
\label{sera}Let $R$ be an under radical, and let $P$ be an over radical. Then

\begin{enumerate}
\item $\mathbf{Sem}\left(  R^{\ast}\right)  =\mathbf{Sem}\left(  R\right)  $;

\item $\mathbf{\mathbf{Rad}}\left(  P^{\circ}\right)  =\mathbf{\mathbf{Rad}%
}\left(  P\right)  $.
\end{enumerate}
\end{theorem}

\begin{proof}
Let $A$ be an algebra, and let $\left(  R_{\alpha}\right)  $ and $\left(
P_{\alpha}\right)  $ be the convolution and superposition chains of $R$ and
$P$, respectively.

$\left(  1\right)  $ If $R\left(  A\right)  =0$ then $R_{1}\left(  A\right)
=R\left(  A\right)  =0$, whence $R^{\ast}\left(  A\right)  =0$.

$\left(  2\right)  $ If $A=P\left(  A\right)  $ then $P_{1}\left(  A\right)
=P\left(  A\right)  =A$, whence $P^{\circ}\left(  A\right)  =A$.
\end{proof}

\begin{theorem}
\label{isotone}Let $R,R_{1}$ be under radicals, and let $P,P_{1}$ be over
radicals. Then

\begin{enumerate}
\item If $R_{1}\leq R$ then $R_{1}^{\ast}\leq R^{\ast}$;

\item If $P_{1}\leq P$ then $P_{1}^{\circ}\leq P^{\circ}$;

\item If $R\leq P$ then $R^{\ast}\leq P^{\circ}$;

\item If $P\leq R$ then $P^{\circ}\leq R^{\ast}$.
\end{enumerate}
\end{theorem}

\begin{proof}
$\left(  1\right)  $ $R^{\ast}$ is a radical and $R_{1}\leq R\leq R^{\ast}$.
As $R_{1}^{\ast}$ is the smallest radical among radicals $R^{\prime}$ such
that $R_{1}\leq R^{\prime}$ then $R_{1}^{\ast}\leq R^{\ast}$.

$\left(  2\right)  $ is similar to $\left(  1\right)  $: $P_{1}^{\circ}$ is a
radical and $P_{1}^{\circ}\leq P_{1}\leq P$. As $P^{\circ}$ is the largest
radical among radicals $P^{\prime}$ such that $P^{\prime}\leq P$ then
$P_{1}^{\circ}\leq P^{\circ}$.

$\left(  3\right)  $ This is \cite[Theorem 6.11]{D97}, but the following
simple observation somewhat shortens the proof: as $R\leq P$, we have that%
\[
\mathbf{Sem}\left(  P\right)  \subset\mathbf{Sem}\left(  R\right)
=\mathbf{Sem}\left(  R^{\ast}\right)  .
\]
Since $R^{\ast}$ is an over radical then $R^{\ast}\leq P$ by Theorem
\ref{equality}. As $R^{\ast}$ is a radical and $P^{\circ}$ is a largest
radical smaller than or equal to $P$ then $R^{\ast}\leq P^{\circ}$.

$\left(  4\right)  $ is trivial.
\end{proof}

In other words, $\left(  1\right)  $ and $\left(  2\right)  $ of Proposition
\ref{isotone} state that the convolution and superposition procedures are isotone.

\subsection{The closure procedure}

The \textit{closure procedure} $P\longmapsto\overline{P}$ makes a preradical
$P$ on a class $\mathfrak{U}$ (containing $\mathfrak{U}_{\mathrm{n}}$) into a
topological preradical $\overline{P}$ on $\mathfrak{U}_{\mathrm{n}}$ by the
rule
\[
\overline{P}\left(  A\right)  =\overline{P\left(  A\right)  }%
\]
for every $A\in\mathfrak{U}_{\mathrm{n}}$; $\overline{P}$ is called the
\textit{closure} of $P$. Dixon pointed out in \cite[Example 6.4]{D97} that if
$P$ is a radical then the closure $\overline{P}$ may be not a topological
radical, but always is a topological under radical. Recall that the
restrictions of algebraic radicals from $\mathfrak{U}_{\mathrm{a}}$ to a class
of normed algebras are algebraic under radicals. One can state Dixon's result
\cite[Theorem 6.3]{D97} in the following stronger form.

\begin{theorem}
\label{closunder}The closure of the restriction of an algebraic under radical
to a class of normed algebras is a strict topological under radical.
\end{theorem}

\begin{proof}
We only prove the strictness for $\overline{P}$. Let $f:A\longrightarrow B$ be
a continuous isomorphism of normed algebras. Then $f\left(  P\left(  A\right)
\right)  =P\left(  B\right)  $ by definition whence $f\left(  \overline
{P\left(  A\right)  }\right)  \ $is dense in $\overline{P\left(  B\right)  }$.
Therefore $\overline{f\left(  \overline{P}\left(  A\right)  \right)
}=\overline{P}\left(  B\right)  $.
\end{proof}

Starting with a (restricted) algebraic radical $P$, one can take its closure
and then apply the topological convolution procedure. The following theorem
shows that the result will not change if we do the same with the algebraic
convolution of $P$.

\begin{theorem}
\label{csscs}Let $P$ be an algebraic under radical. Then $\overline{P}^{\ast
}=\overline{P^{\ast}}^{\ast}$ and $\overline{P}^{\ast}$ is a strong
topological radical.
\end{theorem}

\begin{proof}
By Theorem \ref{sera}, $\mathbf{Sem}\left(  P\right)  =\mathbf{Sem}\left(
P^{\ast}\right)  $. Then
\begin{align*}
\mathbf{Sem}\left(  \overline{P}^{\ast}\right)   &  =\mathbf{Sem}\left(
\overline{P}\right)  =\mathfrak{U}_{\mathrm{n}}\cap\mathbf{Sem}\left(
P\right)  =\mathfrak{U}_{\mathrm{n}}\cap\mathbf{Sem}\left(  P^{\ast}\right) \\
&  =\mathbf{Sem}\left(  \overline{P^{\ast}}\right)  =\mathbf{Sem}\left(
\overline{P^{\ast}}^{\ast}\right)  .
\end{align*}
As $\overline{P}^{\ast}$and $\overline{P^{\ast}}^{\ast}$ are radicals,
$\overline{P}^{\ast}=\overline{P^{\ast}}^{\ast}$ by Corollary \ref{ineq}.

Let $f:A\longrightarrow B$ be a continuous surjective homomorphism of normed
algebras, and let $\left(  P_{\alpha}\right)  $ be the convolution series of
$\overline{P}$. As $P_{1}=\overline{P}$ is strict by Theorem \ref{closunder},
it is strong. Let us prove by transfinite induction that every $P_{\alpha}$ is strong.

Let $\alpha$ be a limit ordinal, and let all $P_{\alpha^{\prime}}$ be strong
for $\alpha^{\prime}<\alpha$. Then
\begin{align*}
f\left(  P_{\alpha}\left(  A\right)  \right)   &  =f\left(  \overline
{\cup_{\alpha^{\prime}<\alpha}P_{\alpha^{\prime}}\left(  A\right)  }\right)
\subset\overline{f\left(  \cup_{\alpha^{\prime}<\alpha}P_{\alpha^{\prime}%
}\left(  A\right)  \right)  }\\
&  =\overline{\cup_{\alpha^{\prime}<\alpha}f\left(  P_{\alpha^{\prime}}\left(
A\right)  \right)  }\subset\overline{\cup_{\alpha^{\prime}<\alpha}%
P_{\alpha^{\prime}}\left(  B\right)  }=P_{\alpha}\left(  B\right)  .
\end{align*}

The step $\alpha\mapsto\alpha+1$ of the induction is reduced to the
consideration of $\overline{P}$: it suffices to check that
\begin{equation}
f\left(  q_{I}^{-1}\left(  \overline{P}\left(  A/I\right)  \right)  \right)
\subset q_{J}^{-1}\left(  \overline{P}\left(  B/J\right)  \right)  \label{aa}%
\end{equation}
whenever $f\left(  I\right)  \subset J$ where $I=P_{\alpha}\left(  A\right)  $
and $J=P_{\alpha}\left(  B\right)  $. As $\overline{P}$ is strong, $\left(
\ref{aa}\right)  $ is easily checked. Indeed, $f$ induces a continuous
homomorphism $f^{\prime}$ from $A/I$ onto $B/J$ such that $f^{\prime}\circ
q_{I}=q_{J}\circ f$; then
\begin{align*}
f\left(  q_{I}^{-1}\left(  \overline{P}\left(  A/I\right)  \right)  \right)
&  \subset q_{J}^{-1}\left(  q_{J}\left(  f\left(  q_{I}^{-1}\left(
\overline{P}\left(  A/I\right)  \right)  \right)  \right)  \right) \\
&  =q_{J}^{-1}\left(  f^{\prime}\left(  \overline{P}\left(  A/I\right)
\right)  \right)  \subset q_{J}^{-1}\left(  \overline{P}\left(  B/J\right)
\right)  .
\end{align*}
As a consequence, $\overline{P}^{\ast}$ is strong.
\end{proof}

We apply the above theorem to classical algebraic radicals related to the
property of nilpotency. An algebra $A$ is called \textit{nilpotent} if there
is an integer $n>0$ such that $a_{1}\cdots a_{n}=0$ for every $a_{1}%
,\ldots,a_{n}\in A$, \textit{locally nilpotent} if every finite subset of $A$
generates a nilpotent subalgebra, \textit{nil }if every element of $A$ is
nilpotent. All these notions transfer to ideals.

An algebra $A$ is called \textit{semiprime} if it has no non-zero ideals with
zero square. For semiprimeness of $A$ it is sufficient to have no non-zero
left or right ideals with zero square \cite[Lemma 30.4]{BD73}.

One can define the \textit{Baer }(\textit{prime} or \textit{lower
nil-})\textit{ radical} $\mathfrak{P}_{\beta}\left(  A\right)  $ of $A$ as the
smallest ideal $I$ of $A$ with semiprime quotient $A/I$; the \textit{Levitzki
}(or \textit{locally nilpotent}) and\textit{ K\"{o}te }(or \textit{upper
nil-}) \textit{radicals} $\mathfrak{P}_{\lambda}\left(  A\right)  $ and
$\mathfrak{P}_{\kappa}\left(  A\right)  $ are defined as the largest locally
nilpotent and nil ideals of $A$, respectively. They all are hereditary
radicals and
\begin{equation}
\mathfrak{P}_{\beta}<\mathfrak{P}_{\lambda}<\mathfrak{P}_{\kappa}.
\label{compnil}%
\end{equation}

Let $\mathcal{P}_{\beta}=\overline{\mathfrak{P}_{\beta}}^{\ast}$,
$\mathcal{P}_{\lambda}=\overline{\mathfrak{P}_{\lambda}}^{\ast}$ and
$\mathcal{P}_{\kappa}=\overline{\mathfrak{P}_{\kappa}}^{\ast}$. They are
strong topological radicals by Theorems \ref{closunder} and \ref{ovunt}, and
are called the \textit{closed-Baer}, \textit{closed-Levitzki} and
\textit{closed-K\"{o}te} \textit{radical}, respectively.

Let $A$ be an algebra in $\mathfrak{U}_{a}$ and let $\Sigma_{\beta}\left(
A\right)  $ be the sum of all nilpotent ideals of $A$. It is clear that
$\Sigma_{\beta}$ is an algebraic under radical on $\mathfrak{U}_{\mathrm{a}}$;
by construction, the (algebraic) Baer radical $\mathfrak{P}_{\beta}$ is equal
to $\Sigma_{\beta}^{\ast}$.

Theorem \ref{csscs} yields

\begin{corollary}
\label{baer}$\mathcal{P}_{\beta}=\overline{\Sigma_{\beta}}^{\ast}$.
\end{corollary}

\begin{theorem}
\label{classical}$\mathcal{P}_{\beta}<\mathcal{P}_{\lambda}<$ $\mathcal{P}%
_{\kappa}$ on $\mathfrak{U}_{\mathrm{n}}$ and $\mathcal{P}_{\beta}%
=\mathcal{P}_{\lambda}=$ $\mathcal{P}_{\kappa}$ on $\mathfrak{U}_{\mathrm{b}}$.
\end{theorem}

\begin{proof}
As all procedures are isotone, the non-strict inequalities follows from
$\left(  \ref{compnil}\right)  $. Also, $\mathcal{P}_{\beta}<\mathcal{P}%
_{\lambda}$ follows from \cite[Proposition 3.3 and Corollary 9.4]{D97}, and
for the proof of $\mathcal{P}_{\lambda}<$ $\mathcal{P}_{\kappa}$ it is
sufficient to point out a normed nil algebra with zero Levitzki radical; for
instance, it was done in \cite{T85} as some realization of Golod's solution of
the Bernside problem for algebras.

By \cite{D73},
\[
\mathfrak{P}_{\beta}=\mathfrak{P}_{\lambda}=\mathfrak{P}_{\kappa}%
=\Sigma_{\beta}%
\]
on $\mathfrak{U}_{\mathrm{b}}$. Therefore $\mathcal{P}_{\beta}=\mathcal{P}%
_{\lambda}=$ $\mathcal{P}_{\kappa}$ on $\mathfrak{U}_{\mathrm{b}}$.
\end{proof}

The common radical on $\mathfrak{U}_{\mathrm{b}}$ is denoted by $\mathcal{P}%
_{\mathrm{nil}}$ and called the \textit{closed-nil radical}. One can consider
$\mathcal{P}_{\beta}$, $\mathcal{P}_{\lambda}$, $\mathcal{P}_{\kappa}$ as
topological extensions of $\mathcal{P}_{\mathrm{nil}}$ to normed algebras.

Let $A$ be an algebra, and let $\Sigma_{\mathrm{hf}}\left(  A\right)  $ be the
sum of all bifinite ideals of $A$. If $A$ is normed, let $\Sigma_{\mathrm{hc}%
}\left(  A\right)  $ be the closure of sum of all bicompact ideals of $A$. It
is easy to see that $\Sigma_{\mathrm{hf}}$ and $\Sigma_{\mathrm{hc}}$ are
under radicals.

In the following theorem we apply the typical method of comparing radicals.

\begin{theorem}
\label{hc}$\mathfrak{R}_{\mathrm{hf}}=\Sigma_{\mathrm{hf}}^{\ast}$ and
$\mathcal{R}_{\mathrm{hf}}=\overline{\mathfrak{R}_{\mathrm{hf}}}^{\ast
}=\overline{\Sigma_{\mathrm{hf}}}^{\ast}\leq\mathcal{R}_{\mathrm{hc}}=$
$\Sigma_{\mathrm{hc}}^{\ast}$.
\end{theorem}

\begin{proof}
Let $B$ be an algebra. As every bifinite ideal is hypofinite, $\Sigma
_{\mathrm{hf}}^{\ast}\left(  B\right)  \subset\mathfrak{R}_{\mathrm{hf}}$.
Since $\mathfrak{R}_{\mathrm{hf}}$ is a radical, $\Sigma_{\mathrm{hf}}^{\ast
}\leq\mathfrak{R}_{\mathrm{hf}}$. As every $\Sigma_{\mathrm{hf}}^{\ast}%
$-semisimple algebra is $\mathfrak{R}_{\mathrm{hf}}$-semisimple,
$\mathfrak{R}_{\mathrm{hf}}\leq\Sigma_{\mathrm{hf}}^{\ast}$. One can show
similarly that $\mathcal{R}_{\mathrm{hc}}=$ $\Sigma_{\mathrm{hc}}^{\ast}$.

Let $A$ be a normed algebra and $J=\mathfrak{R}_{\mathrm{hf}}\left(  A\right)
$. Let $\left(  J_{\alpha}\right)  _{\alpha\leq\gamma}$be the increasing
transfinite chain of ideals of $A$ constructed in Theorem \ref{hf}$\left(
5\right)  $ with $J_{0}=0$ and $J_{\gamma}=J$. As each $J_{\alpha+1}%
/J_{\alpha}$ is bifinite, $\overline{J_{\alpha+1}}/\overline{J_{\alpha}}$ is
approximable, so $\overline{J}$ is closed-hypofinite (see \cite[Proposition
3.48]{TR2}). As $\mathcal{R}_{\mathrm{hf}}\left(  A\right)  $ is the largest
closed-hypofinite ideal of $A$ then $\overline{J}\subset\mathcal{R}%
_{\mathrm{hf}}\left(  A\right)  $. Therefore $\overline{\mathfrak{R}%
_{\mathrm{hf}}}\leq\mathcal{R}_{\mathrm{hf}}$, whence $\overline
{\mathfrak{R}_{\mathrm{hf}}}^{\ast}\leq\mathcal{R}_{\mathrm{hf}}$. But every
$\overline{\mathfrak{R}_{\mathrm{hf}}}^{\ast}$-semisimple algebra has no
non-zero finite rank elements and is therefore $\mathcal{R}_{\mathrm{hf}}%
$-semisimple, whence $\mathcal{R}_{\mathrm{hf}}\leq\overline{\mathfrak{R}%
_{\mathrm{hf}}}^{\ast}$ by Theorem \ref{equality}.

By Theorem \ref{csscs}, $\mathcal{R}_{\mathrm{hf}}=\overline{\Sigma
_{\mathrm{hf}}}^{\ast}$. As $\overline{\Sigma_{\mathrm{hf}}}\leq
\Sigma_{\mathrm{hc}}$ then $\mathcal{R}_{\mathrm{hf}}=\overline{\Sigma
_{\mathrm{hf}}}^{\ast}\leq\Sigma_{\mathrm{hc}}^{\ast}=\mathcal{R}%
_{\mathrm{hc}}$ by Theorem \ref{isotone}.
\end{proof}

\subsection{The regularization procedure\label{regular}}

Let $P$ be a hereditary topological radical defined on some class
$\mathfrak{U}$ of normed algebras such that $\mathfrak{U}_{\mathrm{b}}%
\subset\mathfrak{U}$. By \cite[Theorem 2.21]{TR1}, the map $P^{r}$ defined by
\begin{equation}
P^{r}\left(  A\right)  =A\cap P\left(  \widehat{A}\right)  \label{rproc}%
\end{equation}
for each $A\in\mathfrak{U}_{\mathrm{n}}$, where $\widehat{A}$ is a completion
of $A$, is a hereditary topological radical on $\mathfrak{U}_{\mathrm{n}}$.
The map $P\longmapsto P^{r}$ is called the \textit{regularization procedure};
it is clear that%
\[
P^{rr}=P^{r}.
\]
A hereditary topological radical $P$ on $\mathfrak{U}_{\mathrm{n}}$ is called
\textit{regular} if $P=P^{r}$. One of the important reasons for the
consideration of over radicals is the fact \cite{TR1} that preradicals
obtained from topological radicals by the regular procedure are over radicals.

\begin{theorem}
The radicals $\mathcal{R}_{\mathrm{cq}}$ and $\mathcal{R}_{t}$ are regular.
\end{theorem}

\begin{proof}
Let $A$ be a normed algebra. By definition, $A\cap\mathcal{R}_{\mathrm{cq}%
}\left(  \widehat{A}\right)  \subset\mathcal{R}_{\mathrm{cq}}\left(  A\right)
$. On the other hand, the completion of a compactly quasinilpotent algebra is
again compactly quasinilpotent \cite[Lemma 4.11]{TR1}. So $\overline
{\mathcal{R}_{\mathrm{cq}}\left(  A\right)  }^{\left(  \widehat{A}\right)  }$
is compactly quasinilpotent, and as it is an ideal of $\widehat{A}$ then
$\overline{\mathcal{R}_{\mathrm{cq}}\left(  A\right)  }^{\left(
\widehat{A}\right)  }\subset\mathcal{R}_{\mathrm{cq}}\left(  \widehat{A}%
\right)  $. But $A\cap\overline{\mathcal{R}_{\mathrm{cq}}\left(  A\right)
}^{\left(  \widehat{A}\right)  }=\mathcal{R}_{\mathrm{cq}}\left(  A\right)  $
because $\mathcal{R}_{\mathrm{cq}}\left(  A\right)  $ is a closed ideal of
$A$. Thus $\mathcal{R}_{\mathrm{cq}}$ is regular.

The proof for $\mathcal{R}_{t}$ is similar. For the other proof we refer to
Proposition \ref{regt}.
\end{proof}

\begin{theorem}
The radicals $\mathcal{R}_{\beta}$ and $\mathcal{R}_{\mathrm{hf}}$ are not regular.
\end{theorem}

\begin{proof}
Let $B$ be the Banach space $L^{1}[0,1]$ considered as a Banach algebra with
multiplication $(f\ast g)(t)=\int_{0}^{t}f(s)g(t-s)ds.$ It can be shown that
nilpotent elements are dense in $B$ (and therefore all elements of $B$ are
quasinilpotent since $B$ is commutative): indeed, it follows from the
Titchmarsh convolution theorem that a function $f\in B$ is nilpotent if and
only if it vanishes a.e. on some interval $(0,a)$. Therefore $\mathcal{R}%
_{\beta}(B)=B$. Moreover, $f$ is an element of finite rank if it is nilpotent
of index 2 (equivalently if $f(t)=0$ for $t\in(0,1/2)$), so $\mathcal{R}%
_{\mathrm{hf}}(B)\neq0$. In fact, $\mathcal{R}_{\mathrm{hf}}(B)=B$.

Let $A$ be the subalgebra of $B$ that consists of polynomials. It is dense in
$B$ so $B$ can be considered as the completion of $A$. Clearly $A$ has no
nilpotent elements (a non-zero polynomial cannot vanish on an interval), so
$\mathcal{R}_{\beta}(A)=0$. As a quasinilpotent finite rank element is
nilpotent, $A$ has no finite rank elements and $R_{\mathrm{hf}}(A)=0$.
\end{proof}

\begin{problem}
Is the hypocompact radical $\mathcal{R}_{\mathrm{hc}}$ regular?
\end{problem}

It is clear that the regular Jacobson radical $\operatorname{Rad}^{r}$ defined
by $\left(  \ref{rjr}\right)  $ is a regular radical. Note that%
\begin{equation}
\mathcal{R}_{\mathrm{cq}}\leq\mathcal{R}_{t}\leq\operatorname{Rad}%
^{r}<\operatorname{rad}_{\mathrm{n}}<\operatorname{rad}_{\mathrm{b}};
\label{comprad}%
\end{equation}
also $\operatorname{rad}=\operatorname{rad}_{\mathrm{n}}$ on $\mathfrak{U}%
_{\mathrm{q}}$ and $\operatorname{rad}=\operatorname{rad}_{\mathrm{b}}$ on
$\mathfrak{U}_{\mathrm{q}_{\mathrm{b}}}$ \cite[Theorem 2.18]{TR1}.

Consider the \textit{closed-Jacobson radical} $\overline{\operatorname{rad}%
}^{\ast}$; then $\overline{\operatorname{rad}}^{\ast}=\operatorname{Rad}$ on
$\mathfrak{U}_{\mathrm{b}}$. Define also the map $\Pi_{\mathrm{pc}}$ on normed
algebras by
\begin{equation}
\Pi_{\mathrm{pc}}\left(  A\right)  =\cap\left\{  I\in\operatorname{Prim}%
\left(  A^{1}\right)  :I\text{ is closed in }A\right\}  . \label{pc}%
\end{equation}

\begin{remark}
The map $A\longmapsto\left\{  I\in\operatorname{Prim}\left(  A^{1}\right)
:I\text{ is closed in }A\right\}  $ is not a primitive map because, for an
ideal $J$ of a normed algebra $A$, there can exist a non-closed primitive
ideal $I$ of $A$ such that $I\cap J$ is a closed primitive ideal of $J$.
\end{remark}

\begin{lemma}
$\Pi_{\mathrm{pc}}$ is a topological over radical on normed algebras, and
$\overline{\operatorname{rad}}\leq\Pi_{\mathrm{pc}}$.
\end{lemma}

\begin{proof}
Let $A$ be a normed algebra, and let $I$ be an ideal of $A$. If $J\in
\operatorname{Prim}\left(  A^{1}\right)  $ is closed then $I\cap
J\in\operatorname{Prim}\left(  I^{1}\right)  $ is closed in $I$, whence
$\Pi_{\mathrm{pc}}\left(  I\right)  \subset\Pi_{\mathrm{pc}}\left(  A\right)
$ is an ideal of $A$ by $\left(  \ref{p4}\right)  $.

Let $B$ be a normed algebra, and let $f:A\longrightarrow B$ be a morphism. Let
$\pi\in\operatorname{Irr}\left(  B^{1}\right)  $ and $K:=\ker\pi$ is closed.
Then $\pi\circ f\in\operatorname{Irr}\left(  A^{1}\right)  $ and $\ker\left(
\pi\circ f\right)  =f^{-1}\left(  K\right)  $ is closed.

Let $\mathcal{F}_{A}$ be the set of all closed primitive ideals of $A^{1}$.
Assume that $I\subset\Pi_{\mathrm{pc}}\left(  A\right)  $ is closed. Then
$I\subset J$ for every $J\in\mathcal{F}_{A}$ and%
\begin{align*}
\Pi_{\mathrm{pc}}\left(  A\right)  /I  &  =\cap\left\{  J:J\in\mathcal{F}%
_{A}\right\}  /I=\cap\left\{  J/I:J\in\mathcal{F}_{A}\right\} \\
&  =\cap\left\{  J/I:J/I\in\mathcal{F}_{A/I}\right\}  =\Pi_{\mathrm{pc}%
}\left(  A/I\right)
\end{align*}
by $\left(  \ref{p4}\right)  $ and $\left(  \ref{p5}\right)  $. Hence
$\Pi_{\mathrm{pc}}$ is an over radical.

Furthermore, $\overline{\operatorname{rad}}\left(  A\right)  =\overline
{\cap_{I\in\operatorname{Prim}\left(  A^{1}\right)  }I}\subset\Pi
_{\mathrm{pc}}\left(  A\right)  $ in virtue of $\left(  \ref{pc}\right)  $.
\end{proof}

As a consequence, $\Pi_{\mathrm{pc}}^{\circ}$ is a topological radical which
also extends $\operatorname{Rad}$ to normed algebras; it is called the
\textit{primitively closed Jacobson radical.}

\begin{theorem}
$\overline{\operatorname{rad}}^{\ast}<\Pi_{\mathrm{pc}}^{\circ}\leq
\operatorname{rad}_{\mathrm{n}}$; $\overline{\operatorname{rad}}^{\ast}$and
$\operatorname{Rad}^{r}$ are not comparable.
\end{theorem}

\begin{proof}
Indeed, it is clear that $\overline{\operatorname{rad}}\left(  A\right)
\subset\Pi_{\mathrm{pc}}\left(  A\right)  \subset\operatorname{rad}%
_{\mathrm{n}}\left(  A\right)  $ for every normed algebra $A$. Then
$\overline{\operatorname{rad}}^{\ast}\leq\Pi_{\mathrm{pc}}^{\circ}%
\leq\operatorname{rad}_{\mathrm{n}}$ by Theorems \ref{equality} and
\ref{isotone}.

Let $a$ be a quasinilpotent element of a Banach algebra and $a^{n}\neq0$ for
every $n>0$. Let $A$ be the algebra generated by $a$, and $B=\overline{A}$.
Then $\operatorname{rad}\left(  A\right)  =0$, but $\operatorname{rad}\left(
B\right)  =B$. Hence $\overline{\operatorname{rad}}^{\ast}\left(  A\right)
=\overline{\operatorname{rad}}\left(  A\right)  =0$, but $\operatorname{Rad}%
^{r}\left(  A\right)  =A\cap\operatorname{rad}\left(  B\right)  =A$.

Also, $A$ has no closed primitive ideals in $\operatorname{Prim}\left(
A^{1}\right)  $ besides of $A$. Indeed, a primitive ideal $I$ is the kernel of
multiplicative functional $f$ on $A^{1}$. If $I\neq A$ is closed then
$f\left(  a\right)  \neq0$, $f$ is continuous and $f\left(  a\right)  $ in the
spectrum $\sigma_{B}\left(  a\right)  $, a contradiction. Therefore
$\Pi_{\mathrm{pc}}^{\circ}\left(  A\right)  =$ $\Pi_{\mathrm{pc}}\left(
A\right)  =A$ and $\overline{\operatorname{rad}}^{\ast}<\Pi_{\mathrm{pc}%
}^{\circ}$.

By \cite[Example 9.3]{D97}, there are Banach algebras $C$, $D$ and an
injective embedding $\phi:C\longrightarrow D$ with the dense image such that
$\operatorname{rad}\left(  C\right)  =C$ and $\operatorname{rad}\left(
D\right)  =0$. Let $E=\phi\left(  C\right)  $. As $\phi$ and $\phi^{-1}$ are
algebraic morphisms for appropriate algebras then $\operatorname{rad}\left(
E\right)  =E$ and $\overline{\operatorname{rad}}^{\ast}\left(  E\right)  =E$,
but $\operatorname{Rad}^{r}\left(  E\right)  =E\cap\operatorname{rad}\left(
D\right)  =0$. So the observations with algebras $A$ and $E$ show that
$\overline{\operatorname{rad}}^{\ast}$and $\operatorname{Rad}^{r}$ are not comparable.
\end{proof}

As we will see in Theorem \ref{class}, the closed-Jacobson radical
$\overline{\operatorname{rad}}^{\ast}$ is hereditary.

\section{Operations with radicals}

Operations are multiplace procedures that act on a class of ideal maps.

\subsection{Supremum and infimum}

Let $\mathcal{F}$ be a non-empty family (set, class) of radicals. Then the
maps $\mathrm{H}_{\mathcal{F}}$ and $\mathrm{B}_{\mathcal{F}}$\textrm{
}defined by
\begin{align*}
\mathrm{H}_{\mathcal{F}}\left(  A\right)   &  =\sum_{P\in\mathcal{F}}P\left(
A\right)  \text{\quad\quad\lbrack}\mathrm{H}_{\mathcal{F}}\left(  A\right)
=\overline{\sum_{P\in\mathcal{F}}P\left(  A\right)  }\text{]}\\
\mathrm{B}_{\mathcal{F}}\left(  A\right)   &  =\cap_{P\in\mathcal{F}}P\left(
A\right)
\end{align*}
are the under and over radicals, respectively (see \cite[Lemma 3.2]{TR3}). We
extend the statement to under/over radicals as follows.

\begin{theorem}
\label{dix0}

\begin{enumerate}
\item If $\mathcal{F}_{u}$ is a family of under radicals then $\mathrm{H}%
_{\mathcal{F}_{u}}$ is an under radical, and
\[
\mathbf{Sem}\left(  \mathrm{H}_{\mathcal{F}_{u}}\right)  =\cap_{P\in
\mathcal{F}_{u}}\mathbf{Sem}\left(  P\right)  .
\]

\item If $\mathcal{F}_{o}$ is a family of over radicals then $\mathrm{B}%
_{\mathcal{F}_{o}}$ is an over radical, and
\[
\mathbf{\mathbf{Rad}}\left(  \mathrm{B}_{\mathcal{F}_{o}}\right)  =\cap
_{P\in\mathcal{F}_{o}}\mathbf{\mathbf{Rad}}\left(  P\right)  .
\]

\end{enumerate}
\end{theorem}

\begin{proof}
The equalities for radical and semisimple classes follow from the definitions
of $\mathrm{H}_{\mathcal{F}_{u}}$ and $\mathrm{B}_{\mathcal{F}_{o}}$. It is
clear that $\mathrm{H}_{\mathcal{F}_{u}}$ and $\mathrm{B}_{\mathcal{F}_{o}}$
satisfy Axioms $1$ and $4$.

Let $A$ be an algebra. If $P\in\mathcal{F}_{u}$ then $P\left(  A\right)  $ is
an ideal in $\mathrm{H}_{\mathcal{F}_{u}}\left(  A\right)  $ and therefore
\[
P\left(  A\right)  =P\left(  P\left(  A\right)  \right)  \subset P\left(
\mathrm{H}_{\mathcal{F}_{u}}\left(  A\right)  \right)  \subset\mathrm{H}%
_{\mathcal{F}_{u}}\left(  \mathrm{H}_{\mathcal{F}_{u}}\left(  A\right)
\right)
\]
whence $\mathrm{H}_{\mathcal{F}_{u}}\left(  A\right)  \subset\mathrm{H}%
_{\mathcal{F}_{u}}\left(  \mathrm{H}_{\mathcal{F}_{u}}\left(  A\right)
\right)  $. This proves that $\mathrm{H}_{\mathcal{F}_{u}}$ is an under radical.

Let $I=\mathrm{B}_{\mathcal{F}_{o}}\left(  A\right)  $ and $J=q_{I}%
^{-1}\left(  \mathrm{B}_{\mathcal{F}_{o}}\left(  A/I\right)  \right)  $ where
$q_{I}:A\longrightarrow A/I$ is the standard quotient map. Let $R\in
\mathcal{F}_{o}$; then $R\left(  A\right)  \subset I$ and there is an [open
continuous] epimorphism $p:A/I\longrightarrow A/R\left(  A\right)  $ such that
$q=p\circ q_{I}$ where $q:A\longrightarrow A/R\left(  A\right)  $ is the
standard quotient map. Then
\[
q\left(  J\right)  =\left(  p\circ q_{I}\right)  \left(  J\right)  =p\left(
\mathrm{B}_{\mathcal{F}_{o}}\left(  A/I\right)  \right)  \subset p\left(
R\left(  A/I\right)  \right)  \subset R\left(  A/R\left(  A\right)  \right)
=0,
\]
whence $J\subset R\left(  A\right)  $ for every $R\in\mathcal{F}_{o}$. Then
$J\subset I$ and $\mathrm{B}_{\mathcal{F}_{o}}\left(  A/I\right)
=q_{I}\left(  J\right)  =0$. This means that $\mathrm{B}_{\mathcal{F}_{o}}$ is
an over radical.
\end{proof}

From the viewpoint of the order, $\mathrm{H}_{\mathcal{F}_{u}}$ is the
supremum of $\mathcal{F}_{u}$ in the class of under radicals and
$\mathrm{B}_{\mathcal{F}_{o}}$ is the infimum of $\mathcal{F}_{o}$ in the
class of over radicals.

In general, for a family $\mathcal{F}$ of preradicals satisfying Axiom $4$,
let us define the ideal maps $\vee\mathcal{F}$ and $\wedge\mathcal{F}$ by
\[
\vee\mathcal{F}=\mathrm{H}_{\mathcal{F}}^{\ast}\;\text{ and }\;\wedge
\mathcal{F}=\mathrm{B}_{\mathcal{F}}^{\circ},
\]
respectively; if $\mathcal{F}$ is $\left\{  P_{1},P_{2}\right\}  $ then
instead we write $P_{1}\vee P_{2}$ or $P_{1}\wedge P_{2}$, respectively.

\begin{remark}
It is easy to check (arguing as in Theorems \ref{dix0} and \ref{ovunt}) that,
for a family $\mathcal{F}$ of preradicals satisfying Axiom $4$, $\vee
\mathcal{F}$ determines the \textit{smallest over} \textit{radical} larger
than or equal to each $R\in\mathcal{F}$ and $\wedge\mathcal{F}$ determines the
largest \textit{under} \textit{radical} smaller than or equal to each
$P\in\mathcal{F}$.
\end{remark}

\begin{corollary}
The operations $\vee$ and $\wedge$ produce supremum and infimum of a family of
radicals in the class of radicals.
\end{corollary}

\begin{proof}
Let $\mathcal{F}$ be a family of radicals. Let $P$ be a radical such that
$R\leq P$ for every $R\in\mathcal{F}$. Then $\mathrm{H}_{\mathcal{F}}\leq P$,
whence $\vee\mathcal{F}\leq P$ by Theorem \ref{dix0} and Proposition
\ref{isotone}. As $R\leq\vee\mathcal{F}$ for every $R\in\mathcal{F}$, we
conclude that $\vee\mathcal{F}$ is the supremum of $\mathcal{F}$.

Similarly, $\wedge$ $\mathcal{F}$ is the infimum of $\mathcal{F}$.
\end{proof}

Taking into account Theorems \ref{equality} and \ref{sera}, one can
reformulate Theorem \ref{ovunt} as follows.

\begin{theorem}
\label{dix}Let $R$ be an under radical, and let $P$ be an over radical. Then

\begin{enumerate}
\item $R^{\ast}$ is the only radical $S$ such that $\mathbf{Sem}\left(
S\right)  =\mathbf{Sem}\left(  R\right)  $;

\item $P^{\circ}$ is the only radical $T$ such that $\mathbf{\mathbf{Rad}%
}\left(  T\right)  =\mathbf{\mathbf{Rad}}\left(  P\right)  $.
\end{enumerate}
\end{theorem}

In assumptions of Theorem \ref{dix}, we also recall that
\[
\mathbf{\mathbf{Rad}}\left(  R\right)  \subset\mathbf{\mathbf{Rad}}\left(
R^{\ast}\right)  \text{ and }\mathbf{Sem}\left(  P\right)  \subset
\mathbf{Sem}\left(  P^{\circ}\right)  .
\]

\begin{corollary}
Let $V$ be a class of algebras. Then there exist the largest radical $S$ such
that $\mathbf{Sem}\left(  S\right)  \supset V$, and the smallest radical $T$
such that $\mathbf{\mathbf{Rad}}\left(  T\right)  \supset V$.

\begin{enumerate}
\item $V=\mathbf{Sem}\left(  S\right)  $ if and only if there is a family
$\mathcal{F}_{1}$ of under radicals such that $\mathbf{Sem}\left(
\mathrm{H}_{\mathcal{F}_{1}}\right)  =V$;

\item $V=\mathbf{\mathbf{Rad}}\left(  T\right)  $ if and only if there is a
family $\mathcal{F}_{2}$ of over radicals such that $\mathbf{\mathbf{Rad}%
}\left(  \mathrm{B}_{\mathcal{F}_{2}}\right)  =V$.
\end{enumerate}
\end{corollary}

\begin{proof}
Let $\mathcal{F}_{1}$ be the family of all under radicals $P$ such that
$\mathbf{Sem}\left(  P\right)  \supset V$, and let $\mathcal{F}_{2}$ be the
family of all over radicals $P$ such that $\mathbf{\mathbf{Rad}}\left(
P\right)  \supset V$. By Theorems \ref{dix} and \ref{dix0}, there exist
radicals $S$ and $T$ such that $\mathbf{Sem}\left(  S\right)  =\cap
_{P\in\mathcal{F}_{1}}\mathbf{Sem}\left(  P\right)  \supset V$ and
$\mathbf{\mathbf{Rad}}\left(  T\right)  =\cap_{P\in\mathcal{F}_{2}%
}\mathbf{\mathbf{Rad}}\left(  P\right)  \supset V$.
\end{proof}

For instance, let $W_{\rho}$ be the class of all normed algebras on which the
usual spectral radius $a\longmapsto\rho\left(  a\right)  $ is continuous; this
class contains all commutative normed algebras and Banach algebras whose
elements have at most countable spectrum. Then there exists a largest
topological radical $S_{\rho}$ such that $W_{\rho}\subset\mathbf{Sem}\left(
S_{\rho}\right)  $, but it is a problem whether there is a largest topological
radical $T$ for which every $T$-radical normed algebra has the property of
spectral radius continuity, i.e. lies in $W_{\rho}$.

Let $V$ be a class of algebras. For being a semisimple or radical class it is
necessary that $V$ is stable under extensions.

\begin{theorem}
\label{ext}Let\ $A$ be an algebra, and let $I$ be a [closed] ideal of $A$. Then

\begin{enumerate}
\item If $P$ is an under radical and $I$, $A/I$ are $P$-semi\-simp\-le then
$A$ is $P$-semi\-simp\-le;

\item If $P$ is an over radical and $I$, $A/I$ are $P$-radical then $A$ is $P$-radical.
\end{enumerate}
\end{theorem}

\begin{proof}
$\left(  1\right)  $ As $q_{I}\left(  P\left(  A\right)  \right)  \subset
P\left(  A/I\right)  =0$ then $P\left(  A\right)  \subset I$, whence $P\left(
A\right)  $ is an ideal of $I$ and $P\left(  A\right)  =P\left(  P\left(
A\right)  \right)  \subset P\left(  I\right)  =0$.

$\left(  2\right)  $ As $I\subset P\left(  A\right)  $ then there is a
morphism $p:A/I\longrightarrow A/P\left(  A\right)  $ such that $q_{P\left(
A\right)  }=p\circ q_{I}$ where $q_{P\left(  A\right)  }$ and $q_{I}$ are
standard quotient maps. Then
\[
A/P\left(  A\right)  =p\left(  A/I\right)  =p\left(  P\left(  A/I\right)
\right)  \subset P\left(  A/P\left(  A\right)  \right)  =0,
\]
whence $A=P\left(  A\right)  $.
\end{proof}

In other words, the class $\mathbf{Rad}(P)$ (respectively, $\mathbf{Sem}(P)$)
is stable under extensions if $P$ is an over radical (respectively, under
radical). We will need also partially converse statements.

\begin{lemma}
\label{conv-ext}Let $R$ be an under radical, and let $P$ be an over radical. Then

\begin{enumerate}
\item If $\mathbf{Rad}(R)$ is stable under extensions then $R$ is a radical;

\item If $\mathbf{Sem}(P)$ is stable under extensions then $P$ is a radical.
\end{enumerate}
\end{lemma}

\begin{proof}
$\left(  1\right)  $ Let $J=q_{R(A)}^{-1}(R(A/R(A)))$. Then the ideal $J$ is
an extension of $R(A)$ by $R(A/R(A))$: both algebras are $R$-radical because
$R$ is an under radical. By our assumption, $J$ is a $R$-radical ideal of $A$.
Therefore $J\subset R(A)$ which means that $R\left(  A/R(A)\right)  =0$, i.e.,
$R$ is a radical.

$\left(  2\right)  $ Let $J=P\left(  P(A)\right)  $. Then the algebra $A/J$ is
an extension of $P(A)/J$ by $(A/J)/(P(A)/J)\cong A/P(A)$. Both algebras are
$P$-semisimple because $P$ is an over radical. By our assumption, $A/J$ is
$P$-semisimple. Then%
\[
P(A)/J=q_{J}\left(  P(A)\right)  \subset P(A/J)=0,
\]
whence $J=P(A)$, i.e., $P$ is a radical.
\end{proof}

We apply this result to preradicals $\overline{\Pi_{\Omega_{\mathrm{n}}}}$ and
$\overline{\Pi_{\Omega_{\mathrm{b}}}}$, generated by the primitive maps
$\Omega_{\mathrm{n}}:=$ $\operatorname{Prim}\backslash\operatorname{Prim}%
_{\mathrm{n}}$ and $\Omega_{\mathrm{b}}:=$ $\operatorname{Prim}\backslash
\operatorname{Prim}_{\mathrm{b}}$ on $\mathfrak{U}_{\mathrm{n}}$ (see Example
\ref{pm5}).

\begin{theorem}
$\overline{\Pi_{\Omega_{\mathrm{n}}}}$ and $\overline{\Pi_{\Omega_{\mathrm{b}%
}}}$ are radicals on $\mathfrak{U}_{\mathrm{n}}$; for every normed algebra
$A$, $\overline{\Pi_{\Omega_{\mathrm{n}}}}\left(  A\right)  $ is the largest
$Q$-ideal of $A$ and $\overline{\Pi_{\Omega_{\mathrm{b}}}}\left(  A\right)  $
is the largest $Q_{\mathrm{b}}$-ideal of $A$.
\end{theorem}

\begin{proof}
As $\Pi_{\Omega_{\mathrm{n}}}$ is a hereditary preradical with topological
morphisms then $\overline{\Pi_{\Omega_{\mathrm{n}}}}$ is an under radical by
Theorem \ref{closunder}. Since $\Pi_{\Omega_{\mathrm{n}}}\left(  A\right)  $
is a $Q$-algebra, $\overline{\Pi_{\Omega_{\mathrm{n}}}}\left(  A\right)
:=\overline{\Pi_{\Omega_{\mathrm{n}}}\left(  A\right)  }$ is a $Q$-algebra.
Taking into account that every $Q$-algebra is a $\Pi_{\Omega_{\mathrm{n}}}%
$-radical algebra by definition, and the class of all $Q$-algebras is stable
under extensions (see for instance \cite[Theorem 2.5]{TR1}), we conclude that
$\overline{\Pi_{\Omega_{\mathrm{n}}}}$ is a radical, by Lemma \ref{conv-ext}.
It is clear that $\overline{\Pi_{\Omega_{\mathrm{n}}}}\left(  A\right)  $ is
the largest $Q$-ideal of $A$.

A similar argument gives the result for $\overline{\Pi_{\Omega_{\mathrm{b}}}}$.
\end{proof}

\subsection{Superposition and convolution operations}

The \textit{superposition }of preradicals $P$ and $R$ satisfying Axiom $4$ is
defined by the usual rule:
\[
\left(  P\circ R\right)  \left(  A\right)  =P\left(  R\left(  A\right)
\right)
\]
for each algebra $A$. This map is a preradical satisfying Axiom $4$ by
Proposition \ref{ax4}. The superposition operation is clearly associative.

Let $A$ be an algebra, and let $P$ and $R$ be ideal maps. For a [closed] ideal
$I$ of $A$, define the ideal $P\ast I$ of $A$ by
\[
P\ast I:=q_{I}^{-1}\left(  P\left(  A/I\right)  \right)  ,
\]
where $q_{I}:A\longrightarrow A/I$ is the standard quotient map; in particular
$P\ast0=P\left(  A\right)  $, $P\ast P\left(  A\right)  =q_{P\left(  A\right)
}^{-1}\left(  P\left(  A/P\left(  A\right)  \right)  \right)  $ and $P\ast
A=A$. Sometimes we will write $\left(  P\ast I;A\right)  $ instead of $P\ast
I$ to avoid a misunderstanding.

Define the \textit{convolution } $P{\ast R}$ of ideal maps $P,R$ by
\begin{equation}
(P{\ast R})(A):=P\ast R\left(  A\right)  =q_{R\left(  A\right)  }^{-1}P\left(
A/R\left(  A\right)  \right)  . \label{convfor}%
\end{equation}
Proposition \ref{ax4} shows that if $P$ and $R$ are preradicals satisfying
Axiom $4$ then $P{\ast R}$ is a preradical satisfying Axiom $4$.

The operations of superposition and convolution are related by the dual
procedure $P\longmapsto P^{\dagger}$, where $P^{\dagger}$ is a
\textit{quotient-valued map }corresponding to $P$:
\[
P^{\dagger}\left(  A\right)  =A/P\left(  A\right)
\]
for every algebra $A$. This relation reflected in Proposition \ref{dua} below
reminds the Fourier transform which relates convolution and product of
functions. Rewrite $\left(  \ref{convfor}\right)  $ as
\[
(P{\ast}R)(A)=q_{R\left(  A\right)  }^{-1}\left(  P\left(  R^{\dagger
}(A)\right)  \right)
\]
for every $A$.

Let $S$ be a quotient-valued map, i.e. let $S$ send every algebra $A$ to $A/I$
for some ideal $I$. The dependence of $I$ on $A$ can be written as $I=S^{\dag
}(A)$, so we have an ideal-valued map $S^{\dag}$ such that
\[
S\left(  A\right)  =A/S^{\dag}\left(  A\right)  =q_{S^{\dag}\left(  A\right)
}\left(  A\right)
\]
for every algebra $A$. If $T,S$ are quotient-valued maps then their
\textit{superposition} $T\circ S$ is defined by
\[
\left(  T\circ S\right)  (A):=A/q_{S^{\dag}\left(  A\right)  }^{-1}(T^{\dag
}\left(  S\left(  A\right)  \right)  )\cong S\left(  A\right)  /T^{\dag
}\left(  S\left(  A\right)  \right)  =T\left(  S\left(  A\right)  \right)
\]
for every algebra $A$.

\begin{proposition}
\label{dua}Let $P,R$ be ideal-valued maps. Then
\[
(R{\ast}P)^{\dagger}=R^{\dagger}\circ P^{\dagger}.
\]

\end{proposition}

\begin{proof}
If $J$ is a [closed] ideal of $A$ and $I$ is a [closed] ideal of $A/J$ then
one identifies in a standard way the algebras $(A/J)/I$ and $A/q_{J}^{-1}(I)$.
Therefore
\begin{align*}
R^{\dagger}\circ P^{\dagger}(A)  &  =(A/P(A))/R(A/P(A))=A/q_{P(A)}%
^{-1}(R(A/P(A)))\\
&  =A/(R\ast P)(A)=(R\ast P)^{\dagger}(A).
\end{align*}

\end{proof}

\begin{lemma}
\label{conv}Let $P$ and $R$ be preradicals. Then

\begin{enumerate}
\item $P{\ast R}$ is a preradical;

\item The convolution operation is associative.
\end{enumerate}
\end{lemma}

\begin{proof}
$\left(  1\right)  $ Let $f:A\longrightarrow B$ be a morphism. Then $f\left(
P\left(  A\right)  \right)  \subset P\left(  B\right)  $, so there is a
morphism $g:q_{P\left(  A\right)  }\left(  A\right)  \longrightarrow
q_{P\left(  B\right)  }\left(  B\right)  $ such that $g\circ q_{P\left(
A\right)  }=q_{P\left(  B\right)  }\circ f$; then $g\left(  R\left(
q_{P\left(  A\right)  }\left(  A\right)  \right)  \right)  \subset R\left(
q_{P\left(  B\right)  }\left(  B\right)  \right)  $. Let $a\in\left(  P{\ast
R}\right)  \left(  A\right)  $; then
\[
\left(  g\circ q_{P\left(  A\right)  }\right)  \left(  a\right)  \in R\left(
q_{P\left(  B\right)  }\left(  B\right)  \right)  ,
\]
whence $\left(  q_{P\left(  B\right)  }\circ f\right)  \left(  a\right)  \in
R\left(  q_{P\left(  B\right)  }\left(  B\right)  \right)  $ and $f\left(
a\right)  \in\left(  P{\ast R}\right)  \left(  B\right)  $.

$\left(  2\right)  $ Let $K$ be a preradical. We have that
\[
\left(  R\ast K\right)  \ast P\left(  A\right)  =q_{I_{1}}^{-1}\left(  R\ast
K\left(  A/I_{1}\right)  \right)  =q_{I_{1}}^{-1}\left(  q_{I_{2}}^{-1}\left(
R\left(  \left(  A/I_{1}\right)  /I_{2}\right)  \right)  \right)
\]
where $I_{1}=P\left(  A\right)  $, $I_{2}=K\left(  A/I_{1}\right)  $. Also,
$R\ast\left(  K\ast P\right)  \left(  A\right)  =q_{I_{3}}^{-1}\left(
R\left(  A/I_{3}\right)  \right)  $ where
\[
I_{3}=K\ast P\left(  A\right)  =q_{I_{1}}^{-1}\left(  K\left(  A/I_{1}\right)
\right)  =q_{I_{1}}^{-1}\left(  I_{2}\right)  .
\]
It remains to note that $q_{I_{3}}=q_{I_{2}}\circ q_{I_{1}}$ and $\left(
A/I_{1}\right)  /I_{2}\cong A/q_{I_{1}}^{-1}\left(  I_{2}\right)  =A/I_{3}$.
\end{proof}

The following result is a direct consequence of Proposition \ref{ax4}.

\begin{corollary}
\label{oper}

\begin{enumerate}
\item If $P$ and $R$ are under radicals then $P{\ast}R$ is an under radical.

\item If $P$ and $R$ are over radicals then $P{\circ R}$ is an over radical.
\end{enumerate}
\end{corollary}

\begin{theorem}
\label{oper3} If $P$ and $R$ are hereditary under radicals defined on a
universal class $\mathfrak{U}$, and $P$ is pliant then $P{\ast}R$ is a
hereditary under radical.
\end{theorem}

\begin{proof}
Let $A$ be an algebra, $J$ an ideal of $A$. As $P$ and $R$ are under radicals
then $P{\ast}R$ is also an under radical by Corollary \ref{oper}. Therefore
\[
\left(  P{\ast}R\right)  \left(  J\right)  \subset J\cap\left(  P{\ast
}R\right)  \left(  A\right)  .
\]
We have to prove the converse inclusion.

Since $R(J)=J\cap R(A)$, setting $f(x/R(J))=q_{R(A)}(x)$ we obtain an
isomorphism $f$ of the algebra $J/R(J)$ onto the ideal $q_{R(A)}(J)$ of
$A/R(A)$. Since $P$ is pliant, $f(P(J/R(J)))=P(q_{R(A)}(J))$.

If $x\in J\cap(P{\ast}R)(A)$ then
\[
q_{R(A)}(x)\in P(A/R(A))\cap q_{R(A)}(J)=P(q_{R(A)}(J))=f(P(J/R(J)).
\]
Thus $f(x/R(J))\in f(P(J/R(J))$, whence $x/R(J)\in P(J/R(J))$ and $x\in
(P{\ast}R)(J)$. Therefore $J\cap(P{\ast}R)(A)\subset(P{\ast}R)(J)$ and the
proof is complete.
\end{proof}

\begin{remark}
The proof of Theorem $\ref{oper3}$ shows that it remains true if $P,Q$ are
topological radicals on Banach algebras and $P$ satisfies the condition of
Banach heredity $(\ref{BanRad})$.
\end{remark}

For brevity, let
\[
P+R:=\mathrm{H}_{\left\{  P,R\right\}  }\text{ and }P\cdot R:=\mathrm{B}%
_{\{P,R\}}%
\]
for preradicals $P$ and $R$.

\begin{lemma}
\label{ovun2}

\begin{enumerate}
\item Let $P$ and $R$ be under radicals. Then
\[
P+R\leq P{\ast}R\leq\left(  P+R\right)  \ast\left(  P+R\right)  .
\]

\item Let $P$ and $R$ be over radicals. Then
\[
\left(  P\cdot R\right)  \circ\left(  P\cdot R\right)  \leq P{\circ}R\leq
P\cdot R.
\]

\end{enumerate}
\end{lemma}

\begin{proof}
$\left(  1\right)  $ Let $A$ be an arbitrary algebra,\ $I=P\left(  A\right)
$, $J=\left(  P+R\right)  \left(  A\right)  $, and let $q_{I}:A\longrightarrow
A/I$ and $q_{J}:A\longrightarrow A/J$ be standard quotient maps. As
\[
q_{I}\left(  R\left(  A\right)  \right)  \subset R\left(  A/I\right)
=q_{I}\left(  \left(  P{\ast}R\right)  \left(  A\right)  \right)
\]
then $P+R\leq P{\ast}R$. As $I\subset J$, there is a morphism
$q:A/I\longrightarrow A/J$ such that $q\circ q_{I}=q_{J.}$. Therefore%
\[
q_{J}\left(  \left(  P{\ast}R\right)  \left(  A\right)  \right)  =q\left(
R\left(  A/I\right)  \right)  \subset R\left(  A/J\right)  \subset\left(
P+R\right)  \left(  A/J\right)
\]
whence $\left(  P{\ast}R\right)  \left(  A\right)  \subset\left(  \left(
P+R\right)  \ast\left(  P+R\right)  \right)  \left(  A\right)  $.

$\left(  2\right)  $ Clearly
\[
\left(  P{\circ}R\right)  \left(  A\right)  =P\left(  R\left(  A\right)
\right)  \subset P\left(  A\right)  \cap R\left(  A\right)  =\left(  P\cdot
R\right)  \left(  A\right)  .
\]
Let $I=P\left(  A\right)  \cap R\left(  A\right)  $. As $I$ is an ideal of
$R\left(  A\right)  $ then
\[
\left(  P\cdot R\right)  \left(  I\right)  \subset P\left(  I\right)  \subset
P\left(  R\left(  A\right)  \right)  .
\]

\end{proof}

\begin{corollary}
\label{ovun3}

\begin{enumerate}
\item Let $P$ and $R$ be under radicals. Then
\[
\left(  P{\ast}R\right)  ^{\ast}=\left(  R\ast P\right)  ^{\ast}=P\vee
R=P^{\ast}\vee R^{\ast}.
\]

\item Let $P$ and $R$ be over radicals. Then
\[
\left(  P\circ R\right)  ^{\circ}=\left(  R\circ P\right)  ^{\circ}=P\wedge
R=P^{\circ}\wedge R^{\circ}.
\]

\end{enumerate}
\end{corollary}

\begin{proof}
$\left(  1\right)  $ follows from Lemma \ref{ovun2} if one takes into account
that
\[
\left(  \left(  P+R\right)  \ast\left(  P+R\right)  \right)  ^{\ast}=P\vee R
\]
by associativity of the convolution.

$\left(  2\right)  $ follows from Lemma \ref{ovun2} if one takes into account
that
\[
\left(  \left(  P\cdot R\right)  \circ\left(  P\cdot R\right)  \right)
^{\circ}=P\wedge R
\]
by associativity of the superposition.
\end{proof}

This corollary can be easily extended to finite and even transfinite
\textquotedblleft products\textquotedblright\ of corresponding maps.

It is straightforward that if $P$ and $R$ are hereditary radicals then $P\cdot
R$ is a hereditary radical. As a consequence of this, Lemma \ref{ovun2}%
$\left(  2\right)  $ and Corollary \ref{ovun3}, we obtain the following

\begin{corollary}
\label{herr}Let $P$ and $R$ be the hereditary radicals. Then $P\circ R$ is a
hereditary radical and
\[
P\circ R=R\circ P=P\wedge R=R\cdot P.
\]

\end{corollary}

\subsection{Transfinite chains of ideals and radicals}

We start with a simple observation on the quotient algebras related to preradicals.

\begin{proposition}
\label{ov}Let\ $A$ be an algebra, and let $I$ be an ideal of $A$. Then

\begin{enumerate}
\item If $P$ is an over radical and

\begin{enumerate}
\item $P\left(  A\right)  \subset I$ then $I/P\left(  A\right)  $ is $P$-semisimple;

\item $I\subset P\left(  A\right)  $ [for closed $I$] then $P\left(
A/I\right)  =P\left(  A\right)  /I$.
\end{enumerate}

\item If $P$ is a hereditary preradical and $I\subset P\left(  A\right)  $
[for closed $I$] then $P\left(  A\right)  /I$ is $P$-radical.
\end{enumerate}
\end{proposition}

\begin{proof}
$\left(  1\text{a}\right)  $ $P\left(  I/P\left(  A\right)  \right)  \subset
P\left(  A/P\left(  A\right)  \right)  =0$.

$\left(  1\text{b}\right)  $ There is a morphism $p:A/I\longrightarrow
A/P\left(  A\right)  $ such that $q_{P\left(  A\right)  }=p\circ q_{I}$ where
$q_{P\left(  A\right)  }$, $q_{I}$ are the corresponding standard quotient
maps. Then
\[
p\left(  P\left(  A/I\right)  \right)  \subset P\left(  A/P\left(  A\right)
\right)  =0
\]
whence $P\left(  A/I\right)  \subset\ker p=P\left(  A\right)  /I=q_{I}\left(
P\left(  A\right)  \right)  \subset P\left(  A/I\right)  $.

$\left(  2\right)  $ Indeed, $P\left(  A\right)  /I=q_{I}\left(  P\left(
A\right)  \right)  \subset P\left(  A/I\right)  $. Hence, by heredity of $P$,
\[
P\left(  P\left(  A\right)  /I\right)  =\left(  P\left(  A\right)  /I\right)
\cap P\left(  A/I\right)  =P\left(  A\right)  /I.
\]

\end{proof}

Let $P$ be a preradical, and let $A$ be an algebra. A [closed] ideal $I$ of
$A$ is called $P$\textit{-absorbing} if $A/I$ is $P$-semisimple.

\begin{theorem}
\label{ch1}Let $A$ be an algebra, $P$ an over radical, $R$ an under radical. Then

\begin{enumerate}
\item {}[The closure of] the sum of any family of $R$-radical ideals is $R$-radical;

\item If $\left(  J_{\alpha}\right)  _{\alpha\leq\gamma}$ is an increasing
transfinite chain of ideals of $A$ with $J_{0}=0$ and all gap-quotients are
$R$-radical then $J_{\gamma}$ is $R$-radical;

\item The intersection of any number of $P$-absorbing ideals is $P$-absorbing;

\item If $\left(  I_{\alpha}\right)  _{\alpha\leq\delta}$ is a decreasing
transfinite chain $\left(  I_{\alpha}\right)  _{\alpha\leq\gamma}$ of ideals
of $A$ with $I_{0}=A$ and all $I\ _{\alpha}/I_{\alpha+1}$ are $P$-semisimple
then $I_{\delta}$ is $P$-absorbing.
\end{enumerate}
\end{theorem}

\begin{proof}
$\left(  1\right)  $ Let $I$ be the [closure of the] sum of a family
$(J_{\alpha})_{\alpha\in\Lambda}$ of $P$-radical ideals of $A$. Then
$J_{\alpha}=R\left(  J_{\alpha}\right)  \subset R\left(  I\right)  \subset I$
for every $\alpha$, whence $I=R\left(  I\right)  $.

$\left(  2\right)  $ follows by transfinite induction. The step $\alpha
\mapsto\alpha+1$ follows from Theorem \ref{ext}, and the case of limit
ordinals follows from $\left(  1\right)  $.

$\left(  3\right)  $ Let $I=\cap_{\alpha\in\Lambda}J_{\alpha}$ be the
intersection of $P$-absorbing ideals of $A$. As $I\subset J_{\alpha}$ for
every $\alpha$, then, for the standard morphism $q_{\alpha}:A/I\longrightarrow
A/J_{\alpha}$, we have $q_{\alpha}\left(  P\left(  A/I\right)  \right)
\subset P\left(  A/J_{\alpha}\right)  =0$. Therefore $P\left(  A/I\right)
\subset\cap_{\alpha}\ker q_{\alpha}=\cap_{\alpha\in\Lambda}J_{\alpha
}/I=\left(  \cap_{\alpha\in\Lambda}J_{\alpha}\right)  /I=0$.

$\left(  4\right)  $ follows by transfinite induction. The step $\alpha
\mapsto\alpha+1$ follows from Theorem \ref{ext}, and the case of limit
ordinals follows from $\left(  3\right)  $.
\end{proof}

\begin{theorem}
\label{ch2}Let $\mathcal{F}_{u}$ be a family of under radicals, let
$\mathcal{F}_{o}$ be a family of over radicals, and let $A$ be a [normed]
algebra. Then

\begin{enumerate}
\item There is an increasing transfinite chain $\left(  I_{\alpha}\right)
_{\alpha\leq\gamma}$ of [closed] ideals of $A$ such that $I_{0}=0$,
$I_{\gamma}=\left(  \vee\mathcal{F}_{u}\right)  \left(  A\right)  $ and every
gap-quotient $I_{\alpha+1}/I_{\alpha}$ of the chain is $P$-radical and is
equal to $P\left(  A/I_{\alpha}\right)  $ for some $P\in\mathcal{F}_{u}$;

\item There is an increasing transfinite chain $\left(  J_{\alpha}\right)
_{\alpha\leq\delta}$ of [closed] ideals of $A$ such that $J_{0}=A$,
$J_{\delta}=\left(  \wedge\mathcal{F}_{o}\right)  \left(  A\right)  $ and
every gap-quotient of the chain is $P$-semisimple for some $P\in
\mathcal{F}_{o}$.
\end{enumerate}
\end{theorem}

\begin{proof}
$\left(  1\right)  $ Let $I=\left(  \vee\mathcal{F}_{u}\right)  \left(
A\right)  $. Arguing by transfinite induction, one can assume that we already
have the required chain $\left(  I_{\alpha^{\prime}}\right)  _{\alpha^{\prime
}\leq\alpha}$ but with the only distinction $I_{\alpha}\subset I$. If
$I_{\alpha}\neq I$ then $I/I_{\alpha}$ is a non-zero $\vee\mathcal{F}_{u}%
$-radical algebra by Proposition \ref{ov}$\left(  1\text{b}\right)  $, because
$\vee\mathcal{F}_{u}$ is a radical. Therefore \textrm{H}$_{\mathcal{F}_{u}%
}\left(  I/I_{\alpha}\right)  \neq0$. Then \textrm{H}$_{\mathcal{F}_{u}%
}\left(  A/I_{\alpha}\right)  \neq0$, and there is some $P\in\mathcal{F}_{u}$
such that $P\left(  A/I_{\alpha}\right)  \neq0$. Let $I_{\alpha+1}%
=q_{I_{\alpha}}^{-1}\left(  P\left(  A/I_{\alpha}\right)  \right)  $. Then
$I_{\alpha+1}/I_{\alpha}$ is $P$-radical and
\[
I_{\alpha+1}/I_{\alpha}=P\left(  A/I_{\alpha}\right)  \subset\left(
\vee\mathcal{F}_{u}\right)  \left(  A/I_{\alpha}\right)  =\left(
\vee\mathcal{F}_{u}\right)  \left(  A\right)  /I_{\alpha}=I/I_{\alpha}%
\]
by Proposition \ref{ov}$\left(  1\text{b}\right)  $. So $I_{\alpha+1}\subset
I$.

$\left(  2\right)  $ Let $J=\left(  \wedge\mathcal{F}_{o}\right)  \left(
A\right)  $. Assume, by induction, that we already have the required chain
$\left(  J_{\alpha^{\prime}}\right)  _{\alpha^{\prime}\leq\alpha}$ but with
$J\subset J_{\alpha}$. If $J_{\alpha}\neq J$ then $J_{\alpha}/J$ is a non-zero
$\wedge\mathcal{F}_{o}$-semisimple algebra by Proposition \ref{ov}$\left(
1\text{a}\right)  $. Then \textrm{H}$_{\mathcal{F}_{o}}\left(  J_{\alpha
}/J\right)  =0$, and there is some $P\in\mathcal{F}$ such that $P\left(
\left(  J_{\alpha}/J\right)  \right)  \neq J_{\alpha}/J$. Let $J_{\alpha
+1}=q_{J}^{-1}\left(  P\left(  \left(  J_{\alpha}/J\right)  \right)  \right)
$. Then $J_{\alpha}/J_{\alpha+1}\cong\left(  J_{\alpha}/J\right)  /P\left(
J_{\alpha}/J\right)  $ is $P$-semisimple.
\end{proof}

Let $\mathfrak{U}$ be a base class, and let $V$ be a class of [normed]
algebras from $\mathfrak{U}$. We say that $V$ is \textit{closed under
increasing transfinite chains} if $V$ contains any [normed] algebra $A$ for
which there is an increasing transfinite chain $\left(  I_{\alpha}\right)
_{\alpha\leq\gamma}$ of [closed] ideals $A$ such that $I_{0}=0$, $I_{\gamma
}=A$ and all gap-quotients of the chain are algebras from $V$; we say that $V$
is \textit{closed under decreasing transfinite chains} if $V$ contains any
[normed] algebra $A$ for which there is a decreasing transfinite chain
$\left(  J_{\alpha}\right)  _{\alpha\leq\delta}$ of [closed] ideals of $A$
such that $J_{0}=A$, $J_{\delta}=0$ and all gap-quotients of the chain are
algebras from $V$.

\begin{corollary}
\label{ch3}Let $V$ be a class of [normed] algebras from $\mathfrak{U}$. Then

\begin{enumerate}
\item If $V$ is closed under increasing transfinite chains then there is the
largest radical on $\mathfrak{U}$ whose radical algebras lie in $V$;

\item If $V$ is closed under decreasing transfinite chains then there is the
smallest radical on $\mathfrak{U}$ whose semisimple algebras lie in $V$.
\end{enumerate}
\end{corollary}

\begin{proof}
$\left(  1\right)  $ Let $\mathcal{F}$ be the family of all radicals $P$ on
$\mathfrak{U}$ with $\mathbf{Rad}\left(  P\right)  \subset V$. Let
$R=\vee\mathcal{F}$, and let $A$ be an $R$-radical algebra. By Theorem
\ref{ch2}, there is an increasing transfinite chain $\left(  I_{\alpha
}\right)  _{\alpha\leq\gamma}$ of [closed] ideals of $A$ such that $I_{0}=0$,
$I_{\gamma}=A$ and every gap-quotient of the chain is $P$-radical for some
$P\in\mathcal{F}$. Then all gap-quotients of the chain lie in $V$, and $R\in$
$\mathcal{F}$.

$\left(  2\right)  $ The proof is similar to the proof of $\left(  1\right)  $.
\end{proof}

A transfinite sequence $\left(  S_{\alpha}\right)  _{\alpha\leq\delta}$ of
over radicals is called a \textit{decreasing transfinite chain} of over
radicals if $\left(  S_{\alpha}\left(  A\right)  \right)  _{\alpha\leq\delta}$
is a decreasing transfinite chain of ideals of $A$ for every algebra $A$;
similarly, a transfinite sequence $\left(  T_{\alpha}\right)  _{\alpha
\leq\gamma}$ of under radicals is called an \textit{increasing transfinite
chain} of under radicals if $\left(  T_{\alpha}\left(  A\right)  \right)
_{\alpha\leq\gamma}$ is an increasing transfinite chain of ideals of $A$ for
every algebra $A$.

Examples are:

\begin{enumerate}
\item A \textit{convolution chain} $\left(  T_{\alpha}\right)  _{\alpha
\leq\gamma}$ obtained from family $\mathcal{F}_{u}$ of under radicals by the
rule $T_{0}:=\mathcal{P}_{0}:A\longmapsto0$, \textit{the zero radical,} and
$T_{\alpha+1}=P\ast T_{\alpha}$ for some $P\in\mathcal{F}_{u}$;

\item A \textit{superposition chain} $\left(  S_{\alpha}\right)  _{\alpha
\leq\delta}$ obtained from a family $\mathcal{F}_{o}$ of over radicals by the
rule $S_{0}:=\mathcal{P}_{1}:A\longmapsto A$,\textit{ the identity radical,}
and $S_{\alpha+1}=R\circ S_{\alpha}$ for some $R\in\mathcal{F}_{o}$.
\end{enumerate}

Theorem \ref{dix0} and Corollary \ref{oper} guarantee that $\left(  T_{\alpha
}\right)  _{\alpha\leq\gamma}$ consists of over radicals and $\left(
S_{\alpha}\right)  _{\alpha\leq\delta}$ consists of under radicals.

\begin{theorem}
\label{chain}Let $\mathcal{F}_{u}^{\ast}$ be the family of all convolution
chains of a family $\mathcal{F}_{u}$ of under radicals, and let $\mathcal{F}%
_{o}^{\circ}$ be the family of all superposition chains of a family
$\mathcal{F}_{o}$ of over radicals. Then

\begin{enumerate}
\item $T_{\gamma}\leq\vee\mathcal{F}_{u}$ for every $\left(  T_{\alpha
}\right)  _{\alpha\leq\gamma}\in\mathcal{F}_{u}^{\ast}$ and $S_{\delta}%
\geq\wedge\mathcal{F}_{o}$ for every $\left(  S_{\alpha}\right)  _{\alpha
\leq\delta}\in\mathcal{F}_{o}^{\circ}$.

\item For any algebras $A_{1},\ldots,A_{n}$, there exist

\begin{enumerate}
\item $\left(  T_{\alpha}\right)  _{\alpha\leq\gamma}\in\mathcal{F}_{u}^{\ast
}$ such that $T_{\gamma}\left(  A_{i}\right)  =\left(  \vee\mathcal{F}%
_{u}\right)  \left(  A_{i}\right)  $, for $i\leq n;$

\item $\left(  S_{\alpha}\right)  _{\alpha\leq\delta}\in\mathcal{F}_{o}%
^{\circ}$ such that $S_{\delta}\left(  A_{i}\right)  =\left(  \wedge
\mathcal{F}_{o}\right)  \left(  A_{i}\right)  $, for $i\leq n$.
\end{enumerate}
\end{enumerate}
\end{theorem}

\begin{proof}
$\left(  1\right)  $ follows by transfinite induction: the step $\alpha
\mapsto\alpha+1$ follows by Lemma \ref{ovun2} and the case of limit ordinals
is obvious.

$\left(  2\text{a}\right)  $ Construct $\left(  T_{\alpha}\right)
\in\mathcal{F}_{u}^{\ast}$ taking $T_{\alpha+1}=P\ast T_{\alpha}$ for some
$P\in\mathcal{F}_{u}$ with $P\left(  A_{1}/T_{\alpha}\left(  A_{1}\right)
\right)  \neq0$. The chain $\left(  T_{\alpha}\left(  A_{1}\right)  \right)  $
of ideals of $A_{1}$ must be stabilized: there is $T_{\gamma_{1}}$ such that
$P\left(  A_{1}/T_{\gamma_{1}}\left(  A_{1}\right)  \right)  =0$ for every
$P\in\mathcal{F}_{u}$. Continuing the construction of $\left(  T_{\alpha
}\right)  _{\alpha\leq\gamma_{1}}$ first for $A_{2},\ldots,A_{n}$ and then for
every other algebra $B$ in such a way if necessary, one can find $\left(
T_{\alpha}\right)  _{\alpha\leq\gamma_{i}}$ with $\gamma_{1}\leq\gamma_{2}%
\leq\cdots\leq\gamma_{n}=\gamma$ such that $P\left(  A_{i}/T_{\gamma_{i}%
}\left(  A_{i}\right)  \right)  =0$ for every $P\in\mathcal{F}_{u}$, and
define a preradical $T$ such that $T\left(  A_{i}\right)  =T_{\gamma}\left(
A_{i}\right)  $ for $i=1,\ldots,n$. By Theorem \ref{dix0} and Corollary
\ref{oper}, $T$ is an under radical and $T\leq\vee\mathcal{F}_{u}$ by $\left(
1\right)  $. It is clear that $T\left(  B/T\left(  B\right)  \right)  =0$,
i.e. $T$ is a radical. As%
\[
q_{T\left(  B\right)  }\left(  P\left(  B\right)  \right)  \subset P\left(
B/T\left(  B\right)  \right)  =0
\]
for every $P\in\mathcal{F}_{u}$, then $P\left(  B\right)  \subset T\left(
B\right)  $, for every algebra $B$. Hence $\vee\mathcal{F}_{u}\leq T$.

$\left(  2\text{b}\right)  $ Construct $\left(  S_{\alpha}\right)
\in\mathcal{F}_{o}^{\circ}$ taking $S_{\alpha+1}\left(  A_{1}\right)
=R\left(  S_{\alpha}\left(  A_{1}\right)  \right)  \neq S_{\alpha}\left(
A_{1}\right)  $ for some $R\in\mathcal{F}_{o}$. The chain $\left(  S_{\alpha
}\left(  A_{1}\right)  \right)  $ of ideals must be stabilized: there is
$S_{\delta_{1}}$ such that $R\left(  S_{\delta_{1}}\left(  A_{1}\right)
\right)  =S_{\delta_{1}}\left(  A_{1}\right)  $ for every $R\in\mathcal{F}%
_{o}$. Continuing the construction of $\left(  S_{\alpha}\right)  $ first for
$A_{2},\ldots,A_{n}$ and then for every other algebra $B$ in such a way if
necessary, one can find $\left(  S_{\alpha}\right)  _{\alpha\leq\delta_{i}}$
with $\delta_{1}\leq\delta_{2}\leq\cdots\leq\delta_{n}=\delta$ such that
$R\left(  A_{i}/S_{\delta_{i}}\left(  A_{i}\right)  \right)  =0$ for every
$R\in\mathcal{F}_{o}$, and define a preradical $S$ such that $S\left(
A_{i}\right)  =S_{\delta}\left(  A_{i}\right)  $ for $i=1,\ldots,n$. By
Theorem \ref{dix0} and Corollary \ref{oper}, $S$ is an over radical and
$S\geq\wedge\mathcal{F}_{o}$ by $\left(  1\right)  $. It is clear that
$S\left(  S\left(  B\right)  \right)  =S\left(  B\right)  $, i.e. $S$ is a
radical. As $S\left(  B\right)  =R\left(  S\left(  B\right)  \right)  \subset
R\left(  B\right)  $ for every $R\in\mathcal{F}_{o}$ and every algebra $B$,
then $S\leq\wedge\mathcal{F}_{o}$.
\end{proof}

\subsection{The heredity problem}

The closure procedure does not preserve the class of hereditary radicals:
$\overline{P}$ can be non-hereditary when $P$ is a hereditary preradical (see
\cite[Example 7.1]{D97}). The operation of the closure (of an ideal) is
involved also in the construction of $P^{\ast}$; this leads to a possible loss
of heredity. The following lemma shows that it is not the case for the
convolution chains in $\mathfrak{U}_{\mathrm{a}}$ (in particular, for the
algebraic convolution procedure) and for the superposition chains in general.

\begin{lemma}
\label{her}Let $\delta$ be a limit ordinal. Then

\begin{enumerate}
\item If $\left(  P_{\alpha}\right)  _{\alpha\leq\delta}$ is a decreasing
transfinite chain of preradicals and $P_{\alpha}$ is hereditary for
$\alpha<\delta$, then $P_{\delta}$ is hereditary.

\item If $\left(  P_{\alpha}\right)  _{\alpha\leq\delta}$ is an increasing
transfinite chain of algebraic preradicals and $P_{\alpha}$ is hereditary for
$\alpha<\delta$, then $P_{\delta}$ is hereditary.
\end{enumerate}
\end{lemma}

\begin{proof}
$\left(  1\right)  $ Let $A$ be an algebra, and let $J$ be an ideal of $A$. By
definition, the chain $\left(  P_{\alpha}\left(  A\right)  \right)
_{\alpha\leq\delta}$ is a decreasing transfinite chain of [closed] ideals.
Then
\[
P_{\delta}\left(  A\right)  =\cap_{\alpha<\delta}P_{\alpha}\left(  A\right)
\]
by (\ref{bb}), and $P_{\delta}\left(  J\right)  =J\cap P_{\delta}\left(
A\right)  $.

$\left(  2\right)  $ By definition, $\left(  P_{\alpha}\left(  A\right)
\right)  _{\alpha\leq\delta}$ is an increasing transfinite chain of ideals.
So
\begin{align*}
P_{\delta}\left(  J\right)   &  =\cup_{\alpha<\delta}P_{\alpha}\left(
J\right)  =\cup_{\alpha<\delta}\left(  J\cap P_{\alpha}\left(  A\right)
\right)  =J\cap\left(  \cup_{\alpha<\delta}P_{\alpha}\left(  A\right)  \right)
\\
&  =J\cap P_{\delta}\left(  A\right)  .
\end{align*}

\end{proof}

\begin{corollary}
\label{hera}If $P$ is an algebraic hereditary preradical then $P^{\ast}$ is a
hereditary radical.
\end{corollary}

\begin{proof}
It follows from Theorem \ref{oper} and Lemma \ref{her} that the convolution
chain generated by $P$ consists of hereditary preradicals. Therefore $P^{\ast
}$ is hereditary.
\end{proof}

Let $\mathcal{F}$ be a family of hereditary radicals. Then $\wedge\mathcal{F}$
is a hereditary radical and therefore is the infimum of $\mathcal{F}$ in the
class of hereditary radicals. Besides of this $\wedge\mathcal{F}$ satisfies
the following equality \cite[Lemma 3.2]{TR3}%
\begin{equation}
\wedge\mathcal{F}=\mathrm{B}_{\mathcal{F}}. \label{hp}%
\end{equation}

\begin{theorem}
\label{sih}The operations $\vee$ and $\wedge$ restricted to the class of
algebraic hereditary radicals produce supremum and infimum in this class;
moreover $\wedge$ produces infimum in the class of topological hereditary radicals.
\end{theorem}

\begin{proof}
Let $\mathcal{F}$ be a family of algebraic hereditary radicals, $A$ an
algebra, and let $J$ be an ideal of $A$. By Theorem \ref{chain}, there is a
convolution chain $\left(  T_{\alpha}\right)  _{\alpha\leq\gamma}$ obtained
from $\mathcal{F}$ such that $\left(  \vee\mathcal{F}\right)  \left(
A\right)  =T_{\gamma}\left(  A\right)  $ and $\left(  \vee\mathcal{F}\right)
\left(  J\right)  =T_{\gamma}\left(  J\right)  $. By Theorem \ref{oper} and
Lemma \ref{her}, $\left(  T_{\alpha}\right)  _{\alpha\leq\gamma}$ consists of
hereditary preradicals. Therefore $T_{\gamma}\left(  J\right)  =J\cap
T_{\gamma}\left(  A\right)  $. This shows that $\vee\mathcal{F}$ is a
hereditary radical. Therefore $\vee\mathcal{F}$ is the supremum of
$\mathcal{F}$ in the class of algebraic hereditary radicals.

The assertion about the infimum is evident.
\end{proof}

\begin{theorem}
Let $\mathcal{F}$ be a family of hereditary radicals and $V=\cup
_{P\in\mathcal{F}}\mathbf{\mathbf{Rad}}\left(  P\right)  $. Let $\mathcal{F}%
^{\prime}$ be the family of all hereditary radicals $R$ such that
$\mathbf{\mathbf{Rad}}\left(  R\right)  \supset V$. Then the supremum of
$\mathcal{F}$ in the class of hereditary radicals is equal to
\[
\wedge\mathcal{F}^{\prime}=\mathrm{B}_{\mathcal{F}^{\prime}}.
\]

\end{theorem}

\begin{proof}
Let $T=\wedge\mathcal{F}^{\prime}$. Then $T$ is a hereditary radical by
\cite[Lemma 3.2]{TR3}, and $T\leq R$, for every $R\in\mathcal{F}^{\prime}$, by
definition. As $\wedge\mathcal{F}^{\prime}=\mathrm{B}_{\mathcal{F}^{\prime}}$
then, for every $P\in\mathcal{F}$, we obtain that
\[
\mathbf{\mathbf{Rad}}\left(  T\right)  =\cap_{R\in\mathcal{F}^{\prime}%
}\mathbf{\mathbf{Rad}}\left(  R\right)  \supset V\supset\mathbf{\mathbf{Rad}%
}\left(  P\right)
\]
and therefore $T\geq P$. So $T$ is the supremum of $\mathcal{F}$ in the class
of hereditary radicals.
\end{proof}

Now we consider a way to show heredity of some radicals that are constructed
via the closure procedure and the topological convolution procedure.

Let $A$ be an algebra, and let $I$ be an ideal of $A$. Let $I^{\bot}=\left\{
x\in A:xI=0\right\}  $, the left annihilator of $I$ in $A$. A preradical $P$
is called\textit{ nil-exact} if
\[
I^{\bot}\left(  P\ast I\right)  \subset P\left(  A\right)
\]
for every [closed] ideal $I$ of any algebra $A$.

A preradical $P$ is called \textit{weakly hereditary} if, for any $P$-radical
algebra $A$, $P\left(  I\right)  \neq0$ for every non-zero ideal $I$ of $A$.

Recall that $\Sigma_{\beta}\left(  A\right)  $ means the sum of all nilpotent
ideals of $A$; if $P$ is a topological preradical then $\Sigma_{\beta}\leq P$
means that the inequality holds for normed algebras on which $P$ is defined.

\begin{lemma}
\label{aur}Let $P$ be an algebraic under radical. Then

\begin{enumerate}
\item If $P$ is nil-exact then $\overline{P}$ is nil-exact;

\item If $P$ is weakly hereditary and $\Sigma_{\beta}\leq P$ then
$\overline{P}$ is weakly hereditary.
\end{enumerate}
\end{lemma}

\begin{proof}
$\left(  1\right)  $ Let $A$ be a normed algebra, and let $I$ be a closed
ideal of $A$. We have that $I^{\bot}\left(  P\ast I\right)  \subset P\left(
A\right)  $, i.e. $I^{\bot}q_{I}^{-1}\left(  P\left(  A/I\right)  \right)
\subset P\left(  A\right)  $. By $\left(  \ref{gen0}\right)  $,
\[
q_{I}^{-1}\left(  \overline{P}\left(  A/I\right)  \right)  =q_{I}^{-1}\left(
\overline{P\left(  A/I\right)  }\right)  =\overline{q_{I}^{-1}\left(  P\left(
A/I\right)  \right)  },
\]
so
\[
I^{\bot}q_{I}^{-1}\left(  \overline{P}\left(  A/I\right)  \right)  =I^{\bot
}\overline{q_{I}^{-1}\left(  P\left(  A/I\right)  \right)  }\subset
\overline{I^{\bot}q_{I}^{-1}\left(  P\left(  A/I\right)  \right)  }%
\subset\overline{P(A)}=\overline{P}(A),
\]
i.e. $I^{\bot}\left(  \overline{P}\ast I\right)  \subset\overline{P}(A)$.

$\left(  2\right)  $ Let $A$ be a $\overline{P}$-radical algebra, and let $I$
be a non-zero ideal of $A$. Set $J=IA$.

If $J\neq0$ then $P(J)\neq0$ by the weak heredity of $P$. As $J$ is an ideal
of $I$, it follows that $P(I)\neq0$, whence $\overline{P}(I)\neq0$.

If $IA=0$ then $I^{2}=0$, whence $I=\Sigma_{\beta}(I)\subset P\left(
I\right)  $ and $\overline{P}(I)=I\neq0$.
\end{proof}

\begin{lemma}
\label{whur}Let $P$ be a weakly hereditary under radical, and let $A$ be an
algebra and $I$ an ideal of $A$. Then

\begin{enumerate}
\item If $I\cap P\left(  A\right)  \neq0$ then $P\left(  I\right)  \neq0$;

\item If $I\subset P\left(  A\right)  $ then $I$ is $P^{\ast}$-radical.
\end{enumerate}
\end{lemma}

\begin{proof}
$\left(  1\right)  $ Let $J=I\cap P\left(  A\right)  \neq0$. Then $J$ is a
non-zero ideal of the $P$-radical algebra $P\left(  A\right)  $, whence
$P\left(  J\right)  \neq0$ by weak heredity. As $J$ is an ideal of $I$ then
$P\left(  J\right)  \subset P\left(  I\right)  \neq0$.

$\left(  2\right)  $ Let $A$ be an algebra, and let $I$ be a non-zero ideal of
$A$ and $I\subset P\left(  A\right)  $. Then $I$ is an ideal of the
$P$-radical algebra $B_{1}:=P\left(  A\right)  $. Let $I_{1}=P\ast0=P\left(
I\right)  $ with respect to the algebra $I$; then $I_{1}\neq0$ is an ideal of
$B_{1}$. Let $I_{1}^{\prime}$ be $I_{1}$ for the algebraic case and the
closure of $I_{1}$ in $B_{1}$ for the topological case. Then $B_{1}%
/I_{1}^{\prime}=q_{I_{1}^{\prime}}\left(  P\left(  A\right)  \right)  \subset
B_{2}:=P\left(  A/I_{1}^{\prime}\right)  $ and $I/I_{1}\cong q_{I_{1}^{\prime
}}\left(  I\right)  $.

If $I/I_{1}\neq0$ then $q_{I_{1}^{\prime}}\left(  I\right)  $ is a non-zero
ideal of the $P$-radical algebra $B_{2}$. Let $I_{2}=P\ast I_{1}$ with respect
to the algebra $I$: $I_{2}=q^{-1}\left(  P\left(  I/I_{1}\right)  \right)  $
where $q:I\longrightarrow I/I_{1}$ is the standard quotient map. If
$q_{I_{1}^{\prime}}\left(  I\right)  \neq0$ then $P\left(  q_{I_{1}^{\prime}%
}\left(  I\right)  \right)  \neq0$ by the assumptions and therefore $I_{2}\neq
I_{1}$.

Using transfinite induction, we can build the increasing transfinite chain of
ideals $\left(  I_{\alpha}\right)  $ of ideals of $A$ such that all
$I_{\alpha}\subset I$ and $I_{\alpha+1}=P\ast I_{\alpha}\neq I_{\alpha}$ if
$I\neq I_{\alpha}$, and also the corresponding transfinite chain $\left(
I_{\alpha}^{\prime}\right)  $ (consisting of closures of $I_{\alpha}$'s for
the topological case) such that $q_{I_{\alpha}^{\prime}}\left(  I\right)  $ is
an ideal of the $P$-radical algebra $P\left(  A/I_{\alpha}^{\prime}\right)  $.
The chain $\left(  I_{\alpha}\right)  $ must be stabilized: there is an
ordinal $\gamma$ such that $I=$ $I_{\gamma}$. By construction, $I$ is
$P^{\ast}$-radical.
\end{proof}

\begin{lemma}
\label{aeur}Let $P$ be a nil-exact and weakly hereditary under radical, and
let $I$ be an ideal of an algebra $A$. If $IP^{\ast}(A)\neq0$ then $P\left(
I\right)  \neq0$.
\end{lemma}

\begin{proof}
If $IP\left(  A\right)  \neq0$ then $P\left(  I\right)  \neq0$ by Lemma
\ref{whur}$\left(  1\right)  $. So one may assume that $IP\left(  A\right)
=0$. For the convolution chain $\left(  P_{\alpha}\right)  $ (with
$P_{0}=\mathcal{P}_{0}$ and $P_{\alpha+1}=P\ast P_{\alpha}$ for all $\alpha$),
one can find the first ordinal $\gamma$ such that $I\left(  P\ast P_{\gamma
}\left(  A\right)  \right)  \neq0$ (this is possible because it is clear that
if $IP_{\alpha^{\prime}}\left(  A\right)  =0$ for $\alpha^{\prime}<\alpha$ and
$\alpha$ is a limit ordinal then $IP_{\alpha}\left(  A\right)  =0$). Let
$J=P_{\gamma}\left(  A\right)  $. Then $Iq_{J}^{-1}\left(  P\left(
A/J\right)  \right)  \neq0$ and $IJ=0$ by the choice of $\gamma$. By
nil-exactness of $P$, $Iq_{J}^{-1}\left(  P\left(  A/J\right)  \right)
\subset P\left(  A\right)  $. As $Iq_{J}^{-1}\left(  P\left(  A/J\right)
\right)  \subset I$ then $I\cap P(A)\neq0$, whence $P\left(  I\right)  \neq0$
by Lemma \ref{whur}$\left(  1\right)  $.
\end{proof}

\begin{lemma}
\label{whae}Let $P$ be a weakly hereditary and nil-exact under radical, and
$\Sigma_{\beta}\leq P$. Then $P^{\ast}$ is weakly hereditary.
\end{lemma}

\begin{proof}
Let $A$ be a $P^{\ast}$-radical algebra, and let $I$ be a non-zero ideal of
$A$. If $I^{2}=0$ then $I=\Sigma_{\beta}\left(  I\right)  \subset P\left(
I\right)  \subset P^{\ast}\left(  I\right)  $ and $P^{\ast}\left(  I\right)
=I\neq0$. If $I^{2}\neq0$ then $IP^{\ast}\left(  A\right)  \neq0$ and
$P\left(  I\right)  \neq0$ by Lemma \ref{aeur}. Therefore $P\left(  I\right)
\subset P^{\ast}\left(  I\right)  \neq0$.
\end{proof}

\begin{theorem}
\label{heredit}Let $P$ be a weakly hereditary and nil-exact under radical, and
$\Sigma_{\beta}\leq P$. Then $P^{\ast}$ is hereditary.
\end{theorem}

\begin{proof}
Let $A$ be an algebra, and let $I$ be an ideal of $A$. By Lemma \ref{whae},
$P^{\ast}$ is weakly hereditary. We apply Lemma \ref{whur}$\left(  2\right)  $
to $P^{\ast}$ instead of $P$. As $P^{\ast}$ is a radical (in particular,
$P^{\ast\ast}=P^{\ast}$), every ideal $J$ of $A$ with $J\subset P^{\ast
}\left(  A\right)  $ is $P^{\ast}$-radical. Then $J:=I\cap P^{\ast}\left(
A\right)  $ is $P^{\ast}$-radical; as $J$ is an ideal of $I$, it follows that
\[
J=P^{\ast}\left(  J\right)  \subset P^{\ast}\left(  I\right)  \subset I\cap
P^{\ast}\left(  A\right)  =J,
\]
i.e. $P^{\ast}\left(  I\right)  =I\cap P^{\ast}\left(  A\right)  $. Therefore
$P^{\ast}$ is hereditary.
\end{proof}

As is well known (see for instance \cite{AR79}), the radicals $\mathfrak{P}%
_{\beta}$, $\mathfrak{P}_{\lambda}$, $\mathfrak{P}_{\kappa}$ and
$\operatorname{rad}$ are hereditary.

\begin{theorem}
\label{class}$\mathcal{P}_{\beta}$, $\mathcal{P}_{\lambda}$, $\mathcal{P}%
_{\kappa}$ and $\overline{\operatorname{rad}}^{\ast}$ are hereditary.
\end{theorem}

\begin{proof}
To apply Lemma \ref{aur} and Theorem \ref{heredit}, one has to show that the
algebraic under radicals $\Sigma_{\beta}$ (see Corollary \ref{baer}),
$\mathfrak{P}_{\lambda}$, $\mathfrak{P}_{\kappa}$ and $\operatorname{rad}$ are
weakly hereditary and nil exact. Since any ideal of a $\Sigma_{\beta}%
$-radical, locally nilpotent, nil or Jacobson-radical algebra have the same
structure, all the under radicals are weakly hereditary. It remains to
establish nil exactness. Let $A$ be an algebra, and let $I$ be an ideal of
$A$, $B=A/I$ and $a\in I^{\bot}$.

Let $x\in q_{I}^{-1}\left(  \Sigma_{\beta}\left(  B\right)  \right)  $; then
$q_{I}\left(  x\right)  =b_{1}+\ldots+b_{n}$, where $b_{i}$ is in a nilpotent
ideal of $B$, for each $i$. We take $x_{i}\in A$ such that $b_{i}=q_{I}\left(
x_{i}\right)  $. Then $q\left(  ax_{i}\right)  $ lies in a nilpotent ideal of
$B$ for every $i$: there is $n_{i}>0$ such that $q_{I}\left(  ax_{i}A\right)
^{n_{i}}=0$, i.e. $\left(  ax_{i}A\right)  ^{n_{i}}\subset I$. As
$ax_{i}A\subset I^{\bot}=0$ then $\left(  ax_{i}A\right)  ^{n_{i}+1}=0$. Hence
$ax_{i}\in\Sigma_{\beta}\left(  A\right)  $ for every $i$ and therefore
$ax=ax_{1}+\ldots+ax_{n}\in\Sigma_{\beta}\left(  A\right)  $.

Let $N\subset q_{I}^{-1}\left(  \mathfrak{P}_{\lambda}\left(  B\right)
\right)  $ be finite, and let $A_{aN}$ be the subalgebra generated by $aN$.
Then $q_{I}\left(  aN\right)  \subset\mathfrak{P}_{\lambda}\left(  B\right)  $
generates the nilpotent subalgebra: there is $m>0$ such that $q_{I}\left(
A_{aN}\right)  ^{m}=0$, i.e. $\left(  A_{aN}\right)  ^{m}\subset I$. But
$A_{aN}\subset I^{\bot}$, whence $\left(  A_{aN}\right)  ^{m+1}=0$. As $N$ is
arbitrary, we conclude that $aq_{I}^{-1}\left(  \mathfrak{P}_{\lambda}\left(
B\right)  \right)  \subset\mathfrak{P}_{\lambda}\left(  A\right)  $.

Let $x\in q_{I}^{-1}\left(  \mathfrak{P}_{\kappa}\left(  B\right)  \right)  $,
and let $J$ be the ideal of $A$ generated by $ax$. Then $q_{I}\left(
J\right)  $ consists of nilpotents. In particular, if $b\in J$ is arbitrary
then there is $k>0$ such that $b^{k}\in I$. As $b\in I^{\bot}$ then
$b^{k+1}=0$. This means that $ax\in\mathfrak{P}_{\kappa}\left(  A\right)  $.

Let $x\in q_{I}^{-1}\left(  \operatorname{rad}\left(  B\right)  \right)  $,
and let $\pi\in\operatorname{Irr}\left(  A\right)  $ be arbitrary. If
$\pi\left(  I\right)  \neq0$ then $\pi\left(  I\right)  $ is a strictly
irreducible algebra of operators and $\pi(a)\pi(I)=\pi(aI)=0$ implies that
$\pi(a)=0$, whence $\pi(ax)=0$. If $\pi\left(  I\right)  =0$ then there exists
$\tau\in\operatorname{Irr}\left(  B\right)  $ such that $\pi=\tau\circ q_{I}$;
as $\tau\left(  q_{I}\left(  ax\right)  \right)  =0$ then $\pi\left(
ax\right)  =0$. Therefore $ax\in$ $\operatorname{rad}\left(  A\right)  $.

All under radicals in consideration are nil exact. By Lemma \ref{aur} and
Theorem \ref{heredit}, $\mathcal{P}_{\beta}$, $\mathcal{P}_{\lambda}$,
$\mathcal{P}_{\kappa}$ and $\overline{\operatorname{rad}}^{\ast}$ are hereditary.
\end{proof}

Establishing heredity of the radical $\mathcal{P}_{\beta}$, Theorem
\ref{class} answers a question posed by Dixon \cite[Page 188]{D97}.

Recall that, for any algebra $A$, $\digamma\!\left(  A\right)  $ is the set of
all finite rank elements of $A$.

\begin{lemma}
\label{spfr}Let $A$ be a semiprime algebra. Then

\begin{enumerate}
\item $\digamma\!\left(  A\right)  \cap\operatorname{rad}\left(  A\right)  =0$.

\item If $A$ is normed then $\overline{\digamma\!\left(  A\right)  }%
\cap\overline{\operatorname{rad}\left(  A\right)  }=0$.
\end{enumerate}
\end{lemma}

\begin{proof}
$\left(  1\right)  $ Assume that $x\in\digamma\!\left(  A\right)
\cap\operatorname{rad}\left(  A\right)  $ is nonzero, and let $n$ be the rank
of \textrm{W}$_{x}$. Then $xAx$ is $n$-dimensional and the subalgebra
generated by $x$ in $A^{1}$ is at most $\left(  n+2\right)  $-dimensional. So
$x$ is an algebraic element in $\operatorname{rad}\left(  A\right)  $, whence
there is a minimal polynomial $p\left(  x\right)  $ on $x$ with degree at most
$n+2$. The roots of $p$ lie in the spectrum of $x$ with respect to $A^{1}$ by
\cite[Section 1.1.2]{B67}. As $x$ is in $\operatorname{rad}\left(  A\right)
$, these roots are equal to $0$, whence $x^{n+2}=0$.

Let $a\in A^{1}$ be arbitrary. As \textrm{W}$_{ax}=\mathrm{L}_{a}%
\mathrm{W}_{x}\mathrm{R}_{a}$, the rank of \textrm{W}$_{ax}$ is at most $n$.
As $ax\in$ $\operatorname{rad}\left(  A\right)  $, it follows that $\left(
ax\right)  ^{n+2}=0$ by above. Since $A^{1}x$ consists of nilpotents of fixed
degree, $A^{1}x$ is a nilpotent algebra by the Nagata-Higman theorem (in fact,
Dubnov-Nagata-Higman theorem, see \cite{Form}). As $A^{1}x$ is a non-zero left
nilpotent ideal of $A$, we obtain that $A$ is not semiprime, a contradiction.
Therefore $\digamma\!\left(  A\right)  \cap\operatorname{rad}\left(  A\right)
=0$.

$\left(  2\right)  $ Let $x,y\in I:=\overline{\digamma\!\left(  A\right)
}\cap\overline{\operatorname{rad}\left(  A\right)  }$ be arbitrary. Let
$\left(  a_{n}\right)  \subset\digamma\!\left(  A\right)  $ and $\left(
b_{m}\right)  \subset\operatorname{rad}\left(  A\right)  $ be such that
$a_{n}\rightarrow x$ and $b_{n}\rightarrow y$ as $n\rightarrow\infty$. Then
$a_{n}b_{n}=0$ for all $n$ by $\left(  1\right)  $, but $a_{n}b_{n}\rightarrow
xy$ as $n\rightarrow\infty$, whence $xy=0$ and $I^{2}=0$. As $I$ is an ideal
of $A$ and $A$ is semiprime, $I=0$.
\end{proof}

Let us, for brevity, call a subset $M$ of an algebra $A$ \textit{square zero}
if $ab=0$ for all $a,b\in M$.

\begin{theorem}
\label{baerher}$\mathfrak{P}_{\beta}=\mathfrak{R}_{\mathrm{hf}}\wedge
\operatorname{rad}$ and $\mathcal{R}_{\beta}=\mathcal{R}_{\mathrm{hf}}%
\wedge\overline{\operatorname{rad}}^{\ast}\leq\mathcal{R}_{\mathrm{jhc}%
}=\mathcal{R}_{\mathrm{hc}}\wedge\mathcal{R}_{\mathrm{cq}}$.
\end{theorem}

\begin{proof}
Let $A$ be an algebra, and let $\Sigma_{0}\left(  A\right)  $ denote the sum
of all square zero left ideals of $A$. If $L$ is a square zero left ideal of
$A$ then $La$ is also a square zero left ideal of $A$: $L(aL)a\subset(LL)a=0$.
So $\Sigma_{0}\left(  A\right)  $ is an ideal of $A$. It is easy to check that
$\Sigma_{0}$ is an under radical. If $I$ is a nilpotent ideal of $A$ then it
is obvious that $I\subset\Sigma_{0}^{\ast}\left(  A\right)  $. Therefore
$\Sigma_{\beta}\left(  A\right)  \subset\Sigma_{0}^{\ast}\left(  A\right)  $.
As $\Sigma_{0}^{\ast}$ is a radical then $\mathfrak{P}_{\beta}=\Sigma_{\beta
}^{\ast}\subset\Sigma_{0}^{\ast}$. On the other hand, every semiprime algebra
is $\Sigma_{0}^{\ast}$-semisimple, whence
\[
\mathfrak{P}_{\beta}=\Sigma_{0}^{\ast}.
\]
\ Note that every square zero left ideal consists of zero rank elements of
$A$. So $\Sigma_{0}\left(  A\right)  $ is an ideal of $A$ consisting of finite
rank elements. Therefore $\Sigma_{0}\leq\mathfrak{R}_{\mathrm{hf}}$, whence
\[
\mathfrak{P}_{\beta}=\Sigma_{0}^{\ast}\leq\mathfrak{R}_{\mathrm{hf}}.
\]
It is well-known that $\mathfrak{P}_{\beta}\leq\operatorname{rad}$. Thus
\[
\mathfrak{P}_{\beta}\leq\mathfrak{R}_{\mathrm{hf}}\wedge\operatorname{rad}.
\]

Let now $A$ be semiprime. We must to prove that $A$ is $\left(  \mathfrak{R}%
_{\mathrm{hf}}\wedge\operatorname{rad}\right)  $-semisimple. Assume, to the
contrary, that
\[
I:=\left(  \mathfrak{R}_{\mathrm{hf}}\wedge\operatorname{rad}\right)  \left(
A\right)  =\mathfrak{R}_{\mathrm{hf}}\left(  A\right)  \cap\operatorname{rad}%
\left(  A\right)  \neq0
\]
by $\left(  \ref{hp}\right)  $. By heredity of $\mathfrak{R}_{\mathrm{hf}}$,
$I$ is an $\mathfrak{R}_{\mathrm{hf}}$-radical ideal of $A$. By Theorem
\ref{hf}$\left(  5\right)  $, there is a non-zero finite rank element $x$ of
$A$ such that $x\in\operatorname{rad}\left(  A\right)  $, a contradiction to
Lemma \ref{spfr}. So $A$ is $\left(  \mathfrak{R}_{\mathrm{hf}}\wedge
\operatorname{rad}\right)  $-semisimple, whence
\[
\mathfrak{R}_{\mathrm{hf}}\wedge\operatorname{rad}\leq\mathfrak{P}_{\beta}.
\]

It is clear that $\overline{\mathfrak{P}_{\beta}}^{\ast}\leq\overline
{\mathfrak{R}_{\mathrm{hf}}}^{\ast}\wedge\overline{\operatorname{rad}}^{\ast}%
$, whence $\mathcal{P}_{\beta}\leq\mathcal{R}_{\mathrm{hf}}\wedge
\overline{\operatorname{rad}}^{\ast}$.

Let $A$ be a $\mathcal{P}_{\beta}$-semisimple normed algebra; then $A$ is
semiprime. Assume, to the contrary, that $I:=\left(  \mathcal{R}_{\mathrm{hf}%
}\wedge\overline{\operatorname{rad}}^{\ast}\right)  \left(  A\right)  \neq0$.
By Corollary \ref{herr}, $I=\overline{\operatorname{rad}}^{\ast}\left(
\mathcal{R}_{\mathrm{hf}}\left(  A\right)  \right)  $. Therefore
$J:=\overline{\operatorname{rad}}\left(  \mathcal{R}_{\mathrm{hf}}\left(
A\right)  \right)  \neq0$. Arguing as above, one has that there is a non-zero
element $x\in\overline{\digamma\!\left(  A\right)  }$ such that $x\in
J\subset\overline{\operatorname{rad}}\left(  A\right)  $, whence $x=0$ by
Lemma \ref{spfr}, a contradiction. Hence $A$ is $\left(  \mathcal{R}%
_{\mathrm{hf}}\wedge\overline{\operatorname{rad}}^{\ast}\right)  $-semisimple.
Therefore $\mathcal{P}_{\beta}\leq\mathcal{R}_{\mathrm{hf}}\wedge
\overline{\operatorname{rad}}^{\ast}\leq\mathcal{P}_{\beta}$.

By definition, $\mathcal{R}_{\mathrm{jhc}}=\mathcal{R}_{\mathrm{hc}}%
\wedge\operatorname{Rad}^{r}$. By $\left(  \ref{aff}\right)  $, $\mathcal{R}%
_{\mathrm{hc}}\wedge\operatorname{Rad}^{r}\leq\mathcal{R}_{\mathrm{cq}}$,
whence
\[
\mathcal{R}_{\mathrm{jhc}}=\left(  \mathcal{R}_{\mathrm{hc}}\wedge
\operatorname{Rad}^{r}\right)  \wedge\mathcal{R}_{\mathrm{cq}}=\mathcal{R}%
_{\mathrm{hc}}\wedge\mathcal{R}_{\mathrm{cq}}.
\]
As every nilpotent ideal of $A$ lies in $\mathcal{R}_{\mathrm{cq}}\left(
A\right)  $ then $\mathfrak{P}_{\beta}\leq\mathcal{R}_{\mathrm{cq}}$, whence
also $\mathcal{P}_{\beta}\leq\mathcal{R}_{\mathrm{cq}}$. Taking into
consideration that $\mathcal{P}_{\beta}=\mathcal{R}_{\mathrm{hf}}%
\wedge\overline{\operatorname{rad}}^{\ast}<\mathcal{R}_{\mathrm{hc}}$, we
obtain that $\mathcal{P}_{\beta}\leq\mathcal{R}_{\mathrm{jhc}}$.
\end{proof}

\section{Centralization}

\subsection{Commutative ideals and centralization of radicals}

Let $A\in\mathfrak{U}_{\mathrm{a}}$ be an algebra, and let $Z\left(  A\right)
$ be the center of $A$. An ideal $J$ of $A$ is called \textit{central} if
$J\subset Z\left(  A\right)  $. Let $\Sigma_{a}\left(  A\right)  $ be the sum
of all commutative ideals of $A$.

\begin{lemma}
\label{commut}

\begin{enumerate}
\item $\Sigma_{a}$ is a preradical.

\item If $A$ is a semiprime algebra then $\Sigma_{a}\left(  A\right)  $ is the
largest central ideal of $A$ and

\begin{enumerate}
\item $\Sigma_{a}\left(  A\right)  =\left\{  x\in A:x\left[  a,b\right]
=0\text{ }\forall a,b\in A\right\}  $;

\item $\Sigma_{a}\left(  J\right)  =J\cap\Sigma_{a}\left(  A\right)  $ for
every ideal $J$ of $A$.
\end{enumerate}
\end{enumerate}
\end{lemma}

\begin{proof}
$\left(  1\right)  $ is straightforward.

$\left(  2\right)  $ Let $J$ be a commutative ideal of $A$, and let $x,y\in
J$, $a,b\in A$ be arbitrary. Then
\[
\left[  x,a\right]  y=x\left(  ay\right)  -axy=\left(  ay\right)
x-axy=a\left[  y,x\right]  =0.
\]
As $b\left[  x,a\right]  \in J$, it follows that $\left[  x,a\right]  b\left[
x,a\right]  =0$, i.e. $\left[  x,a\right]  A\left[  x,a\right]  =0$. As $A$ is
semiprime, $\left[  x,a\right]  =0$, i.e. $J\subset Z\left(  A\right)  $,
whence $\Sigma_{a}\left(  A\right)  \subset Z\left(  A\right)  $. So
$\Sigma_{a}\left(  A\right)  $ is a commutative ideal of $A$. Then, by
definition, $\Sigma_{a}\left(  A\right)  $ is the largest ideal that lies in
$Z\left(  A\right)  $.

(a) Let $K=\left\{  x\in A:x\left[  a,b\right]  =0\text{ }\forall a,b\in
A\right\}  $. Let $x\in K$, $a,b,c\in A$ be arbitrary. It is clear that $K$ is
an ideal of $A$: indeed, one has
\[
xa\left[  b,c\right]  =xa\left[  b,c\right]  +x\left[  a,c\right]  b=x\left[
ab,c\right]  =0.
\]
Then $\left[  x,a\right]  b\in K$, whence $\left[  x,a\right]  b\left[
x,a\right]  =0$. Since $A$ is semiprime, $\left[  x,a\right]  =0$, i.e. $x\in
Z\left(  A\right)  $. Therefore $K\subset\Sigma_{a}\left(  A\right)  $.

Let $y\in\Sigma_{a}\left(  A\right)  $ be arbitrary. As $y,ya\in Z\left(
A\right)  $, we obtain that
\[
y\left[  a,b\right]  =\left[  ya,b\right]  -\left[  y,b\right]  a=0.
\]
Thus $\Sigma_{a}\left(  A\right)  \subset$ $K$.

$\left(  \text{b}\right)  $ Obviously $J\cap\Sigma_{a}\left(  A\right)
\subset\Sigma_{a}\left(  J\right)  $. As $J$ is a semiprime algebra,
$\Sigma_{a}\left(  J\right)  $ is described as in $\left(  \text{a}\right)  $
for $A=J$. Then $\Sigma_{a}\left(  J\right)  $ is a left ideal of $A$ that
lies in the center of $J$. By symmetry, $\Sigma_{a}\left(  J\right)  $ is a
commutative ideal of $A$ and so $\Sigma_{a}\left(  J\right)  \subset\Sigma
_{a}\left(  A\right)  $ by definition.
\end{proof}

\begin{remark}
\label{comsem}It follows from Lemma $\ref{commut}\left(  2\right)  $ that if
$A$ is a semiprime normed algebra then $\Sigma_{a}\left(  A\right)  $ is closed.
\end{remark}

Let $P$ be a preradical and let%
\[
P^{a}=\Sigma_{a}\ast P;
\]
$P^{a}$ is called the \textit{centralization} of $P$. The map $P\longmapsto
P^{a}$ is called the \textit{centralization procedure}.

Let $P$ be a radical such that $P\geq\mathfrak{P}_{\beta}$. If $P$ is
topological on $\mathfrak{U}_{\mathrm{n}}$ or $\mathfrak{U}_{\mathrm{b}}$, the
inequality $P\geq\mathfrak{P}_{\beta}$ is assumed to keep on $\mathfrak{U}%
_{\mathrm{n}}$ or $\mathfrak{U}_{\mathrm{b}}$, respectively; in such a case
the inequality $P\geq\mathcal{P}_{\beta}$ also holds.

\begin{theorem}
\label{pab}Let $P$ be a radical such that $P\geq\mathfrak{P}_{\beta}$. Then

\begin{enumerate}
\item $P^{a}$ is an under radical;

\item If $P$ is hereditary then $P^{a}$ is hereditary.
\end{enumerate}
\end{theorem}

\begin{proof}
Note that $P^{a}$ is a preradical by Lemma \ref{conv}. Remark \ref{comsem}
guarantees that $P^{a}$ is a topological preradical in the topological case.

Let $A$ be an algebra, and let $J$ be an ideal of $A$. We claim that
$A/P\left(  A\right)  $ is semiprime: indeed, $\mathfrak{P}_{\beta}\left(
A/P\left(  A\right)  \right)  \subset P\left(  A/P\left(  A\right)  \right)
=0$. Similarly, $J/P\left(  J\right)  $ is a semiprime algebra. It follows
from Lemma \ref{commut}(2a) that
\begin{align}
P^{a}\left(  A\right)   &  =\left\{  x\in A:x\left[  a,b\right]  \in P\left(
A\right)  \text{ }\forall a,b\in A\right\}  ,\label{pa}\\
P^{a}\left(  J\right)   &  =\left\{  x\in J:x\left[  y,z\right]  \in P\left(
J\right)  \text{ }\forall y,z\in J\right\}  .\nonumber
\end{align}

Let $x\in P^{a}\left(  J\right)  $ be arbitrary, and let $q:A/P\left(
J\right)  \longrightarrow A/P\left(  A\right)  $ be the standard quotient map.
Then
\[
\left(  q\circ q_{P\left(  J\right)  }\right)  \left(  x\right)  \in\Sigma
_{a}\left(  q\left(  J/P\left(  J\right)  \right)  \right)  \subset\Sigma
_{a}\left(  A/P\left(  A\right)  \right)
\]
by Lemma \ref{commut}(2b), whence $x\in P^{a}\left(  A\right)  $. So
\[
P^{a}\left(  J\right)  \subset P^{a}\left(  A\right)  .
\]

$\left(  \text{2}\right)  $ Let $x\in J\cap P\left(  A\right)  $. Then
$x\left[  a,b\right]  \in J\cap P\left(  A\right)  $ for every $a,b\in A$, in
particular for every $a,b\in J$. As $J\cap P\left(  A\right)  =P\left(
J\right)  $, it follows that $x\in P^{a}\left(  J\right)  $. Therefore,
$P^{a}\left(  J\right)  =J\cap P^{a}\left(  A\right)  $. Thus $P^{a}$ is a
hereditary preradical, in particular, it is an under radical. This completes
the proof of $\left(  \text{2}\right)  $.

$\left(  \text{1}\right)  $ It follows from (\ref{pa}) applied to the opposite
algebra of $A$ that $P^{a}\left(  J\right)  $ is an ideal of $A$. It is clear
that $P\left(  P^{a}\left(  A\right)  \right)  \subset P\left(  A\right)
\subset P^{a}\left(  A\right)  $, and then we obtain that $P\left(  A\right)
=P\left(  P\left(  A\right)  \right)  \subset P\left(  P^{a}\left(  A\right)
\right)  $, whence
\begin{equation}
P\left(  A\right)  =P\left(  P^{a}\left(  A\right)  \right)  . \label{ppa}%
\end{equation}
It follows from Lemma \ref{commut}(2) that $\left[  a,b\right]  \in P\left(
A\right)  $ for every $a,b\in P^{a}\left(  A\right)  $, and from (\ref{pa})
for $J=P^{a}\left(  A\right)  $ and (\ref{ppa}) that
\[
P^{a}\left(  P^{a}\left(  A\right)  \right)  =\left\{  x\in P^{a}\left(
A\right)  :x\left[  a,b\right]  \in P\left(  A\right)  \text{ }\forall a,b\in
P^{a}\left(  A\right)  \right\}  =P^{a}\left(  A\right)  .
\]
This completes the proof of $\left(  \text{1}\right)  $.
\end{proof}

We see from Theorem \ref{pab} that in general $P^{a}$ is an under radical (if
$P\geq\mathfrak{P}_{\beta}$), and one can apply the convolution procedure to
obtain the corresponding radical $P^{a\ast}$ (\textit{the centralized }%
$P$\textit{ radical}). But $P^{a\ast}$ has a much more complicated description
than $P^{a}$ which can be transparently defined by the condition: $P^{a}(A)$
is the largest ideal of $A$ commutative modulo $P(A)$.

Below we study the question:\textit{ when }$P^{a}$\textit{ itself is a
radical?} Earlier the answer was obtained in the case of the compactly
quasinilpotent radical $R_{\mathrm{cq}}$ --- it was proved in \cite{TR3} that
$\mathcal{R}_{\mathrm{cq}}^{a}$ is a topological radical on $\mathfrak{U}%
_{\mathrm{n}}$.

Theorem \ref{pab} and Lemma \ref{conv-ext} immediately imply the following result.

\begin{lemma}
\label{exc1}Let $P$ be a radical such that $P\geq\mathfrak{P}_{\beta}$. Then
$P^{a}$ is a radical if and only if, for every [closed] ideal $J$ of an
algebra $A$ such that the algebras $A/J$ and $J$ are $P^{a}$-radical, the
algebra $A$ is $P^{a}$-radical.
\end{lemma}

In the following proposition we obtain a tool for the proof that $P^{a}$ is a
radical, which turns out to be convenient for many important examples of radicals.

\begin{proposition}
\label{exc2}Let $P$ be a radical such that $P\geq\mathfrak{P}_{\beta}$. Then
$P^{a}$ is a radical if and only if

\begin{itemize}
\item[(*)] any $P$-semisimple algebra $B$ with $P^{a}$-radical $B/\Sigma
_{a}\left(  B\right)  $, is commutative.
\end{itemize}
\end{proposition}

\begin{proof}
Suppose that (*) is true. By Lemma \ref{exc1}, it suffices to show that if a
[closed] ideal $J$ of an algebra $A$ and the quotient $A/J$ are $P^{a}%
$-radical then $A$ is $P^{a}$-radical.

Let
$B=A/P\left(  A\right)  $ and $I$ be the [closure of the] image of $J$ under
the standard map $A\longrightarrow A/P\left(  A\right)  $. Then $B$ is
$P$-semisimple, in particular semiprime. Since $J$ is commutative modulo
$P(J)\subset P(A)$ then $I$ is a commutative ideal of $B$. Therefore
$I\subset\Sigma_{a}\left(  B\right)  $ and, using the standard quotient map
$q:$ $A/J\longrightarrow B/I\longrightarrow B/\Sigma_{a}\left(  B\right)  $,
we get
\[
B/\Sigma_{a}\left(  B\right)  =q\left(  A/J\right)  =q\left(  P^{a}\left(
A/J\right)  \right)  \subset P^{a}\left(  B/\Sigma_{a}\left(  B\right)
\right)
\]
i.e., $B/\Sigma_{a}\left(  B\right)  $ is $P^{a}$-radical.

By (*), $B$ is commutative whence $B = \Sigma_{a}(B)$ and
\[
A=q_{P\left(  A\right)  }^{-1}\left(  \Sigma_{a}\left(  A/P\left(  A\right)
\right)  \right)  =P^{a}\left(  A\right)  ,
\]
i.e. $A$ is $P^{a}$-radical. By Proposition \ref{exc1}, $P^{a}$ is a radical.

Conversely, let $P^{a}$ be a radical, and let $A$ be a $P$-semisimple algebra
such that $A/\Sigma_{a}\left(  A\right)  $ is $P^{a}$-radical. Let now
$q:A/\Sigma_{a}(A)\rightarrow A/P^{a}(A)$ be the natural morphism. Since the
first algebra is $P^{a}$-radical then its image is contained in $P^{a}%
(A/P^{a}(A))=0$. Therefore $A/P^{a}(A)=0$, whence $A=P^{a}(A)$. Since
$P(A)=0$, we get that $P^{a}(A)=\Sigma_{a}(A)$. Thus $A=\Sigma_{a}(A)$, i.e.,
$A$ is commutative. We proved (*).
\end{proof}

\begin{lemma}
\label{nilcom}Let $A$ be an algebra, and let $J$ be an ideal of $A$ such that
$\Sigma_{a}\left(  A\right)  \subset J$. If $A$ is semiprime and $J/\Sigma
_{a}\left(  A\right)  $ is commutative or nilpotent then $J=\Sigma_{a}\left(
A\right)  $.
\end{lemma}

\begin{proof}
Let $\Sigma_{a}\left(  A\right)  \neq0$. Suppose, to the contrary, that
$J\neq\Sigma_{a}\left(  A\right)  $. If $J/\Sigma_{a}(A)$ is nilpotent then
there exists an ideal $K$ of $A$ such that $\Sigma_{a}\left(  A\right)
\subsetneqq K\subset J$ and $K^{2}\subset\Sigma_{a}\left(  A\right)  $.
Therefore $[K,K]\subset\Sigma_{a}\left(  A\right)  $. This is also true in the
case when $J/\Sigma_{a}\left(  A\right)  $ is commutative: in this case we can
take $K=J$.

Let $x,y\in K$ be arbitrary. Then $z:=\left[  x,y\right]  \in\Sigma_{a}\left(
A\right)  $. By Lemma \ref{commut},
\[
z^{2}=z\left[  x,y\right]  =0.
\]
If $z\neq0$ then $\mathfrak{P}_{\beta}\left(  \Sigma_{a}\left(  A\right)
\right)  \neq0$, whence $\mathfrak{P}_{\beta}\left(  A\right)  \neq0$, a
contradiction to the assumption that $A$ is semiprime. Therefore $z=0$, whence
$K$ is a commutative ideal of $A$. By definition, $K\subset\Sigma_{a}\left(
A\right)  $, a contradiction. Therefore $J=\Sigma_{a}\left(  A\right)  $.
\end{proof}

\begin{theorem}
\label{cent}Let $P\geq\mathfrak{P}_{\beta}$ be a weakly hereditary under
radical. Then if one of the following conditions holds:

\begin{enumerate}
\item $\mathfrak{U}$ is [algebraically] universal and $P$ is pliant or strict;

\item $P$ is defined on Banach algebras and satisfies the condition of Banach heredity.
\end{enumerate}

\noindent then $P^{\ast a}$ is a radical.
\end{theorem}

\begin{proof}
First we prove the following statement:

(*) \textit{Let} $A$ \textit{be a }$P^{\ast}$\textit{-semisimple algebra and}
$B=A/\Sigma_{a}(A)$. \textit{If} $B$ \textit{is commutative modulo} $P(B)$
\textit{then} $A=\Sigma_{a}\left(  A\right)  $.

Indeed, $A$ is $P$-semisimple and semiprime. Let $J:=\Sigma_{a}(A)$,
$I=J^{\bot}$ (the left annihilator of $J$). Then $[A,A]\subset I$ by Lemma
\ref{commut}, hence either $A$ is commutative or $I\neq0$.

So we assume, to the contrary, that $I\neq0$. If $I\cap J\neq0$ then it is a
nilpotent ideal of $A$, a contradiction to semiprimeness of $A$.

Therefore $I\cap J=0$, and the restriction $f$ of $q_{J}:A\longrightarrow B$
to $I$ is one-to-one onto $K:=q_{J}(I)\cong\left(  I+J\right)  /J$. If $K$ is
commutative then $I+J$ is commutative modulo $J$, whence $I+J$ is commutative
by Lemma \ref{nilcom}, because $A$ is semiprime. This contradicts to the
maximality of $\Sigma_{a}(A)$. So $K$ is not commutative. Since $[K,K]\subset
P(B)$, we see that
\[
K\cap P(B)\neq0.
\]

In the case when $A$ is a Banach algebra, $K$ is a Banach ideal of $B$ (see
Section \ref{banach}). If $P$ satisfies the condition of Banach heredity then
\[
P\left(  K;\left\Vert \cdot\right\Vert _{K}\right)  =K\cap P\left(  B\right)
\neq0.
\]
Taking the norm $\left\Vert \cdot\right\Vert _{K}$ as in Section \ref{banach},
we have that $f$ is an isometric isomorphism between $I$ and $\left(
K;\left\Vert \cdot\right\Vert _{K}\right)  $. Then $P(I)\neq0$ whence
$P(A)\neq0$, a contradiction.

Let $P$ be pliant or strict, and let $L=K\cap P(B)$. As $L$ is a non-zero
ideal of the $P$-radical algebra $P\left(  B\right)  $ then
\[
P\left(  L\right)  \neq0
\]
by the weak heredity of $P$.

Since $f$ is a [bounded] isomorphism $I\longrightarrow K$, then $f\left(
P\left(  f^{-1}(L)\right)  \right)  =P(L)$ whenever $P$ is pliant, or
$\overline{f\left(  P(f^{-1}(L))\right)  }=P(L)$ whenever $P$ is strict. In
any case
\[
P\left(  f^{-1}(L)\right)  \neq0.
\]

As $f^{-1}(L)$ is an ideal of $I$, this implies that%
\[
0\neq P\left(  f^{-1}(L)\right)  \subset P(I)\subset P(A),
\]
a contradiction. The proof of the statement (*) is completed.

Let now $B$ be commutative modulo $P^{\ast}\left(  B\right)  $, and let
$\left(  P_{\alpha}\right)  $ be the convolution chain of $P$; then $P_{1}=P$.
Let $A_{1}=\left\{  x\in A:x/J\in P\left(  B\right)  \right\}  $. Then $A_{1}$
is an ideal of $A$; therefore $A_{1}$ is $P^{\ast}$-semisimple, and
$J\subset\Sigma_{a}\left(  A_{1}\right)  =J\cap A_{1}$ by Lemma \ref{commut}.
So
\[
\Sigma_{a}\left(  A_{1}\right)  =J.
\]
Let $B_{1}=A_{1}/\Sigma_{a}\left(  A_{1}\right)  $. As $B_{1}=P\left(
B\right)  $ then $B_{1}=P\left(  B_{1}\right)  $, in particular $B_{1}$ is
commutative modulo $P\left(  B_{1}\right)  $. Then $A_{1}$ is commutative by
(*). Therefore $A_{1}=J$ and $P\left(  B\right)  =0$. Then $P^{\ast}\left(
B\right)  =0$, $B$ is commutative and $A$ is commutative by Lemma
\ref{nilcom}. As $P^{\ast}$ is a radical, $P^{\ast a}$ is a radical by
Proposition \ref{exc2}.
\end{proof}

\begin{remark}
\label{centr}In the assumptions of Theorem \ref{cent}, if $P^{\ast}$ is
hereditary then $P^{\ast a}$ is hereditary by Theorem $\ref{pab}$. If $P$ is
an algebraic hereditary radical and $P\geq\mathfrak{P}_{\beta}$, then
$\overline{P}$ is a strict and weakly hereditary under radical by Theorem
$\ref{closunder}$ and Lemma $\ref{aur}$, whence $\overline{P}^{\ast a}$ is a radical.
\end{remark}

\subsection{Centralization of classical radicals}

Now we apply Proposition \ref{cent} and Remark \ref{centr} to the classical radicals.

\begin{corollary}
\label{centRad}$\operatorname{rad}^{a}$, $\overline{\operatorname{rad}}^{\ast
a}$, and $\operatorname{Rad}^{a}$ are hereditary radicals.
\end{corollary}

\begin{proof}
The radical $\operatorname{Rad}$ satisfies the condition of Banach heredity,
$\operatorname{rad}$ is pliant and $\overline{\operatorname{rad}}$ is weakly
hereditary. By Theorem \ref{cent}, the centralizations of them are radicals.
As $\overline{\operatorname{rad}}^{\ast}$ is hereditary,  $\overline
{\operatorname{rad}}^{\ast a}$ is a hereditary radical.
\end{proof}

\begin{theorem}
\label{rada}$\operatorname{Rad}^{a}\left(  A\right)  =\operatorname{rad}%
^{\dim>1}\left(  A\right)  =\left\{  x\in A:\rho\left(  x\left[  a,b\right]
\right)  =0\text{ }\forall a,b\in A\right\}  $ for every Banach algebra $A$.
\end{theorem}

\begin{proof}
Set $J=\operatorname{Rad}^{a}\left(  A\right)  $. As $\operatorname{rad}%
^{\dim>1}\left(  A\right)  $ admits only one-dimensional strictly irreducible
representations,  $\operatorname{rad}^{\dim>1}\left(  A\right)  \subset J$. If
$\pi\left(  J\right)  \neq0$ for some $\pi\in\operatorname{Irr}^{\dim
>1}\left(  A\right)  $ then $\dim\pi\left(  A\right)  =\dim\pi\left(
J\right)  =1$ by Jacobson's density theorem and commutativity of $J$ modulo
$\operatorname{Rad}\left(  A\right)  $, a contradiction. So $J\subset
\operatorname{rad}^{\dim>1}\left(  A\right)  $.

Set $I=\left\{  x\in A:\rho\left(  x\left[  a,b\right]  \right)  =0\text{
}\forall a,b\in A\right\}  $. Then $J\subset I$ by (\ref{pa}) for
$P=\operatorname{Rad}$.

Let $x\in I$ be arbitrary. Assume that $\pi\left(  x\right)  \xi=\zeta\neq0$
for some $\pi\in\operatorname{Irr}^{\dim>1}\left(  A\right)  $. Taking a
vector $\eta$ such that $\zeta$ and $\eta$ are linearly independent, one can
find $a,b\in A$ by Jacobson's density theorem such that $\pi\left(  a\right)
\eta=\xi$, $\pi\left(  a\right)  \zeta=0$ and $\pi\left(  b\right)  \zeta
=\eta$. Then $\pi\left(  \left[  a,b\right]  \right)  \zeta=\xi$, whence
\[
\pi\left(  x\left[  a,b\right]  \right)  \zeta=\zeta,
\]
a contradiction with $\rho\left(  x\left[  a,b\right]  \right)  =0$. Therefore
$\pi\left(  x\right)  =0$ for every $\pi\in\operatorname{Irr}^{\dim>1}\left(
A\right)  $. This means that $x\in\operatorname{rad}^{\dim>1}\left(  A\right)
$ and $I\subset J$.
\end{proof}

\begin{remark}
$\operatorname{Rad}^{ar}\left(  A\right)  =\operatorname{Rad}^{ra}\left(
A\right)  =\left\{  x\in A:\rho\left(  x\left[  a,b\right]  \right)  =0\text{
}\forall a,b\in\widehat{A}\right\}  $ for every normed algebra $A$.
\end{remark}

Now we turn to the radicals of Baer, Levitzki, K\"{o}the, to the hypofinite
radical, and to their topological counterparts.

\begin{corollary}
$\mathfrak{P}_{\beta}^{a}$, $\mathfrak{P}_{\lambda}^{a}$, $\mathfrak{P}%
_{\kappa}^{a}$, $\mathfrak{R}_{\mathrm{hf}}^{a}$, $\mathcal{P}_{\beta}^{a}$,
$\mathcal{P}_{\lambda}^{a}$, $\mathcal{P}_{\kappa}^{a}$ and $\mathcal{R}%
_{\mathrm{hf}}^{a}$ are hereditary radicals.
\end{corollary}

\begin{proof}
Follows from Theorem \ref{cent} and Remark \ref{centr} because all radicals
are hereditary, and the closures of $\mathfrak{P}_{\beta}$, $\mathfrak{P}%
_{\lambda}$, $\mathfrak{P}_{\kappa}$, $\mathfrak{R}_{\mathrm{hf}}$ are strict
and weakly hereditary.
\end{proof}

Recall that the closed-nil radical $\mathcal{P}_{\mathrm{nil}}$ is the
restriction of $\mathcal{P}_{\beta}$ to Banach algebras.

\begin{corollary}
$\mathcal{P}_{\mathrm{nil}}^{a}$ is a hereditary radical.
\end{corollary}

\begin{lemma}
\label{compa}Let $P_{1}$ and $P_{2}$ be radicals such that $\mathfrak{P}%
_{\beta}\leq P_{1}<P_{2}$. If $P_{1}=P_{2}$ on commutative algebras then
$P_{1}^{a}<P_{2}^{a}$.
\end{lemma}

\begin{proof}
Let $A$ be an algebra such that $P_{1}\left(  A\right)  \neq P_{2}\left(
A\right)  $. Assume, to the contrary, that $P_{1}^{a}\left(  A\right)
=P_{2}^{a}\left(  A\right)  $. Let $B=A/P_{1}\left(  A\right)  $. Then $B$ is
semiprime and therefore commutative by Lemma \ref{commut}. On the other hand,
$B$ is $P_{1}$-semisimple and therefore $P_{2}$-semisimple by our assumption.
But $B$ is not $P_{2}$-semisimple because $P_{2}\left(  A\right)
/P_{1}\left(  A\right)  $ is $P_{2}$-radical, non-zero and lies in
$P_{2}\left(  B\right)  $, a contradiction.
\end{proof}

\begin{corollary}
$\operatorname{rad}_{\mathrm{n}}^{a}<\operatorname{rad}_{\mathrm{b}}%
^{a}<\operatorname{Rad}^{ra}$, $\mathfrak{P}_{\beta}^{a}<\mathfrak{P}%
_{\lambda}^{a}<\mathfrak{P}_{\kappa}^{a}$ and $\mathcal{P}_{\beta}%
^{a}<\mathcal{P}_{\lambda}^{a}<\mathcal{P}_{\kappa}^{a}$.
\end{corollary}

\begin{proof}
The result follows from Lemma \ref{compa}, (\ref{comprad}), (\ref{compnil}),
Theorem \ref{classical}, the coincidence of $\operatorname{Rad}^{ra}$ with
$\operatorname{rad}_{\mathrm{n}}^{a}$ on commutative algebras (by Theorem
\ref{rada}), and also the coincidence of $\mathfrak{P}_{\kappa}^{a}$ with
$\mathfrak{P}_{\beta}^{a}$ on commutative algebras.
\end{proof}

\begin{problem}
Is $\mathcal{R}_{\mathrm{hc}}^{a}$ a radical?
\end{problem}

\subsection{The centralization of the tensor Jacobson radical}

It is not clear if $\mathcal{R}_{\mathrm{cq}}$ and $\mathcal{R}_{t}$ are
strict, so we cannot apply Theorem \ref{cent} to them. However, $\mathcal{R}%
_{\mathrm{cq}}^{a}$ is a radical by \cite{TR3}, and it will be proved that
$\mathcal{R}_{t}^{a}$ is a radical.

First we note that the maps $\mathcal{R}_{t}$ and $\mathcal{R}_{t}^{a}$ are regular.

\begin{theorem}
\label{regt}Let $A$ be a normed algebra. Then $\mathcal{R}_{t}\left(
B\right)  =B\cap\mathcal{R}_{t}\left(  A\right)  $ and $\mathcal{R}_{t}%
^{a}\left(  B\right)  =B\cap\mathcal{R}_{t}^{a}\left(  A\right)  $ for every
dense subalgebra $B$ of $A$.
\end{theorem}

\begin{proof}
As $\mathcal{R}_{t}\left(  B\right)  $ is the set of all $a\in B$ such that
$a\otimes c\in\operatorname{Rad}\left(  B\widehat{\otimes}C\right)
=\operatorname{Rad}\left(  A\widehat{\otimes}C\right)  $ for every normed
algebra $C$ and every $c\in C$ \cite{TR2}, then clearly
\[
\mathcal{R}_{t}\left(  B\right)  =B\cap\mathcal{R}_{t}\left(  A\right)  .
\]
By $\left(  \ref{pa}\right)  $, $\mathcal{R}_{t}^{a}\left(  B\right)
=\left\{  x\in B:x\left[  a,b\right]  \in\mathcal{R}_{t}\left(  B\right)
\text{ }\forall a,b\in B\right\}  $, whence we obtain that
\[
\mathcal{R}_{t}^{a}\left(  B\right)  =B\cap\mathcal{R}_{t}^{a}\left(
A\right)  .
\]

\end{proof}

We need here and will use also in the further sections a result on tensor
radius of operator families, analogous to the result about the joint spectral
radius established in \cite[Lemma 4.2]{ST00}. If $T$ is an operator on a space
$X$ and $Y$ is an invariant subspace for $T$ then $T|_{Y}$ denotes the
restriction of $T$ to $Y$, and $T|_{X/Y}$ --- the operator induced by $T$ on
$X/Y$. If $Y$ is invariant for a set $M$ of operators then $M|_{Y}=\left\{
T|_{Y}:T\in M\right\}  $ and $M|_{X/Y}=\left\{  T|_{X/Y}:T\in M\right\}  $;
this is transferred similarly to families of operators.

Recall that the tensor spectral radius $\rho_{t}$ is defined in Section
\ref{stensor} and
\[
\rho_{t}\left(  M^{m}\right)  =\rho_{t}\left(  M\right)  ^{m}%
\]
for every $m>0$ and every summable family $M$ in a normed algebra.

\begin{theorem}
\label{invar}Let $M=\left(  a_{i}\right)  _{1}^{\infty}$ be a summable family
of operators in $\mathcal{B}\left(  X\right)  $, and let $Y$ be an invariant
closed subspace for $M$. Then
\[
\rho_{t}\left(  M\right)  =\max\left\{  \rho_{t}\left(  M|_{X/Y}\right)
,\rho_{t}\left(  M|_{Y}\right)  \right\}  .
\]

\end{theorem}

\begin{proof}
It is clear that $\max\left\{  \rho_{t}\left(  M|_{X/Y}\right)  ,\rho
_{t}\left(  M|_{Y}\right)  \right\}  \leq\rho_{t}\left(  M\right)  $. Let us
prove the converse inequality. Let $q:X\longrightarrow V:=X/Y$ be the standard
quotient map. Note that
\[
\left\Vert M^{2n}\right\Vert _{+}=\sum_{i=1}^{\infty}\sum_{j=1}^{\infty
}\left\Vert b_{i}b_{j}\right\Vert
\]
where $M^{n}=\left(  b_{i}\right)  _{1}^{\infty}$. For every $\varepsilon>0$,
there is $x_{ij}\in X$ with $\left\Vert x_{ij}\right\Vert =1$ such that
$\left\Vert b_{i}b_{j}\right\Vert \leq\left\Vert b_{i}b_{j}x_{ij}\right\Vert
+\varepsilon.$ Clearly $\left\Vert b_{i}b_{j}x_{ij}\right\Vert \leq\left\Vert
b_{i}y\right\Vert +\left\Vert b_{i}\right\Vert \left\Vert b_{j}x_{ij}%
-y\right\Vert $ for any $y\in Y$. Take $y=y_{ij}$ such that $\left\Vert
b_{j}x_{ij}-y_{ij}\right\Vert \leq\left\Vert q\left(  b_{j}x_{ij}\right)
\right\Vert +\varepsilon$. Then
\[
\left\Vert y_{ij}\right\Vert \leq\left\Vert b_{j}x_{ij}\right\Vert +\left\Vert
b_{j}x_{ij}-y_{ij}\right\Vert \leq\left\Vert b_{j}x_{ij}\right\Vert
+\left\Vert q\left(  b_{j}x_{ij}\right)  \right\Vert +\varepsilon.
\]
Hence
\begin{align*}
\left\Vert b_{i}b_{j}\right\Vert  &  \leq\left\Vert b_{i}b_{j}x_{ij}%
\right\Vert +\varepsilon\leq\left\Vert b_{i}y_{ij}\right\Vert +\left\Vert
b_{i}\right\Vert \left\Vert b_{j}x_{ij}-y_{ij}\right\Vert +\varepsilon\\
&  \leq\left\Vert b_{i}|_{Y}\right\Vert \left(  \left\Vert b_{j}%
x_{ij}\right\Vert +\left\Vert q\left(  b_{j}x_{ij}\right)  \right\Vert
+\varepsilon\right)  +\left\Vert b_{i}\right\Vert \left(  \left\Vert q\left(
b_{j}x_{ij}\right)  \right\Vert +\varepsilon\right)  +\varepsilon
\end{align*}
Taking the limit for $\varepsilon\rightarrow0$, we get
\[
\Vert b_{i}b_{j}\Vert\leq\left\Vert b_{i}|_{Y}\right\Vert \left(  \left\Vert
b_{j}\right\Vert +\left\Vert b_{j}|_{V}\right\Vert \right)  +\left\Vert
b_{i}\right\Vert \left\Vert b_{j}|_{V}\right\Vert .
\]
Therefore
\begin{align*}
\left\Vert M^{2n}\right\Vert _{+}  &  \leq\sum_{i}\sum_{j}\left(  \left\Vert
b_{i}|_{Y}\right\Vert \left\Vert b_{j}\right\Vert +\left\Vert b_{i}%
|_{Y}\right\Vert \left\Vert b_{j}|_{V}\right\Vert +\left\Vert b_{i}\right\Vert
\left\Vert b_{j}|_{V}\right\Vert \right) \\
&  \leq\left\Vert M^{n}|_{Y}\right\Vert _{+}\left\Vert M^{n}\right\Vert
_{+}+\left\Vert M^{n}|_{Y}\right\Vert _{+}\left\Vert M^{n}|_{V}\right\Vert
_{+}+\left\Vert M^{n}\right\Vert _{+}\left\Vert M^{n}|_{V}\right\Vert _{+}.
\end{align*}
Let $n=2k$ in this inequality; as $\left\Vert M^{n}\right\Vert _{+}%
\leq\left\Vert M^{k}\right\Vert _{+}^{2}$ and
\begin{align*}
\left\Vert M^{n}|_{Y}\right\Vert _{+}  &  \leq\left\Vert M^{k}|_{Y}\right\Vert
_{+}\left\Vert M^{k}|_{Y}\right\Vert _{+}\leq\left\Vert M^{k}|_{Y}\right\Vert
_{+}\left\Vert M^{k}\right\Vert _{+},\\
\left\Vert M^{n}|_{V}\right\Vert _{+}  &  \leq\left\Vert M^{k}|_{V}\right\Vert
_{+}\left\Vert M^{k}|_{V}\right\Vert _{+}\leq\left\Vert M^{k}|_{V}\right\Vert
_{+}\left\Vert M^{k}\right\Vert _{+}%
\end{align*}
then
\begin{align*}
\left\Vert M^{2n}\right\Vert _{+}  &  \leq\left\Vert M^{k}\right\Vert _{+}%
^{2}\left(  \left\Vert M^{k}|_{Y}\right\Vert _{+}^{2}+\left\Vert M^{k}%
|_{Y}\right\Vert \left\Vert M^{k}|_{V}\right\Vert +\left\Vert M^{k}%
|_{V}\right\Vert _{+}^{2}\right) \\
&  \leq\left\Vert M^{k}\right\Vert _{+}^{2}\left(  \left\Vert M^{k}%
|_{Y}\right\Vert _{+}+\left\Vert M^{k}|_{V}\right\Vert _{+}\right)  ^{2}\\
&  \leq4\left\Vert M^{k}\right\Vert _{+}^{2}\max\left\{  \left\Vert M^{k}%
|_{Y}\right\Vert _{+},\left\Vert M^{k}|_{V}\right\Vert _{+}\right\}  ^{2}%
\end{align*}
Taking roots and setting $k\rightarrow\infty$, we obtain
\[
\rho_{t}\left(  M\right)  =\rho_{t}\left(  M\right)  ^{1/2}\max\left\{
\rho_{t}\left(  M|_{Y}\right)  ,\rho_{t}\left(  M|_{V}\right)  \right\}
^{1/2},
\]
i.e., $\rho_{t}\left(  M\right)  \leq\max\left\{  \rho_{t}\left(
M|_{Y}\right)  ,\rho_{t}\left(  M|_{V}\right)  \right\}  $.
\end{proof}

Let $A$ be a normed algebra. Recall that, for every element $a\in A$,
$\mathrm{L}_{a}$ and $\mathrm{R}_{a}$ are defined as operators $x\longmapsto
ax$ and $x\longmapsto xa$ on $A$, respectively. If $M=\left(  a_{n}\right)
_{1}^{\infty}$ is a family in $A$, let $\mathrm{L}_{M}$ and $\mathrm{R}_{M}$
denote the operator families $\left(  \mathrm{L}_{a_{n}}\right)  _{1}^{\infty
}$ and $\left(  \mathrm{R}_{a_{n}}\right)  _{1}^{\infty}$, respectively.

\begin{theorem}
\label{mul}Let $A$ be a normed algebra, and let $M=\left(  a_{n}\right)
_{1}^{\infty}$ be a summable family in $A$. Then $\rho_{t}\left(  M\right)
=\rho_{t}\left(  \mathrm{L}_{M}\right)  =\rho_{t}\left(  \mathrm{R}%
_{M}\right)  $ and $\rho_{t}\left(  \mathrm{L}_{M}\mathrm{R}_{M}\right)
=\rho_{t}\left(  M\right)  ^{2}$.
\end{theorem}

\begin{proof}
It is clear that $\rho_{t}\left(  \mathrm{L}_{M}\right)  \leq\rho_{t}\left(
M\right)  $, and $\rho_{t}\left(  \mathrm{L}_{M}\mathrm{R}_{M}\right)
\leq\rho_{t}\left(  \mathrm{L}_{M}\right)  \rho_{t}\left(  \mathrm{R}%
_{M}\right)  \leq\rho_{t}\left(  M\right)  ^{2}$ by \cite[Proposition
3.4]{TR2}. Further,
\begin{align*}
\left\Vert M^{m+1}\right\Vert _{+}  &  =\Vert\mathrm{L}_{M}^{m}(M)\Vert
_{+}=\sum_{i_{1}=1}^{\infty}\cdots\sum_{i_{n}=1}^{\infty}\sum_{i_{n+1}%
=1}^{\infty}\left\Vert \mathrm{L}_{a_{i_{1}}}\cdots\mathrm{L}_{a_{i_{n}}%
}a_{i_{n+1}}\right\Vert \\
&  \leq\sum_{i_{1}=1}^{\infty}\cdots\sum_{i_{n}=1}^{\infty}\sum_{i_{n+1}%
=1}^{\infty}\left\Vert \mathrm{L}_{a_{i_{1}}}\cdots\mathrm{L}_{a_{i_{n}}%
}\right\Vert \left\Vert a_{i_{n+1}}\right\Vert =\left\Vert \mathrm{L}_{M}%
^{m}\right\Vert _{+}\left\Vert M\right\Vert _{+}%
\end{align*}
for every $m>0$; this implies $\rho_{t}\left(  M\right)  \leq\rho_{t}\left(
\mathrm{L}_{M}\right)  $. Similarly,
\[
\left\Vert M^{2m+1}\right\Vert _{+}\leq\left\Vert \left(  \mathrm{L}%
_{M}\mathrm{R}_{M}\right)  ^{m}\right\Vert _{+}\left\Vert M\right\Vert _{+},
\]
for every $m>0$, implies $\rho_{t}\left(  M\right)  ^{2}=\rho_{t}\left(
M^{2}\right)  \leq\rho_{t}\left(  \mathrm{L}_{M}\mathrm{R}_{M}\right)  $.
\end{proof}

Recall that $M/I$ is the image of $M$ in $A/I$.

\begin{corollary}
\label{invarmul}Let $A$ be a normed algebra, and let $I$ be a closed ideal of
$A$. Then $\rho_{t}\left(  M\right)  =\max\left\{  \rho_{t}\left(  M/I\right)
,\rho_{t}\left(  \mathrm{L}_{M}|_{I}\right)  \right\}  $ for every summable
family $M=\left(  a_{n}\right)  _{1}^{\infty}$ in $A$.
\end{corollary}

\begin{proof}
Indeed, it follows from Theorems \ref{mul} and \ref{invar} that
\begin{align*}
\rho_{t}\left(  M\right)   &  =\rho_{t}\left(  \mathrm{L}_{M}\right)
=\max\left\{  \rho_{t}\left(  \mathrm{L}_{M}|_{A/I}\right)  ,\rho_{t}\left(
\mathrm{L}_{M}|_{I}\right)  \right\} \\
&  =\max\left\{  \rho_{t}\left(  \mathrm{L}_{M/I}\right)  ,\rho_{t}\left(
\mathrm{L}_{M}|_{I}\right)  \right\}  =\max\left\{  \rho_{t}\left(
M/I\right)  ,\rho_{t}\left(  \mathrm{L}_{M}|_{I}\right)  \right\}  .
\end{align*}

\end{proof}

For a summable family $M=\left(  a_{n}\right)  _{1}^{\infty}$ in $A$, let%
\[
M|^{k}=\left(  a_{n}\right)  _{1}^{k},\;\text{ }M|_{k+1}=\left(  a_{n}\right)
_{k+1}^{\infty}\text{ }\;\text{and }\;\rho_{+}\left(  M\right)  =\sum
_{1}^{\infty}\rho\left(  a_{n}\right)  .
\]
As $\rho\left(  a_{n}\right)  \leq\left\Vert a_{n}\right\Vert $ for every
$n>0$, then $\rho_{+}\left(  M\right)  <\infty$ for every summable family $M$.

\begin{remark}
The inequality $\rho_{+}\left(  M\right)  \leq\rho_{t}\left(  M\right)  $ does
not hold even for commutative Banach algebra. Indeed, if $M=\left(
p,1-p,0,\ldots\right)  $ for a non-trivial idempotent $p$ then $\rho
_{t}\left(  M\right)  =1$ while $\rho_{+}\left(  M\right)  =2$. See the
calculation of $\rho_{t}\left(  M\right)  $ for commutative finite families
$M$ via joint spectra in \cite{M07}.
\end{remark}

\begin{lemma}
\label{comm}Let $M=\left(  a_{n}\right)  _{1}^{\infty}$ be a commutative
summable family in $A.$Then $\rho_{t}\left(  M\right)  \leq\rho_{+}\left(
M\right)  $.
\end{lemma}

\begin{proof}
As $M|^{k}$ and $M|_{k+1}$ commute for every $k>0$, and $M=M|^{k}\sqcup
M|_{k+1}$, we obtain that
\[
\rho_{t}\left(  M\right)  \leq\rho_{t}\left(  M|^{k}\right)  +\rho_{t}\left(
M|_{k+1}\right)  \leq\rho_{+}\left(  M|^{k}\right)  +\rho_{t}\left(
M|_{k+1}\right)
\]
by \cite[Proposition 3.4]{TR2}. For every $\varepsilon>0$ there is $n>0$ such
that $\rho_{t}\left(  M|_{n+1}\right)  \leq\left\Vert M|_{n+1}\right\Vert
_{+}<\varepsilon$, whence
\[
\rho_{t}\left(  M\right)  \leq\sup_{k}\rho_{t}\left(  M|^{k}\right)  =\sup
_{k}\rho_{+}\left(  M|^{k}\right)  =\rho_{+}\left(  M\right)  .
\]

\end{proof}

\begin{corollary}
\label{cen}Let $A$ be a normed algebra, $I$ a closed ideal of $A$, and let
$M=\left(  a_{n}\right)  _{1}^{\infty}$ be a summable family in $A$. Then

\begin{enumerate}
\item If $A/I$ is commutative then $\rho_{t}\left(  M\right)  \leq\max\left\{
\rho_{+}\left(  M\right)  ,\rho_{t}\left(  \mathrm{L}_{M}|_{I}\right)
\right\}  $;

\item If $I$ is central then, for every $k>0$,
\[
\rho_{t}\left(  M\right)  \leq\max\left\{  \rho_{t}\left(  M/I\right)
,\rho_{t}\left(  M|^{k}\right)  +\rho_{t}\left(  M|_{k+1}\right)  \right\}  .
\]

\end{enumerate}
\end{corollary}

\begin{proof}
$\left(  1\right)  $ It follows from Lemma \ref{comm} and Corollary
\ref{invarmul} that
\[
\rho_{t}\left(  M\right)  \leq\max\left\{  \rho_{+}\left(  M/I\right)
,\rho_{t}\left(  \mathrm{L}_{M}|_{I}\right)  \right\}  \leq\max\left\{
\rho_{+}\left(  M\right)  ,\rho_{t}\left(  \mathrm{L}_{M}|_{I}\right)
\right\}
\]
since $\rho\left(  a_{n}/I\right)  \leq\rho\left(  a_{n}\right)  $ for every
$n$.

$\left(  2\right)  $ It is easy to check that $\mathrm{L}_{M}|_{I}$ is
commutative. Then
\begin{align*}
\rho_{t}\left(  \mathrm{L}_{M}|_{I}\right)   &  \leq\rho_{t}\left(
\mathrm{L}_{M|^{k}}|_{I}\right)  +\rho_{t}\left(  \mathrm{L}_{M|_{k+1}}%
|_{I}\right)  \leq\rho_{t}\left(  M|^{k}/I\right)  +\rho_{t}\left(
M|_{k+1}/I\right) \\
&  \leq\rho_{t}\left(  M|^{k}\right)  +\rho_{t}\left(  M|_{k+1}\right)
\end{align*}
by Lemma \ref{comm}, and
\begin{align*}
\rho_{t}\left(  M\right)   &  =\max\left\{  \rho_{t}\left(  M/I\right)
,\rho_{t}\left(  \mathrm{L}_{M}|_{I}\right)  \right\} \\
&  \leq\max\left\{  \rho_{t}\left(  M/I\right)  ,\rho_{t}\left(
M|^{k}\right)  +\rho_{t}\left(  M|_{k+1}\right)  \right\}  .
\end{align*}

\end{proof}

\begin{corollary}
\label{sum}Let $A$ be a normed algebra. If $A/\mathcal{R}_{t}\left(  A\right)
$ is commutative then $\rho_{t}\left(  M\right)  \leq\rho_{t}\left(
M|^{k}\right)  +\rho_{t}\left(  M|_{k+1}\right)  $ for every summable family
$M=\left(  a_{n}\right)  _{1}^{\infty}$ in $A$ and every $k>0$.
\end{corollary}

\begin{proof}
By $\left(  \ref{tr}\right)  $ and Lemma \ref{comm},
\[
\rho_{t}\left(  M\right)  =\rho_{t}\left(  M/\mathcal{R}_{t}\left(  A\right)
\right)  \leq\rho_{t}\left(  M|^{k}\right)  +\rho_{t}\left(  M|_{k+1}\right)
.
\]

\end{proof}

\begin{lemma}
\label{sum1}Let $A$ be an $\mathcal{R}_{t}$-semisimple algebra, and let $I$ be
a closed commutative ideal of $A$. If $A/I$ is commutative modulo
$\mathcal{R}_{t}\left(  A/I\right)  $ then $\rho_{t}\left(  \left\{
a\right\}  \sqcup M\right)  \leq\rho\left(  a\right)  +\rho_{t}\left(
M\right)  $ for every $a\in A$ and every summable family $M=\left(
a_{n}\right)  _{1}^{\infty}$ in $A$.
\end{lemma}

\begin{proof}
As $A$ is a semiprime algebra, $I$ is a central ideal of $A$ by Lemma
\ref{commut}. By Corollary \ref{cen},%
\[
\rho_{t}\left(  \left\{  a\right\}  \sqcup M\right)  \leq\max\left\{  \rho
_{t}\left(  \left(  \left\{  a\right\}  \sqcup M\right)  /I\right)
,\rho\left(  a\right)  +\rho_{t}\left(  M\right)  \right\}  .
\]
By Corollary \ref{sum},
\[
\rho_{t}\left(  \left(  \left\{  a\right\}  \sqcup M\right)  /I\right)
\leq\rho\left(  a/I\right)  +\rho_{t}\left(  M/I\right)  \leq\rho\left(
a\right)  +\rho_{t}\left(  M\right)  .
\]
Therefore $\rho_{t}\left(  \left\{  a\right\}  \sqcup M\right)  \leq
\rho\left(  a\right)  +\rho_{t}\left(  M\right)  $.
\end{proof}

In the following theorem we will use the inequality
\begin{equation}
\rho_{t}\left(  \left\{  a+b\right\}  \sqcup M\right)  \leq\rho_{t}\left(
\left\{  a\right\}  \sqcup\left\{  b\right\}  \sqcup M\right)  \label{abm}%
\end{equation}
for every $a,b\in A$ and every summable family $M$ in $A$; this is a special
case of \cite[Proposition 3.3]{TR2}.

\begin{theorem}
$\mathcal{R}_{t}^{a}$ is a uniform regular topological radical on
$\mathfrak{U}_{\mathrm{n}}$.
\end{theorem}

\begin{proof}
Let $A$ be an $\mathcal{R}_{t}$-semisimple algebra, and let $I$ be a closed
commutative ideal of $A$ such that $A/I$ is commutative modulo $\mathcal{R}%
_{t}\left(  A/I\right)  $. By Proposition \ref{exc2}, to prove that
$\mathcal{R}_{t}^{a}$ is a topological radical, it suffices to show that $A$
is commutative.

Assume that $A$ is a Banach algebra. Let $M=\left(  a_{n}\right)  _{1}%
^{\infty}$ be a summable family in $A$, and $a,b\in A$ be arbitrary. Put
$c_{\lambda}=\exp\left(  \lambda\operatorname*{ad}\left(  b\right)  \right)
\left(  a\right)  $ for every $\lambda\in\mathbb{C}$, and put $f\left(
\lambda\right)  =\left(  c_{\lambda}-a\right)  /\lambda$; then $\lambda
\longmapsto$ $f\left(  \lambda\right)  $ is an analytic function in $A$ and
$f\left(  0\right)  =\left[  b,a\right]  $. As $\rho\left(  c_{\lambda
}\right)  =\rho\left(  a\right)  $ then, by $\left(  \ref{abm}\right)  $ and
Lemma \ref{sum1},
\[
\rho_{t}\left(  \left\{  c_{\lambda}-a\right\}  \sqcup M\right)  \leq\rho
_{t}\left(  \left\{  c_{\lambda}\right\}  \sqcup\left\{  a\right\}  \sqcup
M\right)  \leq2\rho\left(  a\right)  +\rho_{t}\left(  M\right)
\]
for every $\lambda\in\mathbb{C}$. Replace $a$ by $a/\lambda$ for $\lambda
\neq0$; we obtain that
\begin{equation}
\rho_{t}\left(  \left\{  f\left(  \lambda\right)  \right\}  \sqcup M\right)
\leq2\rho\left(  a\right)  /\lambda+\rho_{t}\left(  M\right)  \label{rtb}%
\end{equation}
The function $\lambda\longmapsto\rho_{t}\left(  \left\{  f\left(
\lambda\right)  \right\}  \sqcup M\right)  $ is subharmonic by \cite[Theorem
3.16]{TR2} and bounded on $\mathbb{C}$ by $\left(  \ref{rtb}\right)  $;
therefore it is constant and
\[
\rho_{t}\left(  \left\{  \left[  b,a\right]  \right\}  \sqcup M\right)
=\rho_{t}\left(  \left\{  f\left(  0\right)  \right\}  \sqcup M\right)
=\lim_{\left\vert \lambda\right\vert \rightarrow\infty}\rho_{t}\left(
\left\{  f\left(  \lambda\right)  \right\}  \sqcup M\right)  \leq\rho
_{t}\left(  M\right)  .
\]
As $M$ is an arbitrary summable family in $A$ then $\left[  b,a\right]
\in\mathcal{R}_{t}\left(  A\right)  =0$, therefore $A$ is commutative.

As $\mathcal{R}_{t}$ is hereditary, it follows from Theorem \ref{pab} and
Proposition \ref{exc2} that $\mathcal{R}_{t}^{a}$ is a hereditary topological
radical on $\mathfrak{U}_{\mathrm{b}}$. As $\mathcal{R}_{t}^{a}$ is regular by
Theorem \ref{regt}, $\mathcal{R}_{t}^{a}$ is a hereditary topological radical
on $\mathfrak{U}_{\mathrm{n}}$.

Let $A$ be an $\mathcal{R}_{t}^{a}$-radical algebra, and let $B$ be a
subalgebra of $A$. If $a,b\in B$ then $\rho_{t}\left(  \left\{  \left[
b,a\right]  \right\}  \sqcup M\right)  =\rho_{t}\left(  M\right)  $ for every
summable family in $B$. Therefore $B$ is $\mathcal{R}_{t}^{a}$-radical and
$\mathcal{R}_{t}^{a}$ is uniform.
\end{proof}

\subsection{Spectral applications}

Let $\mathfrak{U}$ be a class of algebras. A property of algebras from
$\mathfrak{U}$ is called \textit{radical} (respectively, \textit{semisimple})
in $\mathfrak{U}$ if algebras with this property form the radical
(respectively, semisimple) class for some radical.

We start with $\operatorname{Rad}^{a}$-radical algebras. It is well known that
Banach algebras commutative modulo the Jacobson radical share with the
commutative ones the advantages of easy calculation of spectra and the
continuity of the functions $a\rightarrow\sigma(a)$ and $a\rightarrow\rho(a)$.
Moreover, it was shown in \cite{PZ77, Z77} that a Banach algebra is
commutative modulo radical if and only if one (or all) of the following
conditions holds:

\begin{enumerate}
\item[(1$_{a}$)] The function $a\rightarrow\rho(a)$ is submultiplicative;

\item[(2$_{a}$)] The function $a\rightarrow\rho(a)$ is subadditive;

\item[(3$_{a}$)] The function $a\rightarrow\rho(a)$ is uniformly continuous.
\end{enumerate}

Using the fact that $\operatorname{Rad}^{a}$ is a radical (Corollary
\ref{centRad}), we obtain the following result:

\begin{corollary}
\label{extCont}Let $A$ be a Banach algebra, let $I$ be a closed ideal of $A$
and let $\mathcal{F}$ be a family of closed ideals of $A$ with dense sum in
$A$. If either $I$ and $A/I$, or all ideals from $\mathcal{F}$ have one of the
properties $\left(  1_{a}\right)  $, $\left(  2_{a}\right)  $, $\left(
3_{a}\right)  $ then $A$ has the same property.
\end{corollary}

In other words, the uniform continuity of $a\rightarrow\rho(a)$ is a radical
property of Banach algebras.

\begin{problem}
Will the statement of Corollary $\ref{extCont}$ stay true if one replaces the
uniform continuity of $a\rightarrow\rho(a)$ by the continuity of this map?
\end{problem}

Some results related to this problem can be found in Section \ref{continuity}.

In the theory of joint spectra and spectral radii the commutativity modulo the
radical $\mathcal{R}_{\mathrm{cq}}$ or $\mathcal{R}_{t}$ plays a role which is
at least partially similar to the role of commutativity modulo
$\operatorname{Rad}$ in the \textquotedblleft individual\textquotedblright%
\ spectral theory.

\begin{theorem}
\label{sp4}Let $A$ be a normed algebra. Then

\begin{enumerate}
\item If $M = \{a_{\alpha}: \alpha\in\Lambda\}$ is a family in $A$ which is
precompact and commutative modulo $\mathcal{R}_{\mathrm{cq}}\left(  A\right)
$ then $\rho\left(  M\right)  =\sup\left\{  \left\vert \lambda\right\vert
:\lambda\in\sigma_{\widehat{A}}^{l}\left(  M\right)  \right\}  ;$ (Here
$\left\vert \lambda\right\vert =\sup_{\alpha}\left\vert \lambda_{\alpha
}\right\vert $.)

\item If $M=\left(  a_{n}\right)  _{1}^{\infty}$ is a summable family which is
commutative modulo $\mathcal{R}_{t}\left(  A\right)  $ then $\rho_{t}\left(
M\right)  =\sup\left\{  \left\vert \lambda\right\vert _{+}:\lambda\in
\sigma_{\widehat{A}}^{l}\left(  M\right)  \right\}  =\sup\left\{  \left\vert
\lambda\right\vert _{+}:\lambda\in\sigma_{\widehat{A}}^{r}\left(  M\right)
\right\}  .$ (Here $\left\vert \lambda\right\vert _{+}=\sum_{1}^{\infty
}\left\vert \lambda_{n}\right\vert $.)
\end{enumerate}
\end{theorem}

\begin{proof}
As $\mathcal{R}_{\mathrm{cq}}\left(  A\right)  =A\cap\mathcal{R}_{\mathrm{cq}%
}\left(  \widehat{A}\right)  $, $\mathcal{R}_{t}\left(  A\right)
=A\cap\mathcal{R}_{t}\left(  \widehat{A}\right)  $ and the values of the joint
spectral and tensor radii do not change under taking the completion of $A$,
one may assume that $A$ is a Banach algebra.

$\left(  1\right)  $ As $\rho\left(  M\right)  =\rho\left(  M/\mathcal{R}%
_{\mathrm{cq}}\left(  A\right)  \right)  $ by \cite[Theorem 3.29]{TR2}, one
may assume that $M$ is commutative. By \cite[Theorem 6.3]{TR3}, the joint
spectral radius is continuous at $M$ and
\begin{equation}
\rho\left(  M\right)  =\sup\left\{  \rho\left(  N\right)  :N\text{ is a finite
subfamily of }M\right\}  . \label{sf1}%
\end{equation}
The first equality below holds by \cite[Theorem 35.5]{M07} while the other
relations are obvious:
\begin{align}
\rho\left(  N\right)   &  =\sup\left\{  \left\vert \lambda\right\vert
:\lambda\in\sigma_{A}^{l}\left(  N\right)  \right\} \label{sf2}\\
&  \leq\sup\left\{  \left\vert \lambda\right\vert :\lambda\in\sigma_{A}%
^{l}\left(  M\right)  \right\}  =\sup_{a_{\alpha}\in M}\rho\left(  a_{\alpha
}\right)  \leq\rho\left(  M\right)  .\nonumber
\end{align}
The result follows from $\left(  \ref{sf1}\right)  $ and $\left(
\ref{sf2}\right)  $.

$\left(  2\right)  $ As $\rho\left(  M\right)  =\rho\left(  M/\mathcal{R}%
_{\mathrm{t}}\left(  A\right)  \right)  $ by \cite[Theorem 4.18]{TR1}, one may
assume that $M$ is commutative. As $M$ is summable then for every
$\varepsilon>0$ there is $n>0$ such that $\sum_{n+1}^{\infty}\left\Vert
a_{k}\right\Vert <\varepsilon$. By Corollary \ref{sum},
\[
\rho\left(  M|^{n}\right)  \leq\rho\left(  M\right)  \leq\rho\left(
M|^{n}\right)  +\rho\left(  M|_{n+1}\right)  \leq\rho\left(  M|^{n}\right)
+\sum_{n+1}^{\infty}\left\Vert a_{k}\right\Vert <\rho\left(  M|^{n}\right)
+\varepsilon
\]
whence $\rho\left(  M\right)  =\lim_{n\rightarrow\infty}\rho\left(
M|^{n}\right)  $. By \cite[Theorem 35.6]{M07},
\[
\rho\left(  M|^{n}\right)  =\sup\left\{  \left\vert \lambda\right\vert
_{+}:\lambda\in\sigma_{\widehat{A}}^{l}\left(  M|^{n}\right)  \right\}  ,
\]
whence the result follows.
\end{proof}

It should be noted that the results of M\"{u}ller, used in the above proof,
were formulated in \cite{M07} for finite families in Banach algebras and for
the Harte spectrum, but they hold also for the left and right spectra due to
\cite[Proposition 35.2]{M07}.

It follows from Theorem \ref{sp4}$\left(  1\right)  $ that $\rho\left(
M\right)  =\sup_{a_{\alpha}\in M}\rho\left(  a_{\alpha}\right)  $ under the
posed conditions (this was proved also in \cite{TR3}). As a consequence, we
obtain that each $\mathcal{R}_{\mathrm{cq}}^{a}$-radical algebra is a
Berger-Wang algebra (i.e. $\rho(M)=r(M)$ for all precompact sets $M$ in $A$).
By \cite[Corollary 5.15]{TR3}, $\mathcal{R}_{\mathrm{hc}}\vee\mathcal{R}%
_{\mathrm{cq}}^{a}\leq\mathcal{R}_{\mathrm{bw}}$ and in particular for every
normed algebra $A$ the following \textit{algebra version of the} \textit{joint
spectral radius formula}
\[
\rho\left(  M\right)  =\max\left\{  \rho\left(  M/\left(  \mathcal{R}%
_{\mathrm{hc}}\vee\mathcal{R}_{\mathrm{cq}}^{a}\right)  \left(  A\right)
\right)  ,r\left(  M\right)  \right\}
\]
holds for every precompact set $M$ in $A$. We do not know however whether
every Berger-Wang algebra is $\mathcal{R}_{\mathrm{bw}}$-radical.

\begin{proposition}
Let $A$ be a normed algebra. If $A/\mathcal{R}_{\mathrm{bw}}\left(  A\right)
$ is a Berger-Wang algebra then $A$ is a Berger-Wang algebra.
\end{proposition}

\begin{proof}
Indeed,
\[
\rho\left(  M\right)  =\max\left\{  \rho\left(  M/\mathcal{R}_{\mathrm{bw}%
}\left(  A\right)  \right)  ,r\left(  M\right)  \right\}  =\max\left\{
r\left(  M/\mathcal{R}_{\mathrm{bw}}\left(  A\right)  \right)  ,r\left(
M\right)  \right\}  =r\left(  M\right)
\]
for every precompact set $M$ in $A$. Therefore $A$ is a Berger-Wang algebra.
\end{proof}

The following theorem supplies us with a class of examples of Banach algebras
which are $\mathcal{R}_{\mathrm{cq}}^{a}$-radical (and therefore
$\mathcal{R}_{t}^{a}$-radical); these algebras hold an important place in the
theory of joint spectra of Lie representations (see for instance \cite{F93,
BS01, D06}).

Recall that algebras can be considered as Lie algebras with respect to the Lie
bracket $\left[  a,b\right]  =ab-ba$. A Lie subalgebra of an algebra $A$ is a
subspace $L\subset A$ with the property that $[a,b]\in L$ for all $a,b\in L$.

Operators $\mathrm{ad}_{L}\left(  a\right)  :x\longmapsto\left[  a,x\right]  $
on a Lie algebra $L$, for $a\in L$, are called \textit{adjoint operators} of
$L$.

For a Lie algebra $L$, one defines the upper (lower) central series $L^{[n]}$
(respectively $L_{[n]}$) by setting $L_{[1]}=L^{[1]}=L$, $L_{[n+1]}%
=[L,L_{[n]}]$ and $L^{[n+1]}=[L^{[n]},L^{[n]}]$; $L$ is called
\textit{nilpotent} (\textit{solvable}) if $L_{[n]}=0$ (respectively
$L^{[n]}=0$), for some $n$.

\begin{theorem}
\label{sp6}Let $A$ be a normed algebra and let $L$ be a Lie subalgebra of $A$.
If one of the following conditions holds:

\begin{enumerate}
\item $L$ is a nilpotent Lie algebra, and the inverse-closed subalgebra
generated by $L$ is dense in $A^{1}$;

\item $L$ is a finite-dimensional solvable Lie algebra, and the subalgebra
generated by $L$ is dense in $A$;
\end{enumerate}

\noindent then $A$ is commutative modulo $\mathcal{R}_{\mathrm{cq}}\left(
A\right)  $, i.e., $A=\mathcal{R}_{\mathrm{cq}}^{a}\left(  A\right)  $.
\end{theorem}

\begin{proof}
Let $A^{\prime}=A/\mathcal{R}_{\mathrm{cq}}\left(  A\right)  $, and let
$L^{\prime}$ be the image of $L$ in $A^{\prime}$.

$\left(  1\right)  $ If $L^{\prime}$ is not commutative then, by the
Kleinecke-Shirokov theorem, there is a non-zero quasinilpotent element
$x\in\left[  L^{\prime},L^{\prime}\right]  $ in the center of $L^{\prime}$.
Let $B_{0}$ be the subalgebra of $A^{\prime}$ generated by $L^{\prime}$. Then
clearly $x$ is in the center of $B_{0}$. Let $B_{1}$ be the subalgebra
generated by $B_{0}$ and all inverses of elements from $B_{0}$. Then $x$ is
again in the center of $B_{1}$. If we have already built $B_{n}$ and proved
that $x$ is in the center of $B_{n}$, then $B_{n+1}$ is defined as a
subalgebra generated by $B_{n}$ and all inverses of elements from $B_{n}$, and
it is clear that $x$ is in the center of $B_{n+1}$.

Let $C=\cup_{n=0}^{\infty}B_{n}$. Then $C$ is an inverse-closed subalgebra of
$A^{\prime}$; indeed, if $a\in B_{n}$ has an inverse $a^{-1}$ in $A^{\prime}$
then $a^{-1}\in B_{n+1}$. It is clear that $C$ is dense in $A^{\prime}$. Thus
$x$ is in the center of $A^{\prime}$. As $x$ commutes with elements of each
$M\in\mathfrak{k}\left(  A^{\prime}\right)  $ then
\[
\rho\left(  \left\{  x\right\}  \cup M\right)  \leq\max\left\{  \rho\left(
x\right)  ,\rho\left(  M\right)  \right\}  =\rho\left(  M\right)
\]
since $x$ is a quasinilpotent element. Therefore $x\in\mathcal{R}%
_{\mathrm{cq}}\left(  A^{\prime}\right)  =0$ (see Section \ref{2.3.4}), a contradiction.

So $L^{\prime}$ is commutative, whence $A$ is commutative modulo
$\mathcal{R}_{\mathrm{cq}}\left(  A\right)  $.

$\left(  2\right)  $ Let $I$ be the set of nilpotents in $L^{\prime}$. By
\cite[Proposition 24.1]{BS01}, $I$ is a Lie ideal of $L^{\prime}$. As $I$ is
finite-dimensional, the nilpotency indexes of elements of $I$ are uniformly
bounded, so $I$ is a nilpotent Lie algebra (by the algebraic Engel condition)
and, moreover, generates nilpotent subalgebra in $A$, whence $I^{n}=0$ for
some $n>0$. As $L^{\prime}I\subset IL^{\prime}+I$, it follows that $I$
generates a nilpotent ideal in the subalgebra generated by $L^{\prime}$ and
generates therefore a closed nilpotent ideal in $A^{\prime}$. But every
nilpotent ideal lies in$\mathcal{\ R}_{\mathrm{cq}}\left(  A^{\prime}\right)
=\mathcal{R}_{\mathrm{cq}}\left(  A/\mathcal{R}_{\mathrm{cq}}\left(  A\right)
\right)  =0$, whence $I=0$. This implies that $L^{\prime}$ is a nilpotent Lie
algebra; indeed, every eigenvector $x$ of an adjoint operator $\mathrm{ad}%
_{L^{\prime}}\left(  a\right)  $ corresponding to a non-zero eigenvalue is
nilpotent which is impossible. By $\left(  1\right)  $, $A^{\prime
}=\mathcal{R}_{\mathrm{cq}}^{a}\left(  A^{\prime}\right)  $, whence
$A^{\prime}$ is commutative and then $A=\mathcal{R}_{\mathrm{cq}}^{a}\left(
A\right)  $.
\end{proof}

Now we will discuss conditions under which the algebras generated (in the
above sense) by nilpotent Lie subalgebras belong to a special subclass of the
class of $\operatorname{Rad}^{a}$-radical algebras; this subclass occupies an
important place in the theory of linear operator equations and, more
generally, in the study of multiplication operators on Banach algebras (see
\cite{TR2}).

A normed algebra $A$ is called an\textit{ Engel algebra} if all operators
$\mathrm{ad}_{A}\left(  a\right)  $, $a\in A$, are quasinilpotent. It was
proved in \cite{ST05} (and can be deduced from earlier results of
\cite{AM2000}) that each Engel Banach algebra $A$ is commutative modulo
$\operatorname{Rad}(A)$. The converse is not true even for finite-dimensional
algebras, for instance for the algebra of all upper-triangular matrices.
Clearly the class of all Engel algebras is stable under taking closed ideals
and quotients; it evidently contains all commutative and all radical Banach
algebras; as a consequence, this class is not stable under extensions.

An operator $T$ on an algebra $A$ is call \textit{elementary }if $T=\sum
_{i=1}^{n}\mathrm{L}_{a_{i}}\mathrm{R}_{b_{i}}$ for some $a_{1},b_{1}%
,\ldots,a_{n},b_{n}\in A^{1}$. For instance, the identity operator on $A$ is
also elementary. Let $\mathcal{E\!\ell}\left(  A\right)  $ be the algebra of
all elementary operators on $A$.

\begin{lemma}
\label{sp12}Let $A$ be a Banach algebra, $L$ be a nilpotent Lie subalgebra of
$A$. Then the closed, inverse-closed subalgebra $B$ of $\mathcal{B}\left(
A^{1}\right)  $ generated by $\mathrm{L}_{L}+\mathrm{R}_{L}$ is commutative
modulo $\mathcal{R}_{\mathrm{cq}}\left(  B\right)  $.
\end{lemma}

\begin{proof}
Since the closure of an inverse-closed subalgebra of a Banach algebra is
inverse-closed and $\mathrm{L}_{L}+\mathrm{R}_{L}$ is a nilpotent Lie algebra,
we obtain that $B$ is commutative modulo $\mathcal{R}_{\mathrm{cq}}\left(
B\right)  $ by Theorem \ref{sp6}.
\end{proof}

\begin{theorem}
\label{sp20}Let $A$ be a Banach algebra, and let $L$ be a nilpotent Lie
subalgebra of $A$ such that the inverse-closed subalgebra generated by $L$ is
dense in $A^{1}$. Then the following statements are equivalent:

\begin{enumerate}
\item $A$ is an Engel algebra;

\item For every elementary operator $\sum\mathrm{L}_{a_{i}}\mathrm{R}_{b_{i}}%
$,
\[
\sigma_{\mathcal{B}\left(  A^{1}\right)  }\left(  \sum_{1}^{n}\mathrm{L}%
_{a_{i}}\mathrm{R}_{b_{i}}\right)  =\sigma_{A^{1}}\left(  \sum_{1}^{n}%
a_{i}b_{i}\right)  ;
\]

\item The closed subalgebra generated by $\mathrm{ad}_{A}\left(  A\right)  $
is compactly quasinilpotent;

\item $\mathrm{ad}_{A}\left(  a\right)  $ is quasinilpotent for every $a\in
L\backslash\left[  L,L\right]  $.
\end{enumerate}
\end{theorem}

\begin{proof}
$\left(  2\right)  \Rightarrow\left(  1\right)  \Rightarrow\left(  4\right)  $
and $\left(  3\right)  \Rightarrow\left(  1\right)  $ are obvious; $\left(
1\right)  \Rightarrow\left(  3\right)  $ follows from Theorem \ref{sp6} and
Lemma \ref{sp12}.

$\left(  1\right)  \Rightarrow\left(  2\right)  $ Let $B$ be the closed,
inverse-closed subalgebra of $\mathcal{B}\left(  A^{1}\right)  $ generated by
$\mathrm{L}_{L}+\mathrm{R}_{L}$, and let $T=\sum_{1}^{n}\mathrm{L}_{a_{i}%
}\left(  \mathrm{R}_{b_{i}}-\mathrm{L}_{b_{i}}\right)  $ and $c=\sum_{1}%
^{n}a_{i}b_{i}$. As $B$ is inverse-closed,
\[
\sigma_{\mathcal{B}\left(  A^{1}\right)  }\left(  \sum_{1}^{n}\mathrm{L}%
_{a_{i}}\mathrm{R}_{b_{i}}\right)  =\sigma_{B}\left(  \sum_{1}^{n}%
\mathrm{L}_{a_{i}}\mathrm{R}_{b_{i}}\right)  =\sigma_{B}\left(  T+\mathrm{L}%
_{c}\right)  .
\]
By Lemma \ref{sp12}, $B$ is commutative modulo $\mathcal{R}_{\mathrm{cq}%
}\left(  B\right)  $. As $\mathcal{R}_{\mathrm{cq}}\left(  B\right)
\subset\operatorname*{Rad}\left(  B\right)  $ then $B$ is commutative modulo
$\operatorname*{Rad}\left(  B\right)  $, whence the set of all quasinilpotent
elements of $B$ coincides with $\operatorname*{Rad}\left(  B\right)  $. Then,
as $A$ is Engel, it follows that $T\in\operatorname*{Rad}\left(  B\right)  $
and $\sigma_{B}\left(  T+\mathrm{L}_{c}\right)  =\sigma_{B}\left(
\mathrm{L}_{c}\right)  $. But $\sigma_{B}\left(  \mathrm{L}_{c}\right)
=\sigma_{\mathcal{B}\left(  A^{1}\right)  }\left(  \mathrm{L}_{c}\right)
=\sigma_{A^{1}}\left(  c\right)  $ by \cite[Proposition 3.19]{BD73}, and the
result follows.

$\left(  4\right)  \Rightarrow\left(  1\right)  $ As $\left[  L,L\right]  $
consists of quasinilpotents, it follows that $\mathrm{ad}_{A}\left(  L\right)
$ consists of quasinilpotents. Let $E=\left\{  a\in A^{1}:\rho\left(
\mathrm{ad}_{A^{1}}\left(  a\right)  \right)  =0\right\}  $. It is easy to see
that $E$ is an algebra by using Lemma \ref{sp12}. If $x\in E$ is invertible in
$A$ then
\[
\rho\left(  \mathrm{ad}_{A^{1}}\left(  x^{-1}\right)  \right)  =\rho\left(
\mathrm{L}_{x^{-1}}\mathrm{ad}_{A^{1}}\left(  x\right)  \mathrm{R}_{x^{-1}%
}\right)  \leq\rho\left(  \mathrm{L}_{x^{-1}}\right)  \rho\left(
\mathrm{ad}_{A^{1}}\left(  x\right)  \right)  \rho\left(  \mathrm{R}_{x^{-1}%
}\right)  =0
\]
by Lemma \ref{sp12}, whence $E$ is an inverse-closed subalgebra of $A$. As
$L\subset E$ then $E$ is dense in $A$. Let $B$ be defined as in Lemma
\ref{sp12}; as $B$ is commutative modulo $\operatorname{Rad}\left(  B\right)
$ then the spectral radius is continuous on $B$. Then $E$ is closed and $A=E$.
\end{proof}

\begin{corollary}
\label{sp21}Let $A$ be a Banach algebra, and let $L$ be a nilpotent Lie
subalgebra of $A$ such that the inverse-closed subalgebra generated by $L$ is
dense in $A^{1}$. If $\sigma_{A^{1}}\left(  a\right)  $ is at most countable
for every $a\in L\backslash\left[  L,L\right]  $, then $A$ is an Engel algebra.
\end{corollary}

\begin{proof}
Let $A_{0}$ be the subalgebra of $A^{1}$ generated by $L$ and the identity
element $1$ of $A^{1}$. It follows that $\sigma_{\mathcal{B}\left(
A^{1}\right)  }\left(  \mathrm{ad}_{A^{1}}\left(  a\right)  \right)  $ is at
most countable for every $a\in L\backslash\left[  L,L\right]  $. As
$\mathrm{ad}_{A^{1}}\left(  a\right)  $ is nilpotent on $L$, it is locally
nilpotent on $A_{0}$ and is quasinilpotent on the closure $B$ of $A_{0}$ by
\cite[Corollary 3.7]{ST05}. To apply Theorem \ref{sp20}, it is sufficient to
show $B$ is an inverse-closed subalgebra of $A$.

Let $x\in B$ and $x^{-1}\in A$. Then there is a sequence $x_{n}$ of elements
of $A_{0}$ such that $x_{n}$ $\rightarrow x$ as $n\rightarrow\infty$. As the
map $a\longmapsto a^{-1}$ is continuous on the set of all invertible elements
of $A$ and this set is open, one can assume that for every $n$ there is
$x_{n}^{-1}\in A$ and the sequence $x_{n}^{-1}\rightarrow x^{-1}$ as
$n\rightarrow\infty$.

Note that, for every $n$, there is a polynomial $p_{n}$ such that $x_{n}%
=p_{n}\left(  a_{1},\ldots,a_{k}\right)  $ for some $a_{1},\ldots,a_{k}\in
L\backslash\left[  L,L\right]  $. By the spectral mapping theorem \cite{T85},
we obtain that
\[
\sigma_{A^{1}}\left(  p_{n}\left(  a_{1},\ldots,a_{k}\right)  \right)
=p_{n}\left(  \sigma_{A^{1}}\left(  a_{1},\ldots,a_{k}\right)  \right)
\subset p_{n}\left(  \sigma_{A^{1}}\left(  a_{1}\right)  \times\cdots
\times\sigma_{A^{1}}\left(  a_{k}\right)  \right)  .
\]
As $\sigma_{A^{1}}\left(  a_{1}\right)  \times\cdots\times\sigma_{A^{1}%
}\left(  a_{k}\right)  $ is at most countable, it follows that the spectrum
$\sigma_{A^{1}}\left(  x_{n}\right)  $ is also at most countable. By
\cite[Theorem 5.11]{BD73}, $x_{n}^{-1}$ lies in the closed subalgebra of
$A^{1}$ generated by $x_{n}$ and $1$. As $x_{n},1\in B$ then $x_{n}^{-1}\in B$
for every $n$. Therefore $x^{-1}\in B$.

We showed that $B$ is inverse-closed. Therefore $B=A^{1}$. As $\mathrm{ad}%
_{A}\left(  a\right)  $ is quasinilpotent for every $a\in L\backslash\left[
L,L\right]  $, it follows that $A$ is an Engel algebra by Theorem \ref{sp20}.
\end{proof}

Some related results are obtained in \cite{C12}.

\begin{problem}
\label{EngRad} Does every Banach algebra have the largest Engel ideal?
\end{problem}

\section{ Socle procedure and radicals}

\subsection{Socle}

Let $A\in\mathfrak{U}_{\mathrm{a}}$ be an algebra. The (left) \textit{socle}
$\operatorname{soc}(A)$ of $A$ is the sum of all minimal left ideals of $A$.
If $A$ has no minimal left ideals then $\operatorname{soc}(A)=0$;
$\operatorname{soc}(A)$ is an ideal of $A$ \cite[Lemma 30.9]{BD73}. If $I$ is
a minimal left ideal of $A$ and $I^{2}\neq0$ then there is a projection $p\in
I$ such that $I=Ap$, and every such projection $p$ is minimal \cite[Lemma
30.2]{BD73}. Recall that a non-zero projection $p$ of $A$ is \textit{minimal}
if $pAp$ is a division algebra.

Let $\mathrm{Min}(A)$ be the set of minimal projections of $A$. If $A$ is
semiprime then $\operatorname{soc}(A)$ equals the sum of all minimal right
ideals of $A$ (i.e. the \textit{right socle}) \cite[Proposition 30.10]{BD73}
and $L$ is a left (right) minimal ideal of $A$ if and only if $L=Ap$
(respectively, $L=pA$) for some $p\in\mathrm{Min}(A)$.

\begin{remark}
If $B\in\mathfrak{U}_{\mathrm{n}}$ is a division algebra then $B$ is
isomorphic to $\mathbb{C}$ by the Gelfand-Mazur theorem. If $A$ is a semiprime
Banach algebra then $\operatorname{soc}(A)$ is closed if and only if $A$ is
finite-dimensional \cite{T76}, $\overline{\operatorname{soc}(A)}%
\cap\operatorname{Rad}\left(  A\right)  =0$ \cite[Lemma 4]{GR91} and
$\operatorname{soc}(A)$ is the largest ideal of algebraic elements of $A$
\cite[Theorem 5]{GR91}.
\end{remark}

\begin{proposition}
\begin{enumerate}
\item The map $\operatorname{soc}$ is a preradical on $\mathfrak{U}%
_{\mathrm{a}}$.

\item If $A\in\mathfrak{U}_{\mathrm{a}}$ is semiprime then $\operatorname{soc}%
\left(  J\right)  =J\cap\operatorname{soc}\left(  A\right)  $ for each ideal
$J$ of $A$; in particular, $\operatorname{rad}\left(  A\right)  \cap
\operatorname{soc}\left(  A\right)  =0$.
\end{enumerate}
\end{proposition}

\begin{proof}
$\left(  1\right)  $ is straightforward.

$\left(  2\right)  $ Let $E=\{p\in\mathrm{Min}(A):Jp\neq0\}$. Clearly $J$ is a
semiprime algebra, so the minimal left ideals of $J$ are determined by
$\mathrm{Min}(J)$. Let $p\in\mathrm{Min}(J)$ and $L=Jp$. Then $L=Jp\subset
Ap\subset L$ because $p=pp\in Jp=L$. Hence $Jp=Ap$ is a left ideal of $A$ and
$p\in\mathrm{Min}(A)$ because $pJp=pAp$ is a division algebra. So $Jp$ is a
minimal left ideal of $A$. This proves that $\operatorname{soc}\left(
J\right)  \subset J\cap\operatorname{soc}\left(  A\right)  $ and
$\mathrm{Min}(J)\subset E$. Let us prove the converse inclusions.

If $p\in E$ then $Ap$ is a minimal left ideal of $A$. As $Jp\subset Ap$ is a
non-zero left ideal of $A$, it follows that $Jp=Ap$, whence $pJp=pAp$ is a
division algebra. Therefore $E\subset\mathrm{Min}(J)$. We proved that
\begin{equation}
\mathrm{Min}(J)=\{p\in\mathrm{Min}(A):Jp\neq0\}. \label{mip}%
\end{equation}

Let $a\in J\cap\operatorname{soc}\left(  A\right)  $ be arbitrary. Then
$a=\sum b_{i}p_{i}+\sum c_{j}p_{j}^{\prime}$ (both sums are finite) for all
$b_{i},c_{j}\in A$, where all $p_{i}\in\mathrm{Min}(J)$ and all $p_{j}%
^{\prime}\in\mathrm{Min}(A)$ with $Jp_{j}^{\prime}=0$. In other words,
$a=b+c$, where
\[
b:=\sum b_{i}p_{i}=\sum\left(  b_{i}p_{i}\right)  p_{i}\in\operatorname{soc}%
\left(  J\right)
\]
and $c:=\sum c_{j}p_{j}^{\prime}=a-b\in J$. By condition, $Jp_{j}^{\prime}=0$
for all $j$, so one has \ $Jc=0$. But the set $K=\left\{  d\in J:Jd=0\right\}
$ is an ideal of $A$ with $K^{2}=0$, whence $K=0$ and therefore $c=0$ and
$a=b\in\operatorname{soc}\left(  J\right)  $. Thus $J\cap\operatorname{soc}%
\left(  A\right)  \subset\operatorname{soc}\left(  J\right)  $.

If $\operatorname{rad}\left(  A\right)  \cap\operatorname{soc}\left(
A\right)  \neq0$ then $\operatorname{rad}\left(  A\right)  $ has a non-zero
socle and there is a non-zero projection $p$ in $\operatorname{rad}\left(
A\right)  $; this is impossible.
\end{proof}

Let $P$ be a preradical. Define the map $P^{\operatorname{soc}}$ by
\[
P^{\operatorname{soc}}=\operatorname{soc}\ast P,
\]
so $P^{\operatorname{soc}}$ is the preimage of $\operatorname{soc}\left(
A/P\left(  A\right)  \right)  $ in $A$, for each algebra $A$. Let
\[
\mathrm{Min}_{P}(A)=\left\{  x\in A:q_{P\left(  A\right)  }\left(  x\right)
\in\mathrm{Min}(A/P\left(  A\right)  )\right\}
\]
where $q_{P\left(  A\right)  }:A\longrightarrow A/P\left(  A\right)  $ is the
standard quotient map.

\begin{theorem}
\label{psoc}Let $P$ be a radical such that $P\geq\mathfrak{P}_{\beta}$. Then

\begin{enumerate}
\item $\overline{P^{\operatorname{soc}}}$ is a topological under radical;

\item If $P$ is topological then $\overline{P^{\operatorname{soc}}%
}=P^{\overline{\operatorname{soc}}}$;

\item If $P$ is algebraic then $P^{\operatorname{soc}}$ is an algebraic under
radical, and if $P$ is hereditary then $P^{\operatorname{soc}\ast}$ is hereditary.
\end{enumerate}
\end{theorem}

\begin{proof}
In $\left(  2\right)  $ the equality $\overline{P^{\operatorname{soc}}%
}=P^{\overline{\operatorname{soc}}}$ is evident; $\left(  1\right)  $ is a
consequence of $\left(  3\right)  $ and Theorem \ref{closunder} in the
algebraic case and is proved similarly to $\left(  3\right)  $ otherwise. So
we will prove only $\left(  3\right)  $.

$\left(  3\right)  $ $P^{\operatorname{soc}}$ is a preradical by Lemma
\ref{conv}. Let $A$ be an algebra, and let $J$ be an ideal of $A$. Since
\ $A/P\left(  A\right)  $ and $J/P\left(  J\right)  $ are semiprime then, to
prove the inclusion $P^{\operatorname{soc}}\left(  J\right)  \subset
P^{\operatorname{soc}}\left(  A\right)  $, it suffices to show that
\begin{equation}
\mathrm{Min}_{P}(J)\subset\mathrm{Min}_{P}(A)\cup\left\{  0\right\}  .
\label{mipp}%
\end{equation}
Let $q^{\prime}:J\longrightarrow J/P\left(  J\right)  $ be the quotient map,
$q=q_{P\left(  A\right)  }$ (as above), and let $g:J/P\left(  J\right)
\longrightarrow A/P\left(  A\right)  $ be the homomorphism $q^{\prime}\left(
a\right)  \longmapsto q\left(  a\right)  $ for $a\in J$. Let $e\in
\mathrm{Min}_{P}(J)$ be arbitrary. Then $q^{\prime}\left(  eJe\right)  $ is a
division algebra. If $q\left(  e\right)  \neq0$ then $q\left(  eJe\right)
=g\left(  q^{\prime}\left(  eJe\right)  \right)  $ is also a division algebra
and $q\left(  e\right)  \in\mathrm{Min}\left(  q\left(  J\right)  \right)
\subset\mathrm{Min}\left(  A/P\left(  A\right)  \right)  $ by (\ref{mip}). In
all cases $e\in\mathrm{Min}_{P}(A)\cup\left\{  0\right\}  $. This completes
the proof of $\left(  \ref{mipp}\right)  $, so the inclusion
$P^{\operatorname{soc}}\left(  J\right)  \subset P^{\operatorname{soc}}\left(
A\right)  $ holds.

If $a\in A$ is arbitrary, then $q\left(  a\right)  q\left(  Je\right)  \subset
q\left(  Je\right)  $ and $q\left(  Je\right)  q\left(  a\right)  \subset
q\left(  Jea\right)  =q\left(  Jeea\right)  $. As $q\left(  ea\right)  \in
q\left(  J\right)  $ then, by \cite[Lemma 30.7]{BD73}, $q\left(  Je\left(
ea\right)  \right)  $ is a minimal left ideal of $q\left(  J\right)  $ or
zero. So $P^{\operatorname{soc}}\left(  J\right)  $ is an ideal of $A$.

Let $B=P^{\operatorname{soc}}\left(  A\right)  $. Then $P\left(  B\right)
\subset P\left(  A\right)  $. But $P\left(  A\right)  $ is a $P$-radical ideal
of $B$; so $P\left(  B\right)  =P\left(  A\right)  $. Hence
$\operatorname{soc}\left(  B/P\left(  B\right)  \right)  =\operatorname{soc}%
\left(  A/P\left(  A\right)  \right)  $. This shows that
$P^{\operatorname{soc}}\left(  P^{\operatorname{soc}}\left(  A\right)
\right)  =P^{\operatorname{soc}}\left(  A\right)  $. We proved that
$P^{\operatorname{soc}}$ is an under radical.

Let now $P$ be hereditary. Since $P\left(  J\right)  =J\cap P\left(  A\right)
$ then $g:q^{\prime}\left(  x\right)  \longmapsto q\left(  x\right)  $ is an
isomorphism of $J/P\left(  J\right)  $ onto the ideal $q\left(  J\right)  $ of
$A/P\left(  A\right)  $. Therefore
\[
g\left(  \mathrm{Min}\left(  J/P\left(  J\right)  \right)  \right)
=\mathrm{Min}\left(  q\left(  J\right)  \right)  =\left\{  q\left(  y\right)
\in\mathrm{Min}\left(  A/P\left(  A\right)  \right)  :q\left(  Jy\right)
\neq0\right\}
\]
by $\left(  \ref{mip}\right)  $. So if $y\in J$ and $q\left(  y\right)
\in\mathrm{Min}\left(  A/P\left(  A\right)  \right)  $ then $q^{\prime}\left(
y\right)  \in\mathrm{Min}\left(  J/P\left(  J\right)  \right)  $. It follows
that $J\cap P^{\operatorname{soc}}\left(  A\right)  \subset
P^{\operatorname{soc}}\left(  J\right)  $. Then $P^{\operatorname{soc}}$ is
hereditary. By Corollary \ref{hera}, $P^{\operatorname{soc}\ast}$ is hereditary.
\end{proof}

As a consequence of Theorem \ref{psoc}, $P^{\operatorname{soc}\ast}$ and
$\overline{P^{\operatorname{soc}}}^{\ast}$ are (algebraic and, respectively,
topological) radicals.

\subsection{Some applications}

\begin{lemma}
\label{onesid}Let $A$ be an algebra, and let $I$ be a one-sided ideal of $A$. Then

\begin{enumerate}
\item $\operatorname{rad}\left(  I\right)  =I\cap\operatorname{rad}\left(
A\right)  $.

\item If $J$ is a one-sided ideal of $I$ then $\operatorname{rad}\left(
J\right)  =J\cap\operatorname{rad}\left(  A\right)  $.
\end{enumerate}
\end{lemma}

\begin{proof}
$\left(  1\right)  $ Let $I$ be a left ideal of $A$. As $\operatorname{rad}%
\left(  A\right)  $ is the largest left ideal consisting of left
quasi-invertible elements of $A$ \cite[Proposition 24.16]{BD73},
$\operatorname{rad}\left(  I\right)  \subset\operatorname{rad}\left(
A\right)  $ and also $I\cap\operatorname{rad}\left(  A\right)  \subset
\operatorname{rad}\left(  I\right)  $.

$\left(  2\right)  $ By $\left(  1\right)  $, we have that $\operatorname{rad}%
\left(  J\right)  =J\cap\operatorname{rad}\left(  I\right)  $. Then
\[
\operatorname{rad}\left(  J\right)  =J\cap\operatorname{rad}\left(  I\right)
=J\cap I\cap\operatorname{rad}\left(  A\right)  =J\cap\operatorname{rad}%
\left(  A\right)  .
\]

\end{proof}

\begin{proposition}
\label{onesided} Let $A$ be an algebra, and let $J$ be a one-sided ideal of
$A$. Then $\operatorname{rad}^{\operatorname{soc}}(J)\subset\operatorname{rad}%
^{\operatorname{soc}}(A)$.
\end{proposition}

\begin{proof}
Let $J$ be a left ideal. As $\operatorname{rad}\left(  J\right)  =J\cap$
$\operatorname{rad}\left(  A\right)  $ by Lemma \ref{onesid}, the standard map
$g:J/\operatorname{rad}\left(  J\right)  \longrightarrow
B:=A/\operatorname{rad}\left(  A\right)  $ defined by $g:x/\operatorname{rad}%
\left(  J\right)  \longmapsto x/\operatorname{rad}\left(  A\right)  $ is
one-to-one. Let $I=g\left(  J/\operatorname{rad}\left(  J\right)  \right)  $.
Then $I$ is a left ideal of $B$ and one may identify the socle of
$J/\operatorname{rad}\left(  J\right)  $ with the socle of $I$ (because $g$
and $g^{-1}$ are morphisms).

If $L$ is a minimal left ideal of $I$ there is $p\in\mathrm{Min}\left(
I\right)  $ such that $L=Ip$. Then $p\in L$, $BL$ is a left ideal of $B$ and
$BL\subset L$. So either $BL=0$ or $BL=L$. But the equality $BL=0$ is
impossible since $B$ is semisimple. So $L$ is a minimal left ideal of $B$.
This proves $\operatorname{soc}(I)\subset\operatorname{soc}(B)$ and the result follows.
\end{proof}

Recall that $\digamma\!\left(  A\right)  $ is the set of all finite rank
elements of $A$.

\begin{lemma}
\label{soc}Let $A$ be a semiprime normed algebra. Then

\begin{enumerate}
\item $\operatorname{soc}\left(  A\right)  =\digamma\!\left(  A\right)  $;

\item $\overline{\operatorname{soc}\left(  A\right)  }\cap\overline
{\operatorname{rad}\left(  A\right)  }=0$;

\item If $a,b\in\digamma\!\left(  A\right)  $ then $\mathrm{L}_{a}%
\mathrm{R}_{b}$ is a finite rank operator on $A$.
\end{enumerate}
\end{lemma}

\begin{proof}
$\left(  1\right)  $ Let $p\in\mathrm{Min}\left(  A\right)  $ be arbitrary. As
$pAp$ is a division normed algebra, $pAp$ is one-dimensional by the
Gelfand-Masur theorem, so $p\in\digamma\!\left(  A\right)  $, whence
$Ap\subset\digamma\!\left(  A\right)  $ and $\operatorname{soc}\left(
A\right)  \subset\digamma\!\left(  A\right)  $.

Let $x\in\digamma\!\left(  A\right)  $. Then the algebra $xA^{1}x$ is
finite-dimensional. As $A^{1}x$ is a left ideal of $A^{1}$ then
$\operatorname{rad}\left(  A^{1}x\right)  =A^{1}x\cap\operatorname{rad}\left(
A\right)  $ by Lemma \ref{onesid}. Let $B=xA^{1}x+\mathbb{C}x$. Then $B$ is a
finite-dimensional algebra and $B\subset\digamma\!\left(  A\right)  $.

As $B$ is a right ideal of $A^{1}x$, it follows that
\[
\operatorname{rad}\left(  B\right)  =B\cap\operatorname{rad}\left(  A\right)
\subset\digamma\!\left(  A\right)  \cap\operatorname{rad}\left(  A\right)  =0
\]
by Lemmas \ref{onesid} and \ref{spfr}. As $B$ is a finite-dimensional
semisimple algebra then $B=\operatorname{soc}\left(  B\right)  $. Therefore
$x\in L_{1}+\ldots+L_{n}$ where $L_{i}$ is a minimal left ideal of $B$ for
every $i$. Then there is $p_{i}$ $\in\mathrm{Min}\left(  B\right)  $, for
every $i$, such that $L_{i}=Bp_{i}$. As
\[
p_{i}A^{1}p_{i}=p_{i}\left(  p_{i}A^{1}p_{i}\right)  p_{i}\subset p_{i}%
Bp_{i}\subset p_{i}A^{1}p_{i}%
\]
and $p_{i}Bp_{i}=\mathbb{C}p_{i}$ then $p_{i}\in\mathrm{Min}\left(
A^{1}\right)  $ and $A^{1}p_{i}$ is a minimal left ideal of $A^{1}$, whence
$x\in\operatorname{soc}\left(  A^{1}\right)  $. It remains to note that if $A$
is infinite-dimensional then clearly $\operatorname{soc}\left(  A^{1}\right)
=\operatorname{soc}\left(  A\right)  $, and if $A$ is finite-dimensional then
$A$ is semisimple and $A=A^{1}$ by the classical Wedderburn results.

$\left(  2\right)  $ Follows from $\left(  1\right)  $ and Lemma \ref{spfr}.

$\left(  3\right)  $ One may assume that $A$ is unital. Let $p_{1},p_{2}%
\in\mathrm{Min}\left(  A\right)  $ be arbitrary. As $p_{1}+p_{2}%
\in\operatorname{soc}\left(  A\right)  $ then $\mathrm{L}_{p_{1}+p_{2}%
}\mathrm{R}_{p_{1}+p_{2}}$ is a finite rank operator, whence $\mathrm{L}%
_{p_{1}}\mathrm{R}_{p_{2}}+\mathrm{L}_{p_{2}}\mathrm{R}_{p_{1}}$ is a finite
rank operator on $A$. Then
\[
\mathrm{L}_{1-p_{2}}\left(  \mathrm{L}_{p_{1}}\mathrm{R}_{p_{2}}%
+\mathrm{L}_{p_{2}}\mathrm{R}_{p_{1}}\right)  =\mathrm{L}_{p_{1}}%
\mathrm{R}_{p_{2}}-\mathrm{L}_{p_{2}}\mathrm{L}_{p_{1}}\mathrm{R}_{p_{2}}%
\]
is a finite rank operator. As $\mathrm{L}_{p_{2}}\mathrm{L}_{p_{1}}%
\mathrm{R}_{p_{2}}=\left(  \mathrm{L}_{p_{2}}\mathrm{R}_{p_{2}}\right)
\mathrm{L}_{p_{1}}$ is a finite rank operator, $\mathrm{L}_{p_{1}}%
\mathrm{R}_{p_{2}}$ is a finite rank operator.

Let $a,b\in\operatorname{soc}\left(  A\right)  $. Then there are $p_{1}%
,\ldots,p_{n}\in\mathrm{Min}\left(  A\right)  $, $x_{1},y_{1},\ldots
,x_{n},y_{n}\in A$ such that $a=\sum_{i}x_{i}p_{i}$ and $b=\sum_{j}y_{j}p_{j}%
$. Then%
\[
\mathrm{L}_{a}\mathrm{R}_{b}=\sum_{i=1}^{n}\sum_{j=1}^{n}\mathrm{L}%
_{x_{i}p_{i}}\mathrm{R}_{y_{j}p_{j}}=\sum_{i=1}^{n}\sum_{j=1}^{n}%
\mathrm{L}_{x_{i}}\left(  \mathrm{L}_{p_{i}}\mathrm{R}_{p_{j}}\right)
\mathrm{R}_{y_{j}}%
\]
is a finite rank operator.
\end{proof}

\begin{remark}
$\left(  1\right)  $ and $\left(  3\right)  $ of Lemma $\ref{soc}$ were proved
for not necessary normed, semiprime algebras with lower socle instead of socle
in \cite[Lemma 3.1 and Theorem 3.3]{BE03}; the lower socle of $A$ is defined
as the ideal generated by all minimal projections $p$ with $pAp<\infty$. It
should be noted that our proofs are completely different from ones in
\cite{BE03}.
\end{remark}

Lemma \ref{soc}$\left(  3\right)  $ yields

\begin{corollary}
Every finite semiprime normed algebra is bifinite.
\end{corollary}

It follows that if $A$ is a semiprime normed algebra and $\mathcal{R}%
_{\mathrm{hf}}\left(  A\right)  \neq0$ then $\mathcal{E\!\ell}\left(
A\right)  $ has nonzero finite rank operators. The converse is also true.

\begin{theorem}
\label{fre}Let $A$ be an algebra. If there is a finite rank elementary
operator $T$ on $A$ then the image of $T$ is contained in $\mathfrak{R}%
_{\mathrm{hf}}\left(  A\right)  $.
\end{theorem}

\begin{proof}
Let $T=\sum_{i=1}^{n}\mathrm{L}_{a_{i}}\mathrm{R}_{b_{i}}$, $I=\mathfrak{R}%
_{\mathrm{hf}}\left(  A\right)  $ and $B=A/I$. Then the operator $S=\sum
_{i=1}^{n}\mathrm{L}_{a_{i}/I}\mathrm{R}_{b_{i}/I}$ is a finite rank
elementary operator on $B$. By \cite[Lemma 7.1]{BT07}, if $S\neq0$ then there
is a non-zero finite rank element of $B$ which is impossible. Therefore $S=0$
and the image of $T$ is contained in $\mathfrak{R}_{\mathrm{hf}}\left(
A\right)  $.
\end{proof}

In particular, if $A$ is normed then $\mathfrak{R}_{\mathrm{hf}}\left(
A\right)  \subset\mathcal{R}_{\mathrm{hf}}\left(  A\right)  $ and the image of
$T$ lies in $\mathcal{R}_{\mathrm{hf}}\left(  A\right)  $.

\begin{remark}
By \cite[Theorem 8.4]{BT07}, if $A$ is a Banach algebra and $T$ is a compact
elementary operator on $A$ then the image of $T$ is contained in
$\mathcal{R}_{\mathrm{hc}}\left(  A\right)  $. It should be noted that this
result holds also for normed algebras, with the same proof.
\end{remark}

\begin{problem}
Let $A$ be a Banach algebra, and let $T=\sum_{i=1}^{\infty}\mathrm{L}_{a_{i}%
}\mathrm{R}_{b_{i}}$ be a compact operator on $A$ with $\sum_{i}^{\infty
}\left\Vert a_{i}\right\Vert \left\Vert b_{i}\right\Vert <\infty$ for some
$a_{i}$, $b_{i}\in A$. Is the image of $T$ contained in $\mathcal{R}%
_{\mathrm{hc}}\left(  A\right)  $?
\end{problem}

\begin{theorem}
$\mathcal{R}_{\mathrm{hf}}=\mathcal{P}_{\beta}^{\overline{\operatorname{soc}%
}\ast}$.
\end{theorem}

\begin{proof}
Let $A$ be a normed algebra. If $A$ is $\mathcal{P}_{\beta}^{\overline
{\operatorname{soc}}\ast}$-semisimple then $A$ is semiprime and
$\operatorname{soc}\left(  A\right)  =0$. By Lemma \ref{soc}, $A$ has no
non-zero finite rank elements, whence $A$ is $\mathcal{R}_{\mathrm{hf}}%
$-semisimple. Therefore $\mathcal{R}_{\mathrm{hf}}\leq\mathcal{P}_{\beta
}^{\overline{\operatorname{soc}}\ast}$ by Theorem \ref{equality}.

Conversely, if $A$ has no non-zero finite rank elements then $A$ is semisimple
and $\operatorname{soc}\left(  A\right)  =0$. Therefore $\mathcal{P}_{\beta
}^{\overline{\operatorname{soc}}\ast}\leq\mathcal{R}_{\mathrm{hf}}$ by Theorem
\ref{equality}.
\end{proof}

\begin{theorem}
$\mathcal{R}_{\mathrm{hc}}=\mathcal{R}_{\mathrm{hf}}\vee\mathcal{R}%
_{\mathrm{jhc}}=\mathcal{R}_{\mathrm{jhc}}^{\overline{\operatorname{soc}}\ast
}\leq\mathcal{R}_{\mathrm{cq}}^{\overline{\operatorname{soc}}\ast}$ on Banach algebras.
\end{theorem}

\begin{proof}
Let $A$ be a Banach algebra. If $A$ is $\mathcal{R}_{\mathrm{jhc}}%
^{\overline{\operatorname{soc}}\ast}$-semisimple then $A$ has no non-zero
compact elements which lie in $\operatorname{Rad}\left(  A\right)  $, and
$\operatorname{soc}\left(  A\right)  =0$. Hence $A$ is semiprime, so $A$ has
no non-zero finite rank elements by Lemma \ref{soc}. Thus $A$ is $\left(
\mathcal{R}_{\mathrm{hf}}\vee\mathcal{R}_{\mathrm{jhc}}\right)  $-semisimple.

It follows that $A$ has no non-zero compact elements. Indeed, if $a\in A$ is a
compact element then every spectral projection of $a$ corresponding to a
non-zero eigenvalue is clearly a finite rank element. So $a$ is
quasinilpotent. But $A^{1}a$ consists of compact elements, whence
$A^{1}a\subset\operatorname{Rad}\left(  A\right)  $. By the assumption, $a=0$.
Thus $A$ is $\mathcal{R}_{\mathrm{hc}}$-semisimple.

By Theorem \ref{equality},%
\[
\mathcal{R}_{\mathrm{jhc}}^{\overline{\operatorname{soc}}\ast}\leq
\mathcal{R}_{\mathrm{hf}}\vee\mathcal{R}_{\mathrm{jhc}}\leq\mathcal{R}%
_{\mathrm{hc}}%
\]
on Banach algebras. Assume now that $A$ has no non-zero compact elements. Then
it is clear that $A$ is semiprime, has no non-zero finite rank elements and
$\operatorname{soc}\left(  A\right)  =0$. This proves the converse.

Let now $A$ be $\mathcal{R}_{\mathrm{cq}}^{\overline{\operatorname{soc}}\ast}%
$-semisimple. As $\mathcal{R}_{\mathrm{cq}}\left(  A\right)  =0$ then $A$ is
semiprime, and as $\operatorname{soc}\left(  A\right)  =0$ then $A$ has no
non-zero finite rank elements by Lemma \ref{soc}. As above, if $a\in A$ is a
compact element then $a\in\operatorname{Rad}\left(  A\right)  $. But
\[
\mathcal{R}_{\mathrm{hc}}\left(  A\right)  \cap\operatorname{Rad}\left(
A\right)  \subset\mathcal{R}_{\mathrm{cq}}\left(  A\right)  =0
\]
by $\left(  \ref{aff}\right)  $ (the inclusion follows by the algebraic
version of the joint spectral radius formula $\left(  \ref{af}\right)  $),
whence $a=0$. Therefore $A$ has no non-zero compact elements, i.e., $A$ is
$\mathcal{R}_{\mathrm{hc}}$-semisimple. By Theorem \ref{equality},
$\mathcal{R}_{\mathrm{hc}}\leq\mathcal{R}_{\mathrm{cq}}^{\overline
{\operatorname{soc}}\ast}$ on Banach algebras.
\end{proof}

\section{The kernel-hull closures of radicals and the primitivity procedure}

\subsection{The kernel-hull closures of radicals}

Let $A$ be an algebra; $\operatorname{Prim}\left(  A\right)  $ is called the
\textit{structure space} of $A$. For any subset $E$ of $A$, let $\mathrm{hull}%
\left(  E;A\right)  $ or simply $\mathrm{h}(E;A)$ be the set of all
$I\in\operatorname{Prim}\left(  A\right)  $ with $E\subset I$; this set is
called a \textit{hull} of $E$. For any set $W\subset\operatorname{Prim}\left(
A\right)  $, let \textrm{ker}$(W;A)$ or $\mathrm{k}(W;A)$ be defined by
$\mathrm{k}(W;A)=\cap_{I\in W}I$; this ideal is called a \textit{kernel} of
$W$. Note that the Jacobson radical $\operatorname{rad}$ coincides with the
\textit{kernel-hull closure} (briefly, $\mathrm{kh}$\textit{-closure}) of
zero:
\[
\operatorname{rad}\left(  A\right)  :=\cap\operatorname{Prim}\left(  A\right)
=\mathrm{kh}\left(  \left\{  0\right\}  ;A^{1}\right)  :=\mathrm{k}\left(
\mathrm{h}\left(  \left\{  0\right\}  ;A^{1}\right)  ;A^{1}\right)
\]
for each $A\in\mathfrak{U}_{\mathrm{a}}$. The operation of the
\textit{hull-kernel closure}
\[
M\longmapsto\mathrm{hk}\left(  M;A\right)  :=\mathrm{h}\left(  \mathrm{k}%
\left(  M;A\right)  ;A\right)  ,
\]
for $M\subset\operatorname{Prim}\left(  A\right)  $, determines the
\textit{Jacobson topology} on $\operatorname{Prim}\left(  A\right)  $
\cite[Section 1.1]{B67}: closed sets are the sets of form $\mathrm{hk}\left(
M;A^{1}\right)  $. As is known, this topology is not Hausdorff in general.

Let $\Omega$ be a primitive map, and let $\Pi_{\Omega}$ be the related ideal
map on $\mathfrak{U}$ (Section \ref{primitive}). Let $A$ be an algebra, and
let
\[
\operatorname{Irr}_{\Omega}\left(  A\right)  =\left\{  \pi\in
\operatorname{Irr}\left(  A\right)  :\ker\pi\in\Omega\left(  A\right)
\right\}  .
\]
For every ideal $J$ of $A$, put
\[
\mathrm{h}_{\Omega}\left(  J;A\right)  =\mathrm{h}\left(  J;A\right)
\cap\Omega\left(  A\right)  .
\]

\begin{lemma}
\label{khi}Let $A$ be an algebra and $I,J$ be ideals of $A$. Then

\begin{enumerate}
\item If $I\subset J$ then $\mathrm{kh}_{\Omega}\left(  I;A\right)
\subset\mathrm{kh}_{\Omega}\left(  J;A\right)  $;

\item $\mathrm{kh}_{\Omega}\left(  I\cap J;J\right)  =J\cap\mathrm{kh}%
_{\Omega}\left(  I;A\right)  $;

\item $\mathrm{kh}_{\Omega}\left(  J;A\right)  =q_{J}^{-1}\left(  \Pi_{\Omega
}\left(  A/J\right)  \right)  $;

\item If $A$ is a $Q$-algebra then $I\subset\overline{I}\subset\mathrm{kh}%
_{\Omega}\left(  I;A\right)  =\mathrm{kh}_{\Omega}\left(  \overline
{I};A\right)  $.
\end{enumerate}
\end{lemma}

\begin{proof}
$\left(  1\right)  $ It is clear that $\mathrm{kh}\left(  I;A\right)
\subset\mathrm{kh}\left(  J;A\right)  $, whence $\mathrm{kh}_{\Omega}\left(
I;A\right)  \subset\mathrm{kh}_{\Omega}\left(  J;A\right)  $.

$\left(  2\right)  $ Let $K=I\cap J$, and let $\tau\in\operatorname{Irr}%
_{\Omega}(J)$ be arbitrary with $\tau(K)=0$. Then, for any representation
$\pi\in\operatorname{Irr}_{\Omega}(A)$ extending $\tau$ (on the same
representation space), we have that
\[
\pi(I)\tau(J)=\pi(IJ)=\tau(IJ)=0.
\]
It follows that $\pi(I)=0$. Conversely, the restriction $\tau$ of each $\pi
\in\operatorname{Irr}_{\Omega}\left(  A\right)  $ with $\pi(I)=0$ to $J$
vanishes on $K$ and if $\tau\neq0$ then $\tau\in\operatorname{Irr}_{\Omega
}(J)$.

If $a\in J\cap\mathrm{kh}_{\Omega}(I;A)$ then $\pi(a)=0$ for every $\pi
\in\operatorname{Irr}_{\Omega}(A)$ with $\pi\left(  I\right)  =0$. For
$\tau\in\operatorname{Irr}_{\Omega}(J)$ corresponding to $\pi$, we have that
$\tau\left(  K\right)  =0$ and $\tau(a)=\pi(a)=0$, whence $a\in\mathrm{kh}%
_{\Omega}\left(  K;J\right)  $.

If $a\in\mathrm{kh}_{\Omega}\left(  K;J\right)  $ then $\tau\left(  a\right)
=0$ for every $\tau\in\operatorname{Irr}_{\Omega}(J)$ with $\tau(K)=0$. For
$\pi\in\operatorname{Irr}_{\Omega}(A)$ corresponding to $\tau$, we have that
$\pi(I)=0$ and either $\pi(J)=0$ (whence $\pi(a)=0$) or $\pi|_{J}%
\in\operatorname{Irr}_{\Omega}(J)$ with $\pi(a)=\tau\left(  a\right)  =0$,
whence $a\in J\cap\mathrm{kh}_{\Omega}\left(  I;A\right)  $.

$\left(  3\right)  $ is obvious.

$\left(  4\right)  $ As each $J\in\mathrm{h}_{\Omega}\left(  I;A\right)  $ is
closed then $\overline{I}\subset J$, whence $\overline{I}\subset
\mathrm{kh}_{\Omega}\left(  I;A\right)  $.
\end{proof}

Let $\Omega$ be a primitive map, and let $P$ be a preradical. Define the map
$P^{\mathrm{kh}_{\Omega}}$ by
\[
P^{\mathrm{kh}_{\Omega}}\left(  A\right)  =\mathrm{kh}_{\Omega}\left(
P\left(  A\right)  ;A\right)
\]
for every algebra $A$. It follows from Lemma \ref{khi}$\left(  3\right)  $
that
\begin{equation}
P^{\mathrm{kh}_{\Omega}}\left(  A\right)  =\left(  \Pi_{\Omega}\ast P\right)
\left(  A\right)  =q_{P\left(  A\right)  }^{-1}\left(  \Pi_{\Omega}\left(
A/P\left(  A\right)  \right)  \right)  \label{radstarp}%
\end{equation}
for every algebra $A$.

In the following theorem and corollary we assume that $\Omega$ is a primitive
map defined on a base class\textbf{ $\mathfrak{U}$,} and one of the following
conditions, listed in Theorem \ref{pm1}, holds:

\begin{enumerate}
\item[$\left(  1_{kh}\right)  $] $\Omega$ is pliant and either $\mathfrak{U}%
=\mathfrak{U}_{\mathrm{a}}$ or $\mathfrak{U}_{\mathrm{q}}\subset
\mathfrak{U}\subset\mathfrak{U}_{\mathrm{n}}$ (correspondingly, $\Pi_{\Omega}$
is either a hereditary radical or a hereditary preradical);

\item[$(2_{kh})$] $\mathfrak{U}\subset\mathfrak{U}_{\mathrm{q}}$ ($\Pi
_{\Omega}$ is a hereditary topological radical).
\end{enumerate}

\begin{theorem}
\label{kh}Let $\Omega$ be a primitive map satisfying $\left(  1_{kh}\right)  $
or $\left(  2_{kh}\right)  $, and let $P$ be a preradical. Then

\begin{enumerate}
\item $P^{\mathrm{kh}_{\Omega}}$ is a preradical, and $P^{\mathrm{kh}_{\Omega
}}=\overline{P}^{\mathrm{kh}_{\Omega}}$ on $Q$-algebras;

\item If $P$ is an under radical then $P^{\mathrm{kh}_{\Omega}}$ is an under radical;

\item If $P$ is hereditary then $P^{\mathrm{kh}_{\Omega}}$ is hereditary;

\item If $\Omega$ and $P$ are pliant then $P^{\mathrm{kh}_{\Omega}}$ is pliant.
\end{enumerate}
\end{theorem}

\begin{proof}
$\left(  1\right)  $ If $f:A\longrightarrow B$ is a morphism then $f\left(
P\left(  A\right)  \right)  \subset P\left(  B\right)  $ and $f\left(
P\left(  A\right)  \right)  $ is an ideal of $B$. We will show that
\[
f\left(  \mathrm{kh}_{\Omega}\left(  P\left(  A\right)  ;A\right)  \right)
\subset\mathrm{kh}_{\Omega}\left(  f\left(  P\left(  A\right)  \right)
;B\right)  .
\]
Indeed, let $\pi\in\operatorname{Irr}_{\Omega}(B)$ be arbitrary and
$\pi(P\left(  B\right)  )=0$. Then $\tau=\pi\circ f\in\operatorname{Irr}%
_{\Omega}(A)$ and $\tau(P\left(  A\right)  )=0$. Then $\tau(a)=0$ for any
$a\in\mathrm{kh}_{\Omega}(P\left(  A\right)  ;A)$, whence $\pi(f(a))=0$.
Therefore $f(a)\in\mathrm{kh}_{\Omega}(P\left(  B\right)  ,B)$.

As $\mathrm{kh}_{\Omega}\left(  f\left(  P\left(  A\right)  \right)
;B\right)  \subset\mathrm{kh}_{\Omega}\left(  P\left(  B\right)  ;B\right)  $
by Lemma \ref{khi} then
\[
f\left(  P^{\mathrm{kh}_{\Omega}}\left(  A\right)  \right)  =f\left(
\mathrm{kh}_{\Omega}\left(  P\left(  A\right)  ;A\right)  \right)
\subset\mathrm{kh}_{\Omega}\left(  P\left(  B\right)  ;B\right)
=P^{\mathrm{kh}_{\Omega}}\left(  B\right)
\]
and $P^{\mathrm{kh}_{\Omega}}$ is a preradical.

If $A$ is a $Q$-algebra then
\[
P^{\mathrm{kh}_{\Omega}}\left(  A\right)  =\mathrm{kh}_{\Omega}\left(
P\left(  A\right)  ;A\right)  =\mathrm{kh}_{\Omega}\left(  \overline{P\left(
A\right)  };A\right)  =\mathrm{kh}_{\Omega}\left(  \overline{P}\left(
A\right)  ;A\right)  =\overline{P}^{\mathrm{kh}_{\Omega}}\left(  A\right)
\]
by Lemma \ref{khi}$\left(  4\right)  $.

It remains to note that primitive ideals of a $Q$-algebra are closed. Hence
$P^{\mathrm{kh}_{\Omega}}$ is topological on $Q$-algebras.

$\left(  2\right)  $ follows from $\left(  \ref{radstarp}\right)  $ and
Theorem \ref{oper}.

$\left(  3\right)  $ Let $A$ be an algebra, and let $J$ be an ideal of $A$. If
$P\left(  J\right)  =J\cap P\left(  A\right)  $ then
\[
P^{\mathrm{kh}_{\Omega}}\left(  J\right)  =\mathrm{kh}_{\Omega}\left(
P\left(  J\right)  ;J\right)  =J\cap\mathrm{kh}_{\Omega}\left(  P\left(
A\right)  ;A\right)  =J\cap P^{\mathrm{kh}_{\Omega}}\left(  A\right)
\]
by Lemma \ref{khi}, whence $P^{\mathrm{hk}_{\Omega}}$ is hereditary.

$\left(  4\right)  $ is obvious.
\end{proof}

\begin{corollary}
\label{khs}Let $\Omega$ be a primitive map satisfying $\left(  1_{kh}\right)
$ or $\left(  2_{kh}\right)  $, and let $P$ be an under radical. Then

\begin{enumerate}
\item $P^{\mathrm{kh}_{\Omega}\mathrm{\ast}}=\Pi_{\Omega}\vee P=\Pi_{\Omega
}^{\ast}\vee P^{\ast}$;

\item $P^{\mathrm{kh}_{\Omega}\mathrm{\ast}}$ is a radical, and
$P^{\mathrm{kh}_{\Omega}\mathrm{\ast}}=\overline{P}^{\mathrm{kh}_{\Omega
}\mathrm{\ast}}$ on $Q$-algebras;

\item If $P$ is a hereditary preradical then $\overline{P}^{\mathrm{kh}%
_{\Omega}\mathrm{\ast}}$ is hereditary on $Q$-algebras.

\item If $\Omega$ and $P$ are pliant then $P^{\mathrm{kh}_{\Omega}%
\mathrm{\ast}}$ is pliant.
\end{enumerate}
\end{corollary}

\begin{proof}
$\left(  1\right)  $ is clear.

$\left(  2\right)  $ Assume that $P$ is an (algebraic/topological) under
radical. It is clear that
\[
P^{\mathrm{kh}_{\Omega}\mathrm{\ast}}=\operatorname{rad}\vee P=\left(
\operatorname{rad}\vee P\right)  \vee\operatorname{rad}=\left(  P^{\mathrm{kh}%
_{\Omega}}\ast\operatorname{rad}\right)  ^{\ast}.
\]
Let $\left(  R_{\alpha}\right)  $ and $\left(  S_{\alpha}\right)  $ be the
algebraic and topological convolution chains of under radicals generated by
$P^{\mathrm{kh}_{\Omega}}\ast\Pi_{\Omega}$ and by $\overline{P}^{\mathrm{kh}%
_{\Omega}}\ast\Pi_{\Omega}$, respectively.

Let $A$ be a $Q$-algebra. By Theorem \ref{kh}$\left(  1\right)  $,
\[
P^{\mathrm{kh}_{\Omega}}\ast\Pi_{\Omega}\left(  A\right)  =\overline
{P}^{\mathrm{kh}_{\Omega}}\ast\Pi_{\Omega}\left(  A\right)  .
\]
We prove by transfinite induction that $R_{\alpha+1}\left(  A\right)
=S_{\alpha+1}\left(  A\right)  $ for every ordinal $\alpha$, and $R_{\alpha
}^{\mathrm{kh}_{\Omega}}\left(  A\right)  =S_{\alpha}^{\mathrm{kh}_{\Omega}%
}\left(  A\right)  $ for a limit ordinal $\alpha$. Indeed, the step
$\alpha\mapsto\alpha+1$ is easy if $R_{\alpha}\left(  A\right)  =S_{\alpha
}\left(  A\right)  $. Assume that $R_{\alpha}^{\mathrm{kh}_{\Omega}}\left(
A\right)  =S_{\alpha}^{\mathrm{kh}_{\Omega}}\left(  A\right)  $ for some limit
ordinal $\alpha$. Then
\begin{align*}
R_{\alpha+1}\left(  A\right)   &  =P^{\mathrm{kh}_{\Omega}}\ast\Pi_{\Omega
}\left(  R_{\alpha}\left(  A\right)  \right)  =P^{\mathrm{kh}}\ast\Pi_{\Omega
}\left(  R_{\alpha}^{\mathrm{kh}}\left(  A\right)  \right) \\
&  =\overline{P}^{\mathrm{kh}}\ast\Pi_{\Omega}\left(  S_{\alpha}^{\mathrm{kh}%
}\left(  A\right)  \right)  =\overline{P}^{\mathrm{kh}}\ast\Pi_{\Omega}\left(
S_{\alpha}\left(  A\right)  \right) \\
&  =S_{\alpha+1}\left(  A\right)  .
\end{align*}
So we proved the step $\alpha\mapsto\alpha+1$.

So it remains to show that $R_{\alpha}^{\mathrm{kh}_{\Omega}}\left(  A\right)
=S_{\alpha}^{\mathrm{kh}_{\Omega}}\left(  A\right)  $ for each limit ordinal
$\alpha$. Assume by induction, that $R_{\alpha^{\prime}+1}\left(  A\right)
=S_{\alpha^{\prime}+1}\left(  A\right)  $ for every ordinal $\alpha^{\prime
}<\alpha$. Then
\begin{align*}
R_{\alpha}^{\mathrm{kh}_{\Omega}}\left(  A\right)   &  =\mathrm{kh}_{\Omega
}\left(  \cup_{\alpha^{\prime}<\alpha}R_{\alpha^{\prime}+1}\left(  A\right)
;A\right)  =\mathrm{kh}_{\Omega}\left(  \cup_{\alpha^{\prime}<\alpha}%
S_{\alpha^{\prime}+1}\left(  A\right)  ;A\right) \\
&  =\mathrm{kh}_{\Omega}\left(  \overline{\cup_{\alpha^{\prime}<\alpha
}S_{\alpha^{\prime}+1}\left(  A\right)  };A\right)  =S_{\alpha}^{\mathrm{kh}%
}\left(  A\right)
\end{align*}
by Lemma \ref{khi}$\left(  4\right)  $.

Therefore there is an ordinal $\gamma$ such that
\[
P^{\mathrm{kh}_{\Omega}\mathrm{\ast}}\left(  A\right)  =R_{\gamma+1}\left(
A\right)  =R_{\gamma+2}\left(  A\right)  =S_{\gamma+2}\left(  A\right)
=S_{\gamma+1}\left(  A\right)  =\overline{P}^{\mathrm{kh}_{\Omega}%
\mathrm{\ast}}\left(  A\right)  .
\]

$\left(  3\right)  $ follows from $\left(  2\right)  $ and Corollary
\ref{hera}.

$\left(  4\right)  $ follows from Theorem \ref{kh}$\left(  4\right)  $.
\end{proof}

\begin{remark}
In $\left(  3\right)  $ of Corollary $\ref{khs}$ we state the equality of two
radicals on $Q$-algebras if one of them is obtained by application of the
algebraic convolution procedure while the other one is obtained by application
of the topological convolution procedure. So the resulting radical can be
considered as an algebraic radical as well as a topological radical,
simultaneously. This gathers all advantages of both classes of radicals.
\end{remark}

\subsection{Primitivity procedure}

Let $\Omega$ be a primitive map, and let $R$ be a preradical. Define
$R^{p_{\Omega}}$ by
\[
R^{p_{\Omega}}\left(  A\right)  =\cap\left\{  \left(  R\ast I;A\right)
:I\in\Omega\left(  A\right)  \cup\left\{  A\right\}  \right\}
\]
for every algebra $A$; we recall that
\[
\left(  R\ast I;A\right)  =q_{I}^{-1}\left(  P\left(  A/I\right)  \right)  .
\]
\textbf{I}n other words, $R^{p_{\Omega}}(A)$ consists of all elements $a\in A$
such that $a/I\in P(A/I)$ for all $I\in\Omega(A)$.

If $\Omega\left(  A\right)  =\operatorname{Prim}\left(  A\right)  $ then we
write $R^{p}$ instead of $R^{p_{\Omega}}$. The procedure $R\longmapsto
R^{p_{\Omega}}$ is called the $\Omega$-\textit{primitivity procedure}.

\begin{remark}
If $R\leq\Pi_{\Omega}$ then $A/I$ is $\Pi_{\Omega}$-semisimple and therefore
$R$-semisimple, whence $\left(  R\ast I;A\right)  =I$ for every $I\in
\Omega\left(  A\right)  \cup\left\{  A\right\}  $; thus $R^{p}\left(
A\right)  =\Pi_{\Omega}\left(  A\right)  $. In particular this is the case for
$R\leq\operatorname{rad}$.
\end{remark}

For $R^{p_{\Omega}}$ to have sufficiently convenient properties, one has to
impose additional requirements on $R$. One of possible approaches is to
require that $R$ well behaves on Banach ideals. We mean the condition of
Banach heredity (see $(\ref{BanRad})$).

\begin{lemma}
\label{ban}Let $R$ be a preradical on $\mathfrak{U}_{\mathrm{b}}$ satisfying
the condition of Banach heredity. If $A$ is a Banach algebra, $I$ and $J$ are
closed ideals of $A$, then
\[
\left(  R\ast\left(  J\cap I\right)  ;J\right)  =J\cap\left(  R\ast
I;A\right)  .
\]

\end{lemma}

\begin{proof}
Let $K=\left(  J+I\right)  /I$ with the norm $\left\Vert \cdot\right\Vert
_{K}$ defined in Lemma \ref{bi}. Then the natural map $\phi:J/\left(  J\cap
I\right)  \longrightarrow\left(  K;\left\Vert \cdot\right\Vert _{K}\right)  $
is an isometry. Let $q:J\longrightarrow J/\left(  J\cap I\right)  $ and
$q_{I}:A\longrightarrow A/I$ be standard quotient maps. Considering $\phi$ as
an algebraic morphism, one may write that $q_{I}|_{J}=\phi\circ q$. By the
assumption,
\[
\phi\left(  R\left(  \left(  J/\left(  J\cap I\right)  \right)  \right)
\right)  =R\left(  K;\left\Vert \cdot\right\Vert _{K}\right)  =K\cap R\left(
A/I\right)
\]
As $q_{I}|_{J}=\phi\circ q$ then
\[
q^{-1}\left(  R\left(  \left(  J/\left(  J\cap I\right)  \right)  \right)
\right)  =\left(  q_{I}|_{J}\right)  ^{-1}\left(  K\cap R\left(  A/I\right)
\right)  .
\]
It is clear that $\left(  q_{I}|_{J}\right)  ^{-1}\left(  K\right)  =\left(
J+I\right)  \cap J=J$. Then
\[
J\cap\left(  q_{I}|_{J}\right)  ^{-1}\left(  R\left(  A/I\right)  \right)
=J\cap q_{I}^{-1}\left(  R\left(  A/I\right)  \right)
\]
and the result follows.
\end{proof}

Another way to investigate the primitivity procedure is connected with
modifications of Axiom $1$. Namely, we may require that $R$ be either pliant
or strict (i.e. satisfies $(\ref{PRER})$); clearly if $\mathfrak{U}%
=\mathfrak{U}_{\mathrm{b}}$ then this holds for all preradicals.

In the following theorem and corollary we join both approaches.

\begin{theorem}
\label{pri}Let $%
\Omega
$ be a primitive map on $\mathfrak{U}$ satisfying $\left(  1_{kh}\right)  $ or
$\left(  2_{kh}\right)  $, and let $R$ be a preradical on $\mathfrak{U}$. Then

\begin{enumerate}
\item If $\Omega$ and $R$ are pliant then

\begin{enumerate}
\item $R^{p_{\Omega}}$ is a pliant preradical;

\item If $R$ is a radical then $R^{p_{\Omega}}$ satisfies Axiom $4$;

\item If $R$ is an hereditary preradical then $R^{p_{\Omega}}$ is a hereditary preradical.
\end{enumerate}

\item If $R$ is strict, $\mathfrak{U}$ is universal and contained in
$\mathfrak{U}_{\mathrm{q}}$ then

\begin{enumerate}
\item $R^{p_{\Omega}}$ is a strict preradical;

\item If $R$ is a topological radical then $R^{p_{\Omega}}$ satisfies Axiom
$4$;

\item If $R$ is a hereditary topological preradical then $R^{p_{\Omega}}$ is a
hereditary topological preradical.
\end{enumerate}

\item If $R$ is defined on $\mathfrak{U}_{\mathrm{b}}$ and satisfies the
condition of Banach heredity then $R^{p_{\Omega}}$ is a hereditary preradical
on $\mathfrak{U}_{\mathrm{b}}$.
\end{enumerate}
\end{theorem}

\begin{proof}
$\left(  1\text{a}\right)  \And\left(  2\text{a}\right)  $ Let
$f:A\longrightarrow B$ be an algebraic morphism [bounded epimorphism]. As
$\left(  \pi\circ f\right)  \in\operatorname{Irr}_{\Omega}\left(  A\right)  $
for every $\pi\in\operatorname{Irr}_{\Omega}\left(  B\right)  $, then
$J:=f^{-1}\left(  I\right)  \in\Omega\left(  A\right)  $ for every $I\in
\Omega\left(  B\right)  $. Then $f$ induces an algebraic morphism [bounded
epimorphism] $f_{0}:A/J\longrightarrow B/I$ such that $f_{0}\circ q_{J}%
=q_{I}\circ f$. Hence
\begin{align}
f\left(  R\ast J\right)   &  =f\left(  q_{J}^{-1}\left(  R\left(  A/J\right)
\right)  \right)  \subset q_{I}^{-1}\left(  f_{0}\left(  R\left(  A/J\right)
\right)  \right) \nonumber\\
&  \subset q_{I}^{-1}\left(  R\left(  B/I\right)  \right)  =R\ast I \label{fr}%
\end{align}
for every $I\in\Omega\left(  B\right)  $. So $f\left(  R^{p_{\Omega}}\left(
A\right)  \right)  \subset R^{p_{\Omega}}\left(  B\right)  $.

Assume that $\mathfrak{U}$ is universal and $R$ satisfies $\left(
\ref{PRER}\right)  $ for continuous isomorphisms. Let $f:A\longrightarrow B$
be a continuous isomorphism. Then $f^{-1}:B\longrightarrow A$ is an algebraic
morphism and $f_{0}:A/J\longrightarrow B/I$ is a continuous isomorphism, where
$J=f^{-1}\left(  I\right)  $. Then
\[
f\left(  R\ast J\right)  =f\left(  \left(  f_{0}\circ q_{J})^{-1}\left(
f_{0}\left(  R\left(  A/J\right)  \right)  \right)  \right)  \right)
=q_{I}^{-1}\left(  \left(  R\left(  A/J\right)  \right)  \right)  =R\ast I.
\]
As $\Omega\left(  A\right)  =\left\{  f^{-1}\left(  I\right)  :I\in
\Omega\left(  B\right)  \right\}  $ then
\begin{align*}
f\left(  R^{p_{\Omega}}\left(  A\right)  \right)   &  =\cap_{I\in\Omega\left(
B\right)  \cup\left\{  B\right\}  }f\left(  R\ast f^{-1}\left(  I\right)
;A\right) \\
&  =\cap_{I\in\Omega\left(  B\right)  \cup\left\{  B\right\}  }\left(  R\ast
I;B\right)  =R^{p_{\Omega}}\left(  B\right)  .
\end{align*}
Therefore $R^{p_{\Omega}}$ satisfies $\left(  \ref{PRER}\right)  $ for
continuous isomorphisms.

$\left(  1\text{b}\right)  \And\left(  2\text{b}\right)  $ If $K$ is an ideal
of $A$ then $\Omega\left(  K\right)  \cup\left\{  K\right\}  =\left\{  K\cap
I:I\in\Omega\left(  A\right)  \cup\left\{  A\right\}  \right\}  $. As $K\cap
I$ is an ideal of $A$ then $\left(  R\ast\left(  K\cap I\right)  ;K\right)  $
is an ideal of $A$. Therefore $R^{p_{\Omega}}\left(  K\right)  $ is an ideal
of $A$.

Let $q_{I}^{\prime}:K\longrightarrow K/\left(  K\cap I\right)  $ and
$q_{I}:A\longrightarrow A/I$ be standard quotient maps. There is a [bounded]
injective imbedding $g_{I}$ of $K/\left(  K\cap I\right)  $ onto an ideal
$q_{I}\left(  K\right)  $ of $A/I$, whence
\begin{equation}
g_{I}\left(  R\left(  K/\left(  K\cap I\right)  \right)  \right)  \subset
R\left(  q_{I}\left(  K\right)  \right)  \subset R\left(  A/I\right)  .
\label{pkk}%
\end{equation}
As $q_{I}|_{K}=g_{I}\circ q_{I}^{\prime}$ then
\begin{align}
R^{p_{\Omega}}\left(  K\right)   &  =\cap_{I\in\Omega\left(  A\right)
\cup\left\{  A\right\}  }\left(  q_{I}^{\prime}\right)  ^{-1}\left(  R\left(
K/\left(  K\cap I\right)  \right)  \right) \nonumber\\
&  =\cap_{I\in\Omega\left(  A\right)  \cup\left\{  A\right\}  }\left(
g_{I}\circ q_{I}^{\prime}\right)  ^{-1}\left(  g_{I}\left(  R\left(  K/\left(
K\cap I\right)  \right)  \right)  \right) \label{pkk2}\\
&  \subset\cap_{I\in\Omega\left(  A\right)  \cup\left\{  A\right\}  }%
g_{I}^{-1}\left(  R\left(  A/I\right)  \right)  =R^{p_{\Omega}}\left(
A\right)  .\nonumber
\end{align}
Therefore $R^{p_{\Omega}}$ satisfies Axiom $4$.

$\left(  1\text{c}\right)  \And\left(  2\text{c}\right)  $ As $R$ is
hereditary, and also is algebraic or satisfies $\left(  \ref{PRER}\right)  $
for continuous isomorphisms, then we rewrite $\left(  \ref{pkk}\right)  $ as
follows:
\[
g_{I}\left(  R\left(  K/\left(  K\cap I\right)  \right)  \right)  =R\left(
q_{I}\left(  K\right)  \right)  =q_{I}\left(  K\right)  \cap R\left(
A/I\right)
\]
for every $I\in\Omega\left(  A\right)  \cup\left\{  A\right\}  $. It follows
from $\left(  \ref{pkk2}\right)  $ that
\begin{align*}
R^{p_{\Omega}}\left(  K\right)   &  =\cap_{I\in\Omega\left(  A\right)
\cup\left\{  A\right\}  }\left(  q_{I}|_{K}\right)  ^{-1}\left(  q_{I}\left(
K\right)  \cap R\left(  A/I\right)  \right) \\
&  =\cap_{I\in\Omega\left(  A\right)  \cup\left\{  A\right\}  }\left(  K\cap
q_{I}^{-1}\left(  R\left(  A/I\right)  \right)  \right) \\
&  =K\cap\left(  \cap_{I\in\Omega\left(  A\right)  \cup\left\{  A\right\}
}q_{I}^{-1}\left(  R\left(  A/I\right)  \right)  \right)  =K\cap R^{p_{\Omega
}}\left(  A\right)  .
\end{align*}
Therefore $R^{p_{\Omega}}$ is a hereditary preradical.

$\left(  3\right)  $ Let $f:A\longrightarrow B$ be a topological morphism of
Banach algebras. Then $\left(  \pi\circ f\right)  \in\operatorname{Irr}%
_{\Omega}\left(  A\right)  $ for every $\pi\in\operatorname{Irr}_{\Omega
}\left(  B\right)  $, whence $J:=f^{-1}\left(  I\right)  \in\Omega\left(
A\right)  \cup\left\{  A\right\}  $ for every $I\in\Omega\left(  B\right)
\cup\left\{  B\right\}  $. Then $f$ induces a topological morphism
$f_{0}:A/J\longrightarrow B/I$ such that $f_{0}\circ q_{J}=q_{I}\circ f$.
Hence
\[
f\left(  R\ast J\right)  \subset R\ast I
\]
as in $\left(  \ref{fr}\right)  $, for every $I\in\Omega\left(  B\right)
\cup\left\{  B\right\}  $. So $f\left(  R^{p_{\Omega}}\left(  A\right)
\right)  \subset R^{p_{\Omega}}\left(  B\right)  $.

By our assumption, for a closed ideal $K$ of $A$,
\begin{align*}
R^{p_{\Omega}}\left(  K\right)   &  =\cap_{I\in\Omega\left(  A\right)
\cup\left\{  A\right\}  }\left(  R\ast\left(  K\cap I\right)  ;K\right) \\
&  =\cap_{I\in\Omega\left(  A\right)  \cup\left\{  A\right\}  }K\cap\left(
R\ast I;A\right) \\
&  =K\cap\left(  \cap_{I\in\Omega\left(  A\right)  \cup\left\{  A\right\}
}\left(  R\ast I;A\right)  \right)  =K\cap R^{p_{\Omega}}\left(  A\right)
\end{align*}
by Lemma \ref{ban}.
\end{proof}

Now we apply the convolution procedure to obtain a radical.

\begin{corollary}
Let $%
\Omega
$ be a primitive map on $\mathfrak{U}$ satisfying $\left(  1_{kh}\right)  $ or
$\left(  2_{kh}\right)  $, and let $R$ be a preradical. Then

\begin{enumerate}
\item If $\Omega$ and $R$ are pliant then $R^{p_{\Omega}\ast}$ is a pliant
hereditary radical;

\item If $R$ is a closed ideal map satisfying $\left(  \ref{PRER}\right)  $
(i.e. $R$ is strict), $\mathfrak{U}$ is universal and contained in
$\mathfrak{U}_{\mathrm{q}}$ then $R^{p_{\Omega}\ast}$ is a topological radical
on $\mathfrak{U}$;

\item If $R$ is defined on $\mathfrak{U}_{\mathrm{b}}$ and satisfies the
condition of Banach heredity then $R^{p_{\Omega}\ast}$ is a radical on Banach algebras;

\item $\Pi_{\Omega}\vee R=\Pi_{\Omega}^{\ast}\vee R^{\ast}\leq R^{p_{\Omega
}\ast}$.
\end{enumerate}
\end{corollary}

\begin{proof}
$\left(  1\right)  $-$\left(  3\right)  $ follow from the properties of
convolution procedure.

$\left(  4\right)  $ It is clear that
\[
\Pi_{\Omega}\left(  A\right)  =\cap_{I\in\Omega\left(  A\right)  \cup\left\{
A\right\}  }q_{I}^{-1}\left(  0\right)  \subset\cap_{I\in\Omega\left(
A\right)  \cup\left\{  A\right\}  }q_{I}^{-1}\left(  R\left(  A/I\right)
\right)  =R^{p_{\Omega}}\left(  A\right)
\]
for an algebra $A$, and every $R$-radical algebra is $R^{p_{\Omega}}$-radical.
This implies the result.
\end{proof}

\section{Scattered algebras}

Our aim in this section is to give a consistent exposition of the theory of
scattered radical. This radical was announced in \cite{TR0} and already
applied in \cite{BT07, TR2}.

\subsection{Thin ideals}

Let $K\subset\mathbb{C}$ be a bounded set. Recall that the
\textit{polynomially convex hull} $\mathrm{{pc}}(K)$ of $K$ is the set of all
$z\in\mathbb{C}$ such that%
\[
\left\vert p\left(  z\right)  \right\vert \leq\sup_{w\in K}\left\vert p\left(
w\right)  \right\vert
\]
for every polynomial $p$, and $K$ is \textit{polynomially convex} if
$K=\mathrm{{pc}}(K)$. By \cite[Lemma 3.1.3]{G69}, a compact subset
$K\subset\mathbb{C}$ is polynomially convex if and only if the complement of
$K$ is connected. We say that $H$ is a \textit{polynomially convex
neighborhood }of\textit{ }$K$ if $K$ is contained in the interior of $H$ and
$H=\mathrm{{pc}}\left(  H\right)  $ (hence $H$ is a compact subset of
$\mathbb{C}$).

An algebra $A$ is called\textit{ inessential} (or \textit{Riesz }%
\cite{BMSW82}) if $\sigma\left(  a\right)  $ is finite or countable with the
only one limit point at $0$ for every $a\in A$; $A$ is \textit{scattered} if
$\sigma\left(  a\right)  $ is at most countable for every $a\in A$. These
notions are transferred to ideals. As ideals are spectral subalgebras, it is
not important with respect to what -- the algebra or the ideal -- the spectrum
of an element of the ideal is considered. The algebras $\mathcal{F}\left(
X\right)  $ of finite rank operators and $\mathcal{K}\left(  X\right)  $ of
compact operators on a Banach space $X$ are inessential, and they are
inessential ideals of $\mathcal{B}\left(  X\right)  $.

By \cite[Corollary 5.7.6]{A91}, if $I$ is an inessential ideal of a Banach
algebra $A$ then $\overline{I}$ and $\mathrm{kh}\left(  I;A\right)  $ are also
inessential ideals of $A$ and have the same set of projections as $I$.

\begin{lemma}
\label{ines}Let $A$ be a Banach algebra, and let $I$ be an inessential ideal
of $A$. If $a\in A$ and $V$ is a polynomially convex neighborhood of
$\sigma(a/I)$ then there is a finite set $Z$ such that $\sigma(a)\subset V\cup
Z$.
\end{lemma}

\begin{proof}
Let $\Delta$ be the set of all isolated points in $\sigma(a)$ with Riesz
(spectral) projections in $I$, and $D=\sigma(a)\setminus\Delta$. By
\cite[Theorem 5.7.4]{A91} and Corollary \ref{sinq},
\[
D\subset\mathrm{pc}\left(  \sigma\left(  a/\overline{I}\right)  \right)
=\mathrm{pc}(\sigma\left(  a/I\right)  ).
\]
Then $D\subset V$, and only a finite number of points in $\sigma(a)$ can be
outside of $V$ because otherwise $\sigma(a)$ would have a limit point outside
of $D$ which contradicts the definition of $D$. Therefore
\[
\sigma(a)\subset V\cup Z
\]
for some finite set $Z\subset\mathbb{C}$.
\end{proof}

Let us say that an ideal $I$ of an algebra $A$ is \textit{thin} if for each
$a\in A$ there is a countable set $Z\subset\mathbb{C}$ such that%
\begin{equation}
\sigma(a)\subset\mathrm{{pc}}(\sigma(a/I))\cup Z. \label{incZ}%
\end{equation}

\begin{lemma}
\label{inc}Let $A$ be a [normed] algebra, and let $J\subset I$ be [closed]
ideals of $A$. If $I$ is a thin ideal of $A$ then $J$ is a thin ideal of $A$.
\end{lemma}

\begin{proof}
It is sufficient to note that $\sigma(a/I)\subset$ $\sigma(a/J)$ for every
$a\in A$.
\end{proof}

If $I$ is a thin ideal of a $Q$-algebra then $\overline{I}$ is also thin by
Corollary \ref{sinq}. Recall that the spectrum of each element of a
$Q$-algebra is a compact set in $\mathbb{C}$.

\begin{lemma}
\label{okr} Let $A$ be a $Q$-algebra. An ideal $I$ of $A$ is thin if and only
if for each $a\in A$ and for every polynomially convex neighborhood $V$ of
$\sigma(a/I)$, there is a countable set $Z\subset\mathbb{C}$ such that
$\sigma(a)\subset V\cup Z$.
\end{lemma}

\begin{proof}
$\Rightarrow$ is obvious.

$\Leftarrow$ Let $a\in A$, and let $Z_{V}:=\sigma(a)\setminus V$ be countable
for each polynomially convex neighborhood $V$ of $\sigma(a/I)$. Let
$D\subset\mathbb{C}$ be an open disk with
\[
\overline{D}\cap\mathrm{{pc}}(\sigma(a/I))=\varnothing.
\]
As $\mathrm{{pc}}(\sigma(a/I))$ is polynomially convex, its complement $C$ is
connected, so it contains $E_{t}=\{z:\left\vert z\right\vert \geq t\}$ for
sufficiently large $t>0$, and a curve $L\subset C$ connecting $D$ with $E_{t}%
$. Let $U\subset C$ be a connected neighborhood of $E_{t}\cup L\cup
\overline{D}$. Then the complement $V$ of $U$ is a compact neighborhood of
$\sigma(a/I)$. Since the complement of $V$ is connected then $V$ is
polynomially convex. By the assumption,
\[
\sigma(a)\subset V\cup Z_{V}%
\]
where $Z_{V}$ is countable (or finite) by assumption.

The set $\mathbb{C}\setminus\mathrm{{pc}}(\sigma(a/I))$ is open and has
therefore an inscribed covering by open disks. Taking an inscribed subcovering
by disks with rational radii and coordinates of centers one can find a
sequence $\left(  D_{k}\right)  _{1}^{\infty}$ of open disks such that
\[
\cup_{k=1}^{\infty}D_{k}=\mathbb{C}\setminus\mathrm{{pc}}(\sigma(a/I))\text{
and all }\overline{D_{k}}\cap\mathrm{{pc}}(\sigma(a/I))=\varnothing;
\]
let $V_{k}$ be compact neighborhoods of $\mathrm{{pc}}(\sigma(a/I))$
non-intersecting $D_{k}$. Then, as each $V_{k}$ lies in the complement
$D_{k}^{\mathrm{c}}$ of $D_{k}$,
\[
\mathrm{{pc}}(\sigma(a/I))\subset\cap_{k=1}^{\infty}V_{k}\subset\cap
_{k=1}^{\infty}D_{k}^{\mathrm{c}}\subset\left(  \cup_{k=1}^{\infty}%
D_{k}\right)  ^{\mathrm{c}}=\mathrm{{pc}}(\sigma(a/I)),
\]
and
\[
\sigma(a)\subset\cap_{k=1}^{\infty}(V_{k}\cup Z_{V_{k}})\subset\left(
\cap_{k=1}^{\infty}V_{k}\right)  \cup\left(  \cup_{k=1}^{\infty}Z_{V_{k}%
}\right)  =\mathrm{{pc}}(\sigma(a/I))\cup Z
\]
where $Z=\cup_{k=1}^{\infty}Z_{V_{k}}$ is a countable set.
\end{proof}

Taking into account Lemma $\ref{ines}$, we see that inessential ideals of
Banach algebras are thin ideals.

\begin{corollary}
\label{th}Let $A$ be a $Q$-algebra. Then each thin ideal $J$ is a scattered algebra.
\end{corollary}

\begin{proof}
Indeed, if $a\in J$ then $\sigma(a/J)=\{0\}$, whence, by (\ref{incZ}),
$\sigma(a)\subset\{0\}\cup Z$ is countable.
\end{proof}

\begin{lemma}
\label{stepext} Let $A$ be a $Q$-algebra, and let $I\subset J$ be closed
ideals of $A$. If $I$ and $J/I$ are thin ideals of $A$ and $A/I$,
respectively, then $J$ is a thin ideal of $A$.
\end{lemma}

\begin{proof}
Let $a\in A$ be arbitrary. Since the standard morphism $\tau:A/I\rightarrow
A/J$ has the kernel $J/I$ then, by Lemma \ref{okr},
\[
\sigma(a/I)\subset\mathrm{{pc}}(\sigma(a/J))\cup Z
\]
where $Z$ is a countable set. As the union of a polynomially convex set and a
(bounded) countable set in $\mathbb{C}$ is polynomially convex then
\[
\mathrm{{pc}}(\sigma(a/I))\subset\mathrm{{pc}}(\sigma(a/J))\cup Z.
\]
Since $I$ is thin then
\[
\sigma(a)\subset\mathrm{{pc}}(\sigma(a/I\ ))\cup Z^{\prime}%
\]
where $Z^{\prime}$ is countable. Thus%
\[
\sigma(a)\subset\mathrm{{pc}}(\sigma(a/J))\cup Z\cup Z^{\prime}.
\]

\end{proof}

The following quite general result establishes a kind of spectral continuity.

\begin{lemma}
\label{net}\label{limit} Let $A$ be a $Q$-algebra, let $\left(  J_{\alpha
}\right)  _{\alpha\in\Lambda}$ be an up-directed net of closed ideals of $A$
and $J=\overline{\bigcup_{\alpha\in\Lambda}J_{\alpha}}$. Then

\begin{enumerate}
\item
\begin{equation}
\sigma(a/J)=\bigcap_{\alpha\in\Lambda}{\sigma}(a/J_{\alpha}); \label{int1}%
\end{equation}

\item If all ideals $J_{\alpha}$ are thin then $J$ is thin.
\end{enumerate}
\end{lemma}

\begin{proof}
$\left(  1\right)  $ If $\lambda\notin\sigma(a/J)$ then there are $b\in A$ and
$x\in J$ with $(a-\lambda)b=1+x$. Hence one can find $\alpha\in\Lambda$ and
$y\in J_{\alpha}$ with $\Vert x-y\Vert<1$. By \cite[Proposition 2.2.7]{P94},
$\sum_{n>0}\left(  y-x\right)  ^{n}$ converges in $A$. Therefore $1-(y-x)$ is
invertible and $1+(x-y)=z^{-1}$ for some $z\in A$, whence $(a-\lambda
)b=z^{-1}+y$ and
\[
(a-\lambda)bz=1+yz.
\]
Therefore $\lambda\notin\sigma^{l}(a/J_{\alpha})$. Similarly, there is
$\delta\in\Lambda$ with $\lambda\notin\sigma^{r}(a/J_{\delta})$. So
$\lambda\notin\sigma(a/J_{\gamma})$ for $\gamma>\alpha$ and $\gamma>\beta$.
Hence
\[
\bigcap_{\alpha\in\Lambda}\sigma(a/J_{\alpha})\subset\sigma(a/J).
\]
The converse inclusion is evident.

$\left(  2\right)  $ Let $V$ be a polynomially convex neighborhood of
$\sigma(a/J)$ then, by (\ref{int1}) and compactness of $\sigma(a/J)$, there is
$\alpha\in\Lambda$ with $\sigma(a/J_{\alpha})\subset V$. Since $J_{\alpha}$ is
thin there is a countable set $Z\subset\mathbb{C}$ such that $\sigma(a)\subset
V\cup Z$. It remains to apply Lemma \ref{okr}.
\end{proof}

\begin{theorem}
\label{spec} Let $A$ be a $Q$-algebra, and let $\left(  J_{\alpha}\right)
_{{\alpha}\leq\gamma}$ be an increasing transfinite chain of closed ideals of
$A$ with $J_{0}=0$. If each ideal $J_{{\alpha}+1}/J_{\alpha}$ is a thin ideal
of $A/J_{\alpha}$ then all ideals $J_{\alpha}$ are thin ideals of $A$ and
therefore scattered algebras.
\end{theorem}

\begin{proof}
Using transfinite induction, assume that for increasing transfinite chains of
ideals with final ordinals $<\alpha$ the statement is true, i.e. all
$J_{\alpha^{\prime}}$ with $\alpha^{\prime}<\alpha$ are thin ideals. If
$\alpha$ is a limit ordinal then the result follows from Lemma \ref{limit}%
$\left(  2\right)  $. Otherwise $\alpha=\alpha^{\prime}+1$ for some
$\alpha^{\prime}$ and the result follows from Lemma \ref{stepext} applied to
ideals $J_{\alpha^{\prime}}\subset J_{{\alpha}^{\prime}+1}$. Therefore all
ideals $J_{\alpha}$ are thin. By Corollary \ref{th}, they are also scattered.
\end{proof}

\subsection{Scattered radical}

Recall that the algebraic under radical $\operatorname{rad}%
^{\operatorname{soc}}=\operatorname{soc}\ast\operatorname{rad}$ sends each
algebra $A$ to the ideal $\{x\in A:x/\operatorname{rad}(A)\in
\operatorname{soc}(A/\operatorname{rad}(A))\}$. This ideal, the socle modulo
radical, was called in \cite[Definition F.3.1]{BMSW82} by the
\textit{presocle} of $A$ and denoted by $\operatorname{psoc}(A)$. We preserve
this notation for the map itself: $\operatorname{psoc}:=\operatorname{soc}%
\ast\operatorname{rad}$.

Now we can apply the algebraic convolution procedure and obtain the hereditary
algebraic radical $\operatorname{psoc}^{\ast}$ on $\mathfrak{U}_{\mathrm{a}}$.

\begin{theorem}
The restriction of $\operatorname{psoc}^{\ast}$ to $\mathfrak{U}_{q}$ is a
hereditary topological radical on $Q$-algebras.
\end{theorem}

\begin{proof}
As $\operatorname{rad}\leq\operatorname{psoc}$ then
\[
\operatorname{rad}\vee\operatorname{psoc}=\operatorname{psoc}^{\ast}.
\]
(Recall that $\operatorname{rad}\vee\operatorname{psoc}$ is the smallest
radical that is larger than or equal to $\operatorname{rad}$ and
$\operatorname{psoc}$.) One can realize $\operatorname{rad}\vee
\operatorname{psoc}$ as the action of the algebraic convolution procedure,
namely
\[
\operatorname{rad}\vee\operatorname{psoc}=\left(  \operatorname{rad}%
\ast\operatorname{psoc}\right)  ^{\ast}.
\]
But $\operatorname{rad}\ast\operatorname{psoc}=\operatorname{psoc}%
^{\mathrm{kh}}$. So
\[
\operatorname{psoc}^{\ast}=\operatorname{psoc}^{\mathrm{kh}\ast}.
\]
By Corollary \ref{khs}, $\operatorname{psoc}^{\mathrm{kh}\ast}$ is a
hereditary topological radical on $Q$-algebras.
\end{proof}

Define map $\mathcal{R}_{s}$ by
\[
\mathcal{R}_{s}=\operatorname{psoc}^{\ast}\text{ on Banach algebras;}%
\]
this map is a hereditary topological radical on $\mathfrak{U}_{\mathrm{b}}$;
it is called the \textit{scattered radical}.

\begin{lemma}
\label{scat} Let $A$ be a Banach algebra. Then $\mathcal{R}_{s}(A)$ is a thin
ideal of $A$ and a scattered algebra.
\end{lemma}

\begin{proof}
Let $\left(  R_{\alpha}\right)  _{\alpha}$ be the convolution chain of under
radicals generated by $\overline{\operatorname{psoc}}$. As $\overline
{\operatorname{psoc}}^{\ast}=\mathcal{R}_{s}$ by Corollary \ref{khs}, then
there is $\gamma$ such that $R_{\gamma}\left(  A\right)  =\mathcal{R}%
_{s}\left(  A\right)  $ and all gap-quotients $R_{\alpha+1}\left(  A\right)
/R_{\alpha}\left(  A\right)  =\overline{\operatorname{psoc}}\left(
A/R_{\alpha}\left(  A\right)  \right)  $ are inessential ideals by
\cite[Section R]{BMSW82}, and therefore thin ideals by Lemmas \ref{ines} and
\ref{okr}. By Theorem \ref{spec}, $\mathcal{R}_{s}(A)$ is a thin ideal of $A$
and a scattered algebra.
\end{proof}

Let $\mathcal{S}\left(  A\right)  $ be the set of all elements of $A$ with
(finite or) countable spectrum.

\begin{theorem}
\label{large} Let $A$ be a Banach algebra. Then

\begin{enumerate}
\item $\mathcal{R}_{s}(A)$ contains all one-sided, non-necessarily closed,
scattered ideals of $A$;

\item $\mathcal{R}_{s}(A)$ is the largest thin ideal and the largest scattered
ideal of $A$;

\item $\mathcal{R}_{s}(A)=\{a\in A:\sigma(ax)\text{ is countable for each
}x\in A\}$;

\item $\mathcal{R}_{s}(A)$ is the largest ideal of $A$ contained in
$\mathcal{S}\left(  A\right)  $;

\item $\mathcal{R}_{s}(A)+\mathcal{S}\left(  A\right)  \subset\mathcal{S}%
\left(  A\right)  $.
\end{enumerate}
\end{theorem}

\begin{proof}
$\left(  1\right)  $ Let $J$ be a scattered left ideal in $A$. Let
$B=A/\mathcal{R}_{s}(A)$, and let $I$ be the image of $J$ in $B$. Then $B$ is
semisimple and $I$ is a scattered left ideal of $B$. If it is non-zero then it
contains an element $a$ with non-zero spectrum and therefore there is an
isolated non-zero point $\lambda$ in its spectrum. As $\sigma_{B}\left(
a\right)  \subset\sigma_{I}\left(  a\right)  $ and $\lambda\in\sigma_{I}%
^{a}\left(  a\right)  \subset\sigma_{B}^{a}\left(  a\right)  $, $\lambda$ is
an isolated point of $\sigma_{B}\left(  a\right)  $. As $\lambda\neq0$, the
corresponding Riesz projection $p$ of $a$ belongs to $Ba\subset I$ (see for
instance the proof of Lemma 5.7.1 of \cite{A91}). As $p\neq0$ then $Bp$ is a
closed non-zero scattered left ideal of $B$. By Barnes' Theorem \cite{B68},
$\operatorname{psoc}(Bp)\neq0$. By Proposition \ref{onesided},
$\operatorname{psoc}(B)\neq0$. Thus $\mathcal{R}_{s}(B)\neq0$, a
contradiction. Therefore $I=0$, whence $J\subset\mathcal{R}_{s}(A)$.

$\left(  2\right)  $ It follows from $\left(  1\right)  $ that $\mathcal{R}%
_{s}(A)$ is the \textit{largest scattered ideal} of $A$. Let $I$ be a thin
ideal of $A$. Then it is a scattered ideal of $A$ by Corollary \ref{th}, and
so $I\subset\mathcal{R}_{s}(A)$.

$\left(  3\right)  $ Let $a\in A$ and $J=Aa$. If $a\in\mathcal{R}_{s}(A)$ then
$J$ is contained in $\mathcal{R}_{s}(A)$ and its elements have countable
spectra. Conversely, if $J$ is scattered then it is contained in
$\mathcal{R}_{s}(A)$ by Proposition \ref{large}.

$\left(  4\right)  $ If $I$ is an ideal of $A$ contained in $\mathcal{S}%
\left(  A\right)  $ then it is a scattered ideal of $A$. By $\left(  2\right)
$, $I\subset\mathcal{R}_{s}(A)$.

$\left(  5\right)  $ Let $a\in\mathcal{S}\left(  A\right)  $ and
$b\in\mathcal{R}_{s}(A)$. As $\mathcal{R}_{s}(A)$ is a thin ideal of $A$ then
\[
\sigma\left(  a+b\right)  \subset\mathrm{pc}\left(  \sigma\left(
a/\mathcal{R}_{s}(A)\right)  \right)  \cup N
\]
for some countable set $N\subset\mathbb{C}$. But $\sigma\left(  a/\mathcal{R}%
_{s}(A)\right)  $ is countable, whence
\[
\mathrm{pc}\left(  \sigma\left(  a/\mathcal{R}_{s}(A)\right)  \right)
=\sigma\left(  a/\mathcal{R}_{s}(A)\right)
\]
is countable and $a+b\in\mathcal{S}\left(  A\right)  $.
\end{proof}

\begin{corollary}
$\mathcal{R}_{s}$ is a uniform radical on $\mathfrak{U}_{\mathrm{b}}$.
\end{corollary}

\begin{proof}
Indeed, every closed subalgebra of a scattered Banach algebra is scattered.
\end{proof}

\begin{corollary}
\label{ScatBH} The radical $\mathcal{R}_{s}$ satisfy the condition of Banach
heredity $(\ref{BanRad})$.
\end{corollary}

\begin{proof}
Let $L$ be a Banach ideal of a Banach algebra $A$. This means that $L$ is an
ideal of $A$ and there is an injective continuous homomorphism $f$ of a Banach
algebra $B$ to $A$ with $f(B)=L$. Since $f$ is an isomorphism of algebras $B$
and $L$, and $\operatorname{psoc}^{\ast}$ is a hereditary radical on
$\mathfrak{U}_{a}$, we have that
\begin{align*}
\mathcal{R}_{s}((L,\Vert\cdot\Vert_{B}))  &  =f(\mathcal{R}_{s}%
(B))=f(\operatorname{psoc}^{\ast}(B))=\operatorname{psoc}^{\ast}(L)\\
&  =L\cap\operatorname{psoc}^{\ast}(A)=L\cap\mathcal{R}_{s}(A).
\end{align*}

\end{proof}

\begin{corollary}
\label{lar}Let $A$ be a Banach algebra, and let $I$ be a (non-necessary
closed) ideal of $A$. Then

\begin{enumerate}
\item If $I$ is scattered then $I$ is a thin ideal of $A$;

\item If $I$ and $A/I$ are scattered then $A$ is scattered;

\item If $f:A\longrightarrow B$ is an algebraic morphism of Banach algebras
then
\[
f\left(  \mathcal{R}_{s}(A)\right)  \subset\mathcal{R}_{s}(B).
\]

\end{enumerate}
\end{corollary}

\begin{proof}
$\left(  1\right)  $ Indeed, $\overline{I}\subset\mathcal{R}_{s}(A)$ by
Theorem \ref{large}. As $\mathcal{R}_{s}$ is hereditary, $\overline
{I}=\mathcal{R}_{s}(\overline{I})$. Taking $\left(  R_{\alpha}\right)  $ as in
Lemma \ref{scat} we see that all $R_{\alpha}\left(  I\right)  $ are thin
ideals of $A$. Therefore $\mathcal{R}_{s}(\overline{I})$ is a thin ideal of
$A$, hence $\overline{I}$ and $I$ are also thin.

$\left(  2\right)  $ It follows that $\overline{I}$ and $A/\overline{I}$ are
also scattered. Then $\overline{I}=\mathcal{R}_{s}(\overline{I})$ and
$A/\overline{I}=\mathcal{R}_{s}\left(  A/\overline{I}\right)  $. By Theorem
\ref{ext}, $A=\mathcal{R}_{s}\left(  A\right)  $. Therefore $A$ is scattered.

$\left(  3\right)  $ Indeed, $\mathcal{R}_{s}$ inherits the algebraic
properties of $\operatorname{psoc}^{\ast}$.
\end{proof}

\begin{theorem}
\label{univ}If an algebra $A$ is a subideal of a Banach algebra (see Section
$\ref{uni}$) or is algebraically isomorphic to the quotient of a subideal of a
Banach algebra by a non-necessarily closed ideal then

\begin{enumerate}
\item $\operatorname{psoc}^{\ast}\left(  A\right)  =I_{A}$ where
$I_{A}:=\{a\in A:\sigma_{A}(ax)$ is countable for each $x\in A\}$;

\item Every scattered ideal of $A$ is a thin ideal of $A$.
\end{enumerate}
\end{theorem}

\begin{proof}
Note that every subideal of a Banach algebra is a normed $Q$-algebra.

Let $C$ be a $Q$-algebra for which $\left(  1\right)  $ and $\left(  2\right)
$ hold (under the substitution $A=C$; for instance these conditions hold for
Banach algebras), and let $A$ be an ideal of $C$. Then $\operatorname{psoc}%
^{\ast}\left(  A\right)  =A\cap I_{C}$ by heredity of $\operatorname{psoc}%
^{\ast}$. As clearly $I_{A}=A\cap I_{C}$ then $\operatorname{psoc}^{\ast
}\left(  A\right)  =I_{A}$.

Let $K$ be a scattered ideal of $A$, and let $J$ be the ideal of $C$ generated
by $K$. Then
\[
J^{3}\subset K\subset J
\]
by \cite[Lemma 1.1.5]{AR79}, whence $J^{3}$ is a scattered ideal of $C$. Then
$J^{3}$ is a thin ideal of $C$ by $\left(  2\right)  $ and therefore a thin
ideal of $A$. As $K^{3}\subset J^{3}$ then $K^{3}$ and $\overline{K^{3}}$ are
thin ideals of $A$ by Lemma \ref{inc}. As $\overline{K}/\overline{K^{3}}$ is
nilpotent then it is a thin ideal of $A/$ $\overline{K^{3}}$. By Lemma
\ref{stepext}, $\overline{K}$ is a thin ideal of $A$, so $K$ is also a thin
ideal of $A$.

Thus $\left(  1\right)  $ and $\left(  2\right)  $ hold for $A$: therefore the
steps $0\mapsto1$ and $n\mapsto n+1$ of induction for $n$-subideals of Banach
algebras are valid, for every $n$. Then the proof is completed for subideals
of Banach algebras.

Let now $f:A\longrightarrow$ $B/J$ be an algebraic isomorphism of $A$ onto
$B/J$, where $B$ is a subideal of a Banach algebra and $J$ is an ideal of $B$.
If $a\in A$ and $b\in B$ is any element such that $f\left(  a\right)  =b/J$,
then clearly $\sigma\left(  a\right)  =\sigma\left(  b/J\right)  $ and it
suffices to check $\left(  1\right)  $ and $\left(  2\right)  $ for $B/J$. As
$\sigma\left(  b/J\right)  =\sigma\left(  b/\overline{J}\right)  $ for every
$b\in B$ by Lemma \ref{sinq}, the proof is reduced to the case of
$B/\overline{J}$. But $B/\overline{J}$ is a subideal of a Banach algebra (see
for instance \cite[Theorem 2.24]{TR1}), and the result follows from the above.
\end{proof}

Theorem \ref{univ} extends the main properties of the scattered radical from
Banach algebras to subideals of Banach algebras. Recall that subideals of
Banach algebras form the smallest universal class $\mathfrak{U}_{\mathrm{b}%
}^{u}$ generated by Banach algebras. Thus we extend the denotation
$\mathcal{R}_{s}$ for $\operatorname{psoc}^{\ast}$ on subideals of Banach algebras.

The \textit{regular scattered radical} $\mathcal{R}_{s}^{r}$, obtained from
$\mathcal{R}_{s}$ by the regular procedure, extends $\mathcal{R}_{s}$ to
normed algebras, and it is determined as
\[
\mathcal{R}_{s}^{r}\left(  A\right)  =\left\{  x\in A:\sigma_{\widehat{A}%
}\left(  ax\right)  \text{ is at most countable }\forall\text{ }%
a\in\widehat{A}\right\}  .
\]

\begin{theorem}
\label{hc-scat} $\mathcal{R}_{\mathrm{hc}}\leq\mathcal{R}_{s}=\mathcal{R}%
_{\mathrm{hf}}\vee\operatorname{Rad}=\mathcal{R}_{\mathrm{hc}}\vee
\operatorname{Rad}$ on Banach algebras.
\end{theorem}

\begin{proof}
Let $A$ be a Banach algebra. Any compact element of a Banach algebra $A$ has
countable spectrum \cite{Al68} and any bicompact Banach algebra is scattered.
Then any closed bicompact ideal of $A$ is scattered and contained in
$\mathcal{R}_{s}(A)$, whence $\Sigma_{\mathrm{hc}}\left(  A\right)
\subset\mathcal{R}_{s}\left(  A\right)  $ and $\Sigma_{\mathrm{hc}}%
\leq\mathcal{R}_{s}$ in general. As $\mathcal{R}_{s}$ is a radical, it follows
that $\mathcal{R}_{\mathrm{hf}}\leq\mathcal{R}_{\mathrm{hc}}=\Sigma
_{\mathrm{hc}}^{\ast}\leq\mathcal{R}_{s}$ by Theorems \ref{ovunt} and \ref{hc}.

It is clear that $\operatorname{Rad}\leq\mathcal{R}_{s}$, so $\mathcal{R}%
_{\mathrm{hf}}\vee\operatorname{Rad}\leq\mathcal{R}_{s}$.

Let $A$ be $\left(  \mathcal{R}_{\mathrm{hf}}\vee\operatorname{Rad}\right)
$-semisimple, and let $I=\mathcal{R}_{s}\left(  A\right)  $. As $A$ is
semisimple and has no non-zero finite rank elements, $I$ is also semisimple
and has no non-zero finite rank elements. However, if $I\neq0$ then $I$ has a
non-zero socle by Barnes' theorem \cite{B68}. Thus, by Lemma \ref{soc}, $I$
has non-zero finite rank elements, and so $A$ has such elements, a
contradiction. Therefore $I=0$, and $A$ is $\mathcal{R}_{s}$-semisimple. By
Theorem \ref{equality}, $\mathcal{R}_{s}\leq\mathcal{R}_{\mathrm{hf}}%
\vee\operatorname{Rad}$.

Furthermore,
\[
\mathcal{R}_{s}=\mathcal{R}_{\mathrm{hc}}\vee\mathcal{R}_{s}=\mathcal{R}%
_{\mathrm{hc}}\vee\mathcal{R}_{\mathrm{hf}}\vee\operatorname{Rad}%
=\mathcal{R}_{\mathrm{hc}}\vee\operatorname{Rad}.
\]

\end{proof}

Let $A$ be an algebra, and let $\widehat{\operatorname{Irr}}\left(  A\right)
$ be the set of classes of equivalent strictly irreducible representations of
$A$. Then $\pi\longmapsto\ker\pi$ is a map from $\widehat{\operatorname{Irr}%
}\left(  A\right)  $ onto $\operatorname{Prim}(A)$, so
$\widehat{\operatorname{Irr}}\left(  A\right)  $ inherits the Jacobson
topology from $\operatorname{Prim}(A)$: the preimages of closed sets in
$\operatorname{Prim}(A)$ determine closed sets in $\widehat{\operatorname{Irr}%
}\left(  A\right)  $. The following statement says that one can identify these
topological spaces for scattered Banach algebras.

\begin{theorem}
\label{over}Let $A$ be a Banach algebra. If $A$ is scattered then for any
primitive ideal $I$, there is only one, up to the equivalence, strictly
irreducible representation with kernel $I$.
\end{theorem}

\begin{proof}
Any representation with kernel $I$ defines a faithful representation of $A/I$.
Since $A/I$ is a primitive Banach algebra, we have only to prove that a
faithful representation of a primitive scattered algebra $B$ is unique (up to
the equivalence).

Since $B$ is scattered and semisimple, the socle of $B$ is non-zero. So it
contains a minimal projection $p$. Taking into account that $\dim pBp=1$ and
applying Lemma \ref{reprfin}, we have that there is only one strictly
irreducible representation $\pi$ with $\pi(p)\neq0$. But the last condition
holds for each faithful representation.
\end{proof}

A topological space is called \textit{dispersed} if it does not contain
perfect subspaces, i.e., closed subsets without isolated points.

\begin{theorem}
\label{sparse} If a Banach algebra $A$ is scattered then the space
$\operatorname{Prim}(A)$ of its primitive ideals is \textit{dispersed}.
\end{theorem}

\begin{proof}
Let $E$ be a closed subset in $\operatorname{Prim}(A)$, $J=\cap_{I\in E}I$ and
$B=A/J$. All primitive ideals of $B$ are of the form $I/J$, $I\in E$ (because
$E$ is closed) and their intersection is trivial. Thus $B$ is a semisimple
scattered algebra, whence it contains a minimal projection $p$. Since $\dim
pBp=1$, there is, by Lemma \ref{reprfin}, only one primitive ideal
$I_{0}^{\prime}=I_{0}/J$ of $B$ that does not contain $p$. It follows that
$I_{0}$ does not contain the intersection of all $I\neq I_{0}$ in $E$. Hence
$I_{0}$ is an isolated point in $E$ and $E$ is not perfect.
\end{proof}

\begin{corollary}
If $A$ is a separable scattered Banach algebra then the spaces
$\operatorname{Prim}(A)$ and $\widehat{\operatorname{Irr}}\left(  A\right)  $
are countable.
\end{corollary}

\begin{proof}
It follows from Theorem \ref{over} that if $\text{Prim}(A)$ is countable then
$\widehat{\operatorname{Irr}}\left(  A\right)  $ is countable.

In any topological space $X$ there is a decreasing transfinite chain of sets:
let $X_{0}=X$, $X_{\alpha+1}$ be the set of all non-isolated points in
$X_{\alpha}$ for every ordinal $\alpha$, and let $X_{\alpha}=\cap
_{\alpha^{\prime}<\alpha}X_{\alpha^{\prime}}$ if $\alpha$ is a limit ordinal.
If $X$ is dispersed then $X_{\delta}=\varnothing$ for some $\delta$.

Let $X=\operatorname{Prim}(A)$. As $X$ is dispersed by Theorem \ref{sparse},
let $\left(  X_{\alpha}\right)  _{\alpha\leq\delta}$ be the chain of the above
sets of $X$, with $X_{\delta}=\varnothing$. Choose $I_{\alpha}\in X_{\alpha
}\backslash X_{\alpha+1}$ for every ordinal $\alpha<\delta$. By the
construction, $I_{\alpha}\notin\overline{\left\{  I_{\alpha^{\prime}}:{\alpha
}^{\prime}>{\alpha}\right\}  }$ in the Jacobson topology. Hence there are
$x_{\alpha}\in(\cap_{{\alpha}^{\prime}>{\alpha}}I_{\alpha^{\prime}})\setminus
I_{\alpha}$. Multiplying by a constant, we can have%
\[
\operatorname{dist}(x_{\alpha},I_{\alpha})>1.
\]
Let now $\alpha^{\prime}>\alpha$. Then $x_{\alpha}\in I_{\alpha^{\prime}}$
and
\[
\Vert x_{\alpha^{\prime}}-x_{\alpha}\Vert\geq\operatorname{dist}%
(x_{\alpha^{\prime}},I_{\alpha^{\prime}})>1.
\]
So the last inequality holds for all ordinals $\alpha,\alpha^{\prime}$ with
$\alpha\neq\alpha^{\prime}$. Since $A$ is separable, we obtain that $\delta$
is a countable ordinal.

Now it suffices to show that each set $X_{\alpha}\setminus X_{\alpha+1}$ is
countable. This can be done by the same trick because each ideal in
$X_{\alpha}\setminus X_{\alpha+1}$ is not contained in the closure of the set
of the others.
\end{proof}

As an example let us look at the algebras $C(X)$ where $X$ is a compact set.

\begin{corollary}
\label{scatcx} An algebra $C(X)$ is scattered if and only if $X$ is dispersed.
\end{corollary}

\begin{proof}
It follows from Theorem \ref{sparse} that if $C(X)$ is scattered then $X$ is
dispersed. Conversely, let $X$ be dispersed. Any quotient of $C(X)$ by a
closed ideal is isomorphic to $C(Y)$, where $Y$ is a compact subset of $X$.
Since $Y$ is not perfect, there is an isolated point $z\in Y$. The ideal of
all functions in $C(Y)$ that vanish outside $z$ is one-dimensional, hence it
is thin. This allows us to construct a transfinite chain $\left(  I_{\alpha
}\right)  _{\alpha\leq\gamma}$ of closed ideals of $C(X)$, such that
$I_{0}=0,$ $I_{\gamma}=C(X)$ and all gap-quotients of the chain are thin
ideals. By Theorem \ref{spec}, $C(X)$ is scattered.
\end{proof}

\begin{remark}
The above corollary gives a proof of the fact that the image of a dispersed
space under a continuous map is dispersed. Indeed, if $Y=f(X)$ then $C(Y)$ is
isomorphic to a closed subalgebra $B$ of $C(X)$. If $X$ is dispersed then
$C(X)$ is scattered. Then $B$ is scattered, so $C(Y)$ is scattered. Hence $Y$
is dispersed.
\end{remark}

\subsection{Scattered radical on hereditarily semisimple Banach algebras}

By Theorem \ref{hc-scat},\textbf{ }all hypocompact Banach algebras are
scattered. The converse is not true in general because all Jacobson radical
algebras are scattered. We will show here that if the radicals of an algebra
and all its quotients are trivial that these conditions are equivalent.

Let $A$ be a [normed] algebra. Let us call $A$ \textit{hereditarily
semisimple} if all its quotients by [closed] ideals are semisimple. It is
obvious that the class of all hereditarily semisimple [normed] algebras has
the following properties:

\begin{enumerate}
\item All quotients by [closed] ideals are in this class;

\item If a [closed] ideal $I$ of a [normed] algebra $A$ and the quotient $A/I$
are in this class then so is $A$.
\end{enumerate}

\begin{proposition}
\label{perPrim} A Banach algebra $A$ is hereditarily semisimple if and only if
each closed ideal of $A$ is the intersection of a family of primitive ideals.
\end{proposition}

\begin{proof}
Let $J$ be a closed ideal of $A$ and let $K=\mathrm{kh}(J)$. Then
$K/J=\operatorname{Rad}(A/J)$. Indeed, if $I\in\operatorname{Prim}(A/J)$ then
there is $I^{\prime}\in\operatorname{Prim}(A)$ with $J\subset I^{\prime}$ and
$I=I^{\prime}/J$. Since $K\subset I^{\prime}$, we have that $K/J\subset I$.
Thus
\[
K/J\subset\cap_{I\in\operatorname{Prim}(A/J)^{1}}I=\operatorname{Rad}(A/J).
\]
Conversely, if $K^{\prime}=q_{J}^{-1}\left(  \operatorname{Rad}(A/J)\right)  $
then $K^{\prime}\subset I^{\prime}$ for any primitive ideal $I^{\prime}\supset
I$. Hence $K^{\prime}\subset K$ and $\operatorname{Rad}(A/J)\subset
q_{J}(K)=K/J$.

Now if $A$ is hereditarily semisimple then $K/J=0$, whence $J=K$, the
intersection of primitive ideals. Conversely, if $\mathrm{kh}(J)=J$ for all
$J$, then $\operatorname{Rad}(A/J)=0$, whence $A$ is hereditarily semisimple.
\end{proof}

In commutative case the condition that each closed ideal of a Banach algebra
$A$ is the intersection of primitive ideals, is sometimes formulated as $A$
possesses the spectral synthesis. Among group algebras $L^{1}(G)$, only the
algebras of compact groups have this property (Malliavin's Theorem, see
\cite{Rud}).

\begin{theorem}
\label{simher}Let $A$ be a hereditarily semisimple Banach algebra. Then the
following conditions are equivalent:

\begin{enumerate}
\item $A$ is scattered;

\item $A$ is hypocompact;

\item $A$ is a closed-hypofinite algebra;

\item Every non-zero quotient of $A$ has a minimal left ideal.
\end{enumerate}
\end{theorem}

\begin{proof}
$(1)\Longrightarrow(4)$ Let $A$ be scattered, and let $J$ be a closed ideal of
$A$. Then $A/J$ is also scattered. So it suffices to show that a scattered
semisimple Banach algebra $A$ has non-zero socle. But if $\operatorname{soc}%
(A)=0$ then $\operatorname{psoc}(A)=0$, whence $\operatorname{psoc}^{\ast
}(A)=0$, i.e., $\mathcal{R}_{s}(A)=0$ while $\mathcal{R}_{s}(A)=A$, a contradiction.

$(4)\Longrightarrow(3)$ follows from the fact that each minimal projection is
a finite rank element.

$(3)\Longrightarrow(2)$ is evident.

$(2)\Longrightarrow(1)$ follows from Theorem \ref{hc-scat}.
\end{proof}

An example of a hereditarily semisimple algebra is the algebra $C(X)$ of all
continuous functions on a compact space $X$. We saw that it is scattered if
and only if $X$ does not contain perfect subsets. Hence it is hypocompact
under the same condition.

Since all C*-algebras are semisimple and their quotients are again
C*-algebras, we have that C*-algebras are hereditarily semisimple.

Note that in the theory of C*-algebras the term \textquotedblleft
scattered\textquotedblright\ is used for a class of C*-algebras $A$ satisfying
the following equivalent conditions:

\begin{enumerate}
\item[(1$_{s}$)] Each positive functional on $A$ is the sum of a sequence of
pure functionals;

\item[(2$_{s}$)] $A$ is a type I C*-algebra (= GCR C*-algebra) and
$\operatorname{Prim}(A)$ is dispersed in the hull-kernel topology;

\item[(3$_{s}$)] $A$ is of type I and the maximal ideal space of its center is dispersed;

\item[(4$_{s}$)] $A$ admits a superposition series $\left(  I_{\gamma}\right)
$ of closed ideals and each gap-quotient $I_{\gamma+1}/I_{\gamma}$ is
isomorphic to $\mathcal{K}(H_{\gamma})$ for some Hilbert space $H_{\gamma}$;

\item[(5$_{s}$)] Each self-adjoint element of $A$ has countable spectrum;

\item[(6$_{s}$)] Each C*-subalgebra of $A$ is AF (approximately finite-dimensional);

\item[(7$_{s}$)] The dual space $A^{\ast}$ of $A$ has the Radon-Nikodym property.
\end{enumerate}

For the proof of the equivalence see \cite{Kus} and the references therein. We
will show that this class of algebras are exactly the scattered C*-algebras.

\begin{theorem}
\label{eq-scat} A C*-algebra $A$ satisfies the equivalent conditions $(1_{s}%
)$-$(7_{s})$ if and only if it is scattered.
\end{theorem}

\begin{proof}
Clearly each scattered C*-algebra satisfies $\left(  6_{s}\right)  $ and
$\left(  5_{s}\right)  $. On the other hand, if $A$ satisfies $\left(
4_{s}\right)  $ then $A$ is hypocompact, so it is scattered by Theorem
\ref{hc-scat}.
\end{proof}

Thus taking into account that C*-algebras are hereditary semisimple and using
Theorem \ref{simher} one can add to the above list of the equivalent
properties of a C*-algebra the following ones:

\begin{enumerate}
\item[(8$_{s}$)] All elements of $A$ have countable spectra;

\item[(9$_{s}$)] Each non-zero quotient of $A$ has a non-zero compact element;

\item[(10$_{s}$)] Each non-zero quotient of $A$ has a non-zero finite element;

\item[(11$_{s}$)] Every non-zero quotient of $A$ has a minimal projection.
\end{enumerate}

It was shown in \cite{AkWr} that an element $a$ of a C*-algebra $A$ is compact
if and only if the map $x\longmapsto axa$ is weakly compact. So one can add
also the condition:

\begin{enumerate}
\item[(12$_{s}$)] Each non-zero quotient of $A$ has a non-zero weakly compact element.
\end{enumerate}

\section{\label{continuity}Spectral continuity and radicals}

Spectral continuity, that is the continuity of such functions of an operator
(or an element of Banach algebra) as spectrum, special parts of spectrum,
spectral radius and so on, is a very convenient property when it holds. For
the information on this subject see for example \cite{Bur} and the references
therein. We consider it from the viewpoint of theory of topological radicals.
An important role here is played by the scattered radical $\mathcal{R}_{s}$,
its primitivity extension $\mathcal{R}_{s}^{p}$ and the radical $\mathcal{R}%
_{s}^{p\ast}$ obtained from $\mathcal{R}_{s}^{p}$ by the convolution
procedure. So we begin with a study of these ideal maps.

\subsection{$\mathcal{R}_{s}^{a}$, $\mathcal{R}_{s}^{p}$, $\mathcal{R}%
_{s}^{p\ast}$ and continuity of the spectrum}

\begin{proposition}
$\mathcal{R}_{s}^{a}$ is a hereditary topological radical.
\end{proposition}

\begin{proof}
Since the radical $\mathcal{R}_{s}$ satisfies the condition of Banach heredity
(Corollary \ref{ScatBH}), its centralization $\mathcal{R}_{s}^{a}$ is a
hereditary topological radical on $\mathfrak{U}_{b}$, by Theorem \ref{cent}
and Remark \ref{centr}.
\end{proof}

The Banach heredity implies, by Theorem \ref{pri}$\left(  3\right)  $, that
the primitivity extension $\mathcal{R}_{s}^{p}$ of $\mathcal{R}_{s}$ is a
hereditary preradical on $\mathfrak{U}_{\mathrm{b}}$. Recall that a Banach
algebra $A$ is $\mathcal{R}_{s}^{p}$-radical if $A/I$ is $\mathcal{R}_{s}%
$-radical (= scattered), for each primitive ideal $I$ of $A$.

\begin{theorem}
\label{dif-C} The classes $\mathbf{Rad}(\mathcal{R}_{s}^{p})\setminus
\mathbf{Rad}(\mathcal{R}_{s}^{a})$ and $\mathbf{Rad}(\mathcal{R}_{s}%
^{a})\setminus\mathbf{Rad}(\mathcal{R}_{s}^{p})$ are non-empty.
\end{theorem}

\begin{proof}
Let $A=C([0,1],\mathcal{K}(H))$ be the C*-algebra of all norm-continuous
$\mathcal{K}(H)$-valued functions on $[0,1]$. It is well known that
$\operatorname{Prim}(A)$ is isomorphic to $[0,1]$ via the map $t\mapsto I_{t}$
where $I_{t}$ is the ideal of all functions equal $0$ at $t$. Since
$A/I_{t}\cong\mathcal{K}(H)$ then $A$ is $\mathcal{R}_{s}^{p}$-radical. On the
other hand, $A$ has no ideals $J$ with commutative $A/J$, so it is not
$\mathcal{R}_{s}^{a}$-radical.

To show the non-voidness of $\mathbf{Rad}(\mathcal{R}_{s}^{a})\setminus
\mathbf{Rad}(\mathcal{R}_{s}^{p})$, let us consider the Toeplitz algebra
$A_{t}$ that is the C*-algebra generated by the unilateral shift $V$ in the
space $H=l_{2}(\mathbb{N})$, acting by the rule $Ve_{n}=e_{n+1}$, where
$(e_{n})_{n=1}^{\infty}$ is the standard basis in $l_{2}(\mathbb{N})$. It is
known \cite[Theorem 5.1.5]{Dav} that $A_{t}$ contains the ideal $\mathcal{K}%
(H)$ and that $A_{t}/\mathcal{K}(H)$ is isomorphic to the algebra
$C(\mathbb{T})$ of all continuous functions on the unit circle. Since
$\mathcal{K}(H)\subset\mathcal{R}_{s}(A_{t})$, we get that $A_{t}%
\in\mathbf{Rad}(\mathcal{R}_{s}^{a})$. On the other hand, the identity
representation of $A_{t}$ is irreducible, while $A_{t}$ is not scattered,
because $\sigma(V)=\mathbb{D}$, the unit disk. Therefore $A_{t}$ is not
$\mathcal{R}_{s}^{p}$-radical.
\end{proof}

Now let us consider the radical $\mathcal{R}_{s}^{p{\ast}}$ which is obtained
from $\mathcal{R}_{s}^{p}$ by the convolution procedure. We will see later
that this radical plays a very important role in continuity of the spectral radius.

\begin{lemma}
\label{upperstar} $\mathcal{R}_{s}^{a}<\mathcal{R}_{s}^{p{\ast}}$ on
$\mathfrak{U}_{\mathrm{b}}$.
\end{lemma}

\begin{proof}
All commutative Banach algebras belong to $\mathcal{R}_{s}^{p}$ because their
strictly irreducible representations are one-dimensional. The quotient
$\mathcal{R}_{s}^{a}(A)/\mathcal{R}_{s}(A)$ is a commutative algebra, for
every Banach algebra $A$. Thus we have that $\mathcal{R}_{s}^{a}(A)$ is an
extension of an $\mathcal{R}_{s}^{p{\ast}}$-radical algebra by an
$\mathcal{R}_{s}^{p{\ast}}$-radical algebra. Since $\mathcal{R}_{s}^{p{\ast}}$
is a radical, $\mathcal{R}_{s}^{a}(A)$ is $\mathcal{R}_{s}^{p{\ast}}$-radical.
By Theorems \ref{equality} and \ref{dif-C}, $\mathcal{R}_{s}^{a}%
<\mathcal{R}_{s}^{p{\ast}}$.
\end{proof}

Since $A_{t}$ is in $\mathbf{Rad}(R_{s}^{a}) \setminus\mathbf{Rad}(R_{s}^{p}%
)$, we conclude that $R_{s}^{p} \neq R_{s}^{p\ast}$. In other words, we obtain
the following:

\begin{corollary}
The under radical $\mathcal{R}_{s}^{p}$ is not a radical.
\end{corollary}

Recall some facts about continuity of the spectrum.

Let $A$ be a Banach algebra. It is well known that $\sigma$ is upper
continuous, i.e. if $a_{n}\rightarrow a$ then for every open neighborhood $V$
of $\sigma\left(  a\right)  $ there is $m>0$ such that $\sigma\left(
a_{n}\right)  \subset V$ for every $n>m$. Moreover, Newburgh's theorem
\cite[Lemma 6.15]{S76} says that for every clopen subset $\sigma_{0}$ of
$\sigma\left(  a\right)  $ and every open neighborhood $V_{0}$ of $\sigma_{0}%
$, there is $m_{0}$ such that $\sigma_{n}\subset V_{0}$ for $n>m_{0}$ and some
clopen subsets $\sigma_{n}$ of $\sigma\left(  a_{n}\right)  $.

In particular, if $\sigma\left(  a\right)  $ is at most countable then the
spectrum $\sigma$ is continuous at $a$, i.e. $\sigma\left(  a_{n}\right)
\rightarrow\sigma\left(  a\right)  $ by the Hausdorff metric. For continuity
of $\sigma$ at $a$ it is sufficient to have that for every $\lambda\in
\sigma\left(  a\right)  $ there is a sequence $\lambda_{n}\in\sigma\left(
a_{n}\right)  $ such that $\lambda_{n}\rightarrow\lambda$ as $n\rightarrow
\infty$.

We need the following result of Zemanek \cite[Remark 1]{Z80}.

\begin{lemma}
\label{zem} Let $A$ be a Banach algebra, and let $a,a_{n}\in A$,
$a_{n}\rightarrow a$ as $n\rightarrow\infty$. If $\sigma_{A/I}\left(
a_{n}/I\right)  \rightarrow\sigma_{A/I}\left(  a/I\right)  $, for every
$I\in\operatorname{Prim}\left(  A\right)  $, then $\sigma_{A}\left(
a_{n}\right)  \rightarrow\sigma_{A}\left(  a\right)  $.
\end{lemma}

It is convenient to use a more general result of this type.

\begin{lemma}
\label{doubleZem} Let $\{f_{\alpha}:\alpha\in\Lambda\}$ be a set of
homomorphisms of a Banach algebra $A$ to some $Q$-algebras $A_{\alpha}$. Let
$a\in A$ be such that
\begin{equation}
\sigma(a)=\overline{\cup_{\alpha\in\Lambda}\sigma(f_{\alpha}(a))}.
\label{unionS}%
\end{equation}
If $a_{n}\rightarrow a$ and $\sigma(f_{\alpha}(a_{n}))\rightarrow
\sigma(f_{\alpha}(a))$, for each $\alpha\in\Lambda$, then $\sigma
(a_{n})\rightarrow\sigma(a)$.
\end{lemma}

\begin{proof}
Let $G=\liminf_{n\rightarrow\infty}\sigma(a_{n}):=\{\lambda\in\mathbb{C}%
:\lambda=\lim_{n\rightarrow\infty}\lambda_{n}\text{ for some }\lambda_{n}%
\in\sigma(a_{n})\}$. By the upper continuity, it suffices to show that
$\sigma(a)\subset G$. But $G$ is closed and contains all $\sigma(f_{\alpha
}(a))$. So the result follows from (\ref{unionS}).
\end{proof}

Note that apart from the case $f_{\alpha}=q_{I_{\alpha}}$, $I_{\alpha}%
\in\operatorname{Prim}\left(  {A}\right)  $, the condition (\ref{unionS})
holds in many other situations, for example when $\Lambda$ is finite and
$\cap_{\alpha\in\Lambda}I_{\alpha}\subset\operatorname{Rad}(A)$ (see Theorem
\ref{spec2}).

\begin{theorem}
\label{takCont} Let $A$ be a Banach algebra. Then each element $a\in
\mathcal{R}_{s}^{p}\left(  A\right)  $ is a point of the spectrum continuity.
\end{theorem}

\begin{proof}
Let $a\in\mathcal{R}_{s}^{p}\left(  A\right)  $, $a_{n}\in A$, $a_{n}%
\rightarrow a$; by Lemma \ref{zem}, it suffices to show that $\sigma
(a_{n}/I)\rightarrow\sigma(a/I)$, for each $I\in\operatorname{Prim}(A)$. Since
$\sigma(a/I)$ is countable of finite, the fact follows from the Newburgh theorem.
\end{proof}

We show now that Theorem \ref{takCont} does not transfer to $\mathcal{R}%
_{s}^{a}$ and therefore to $\mathcal{R}_{s}^{p\ast}$.

\begin{theorem}
\label{discont} The spectrum in an $\mathcal{R}_{s}^{a}$-radical Banach
algebra can be discontinuous.
\end{theorem}

\begin{proof}
We use the example of spectral discontinuity proposed by G. Lumer (see
\cite[Problem 86]{Halm}). Let $H=l^{2}(\mathbb{Z})$, and let $\{e_{n}%
:n\in\mathbb{Z}\}$ be the natural orthonormal basis in $H$. Let us denote by
$W$ the shift operator $W$ on $H$ ($We_{n}=e_{n+1}$). For each $k\in
\mathbb{N}$, let $W_{k}$ be the operator acting on the basis by the formulas
$W_{k}e_{n} = e_{n+1}$, for $n\neq0$, $W_{k}e_{0} = \frac{1}{k}e_{1}$. Then
$W_{k}\to W_{\infty}$ where $W_{\infty}e_{n}= e_{n+1}$ for $n\neq0$,
$W_{\infty}e_{0} = 0$. It is easy to check that $\sigma(W_{k}) = \mathbb{T}$
while $\sigma(W_{\infty}) = \mathbb{D}$, the unit disk.

Since $W_{k}-W$ is a rank one operator, for each $k$, all operators $W_{k}$,
$k=1,2,...,\infty$, belong to the C*-algebra $A$ generated by $W$ and the
ideal $\mathcal{K}(H)$ of all compact operators. Then $A/\mathcal{K}(H)$ is
commutative (it is generated by the image of the normal operator $W$ in the
Calkin algebra $\mathcal{B}(H)/\mathcal{K}(H)$). Therefore $A\in
\mathbf{\mathbf{Rad}}\left(  \mathcal{R}_{s}^{a}\right)  $ since it is
commutative modulo the scattered ideal $\mathcal{K}(H)$.
\end{proof}

\begin{corollary}
The spectrum continuity is not a radical/semisimple property, even in Banach algebras.
\end{corollary}

\begin{proof}
Indeed, assume, to the contrary, that there is a radical $P$ on Banach
algebras such that $\mathbf{Rad}\left(  P\right)  $ or $\mathbf{Sem}\left(
P\right)  $ coincides with the class of algebras with the property of the
spectrum continuity. Then this class contains commutative algebras and
scattered algebras. As $\mathbf{Rad}\left(  P\right)  $ or $\mathbf{Sem}%
\left(  P\right)  $ is stable under extensions, then it contains the algebra
constructed in Theorem \ref{discont}, a contradiction.
\end{proof}

In the same time, $\mathbf{Rad}\left(  \operatorname{rad}^{\dim=\infty
}|_{\mathfrak{U}_{\mathrm{q}}}\right)  $ is contained in the class of algebras
with the property of the spectrum continuity.

\begin{problem}
Is there the largest radical $P$ such that $\mathbf{Rad}\left(  P\right)  $ is
contained in the class of Banach algebras with the property of the spectrum continuity?
\end{problem}

At the moment, the largest radical $P$ about which we know that $P$-radical
algebras enjoy the property of spectral continuity, is $\operatorname{rad}%
^{\operatorname{dim}=\infty}$ (which is clearly majorized by $\mathcal{R}%
_{s}^{p}$).

We finish by a discussion of the restrictions of the ideal maps $\mathcal{R}%
_{s}^{p}$, $\mathcal{R}_{s}^{a}$ and $\mathcal{R}_{s}^{p\ast}$ to the class
$\mathfrak{U}_{\mathrm{c}^{\ast}}$.

It was already mentioned that an irreducible *-representation of a C*-algebra
is strictly irreducible and that for each primitive ideal $I$ of a C*-algebra
$A$, the quotient $A/I$ is isometrically isomorphic to the image of $A$ in an
irreducible *-representation. These facts will be multiply used below.
Moreover, dealing with C*-algebra we usually write representation meaning *-representation.

\begin{theorem}
\label{CCR-R} Let $A$ be a C*-algebra. Then $\mathcal{R}_{s}^{p{\ast}%
}(A)=\mathcal{R}_{\mathfrak{gcr}}(A)$.
\end{theorem}

\begin{proof}
Since $\mathcal{R}_{s}^{p{\ast}}$ and $\mathcal{R}_{\mathfrak{gcr}}$ are
radicals, it suffices to show (taking into account Corollary \ref{ineq}) that
$\mathbf{Rad}(\mathcal{R}_{\mathfrak{gcr}})=\mathbf{Rad}(\mathcal{R}%
_{s}^{p{\ast}})\cap\mathfrak{U}_{\mathrm{c}\ast}$.

Let $A$ be a CCR-algebra; then its image in each irreducible representation
$\pi$ coincides with $\mathcal{K}(H_{\pi})$, for some Hilbert space $H_{\pi}$;
so $\pi\left(  A\right)  $ is scattered and therefore $\mathcal{R}_{s}%
$-radical. Hence $A$ is $\mathcal{R}_{s}^{p{\ast}}$-radical. Since each
GCR-algebra $B$ admits an increasing transfinite chain $\left(  I_{\alpha
}\right)  _{\alpha\leq\gamma}$ with CCR gap-quotients and $I_{0}=0$,
$I_{\gamma}=B$, then $\mathbf{Rad}(\mathcal{R}_{s}^{p{\ast}})$ contains all
GCR C*-algebras.

Conversely, each $\mathcal{R}_{s}$-radical C*-algebra is a GCR-algebra by
Theorem \ref{eq-scat}. If a C*-algebra $A$ is $\mathcal{R}_{s}^{p}$-radical,
then, by definition, for each irreducible representation $\pi$, $\pi(A)$ is
$\mathcal{R}_{s}$-radical, so it is a GCR-algebra. It follows that $\pi(A)$
contains a non-zero compact operator. Thus $A$ is a GCR-algebra.

Since each $\mathcal{R}_{s}^{p{\ast}}$-radical C*-algebra $C$ admits an
increasing transfinite chain $\left(  J\right)  _{\alpha\leq\delta}$ with
$\mathcal{R}_{s}^{p}$-radical gap-quotients and $J_{0}=0$, $J_{\delta}=C$,
then $\mathbf{Rad}(R_{s}^{p{\ast}})\cap\mathfrak{U}_{\mathrm{c}^{\ast}}$
consists of GCR C*-algebras.
\end{proof}

It follows from Theorem \ref{CCR-R} that $\mathcal{R}_{s}^{p{\ast}}$ can be
considered as a natural extension of the GCR-radical from the class
$\mathfrak{U}_{\mathrm{c}^{\ast}}$ to all Banach algebras.

We saw in the proof of Theorem \ref{CCR-R} that the class of all
$\mathcal{R}_{s}^{p}$-radical C*-algebras contains the class of all
CCR-algebras. The inclusion is strict: for example it is not difficult to
construct a C*-algebra $A\subset\mathcal{B}(H)$ that contains $\mathcal{K}(H)$
and such that $A/\mathcal{K}(H)\cong\mathcal{K}(H)$ --- it is scattered but
not CCR.

The examples in Theorem \ref{dif-C} show that CCR-algebras need not be
$\mathcal{R}_{s}^{a}$-radical, and that $\mathcal{R}_{s}^{a}$-radical
C*-algebras are not necessarily CCR. The full description of $\mathcal{R}%
_{s}^{a}$-radical C*-algebras is very non-trivial: the famous work \cite{BDF}
shows that even the classification of commutative extensions of the algebra
$\mathcal{K}(H)$ is related to deep homological constructions.

Concluding this subsection let us show that spectrum is continuous {\it at normal
elements} in a wide class of C*-algebras that are far from being GCR --- for
example, in C*-algebras of free groups.

A C*-algebra $A$ is called \textit{residually finite-dimensional} (RFD, for
short)\textit{ } if there is a family $\{\pi_{\alpha}:\alpha\in\Lambda\}$ of
finite-dimensional representations of $A$ with $\cap_{\alpha\in\Lambda}\ker
\pi_{\alpha}=0$. Such algebras are $\operatorname{rad}^{\operatorname{dim}%
<\infty}$-semisimple.

\begin{theorem}
\label{RFD} Let $A$ be an RFD C*-algebra. Then every normal element $a$ of $A$
is a point of continuity of the spectrum $\sigma_{A}$.
\end{theorem}

\begin{proof}
Let $\left(  \pi_{\alpha}\right)  _{\alpha\in\Lambda}$ be a family of
finite-dimensional representations of $A$ with $\cap_{\alpha\in\Lambda}\ker
\pi_{\alpha}=0$. Let $(a_{n})$ be a sequence of elements of $A$ tending to
$a$. Since $\sigma(\pi_{\alpha}(a_{n}))\rightarrow\sigma(\pi_{\alpha}(a))$ for
each $\alpha$, it is sufficient to show that
\[
\sigma(a)\subset\overline{\cup_{\alpha\in\Lambda}\sigma(\pi_{\alpha}(a))}%
\]
and apply Proposition \ref{doubleZem}.

Let $\lambda_{0}\in\sigma(a)\setminus\overline{\cup_{\alpha\in\Lambda}%
\sigma(\pi_{\alpha}(a))}$, then there is a disk $D_{\varepsilon}(\lambda
_{0})=\{\lambda:|\lambda-\lambda_{0}|<\varepsilon\}$ which does not intersect
$\cup_{\alpha\in\Lambda}\sigma(\pi_{\alpha}(a))$. Therefore for each $\alpha$
there is an element $b_{\alpha}\in\pi_{\alpha}(A)$ which is inverse to
$\pi_{\alpha}\left(  a-\lambda_{0}\right)  $ and
\[
\left\Vert b_{\alpha}\right\Vert =\left\Vert \left(  \pi_{\alpha}\left(
a-\lambda_{0}\right)  \right)  ^{-1}\right\Vert =\rho\left(  \pi_{\alpha
}\left(  a-\lambda_{0}\right)  \right)  ^{-1}<\varepsilon^{-1}%
\]
since $a$ is normal.

Let $H=\oplus_{\alpha\in\Lambda}H_{\alpha}$ and $\pi=\oplus_{\alpha\in\Lambda
}\pi_{\alpha}$, the direct sum of representations $\pi_{\alpha}$. Then
$\ker(\pi)=0$ and $B=\pi(A)$ is a C*-algebra isomorphic to $A$. The element
$b=\oplus_{\alpha\in\Lambda}b_{\alpha}$ belongs to $\mathcal{B}(H)$, because
$\sup_{\alpha\in\Lambda}\Vert b_{\alpha}\Vert<\infty$. It follows easily from
the definition of $b_{\alpha}$ that
\[
b\pi(a-\lambda_{0})=\pi(a-\lambda_{0})b=1.
\]
Therefore $\lambda_{0}\notin\sigma_{\mathcal{B}(H)}(\pi(a))$; since $\pi(A)$
is a C*-algebra then it is an inverse-closed subalgebra of $\mathcal{B}\left(
H\right)  $ and
\[
\lambda_{0}\notin\sigma_{\pi(A)}(\pi\left(  a\right)  )=\sigma(a),
\]
a contradiction.
\end{proof}

\subsection{Continuity of the spectral radius}

It turns out that the situation with the continuity of the spectral radii is different.

Let $A$ be a Banach algebra. Recall that the function $\rho:a\longmapsto$
$\rho\left(  a\right)  $ is upper continuous because it is the infimum of
continuous functions $a\longmapsto\left\Vert a^{n}\right\Vert ^{1/n}$.

Let $\mathcal{V}_{\rho}$ be the class of all Banach algebras $A$ satisfying
the following condition

\begin{enumerate}
\item[(1$_{\rho}$)] For every Banach algebra $B$ containing $A$ as a closed
ideal, any element $b\in B$ with $\rho\left(  b\right)  >\rho\left(
b/A\right)  $ is a point of continuity of the function $\rho$ in $B$.
\end{enumerate}

\begin{theorem}
\label{pcr}Let $B$ be a Banach algebra, and let $A\in\mathcal{V}_{\rho}$ be
its closed ideal. Then $A$ consists of points of continuity of $\rho$ in $B$.
\end{theorem}

\begin{proof}
Let $a\in A$. If $\rho\left(  a\right)  >0$ then $a$ is a point of continuity
of $\rho$ in $B$ by $\left(  1_{\rho}\right)  $. If $\rho\left(  a\right)  =0$
then \textbf{ }continuity of $\rho$ at $a$ follows from the upper continuity.
\end{proof}

\begin{lemma}
\label{410}Let $A$ be a normed algebra, and let $\left(  J_{\alpha}\right)
_{\alpha\in\Lambda}$ be an up-directed net of closed ideals of $A$,
$J=\overline{\bigcup_{\alpha\in\Lambda}J_{\alpha}}$. Then
\begin{align*}
\left\Vert M/J\right\Vert  &  =\lim\left\Vert M/J_{\alpha}\right\Vert
=\inf\left\Vert M/J_{\alpha}\right\Vert ,\\
\rho\left(  M/J\right)   &  =\lim\rho\left(  M/J_{\alpha}\right)  =\inf
\rho\left(  M/J_{\alpha}\right)
\end{align*}
for every precompact set $M$ in $A$.
\end{lemma}

\begin{proof}
This is \cite[Lemma 4.10]{TR3}.
\end{proof}

Let $\mathcal{F}_{\mathcal{V}_{\rho}}$ be the family of all under radicals
whose radical classes are contained in $\mathcal{V}_{\rho}$.

\begin{theorem}
\label{itc}$\vee\mathcal{F}_{\mathcal{V}_{\rho}}$ is a radical and
$\vee\mathcal{F}_{\mathcal{V}_{\rho}}\in\mathcal{F}_{\mathcal{V}_{\rho}}$.
\end{theorem}

\begin{proof}
Indeed, \textrm{H}$_{\mathcal{F}_{\mathcal{V}_{\rho}}}$ is an under radical by
Theorem \ref{dix0}. As $\vee\mathcal{F}_{\mathcal{V}_{\rho}}=$\textrm{H}%
$_{\mathcal{F}_{\mathcal{V}_{\rho}}}^{\ast}$ then $\vee\mathcal{F}%
_{\mathcal{V}_{\rho}}$ is a radical by Theorem \ref{ovunt}.

Let $A$ be a $\left(  \vee\mathcal{F}_{\mathcal{V}_{\rho}}\right)  $-radical
Banach algebra. By Theorem \ref{ch2}, there is an increasing transfinite chain
$\left(  I_{\alpha}\right)  _{\alpha\leq\gamma}$ of closed ideals of $A$ such
that $I_{0}=0$, $I_{\gamma}=A$ and for each $\alpha<\gamma$ there is
$P\in\mathcal{F}_{\mathcal{V}_{\rho}}$ with $I_{\alpha+1}/I_{\alpha}=$
$P\left(  A/I_{\alpha}\right)  $.

Let $B$ be a Banach algebra for which $A$ is an ideal. It follows that
$I_{1}=P_{1}\left(  A\right)  $ for some $P_{1}\in\mathcal{F}_{\mathcal{V}%
_{\rho}}$. So $I_{1}$ is an ideal of $B$. Similarly, $I_{2}/I_{1}=P_{2}\left(
A/I_{1}\right)  $ for some $P_{2}\in\mathcal{F}_{\mathcal{V}_{\rho}}$ and
therefore $I_{2}/I_{1}$ is an ideal of $B/I_{1}$. As $I_{2}=q_{I_{1}}%
^{-1}\left(  P_{2}\left(  A/I_{1}\right)  \right)  $ then $I_{2}$ is an ideal
of $B$. Applying the transfinite induction, it is easy to check that all
$I_{\alpha}$ are ideals of $B$.

Assume, to the contrary, that $A\notin\mathcal{V}_{\rho}$. Then there are a
Banach algebra $B$ and an element $b\in B$ with $\rho\left(  b\right)
>\rho\left(  b/A\right)  $ which is a point of discontinuity of $\rho$. So
there is a sequence $\left(  b_{n}\right)  \subset B$ such that $b_{n}%
\rightarrow b$ as $n\rightarrow\infty$, but $\lim\sup\rho\left(  b_{n}\right)
<\rho\left(  b\right)  $.

Take the first ordinal $\alpha^{\prime}$ for which $\rho\left(  b\right)
\neq\rho\left(  b/I_{\alpha^{\prime}}\right)  $. By Lemma \ref{410},
$\alpha^{\prime}$ is not a limit ordinal. So there is an ordinal
$\alpha<\gamma$ such that
\begin{equation}
\rho\left(  b\right)  =\rho\left(  b/I_{\alpha}\right)  >\rho\left(
b/I_{\alpha+1}\right)  . \label{cc}%
\end{equation}
Let $C=B/I_{\alpha}$, $J=I_{\alpha+1}/I_{\alpha}$ and $x=b/I_{\alpha}$. Then
one can rewrite $\left(  \ref{cc}\right)  $ as%
\[
\rho\left(  x\right)  >\rho\left(  x/J\right)  .
\]
As $J\in\mathcal{V}_{\rho}$ is a closed ideal of $C$ then $x$ is a point of
continuity of $\rho$. Then
\[
\lim\inf\rho\left(  b_{n}\right)  \geq\lim\inf\rho\left(  b_{n}/I_{\alpha
}\right)  =\rho\left(  b/I_{\alpha}\right)  =\rho\left(  b\right)  >\lim
\sup\rho\left(  b_{n}\right)  ,
\]
a contradiction.
\end{proof}

The radical $\vee\mathcal{F}_{\mathcal{V}_{\rho}}$ is denoted by
$\mathcal{R}_{\overrightarrow{\rho}}$ and is called the $\rho$%
\textit{-continuity radical}. Now we extend our knowledge about $\mathbf{Rad}%
\left(  \mathcal{R}_{\overrightarrow{\rho}}\right)  $.

\begin{theorem}
\label{mod-scat-r0}$\mathcal{R}_{s}^{p\ast}\leq$ $\mathcal{R}%
_{\overrightarrow{\rho}}$.
\end{theorem}

\begin{proof}
Let $B$ be a Banach algebra, let $A$ be a closed $\mathcal{R}_{s}^{p}$-radical
ideal of $B$, and let $b\in B$ be such that $\rho\left(  b\right)
>\rho\left(  b/A\right)  $. Let a sequence $b_{n}$ tend to $b$ as
$n\rightarrow\infty$.

Assume, to the contrary, that $\lim\sup\rho(b_{n})<\rho(b)$. Choose
$\varepsilon>0$ such that
\begin{equation}
\rho(b)-\varepsilon>\max\left\{  \rho\left(  b/A\right)  ,\lim\sup\rho
(b_{n})\right\}  . \label{ee}%
\end{equation}
As $\rho(b)=\sup_{I\in\operatorname{Prim}(B^{1})}\rho(b/I)$ then there is a
primitive ideal $I$ of $B$ such that
\[
\rho(b/I)>\rho(b)-\varepsilon.
\]
Thus
\[
\rho(b/I)>\rho\left(  b/A\right)  \geq\rho(q_{I}(b)/\overline{q_{I}(A)}).
\]
Since $q_{I}\left(  A\right)  $ is an $\mathcal{R}_{s}$-radical ideal of $B/I$
then so is $\overline{q_{I}(A)})$. It follows from Theorem \ref{pcr} that
$b/I$ is a point of continuity of $\rho$. In particular,
\[
\lim\inf\left(  \rho\left(  b_{n}\right)  \right)  \geq\lim\inf\left(
\rho\left(  b_{n}/I\right)  \right)  =\rho(b/I),
\]
a contradiction with $\left(  \ref{ee}\right)  $.

We proved that any $\mathcal{R}_{s}^{p}$-radical algebra lies in
$\mathcal{V}_{\rho}$. So $\mathcal{R}_{s}^{p}\in\mathcal{F}_{\mathcal{V}%
_{\rho}}$and $\mathcal{R}_{s}^{p}\leq\mathcal{R}_{\overrightarrow{\rho}}$. As
$\mathcal{R}_{s}^{p\ast}$ is the smallest radical that is larger than or equal
to $\mathcal{R}_{s}^{p}$ then $\mathcal{R}_{s}^{p\ast}\leq\mathcal{R}%
_{\overrightarrow{\rho}}$.
\end{proof}

\begin{corollary}
$\mathcal{R}_{s}^{a}\leq\mathcal{R}_{\overrightarrow{\rho}}$; in particular,
all $\mathcal{R}_{s}^{a}$-radical ideals of a Banach algebra consist of the
points of continuity for $\rho$ in the algebra.
\end{corollary}

\begin{proof}
It follows from Lemma \ref{upperstar} and Theorem \ref{mod-scat-r0}.
\end{proof}

\begin{corollary}
\label{cont-CCR} Let $A$ be a C*-algebra, $J=\mathcal{R}_{\mathfrak{gcr}}(A)$,
the largest GCR-ideal of $A$. Then the spectral radius is continuous at every
point $a\in J$.
\end{corollary}

\begin{proof}
It follows from Theorem \ref{CCR-R} and Theorem \ref{mod-scat-r0}.
\end{proof}

\begin{problem}
Is $\mathcal{R}_{\overrightarrow{\rho}}$ the largest radical among all
radicals $P$ for which $P$-radical ideals consist of the points of continuity
for the spectral radius?
\end{problem}

\begin{problem}
If the spectral radius is continuous on a C*-algebra $A$, is $A$ a GCR-algebra?
\end{problem}

\subsection{Continuity of the joint spectral radius}

Let $A$ be a normed algebra. Here it will be convenient to denote by $\rho
_{j}$ (instead of $\rho$) the function $M\longmapsto$ $\rho\left(  M\right)  $
defined on bounded sets of $A$. It is upper continuous with respect to
Hausdorff's distance \cite[Proposition 3.1]{ST00}, that is
\[
\underset{k\rightarrow\infty}{\lim\sup\;}\rho\left(  M_{k}\right)  \leq
\rho\left(  M\right)
\]
if $M_{k}\rightarrow M$ in the sense that $\operatorname{dist}\left(
M_{k},M\right)  \rightarrow0$ as $k\rightarrow\infty$. If
$\underset{k\rightarrow\infty}{\lim}\rho\left(  M_{k}\right)  =\rho\left(
M\right)  $ for each sequence $M_{k}\rightarrow M$ then we say that $M$ is
a\textit{ point of continuity for the joint spectral radius}.

Let $\mathcal{V}_{\rho_{j}}$ be the class of all normed algebras $A$
satisfying the following condition

\begin{enumerate}
\item[(1$_{j}$)] For every normed algebra $B$ containing $A$ as a closed
ideal, any precompact set $M\subset B$ with $\rho\left(  M\right)
>\rho\left(  M/A\right)  $ is a point of continuity of the function $\rho_{j}$.
\end{enumerate}

Let $\mathcal{F}_{\mathcal{V}_{\rho_{j}}}$ be the family of all topological
under radicals whose radical classes are contained in $\mathcal{V}_{\rho_{j}}$.

\begin{theorem}
Let $\mathcal{R}_{\overrightarrow{\rho_{j}}}=\vee\mathcal{F}_{\mathcal{V}%
_{\rho_{j}}}$. Then

\begin{enumerate}
\item $\mathcal{R}_{\overrightarrow{\rho_{j}}}$ is a radical and
$\mathcal{R}_{\overrightarrow{\rho_{j}}}\in\mathcal{F}_{\mathcal{V}_{\rho_{j}%
}}$;

\item For every normed algebra $A$, every precompact subset $M$ of
$\mathcal{R}_{\overrightarrow{\rho_{j}}}\left(  A\right)  $ is a point of
continuity of $\rho_{j}$;

\item For every normed algebra $A$, $\rho\left(  M\right)  =\sup\left\{
\rho\left(  K\right)  :K\subset M\text{ is finite}\right\}  $ for every
precompact set $M$ in $\mathcal{R}_{\overrightarrow{\rho_{j}}}\left(
A\right)  $;

\item $\mathcal{R}_{\mathrm{hc}}^{r}\vee\mathcal{R}_{\mathrm{cq}}^{a}%
\leq\mathcal{R}_{\overrightarrow{\rho_{j}}}$.
\end{enumerate}
\end{theorem}

\begin{proof}
$\left(  1\right)  \And\left(  4\right)  $ Using Lemma \ref{410} and repeating
the argument in Theorem \ref{itc}, we obtain that $\mathcal{R}%
_{\overrightarrow{\rho_{j}}}$ is a radical and all $\mathcal{R}%
_{\overrightarrow{\rho_{j}}}$-radical algebras lie in $\mathcal{V}_{\rho_{j}}%
$. It was proved in \cite[Theorem 6.3]{TR3} that for $\mathcal{R}%
_{\mathrm{hc}}^{r}\vee\mathcal{R}_{\mathrm{cq}}^{a}$-radical ideals the
condition $\left(  1_{j}\right)  $ holds. Therefore $\left(  4\right)  $ is valid.

$\left(  2\right)  $ is similar to the proof of Theorem \ref{itc}.

$\left(  3\right)  $ follows from $\left(  2\right)  $.
\end{proof}

We call $\mathcal{R}_{\overrightarrow{\rho_{j}}}$ the $\rho_{j}$%
-\textit{continuity radical}.

\begin{corollary}
The joint spectral radius is continuous on precompact subsets of any scattered C*-algebra.
\end{corollary}

In Section \ref{intersection} this result will be extended to all GCR-algebras.

\subsection{Continuity of the tensor radius}

The function $\rho_{t}:N\longmapsto\rho_{t}\left(  N\right)  $ defined on
summable families $N=\left(  a_{m}\right)  _{1}^{\infty}$ of elements of $A$
is upper continuous with respect to the metric $\mathrm{d}\left(  N^{\prime
},N\right)  =\sum_{1}^{\infty}\left\Vert a_{m}^{\prime}-a_{m}\right\Vert $
where $N^{\prime}=\left(  a_{m}^{\prime}\right)  _{1}^{\infty}$ (see
\cite[Proposition 3.12]{TR2}). We say that a family $N$ is a \textit{point of
continuity of the tensor spectral radius} if $\rho_{t}(N_{n})\rightarrow
\rho_{t}(N)$, for any sequence $N_{n}$ of summable families in $A$ that tends
to $N$ with respect to the metric $\mathrm{d}$.

Let $\mathcal{V}_{\rho_{t}}$ be the class of all normed algebras $A$
satisfying the following condition

\begin{enumerate}
\item[(1$_{t}$)] For every normed algebra $B$ containing $A$ as a closed
ideal, any summable family $N$ in $B$ with $\rho_{t}\left(  N\right)
>\rho_{t}\left(  N/A\right)  $ is a point of continuity of the function
$\rho_{t}$.
\end{enumerate}

\begin{lemma}
\label{ten}Let $A$ be a normed algebra, let $\left(  J_{\alpha}\right)
_{\alpha\in\Lambda}$ be an up-directed net of closed ideals of $A$ and
$J=\overline{\bigcup_{\alpha\in\Lambda}J_{\alpha}}$. Then
\begin{align}
\left\Vert N/J\right\Vert _{+}  &  =\lim\left\Vert N/J_{\alpha}\right\Vert
_{+}=\inf\left\Vert N/J_{\alpha}\right\Vert _{+},\label{econt1}\\
\rho_{t}\left(  N/J\right)   &  =\lim\rho_{t}\left(  N/J_{\alpha}\right)
=\inf\rho_{t}\left(  N/J_{\alpha}\right)  \label{econt}%
\end{align}
for every summable family $N=\left(  a_{m}\right)  _{1}^{\infty}$ in $A$.
\end{lemma}

\begin{proof}
As $\left\Vert N/J\right\Vert _{+}\leq\left\Vert N/J_{\alpha}\right\Vert _{+}$
then
\[
\rho_{t}(N/J)=\inf_{n}\left\Vert N^{n}/J\right\Vert _{+}^{1/n}\leq\inf
_{\alpha}\inf_{n}\left\Vert N^{n}/J_{\alpha}\right\Vert _{+}^{1/n}%
=\inf_{\alpha}\rho(N/J_{\alpha})
\]
and
\[
\left\Vert N/J\right\Vert _{+}\leq\inf_{\alpha}\left\Vert N/J_{\alpha
}\right\Vert _{+}\leq\underset{\alpha}{\,\lim\inf}\left\Vert N/J_{\alpha
}\right\Vert _{+}.
\]
By our assumption, for every $a_{m}$ and $\varepsilon_{m}>0$ there exists
$\alpha=\alpha(a_{m},\varepsilon_{m})$ such that
\begin{equation}
\left\Vert a_{m}/J_{\alpha}\right\Vert \leq\left\Vert a_{m}/J\right\Vert
+\varepsilon_{m}. \label{econt2}%
\end{equation}
For $\varepsilon>0$, take $k>0$ such that $\mathrm{d}\left(  N|_{1}%
^{k},N\right)  <\varepsilon$. Let $N_{k}=N|_{1}^{k}$. Take $\varepsilon_{m}>0$
such that $\sum_{1}^{k}\varepsilon_{m}<\varepsilon$. Then $\mathrm{d}\left(
N_{k}/J_{\gamma},N/J_{\gamma}\right)  \leq\mathrm{d}\left(  N_{k},N\right)
<\varepsilon$, and%
\[
\left\Vert N/J_{\gamma}\right\Vert _{+}\leq\mathrm{d}\left(  N_{k}/J_{\gamma
},N/J_{\gamma}\right)  +\left\Vert N_{k}/J_{\gamma}\right\Vert _{+}%
\leq\left\Vert N_{k}/J\right\Vert _{+}+\varepsilon\leq\left\Vert
N/J\right\Vert _{+}+\varepsilon
\]
for $\gamma\geq\max\left\{  \alpha(a_{m},\varepsilon_{m}):m=1,\ldots
,k\right\}  $ by $\left(  \ref{econt2}\right)  $. Therefore
\begin{equation}
\inf_{\alpha}\left\Vert N/J_{\alpha}\right\Vert _{+}\leq\underset{\alpha
}{\,\lim\sup\,}\left\Vert N/J_{\alpha}\right\Vert _{+}\leq\left\Vert
N/J\right\Vert _{+} \label{econt3}%
\end{equation}
that implies (\ref{econt1}). Take $n>0$ such that $\left\Vert N^{n}%
/J\right\Vert _{+}^{1/n}\leq\rho(N/J)+\varepsilon.$ It follows from
(\ref{econt3}) applied to $N^{n}$ that
\begin{align*}
\inf_{\alpha}\rho_{t}(N/J_{\alpha})  &  \leq\underset{\alpha}{\lim\sup
}\,\,\rho_{t}(N/J_{\alpha})\leq\underset{\alpha}{\,\lim\sup\,}\Vert
N^{n}/J_{\alpha}\Vert_{+}^{1/n}\leq\left\Vert N^{n}/J\right\Vert _{+}^{1/n}\\
&  \leq\rho_{t}(N/J)+\varepsilon.
\end{align*}
This implies (\ref{econt}).
\end{proof}

\begin{lemma}
\label{ten2}Let $A$ be a commutative Banach algebra. Then $\rho_{t}$ is
uniformly continuous on $B$ with respect to the metric $\mathrm{d}$.
\end{lemma}

\begin{proof}
One can assume that $A$ is unital. Let $M=\left(  a_{n}\right)  _{1}^{\infty}$
and $N=\left(  b_{n}\right)  _{1}^{\infty}$ be summable families in $A$, and
let $\mathcal{F}$ be the set of all multiplicative functionals $f$ on $A$ with
$\left\Vert f\right\Vert =f\left(  1\right)  =1$. By Theorem \ref{sp4},
\[
\rho_{t}\left(  M\right)  =\sup\left\{  \sum_{1}^{\infty}\left\vert f\left(
a_{n}\right)  \right\vert :f\in\mathcal{F}\right\}
\]
and, by \cite[Propositions 3.3 and 3.4]{TR2}, $\rho_{t}$ is subadditive on
$A$, whence
\begin{align*}
\left\vert \rho_{t}\left(  M\right)  -\rho_{t}\left(  N\right)  \right\vert
&  \leq\rho_{t}\left(  M-N\right)  =\sup\left\{  \sum_{1}^{\infty}\left\vert
f\left(  a_{n}-b_{n}\right)  \right\vert :f\in\mathcal{F}\right\} \\
&  \leq\sum_{1}^{\infty}\left\Vert a_{n}-b_{n}\right\Vert =\mathrm{d}\left(
M,N\right)  .
\end{align*}

\end{proof}

Let $\mathcal{F}_{\mathcal{V}_{\rho_{t}}}$ be the family of all topological
under radicals whose radical classes are contained in $\mathcal{V}_{\rho_{t}}$.

\begin{theorem}
Let $\mathcal{R}_{\overrightarrow{\rho_{t}}}=\vee\mathcal{F}_{\mathcal{V}%
_{\rho_{t}}}$. Then

\begin{enumerate}
\item $\mathcal{R}_{\overrightarrow{\rho_{t}}}$ is a radical and
$\mathcal{R}_{\overrightarrow{\rho_{t}}}$ $\in\mathcal{F}_{\mathcal{V}%
_{\rho_{t}}}$;

\item For every normed algebra $A$, every summable family $N$ of
$\mathcal{R}_{\overrightarrow{\rho_{t}}}\left(  A\right)  $ is a point of
continuity of $\rho_{t}$;

\item For every normed algebra $A$, $\rho_{t}\left(  N\right)  =\sup_{k}%
\rho_{t}\left(  N|_{1}^{k}\right)  $ for every summable family $N$ in
$\mathcal{R}_{\overrightarrow{\rho_{t}}}\left(  A\right)  $;

\item $\mathcal{R}_{t}^{a}\leq\mathcal{R}_{\overrightarrow{\rho_{t}}}$.
\end{enumerate}
\end{theorem}

\begin{proof}
$\left(  1\right)  $ Using Lemma \ref{ten} and repeating the argument in
Theorem \ref{itc}, we obtain that $\mathcal{R}_{\overrightarrow{\rho_{t}}}$ is
a radical and all $\mathcal{R}_{\overrightarrow{\rho_{t}}}$-radical algebras
lie in $\mathcal{V}_{\rho_{t}}$.

$\left(  2\right)  $ is similar to Theorem \ref{itc}.

$\left(  3\right)  $ follows from $\left(  2\right)  $.

$\left(  4\right)  $ It is sufficient to check the condition $\left(
1_{t}\right)  $ for any $\mathcal{R}_{t}$-semisimple Banach algebra $B$. Let
$A$ be a closed $\mathcal{R}_{t}^{a}$-radical ideal of $B$. As $A$ is
$\mathcal{R}_{t}$-semisimple, it is a central ideal. By Theorem \ref{mul},
\begin{equation}
\rho_{t}\left(  N\right)  =\rho_{t}\left(  \mathrm{L}_{N}\right)  \label{d0}%
\end{equation}
for every summable family $N=\left(  a_{m}\right)  _{1}^{\infty}$ in $B$. As
$A$ is a closed invariant subspace for $\mathrm{L}_{B}$ then
\begin{equation}
\rho_{t}\left(  \mathrm{L}_{N}\right)  =\max\left\{  \rho_{t}\left(
\mathrm{L}_{N}|_{B/A}\right)  ,\rho_{t}\left(  \mathrm{L}_{N}|_{A}\right)
\right\}  \label{d1}%
\end{equation}
by Theorem \ref{invar}. It is clear that
\begin{equation}
\rho_{t}\left(  \mathrm{L}_{N}|_{B/A}\right)  =\rho_{t}\left(  \mathrm{L}%
_{N/A}\right)  =\rho_{t}\left(  N/A\right)  . \label{d2}%
\end{equation}

Let now $\rho_{t}\left(  N\right)  >\rho_{t}\left(  N/A\right)  $. It follows
from $\left(  \ref{d0}\right)  $, $\left(  \ref{d1}\right)  $ and $\left(
\ref{d2}\right)  $ that
\[
\rho_{t}\left(  N\right)  =\rho_{t}\left(  \mathrm{L}_{N}|_{A}\right)
\]
Let $\left(  N_{n}\right)  $ be a sequence of summable families of $B$ such
that $N_{n}\rightarrow N$ in the metric \textrm{d }as $n\rightarrow\infty$.
Then \textrm{ } $\mathrm{L}_{N_{n}}|_{A}\rightarrow\mathrm{L}_{N}|_{A}$. It is
easy to check that the algebra $\mathrm{L}_{B}|_{A}$ is commutative. By Lemma
\ref{ten2}, $\rho_{t}$ is continuous on $\mathrm{L}_{B}|_{A}$, whence
\[
\lim\inf\rho_{t}\left(  N_{n}\right)  \geq\lim\inf\rho_{t}\left(
\mathrm{L}_{N_{n}}|_{A}\right)  =\rho_{t}\left(  \mathrm{L}_{N}|_{A}\right)
=\rho_{t}\left(  N\right)
\]
and $\lim\rho_{t}\left(  N_{n}\right)  =\rho_{t}\left(  N\right)  $.

This implies that $\mathbf{Rad}\left(  \mathcal{R}_{t}^{a}\right)
\subset\mathcal{V}_{\rho_{t}}$. As $\mathcal{R}_{\overrightarrow{\rho_{t}}}$
is the largest radical with such property, $\mathcal{R}_{t}^{a}\leq
\mathcal{R}_{\overrightarrow{\rho_{t}}}$.
\end{proof}

We call $\mathcal{R}_{\overrightarrow{\rho_{t}}}$ the $\rho_{t}$%
\textit{-continuity radical}.

\section{Estimations of the joint spectral radius\label{intersection}}

In this section we study how various spectral characteristics of an element or
a subset of a Banach algebra can be expressed via its images in quotients of
the algebra, and how the characteristics of these images depend on the
corresponding ideals. The main attention is devoted to the joint spectral radius.

For example, it is well known that for each element $a$ of any Banach algebra
$A$
\begin{equation}
\rho(a)=\sup\{\rho(a/I):I\in\operatorname{Prim}(A)\}. \label{ind}%
\end{equation}
Can $\operatorname{Prim}(A)$ be changed here by arbitrary family of (primitive
or not) ideals whose intersection is contained in $\operatorname{Rad}(A)$? Is
it possible to extend (\ref{ind}) to $\rho(M)$, where $M$ is a precompact
subset of $A$? How in general depend $\rho(a/I)$ and $\rho(M/I)$ on $I$?
Questions of this type often arise in the spectral theory.

In what follows by an \textit{operator} we mean a bounded linear operator.

\subsection{Finite families of ideals with trivial intersection}

We begin with the spectrum.

\begin{theorem}
\label{spec2} Let $A$ be a Banach algebra, $I_{1}$,...,$I_{n}$ be closed
ideals of $A$, and let $\cap_{k=1}^{n}I_{k}\subset\operatorname{Rad}(A)$. Then%
\[
\sigma(a)=\cup_{k=1}^{n}\sigma(a/I_{k}),
\]
for any element $a\in A$.
\end{theorem}

\begin{proof}
Let $I,J$ be ideals of $A$. Assuming that $0\notin\sigma(a/I)\cup\sigma(a/J),$
we prove that $0\notin\sigma(a/K)$ where $K=I\cap J$. Indeed, since $a$ is
invertible in $A/I$ and $A/J$ then there are $b,c\in A$ with $ba=1+i$ and
$ca=1+j$ where $i\in I$, $j\in J$. Therefore%
\[
(ba-1)(ca-1)=ij\in K.
\]
It follows that $(bac-c-b)a\in1+K$ that is $a$ is left invertible in $A/K$.
Similarly, $a$ is right invertible in $A/K$. Thus $a$ is invertible in $A/K$.

Now using induction, we obtain that
\[
\sigma(a/\cap_{k=1}^{n}I_{k})\subset\cup_{k=1}^{n}\sigma(a/I_{k}).
\]

In our assumptions this implies that
\[
\sigma(a)=\sigma(a/\operatorname{Rad}(A))\subset\cup_{k=1}^{n}\sigma
(a/I_{k});
\]
the converse inclusion is evident.
\end{proof}

Theorem \ref{spec2} implies that if $\mathcal{F}$ is a finite set of closed
ideals with intersection in $\operatorname{Rad}(A)$ then
\begin{equation}
\rho(a)=\sup_{J\in\mathcal{F}}\rho(a/J)\text{ for each }a\in A.
\label{rho-one}%
\end{equation}
Our next aim is to establish a similar result for the joint spectral radius.

We use the following statement proved in \cite[Corollary 4.3]{ST00}.

\begin{lemma}
\label{chain0} Let $M$ be a bounded set of operators on a Banach space $X$,
and let $X\supset Z_{1}\supset...\supset Z_{n}$ be a chain of closed subspaces
invariant for all operators in $M$. Then
\[
\rho(M)=\max\left\{  {\rho}\left(  {M|}_{{X/Z_{1}}}\right)  {,\rho}\left(
{M|}_{{Z_{1}/Z_{2}}}\right)  {,...,\rho}\left(  {M|}_{{Z_{n-1}/Z_{n}}}\right)
{,\rho}\left(  {M|}_{{Z_{n}}}\right)  \right\}  ,
\]
where by $M|_{Z_{k}/Z_{k+1}}$ we mean the family of operators induced by $M$
in the quotient Banach space $Z_{k}/Z_{k+1}$.
\end{lemma}

Let $A$ be a Banach algebra. Recall (sf. Section \ref{banach}) that by a
\textit{Banach ideal} of $A$ we call any ideal $I$ of $A$ which is complete
with respect to a norm $\Vert\cdot\Vert_{I}$ such that
\[
\Vert x\Vert_{I}\geq\Vert x\Vert
\]
for all $x\in I;$ recall also that $\mathrm{W}_{a}:=\mathrm{L}_{a}%
\mathrm{R}_{a}$, for $a\in A$, and $\mathrm{W}_{M}:=\left\{  \mathrm{W}%
_{a}:a\in M\right\}  $, for $M\subset A$. It is clear that all ideals of $A$
are invariant subspaces for the operators $\mathrm{W}_{a}$.

Let $\rho_{\left(  A;\Vert\cdot\Vert\right)  }(M)$ denote the joint spectral
radius of a bounded set $M$ in the algebra $\left(  A;\Vert\cdot\Vert\right)
$.

\begin{lemma}
\label{ineQuot} Let $A$ be a Banach algebra, and let $I$ be a Banach ideal of
$A$. Then

\begin{enumerate}
\item All operators $\mathrm{W}_{a}|_{I}$, $a\in A$, are bounded in the norm
$\Vert\cdot\Vert_{I}$;

\item For each bounded subset $M$ of $A$,%
\begin{equation}
\rho_{\left(  I;\Vert\cdot\Vert_{I}\right)  }(\mathrm{W}_{M}|_{I})\leq
\rho_{\left(  A;\Vert\cdot\Vert\right)  }(M)^{2}. \label{BanId}%
\end{equation}

\end{enumerate}
\end{lemma}

\begin{proof}
$\left(  1\right)  $ It was proved in \cite{B68} that there is a constant
$C>0$ such that
\[
||axb||_{I}\leq C||a||_{A}||x||_{I}||b||_{A}%
\]
for all $a,b\in A^{1}$, $x\in I$. Therefore $\Vert\mathrm{W}_{a}|_{I}\Vert
_{I}\leq C\Vert a\Vert^{2}$ which proves (1).

$\left(  2\right)  $ As $||\mathrm{W}_{M^{n}}||_{I}\leq C\left(  ||M^{n}%
||_{A}\right)  ^{2}$ then $||\mathrm{W}_{M^{n}}||_{I}^{1/n}\leq C^{1/n}\left(
||M^{n}||_{A}\right)  ^{2/n}$, and taking $n\rightarrow\infty$ we obtain (2).
\end{proof}

\begin{proposition}
\label{finPlus} Let $A$ be a Banach algebra, let $I_{1},...,I_{n}$ be closed
ideals of $A$, and let $J=\cap_{i=1}^{n}I_{i}$. Then
\begin{equation}
\rho(M)^{2}=\max\{\max_{1\leq i\leq n}{\rho(M/I_{i})^{2}},\rho(\mathrm{W}%
_{M}|_{J})\}, \label{MorLon}%
\end{equation}
for each bounded subset $M$ of $A$.
\end{proposition}

\begin{proof}
We will prove $\leq$ because the converse inequality is trivial.

For $k\leq n$, let $J_{k}=I_{1}\cap I_{2}\cap...\cap I_{k}$. Applying Lemma
\ref{chain0} to the family $\mathrm{W}=\mathrm{W}_{M}$ of operators on $A$ and
the chain $A\supset J_{1}\supset J_{2}...\supset J_{n}$, we obtain that
\begin{align*}
\rho(M)^{2}  &  =\rho(\mathrm{W}_{M})\\
&  =\max\left\{  \rho(\mathrm{W}_{M}|_{A/J_{1}}),\rho(\mathrm{W}_{M}%
|_{J_{1}/J_{2}}),...,\rho(\mathrm{W}_{M}|_{J_{n-1}/J_{n}}),\rho(\mathrm{W}%
_{M}|_{J_{n}})\right\}  .
\end{align*}
So it suffices to show that
\begin{equation}
\rho(\mathrm{W}_{M}|_{J_{k-1}/J_{k}})\leq\rho(M/I_{k})^{2} \label{ch}%
\end{equation}
for each $k$ (assuming $J_{0}=A$).

For $k=1$, the inequality (\ref{ch}) is in fact an equality:
\[
\rho(\mathrm{W}_{M}|_{A/I_{1}})=\rho(M/I_{1})^{2}.
\]
For a fixed $k$ with $1<k\leq n$, let $B=A/I_{k}$ and $q=q_{I_{k}%
}:A\rightarrow B$. Then the algebraic isomorphism $\phi$ of $I:=q(J_{k-1}%
)=(J_{k-1}+I_{k})/I_{k}$ onto the Banach algebra $C:=J_{k-1}/(I_{k}\cap
J_{k-1})=J_{k-1}/J_{k}$ allows us to supply $I$ with a new norm%
\[
\Vert x\Vert_{I}=\Vert\phi(x)\Vert_{C},
\]
and it is easy to check that $I$ is a Banach ideal of $B$ in this norm. It
follows from the definition of the norm $\Vert\cdot\Vert_{I}$ that
$\Vert\mathrm{W}_{a}|_{J_{k-1}/J_{k}}\Vert=\Vert\mathrm{W}_{q(a)}|_{I}%
\Vert_{I}$, for $a\in A$. Therefore
\begin{equation}
\rho(\mathrm{W}_{M}|_{J_{k-1}/J_{k}})=\rho_{\left(  I;\left\Vert
\cdot\right\Vert _{I}\right)  }(\mathrm{W}_{q(M)}|_{I}). \label{temp-eq}%
\end{equation}
Applying Lemma \ref{ineQuot} to the subset $q(M)$ of $B$ we get:
\[
\rho_{\left(  I;\left\Vert \cdot\right\Vert _{I}\right)  }(\mathrm{W}%
_{q(M)}|_{I})\leq\rho_{B}(q(M))^{2}=\rho(M/I_{k})^{2}.
\]
By (\ref{temp-eq}), this is a reformulation of (\ref{ch}).
\end{proof}

\begin{corollary}
\label{inv-rho} Let $A$ be a Banach algebra, let $\mathcal{F}=\left\{
{I_{1},...,I_{n}}\right\}  $ be a finite family of closed ideals of $A$, such
that $\cap_{i=1}^{n}I_{i}\subset\mathcal{R}_{\mathrm{cq}}(A)$. Then
\begin{equation}
\rho(M)=\max\left\{  {\rho(M/I_{1}),...,\rho(M/I_{n})}\right\}  ,
\label{finRho0}%
\end{equation}
for each precompact subset $M$ of $A$.
\end{corollary}

\begin{proof}
The equality (\ref{finRho0}) is a consequence of Theorem \ref{finPlus}, if
$\cap_{i=1}^{n}I_{i}=0$. It follows that in general
\begin{equation}
\rho(M/(\cap_{i=1}^{n}I_{i}))=\max\left\{  {\rho(M/I_{1}),...,\rho(M/I_{n}%
)}\right\}  . \label{finRho1}%
\end{equation}
Indeed, setting $J=\cap_{i=1}^{n}I_{i}$, $J_{i}=I_{i}/J$ and $\widetilde{M}%
=M/J$, we have that%
\[
\cap_{i=1}^{n}J_{i}=0
\]
whence $\rho(\widetilde{M})=\max_{i}\rho(\widetilde{M}/J_{i})$, and it
suffices to note that $\widetilde{M}/J_{i}$ corresponds to $M/I_{i}$ with
respect to the standard isomorphism of $A/I_{i}$ onto $(A/J)/J_{i}$.

Now, since $\cap_{i=1}^{n}I_{i}\subset\mathcal{R}_{\mathrm{cq}}(A)$ and
$\left(  \ref{cq}\right)  $ holds for precompact sets, we obtain that
\[
\rho(M)=\rho(M/\mathcal{R}_{\mathrm{cq}}(A))\leq\rho(M/\left(  \cap_{i=1}%
^{n}I_{i}\right)  )=\max\left\{  {\rho(M/I_{1}),...,\rho(M/I_{n})}\right\}  .
\]
The converse inequality is evident.
\end{proof}

\subsection{Arbitrary families of ideals with trivial intersection}

The result of Theorem \ref{spec2} does not extend to arbitrary families of
ideals. It suffices to show that the equality (\ref{rho-one}) fails in general.

\begin{example}
Let $\{e_{n}:1\leq n<\infty\}$ be an orthonormal basis in a Hilbert space $H$,
and let $A$ be the algebra of all operators on $H$ preserving the subspaces
$H_{n}=\mathrm{span}(e_{1},...,e_{n})$. Let $K_{n}=\{T\in A:T|_{H_{n}}=0\}$.
Then all $K_{n}$ are closed ideals of $A$, and
\[
\cap_{n=1}^{\infty}K_{n}=0.
\]
Let $S$ be the backward shift: $Se_{n}=e_{n-1}$, $Se_{1}=0$. Then $S\in A$ and
all elements $S/K_{n}$ are nilpotent,
\[
\rho(S/K_{n})=0
\]
while $\sigma(S)=\{\lambda\in\mathbb{C}:|\lambda|\leq1\}$ and
\[
\rho(S)=1.
\]

\end{example}

Let us first of all give a slight extension of Theorem \ref{spec2}.

\begin{proposition}
Let $A$ be an algebra, and let $\mathcal{F}=\left(  I_{\alpha}\right)
_{\alpha\in\Lambda}$ be a family of ideals of $A$. For each $\alpha\in\Lambda
$, set $J_{\alpha}=\cap_{\beta\neq\alpha}I_{\beta}$. If $\cap_{\alpha
\in\Lambda}I_{\alpha}\subset\operatorname{rad}\left(  A\right)  $ then%
\[
\sigma\left(  a\right)  =\left(  \cup_{\alpha\in\Lambda}\sigma\left(
a/I_{\alpha}\right)  \right)  \cup\left(  \cap_{\alpha\in\Lambda}\sigma\left(
a/J_{\alpha}\right)  \right)  ,
\]
for each $a\in A$
\end{proposition}

\begin{proof}
For each $\alpha\in\Lambda$, $I_{\alpha}\cap J_{\alpha}\subset
\operatorname{rad}\left(  A\right)  $, whence%
\[
\sigma\left(  a\right)  =\sigma\left(  a/I_{\alpha}\right)  \cup\sigma\left(
a/J_{\alpha}\right)
\]
by Theorem \ref{spec2}. Now if $\lambda\in\sigma\left(  a\right)  $ and
$\lambda\notin\cup_{\alpha\in\Lambda}\sigma\left(  a/I_{\alpha}\right)  $ then
$\lambda\in\sigma\left(  a/J_{\alpha}\right)  $ for every $\alpha$. Therefore
$\lambda\in\cap_{\alpha\in\Lambda}\sigma\left(  a/J_{\alpha}\right)  $.
\end{proof}

The following result shows that for scattered Banach algebras the situation is
sufficiently satisfactory.

\begin{theorem}
Let $A$ be a Banach algebra. If the spectrum of an element $a$ of $A$ is the
closure of its isolated points then $\sigma\left(  a\right)  =\overline
{\cup_{\alpha\in\Lambda}\sigma\left(  a/I_{\alpha}\right)  }$ for each family
$\mathcal{F}=\left(  I_{\alpha}\right)  _{\alpha\in\Lambda}$ of closed ideals
with $\cap_{\alpha\in\Lambda}I_{\alpha}\subset\operatorname{Rad}(A)$.
\end{theorem}

\begin{proof}
It suffices to show that $\lambda\in\cup_{\alpha\in\Lambda}\sigma\left(
a/I_{\alpha}\right)  $ for each isolated point $\lambda\in\sigma(a)$. Assume,
to the contrary, that%
\[
\lambda\notin\sigma(a/I_{\alpha})
\]
for each $\alpha$. Let $p$ be the Riesz projection of $a$ corresponding to
$\lambda$. Then $(a-\lambda)p$ is quasinilpotent whence $(a-\lambda
)p/I_{\alpha}$ is quasinilpotent for each $\alpha$. But $(a-\lambda
)/I_{\alpha}$ is an invertible element of the algebra $A/I_{\alpha}$ that
commutes with $p/I_{\alpha}$. Therefore $p/I_{\alpha}$ is quasinilpotent. As
it is idempotent, $p/I_{\alpha}=0$ whence $p\in I_{\alpha}$. Then%
\[
p\in\cap_{\alpha\in\Lambda}I_{\alpha}\subset\operatorname{Rad}(A),
\]
a contradiction.
\end{proof}

Let us now come to consideration of the joint spectral characteristics of
subsets in $A$. Recall that the \textit{Berger-Wang spectral radius} $r(M)$ of
a bounded family $M\subset A$ is defined by
\[
r(M)=\limsup_{n\rightarrow\infty}\sup_{a\in M^{n}}\rho(a)^{1/n}.
\]

\begin{proposition}
\label{r-cond} Let $A$ be a Banach algebra, and let $\mathcal{F}$ be a family
of closed ideals such that
\[
\rho(a)=\sup_{J\in\mathcal{F}}\rho(a/J)
\]
for each $a\in A.$ Then
\begin{equation}
r(M)=\sup_{I\in\mathcal{F}}r(M/I) \label{r-bound}%
\end{equation}
for each bounded set $M\subset A$.
\end{proposition}

\begin{proof}
By \cite[(2.1)]{TR3},
\[
r(M)=\sup_{n}\sup_{a\in M^{n}}\rho(a)^{1/n}.
\]
Hence by (\ref{rho-one}),
\[
r(M)=\sup_{n}\sup_{a\in M^{n}}\sup_{I\in\mathcal{F}}\rho(a/I)^{1/n}=\sup
_{I\in\mathcal{F}}\sup_{n}\sup_{a\in M^{n}}\rho(a/I)^{1/n}=\sup_{I\in
\mathcal{F}}r(M/I).
\]

\end{proof}

\begin{corollary}
\label{r-fin}Let $A$ be a Banach algebra, and let $\mathcal{F}$ be a finite
family of closed ideals. If $\cap\{I:I\in\mathcal{F}\}\subset
\operatorname{Rad}(A)$ then
\[
r(M)=\sup_{I\in\mathcal{F}}r(M/I)
\]
for each bounded set $M\subset A$.
\end{corollary}

Now we consider the joint spectral radius $\rho(M)$.

\begin{proposition}
\label{finInt} Let $A$ be a Banach algebra, let $\mathcal{F}$ be a family of
closed ideals of $A$, and let $\mathcal{E}$ be the set of all finite
intersections of ideals from $\mathcal{F}$. Then
\[
\rho(M)^{2}=\max\left\{  \max_{I\in\mathcal{F}}\rho(M/I)^{2},\inf
_{J\in\mathcal{E}}\rho(\mathrm{W}_{M}|_{J})\right\}
\]
for every bounded set $M$ in $A$.
\end{proposition}

\begin{proof}
Only the inequality $\leq$ should be proved. If $\inf_{J\in\mathcal{E}}%
\rho(\mathrm{W}_{M}|_{J})<\rho(M)^{2}$ then there are ideals $I_{1}%
,...,I_{n}\in\mathcal{F}$ with
\[
\rho(\mathrm{W}_{M}|_{J})<\rho(M)^{2}%
\]
where $J=I_{1}\cap I_{2}\cap\cdots\cap I_{n}$. By Proposition \ref{finPlus},
\[
\rho(M)^{2}=\max\left\{  \max_{1\leq i\leq n}{\rho(M/I_{i})^{2}}%
,\rho(\mathrm{W}_{M}|J)\right\}  =\max_{1\leq i\leq n}\rho(M/I_{i})^{2}%
\leq\max_{I\in\mathcal{F}}\rho(M/I)^{2}%
\]
and we are done.
\end{proof}

Now we shall estimate the number $\inf_{J\in\mathcal{E}}\rho(W_{M}|J)\}$. Let
$\mathcal{E}$ be a family of closed subspaces of a Banach space $X$, and let
\begin{align*}
\operatorname{Alg}(\mathcal{E})  &  =\{T\in\mathcal{B}(X):TY\subset Y\text{
for all }Y\in\mathcal{E}\},\\
\left\Vert T\right\Vert _{\mathcal{E}}  &  =\inf\{\Vert T|_{Y}\Vert
:Y\in\mathcal{E}\}\text{ for any }T\in\operatorname{Alg}(\mathcal{E}),\\
\ker\mathcal{E}  &  =\{T\in\operatorname{Alg}(\mathcal{E}):\left\Vert
T\right\Vert _{\mathcal{E}}=0\}.
\end{align*}
For a bounded set $N\subset$ $\operatorname{Alg}(\mathcal{E})$, we write
\[
\left\Vert N\right\Vert _{\mathcal{E}}=\sup_{T\in N}\left\Vert T\right\Vert
_{\mathcal{E}}.
\]

\begin{lemma}
\label{decr} Let $X$ be a Banach space, and let $\mathcal{E}$ be a family of
closed subspaces of $X$, closed under finite intersections and satisfying the
condition $\cap_{Y\in\mathcal{E}}Y=0$. Then

\begin{enumerate}
\item[(1$_{e}$)] $\left\Vert \cdot\right\Vert _{\mathcal{E}}$ is a seminorm on
$\operatorname{Alg}(\mathcal{E})$;

\item[(2$_{e}$)] $\left\Vert TS\right\Vert _{\mathcal{E}}\leq\left\Vert
T\right\Vert _{\mathcal{E}}\left\Vert S\right\Vert _{\mathcal{E}}$ for any
$T,S\in\operatorname{Alg}(\mathcal{E})$;

\item[(3$_{e}$)] $\left\Vert T\right\Vert _{\mathcal{E}}\leq\Vert T\Vert$ for
all $T\in\operatorname{Alg}(\mathcal{E})$;

\item[(4$_{e}$)] $\left\Vert T\right\Vert _{\mathcal{E}}=0$ for each operator
$T\in\operatorname{Alg}(\mathcal{E})$ which is compact on some $Y\in
\mathcal{E}$.
\end{enumerate}
\end{lemma}

\begin{proof}
$\left(  1_{e}\right)  $ and $\left(  2_{e}\right)  $ follow from the fact
that $\Vert T|_{Y_{1}\cap Y_{2}}\Vert\leq\min\{\Vert T|_{Y_{1}}\Vert,\Vert
T|_{Y_{2}}\Vert\}$,

$\left(  3_{e}\right)  $ is evident.

$\left(  4_{e}\right)  $ Let $T\in\operatorname{Alg}(\mathcal{E})$ be compact
on $Y\in\mathcal{E}$. Let $Z_{\odot}$ be the closed unit ball of
$Z\in\mathcal{E}$, and $K_{Z}=\overline{T(Z_{\odot})}$. If $\left\Vert
T\right\Vert _{\mathcal{E}}>0$, choose a number $t$ such that $0<t<\left\Vert
T\right\Vert _{\mathcal{E}}$. For each $Z\in\mathcal{E}$, let $D_{Z}%
=\{\zeta\in K_{Z}:\Vert\zeta\Vert\geq t\}$. Then $\left(  D_{Z\cap Y}\right)
_{Z\in\mathcal{E}}$ is a centered family of compact sets, so there is an
element $z$ in their intersection. It is non-zero and belongs to $\cap\{Y\cap
Z:Z\in\mathcal{E}\}=0$, a contradiction.
\end{proof}

It follows from $\left(  3_{e}\right)  $ that $\left\Vert \cdot\right\Vert
_{\mathcal{E}}$ is continuous with respect to $\Vert\cdot\Vert$ on
$\operatorname{Alg}(\mathcal{E})$. This together with $\left(  2_{e}\right)  $
yield that $\ker\mathcal{E}$ is a closed ideal of $\operatorname{Alg}%
(\mathcal{E})$.

We return to the conditions of Proposition \ref{finInt}. Let $A$ be a Banach
algebra, and let $J$ be a closed ideal of $A$. Let $\mathbb{K}_{J}\left(
A\right)  $ be the set of all closed ideals $I$ of $A$ such that $I$ is
generated by a compact element of $J$. (Recall that $a$ is a \textit{compact
element} of $A$ if $\mathrm{W}_{a}$ is a compact operator on $A$.) Set
\[
\mathbb{K}_{\mathcal{E}}\left(  A\right)  =\cup_{J\in\mathcal{E}}%
\mathbb{K}_{J}\left(  A\right)  .
\]

\begin{lemma}
\label{bik}Let $A$ be a Banach algebra, let $\mathcal{F}$ be a family of
closed ideals of $A$ with zero intersection, and let $\mathcal{E}$ be the set
of all finite intersections of ideals from $\mathcal{F}$. If $K\in
\mathbb{K}_{\mathcal{E}}\left(  A\right)  $ then $\mathrm{L}_{K}\mathrm{R}%
_{K}\subset\ker\mathcal{E}$ and
\[
\inf_{J\in\mathcal{E}}\rho(\mathrm{W}_{M}|_{J})\leq\rho(M)\inf_{K\in
\mathbb{K}_{\mathcal{E}}\left(  A\right)  }\rho(M/K)
\]
for every bounded set $M$ in $A$.
\end{lemma}

\begin{proof}
Let $L\in\mathcal{E}$, let $a$ be a compact element of $L$, $I$ the ideal of
$A$ generated by $a$, and $K=\overline{I}$. Then $\mathrm{W}_{a}%
\in\operatorname{Alg}(\mathcal{E})$ and it is compact on $L$. By Lemma
\ref{decr}, $\left\Vert \mathrm{W}_{a}\right\Vert _{\mathcal{E}}=0$. Note that
the operator $\mathrm{L}_{x}\mathrm{R}_{y}$ for any $x,y\in I$ is represented
as a finite sum $\sum\mathrm{L}_{u_{i}}\mathrm{R}_{v_{i}}\mathrm{W}%
_{a}\mathrm{L}_{s_{i}}\mathrm{R}_{t_{i}}$ for some $u_{i},v_{i},s_{i},t_{i}\in
A^{1}$ whence
\[
\left\Vert \mathrm{L}_{x}\mathrm{R}_{y}\right\Vert _{\mathcal{E}}=\left\Vert
\sum\mathrm{L}_{u_{i}}\mathrm{R}_{v_{i}}\mathrm{W}_{a}\mathrm{L}_{s_{i}%
}\mathrm{R}_{t_{i}}\right\Vert _{\mathcal{E}}\leq\sum\left\Vert \mathrm{L}%
_{u_{i}}\mathrm{R}_{v_{i}}\right\Vert _{\mathcal{E}}\left\Vert \mathrm{W}%
_{a}\right\Vert _{\mathcal{E}}\left\Vert \mathrm{L}_{s_{i}}\mathrm{R}_{t_{i}%
}\right\Vert _{\mathcal{E}}=0.
\]
So $\mathrm{L}_{K}\mathrm{R}_{K}\subset\ker\mathcal{E}$.

Let $M$ be a bounded set in $A$. Denote $\Vert M/K\Vert$ by $d$. For $a,b\in
M$ and $\varepsilon>0$ choose $u,v\in K$ with
\[
\Vert a-u\Vert\leq d+\varepsilon\text{ and }\Vert b-v\Vert\leq d+\varepsilon.
\]
Then $\left\Vert \mathrm{L}_{u}\mathrm{R}_{v}\right\Vert _{\mathcal{E}}=0$
whence
\begin{align*}
\left\Vert \mathrm{L}_{a}\mathrm{R}_{b}\right\Vert _{\mathcal{E}}  &
=\left\Vert \mathrm{L}_{u}\mathrm{R}_{v}+\mathrm{L}_{u}\mathrm{R}%
_{b-v}+\mathrm{L}_{a-u}\mathrm{R}_{v}+\mathrm{L}_{a-u}\mathrm{R}%
_{b-v}\right\Vert _{\mathcal{E}}\\
&  \leq\left\Vert \mathrm{L}_{u}\mathrm{R}_{b-v}\right\Vert _{\mathcal{E}%
}+\left\Vert \mathrm{L}_{a-u}\mathrm{R}_{v}\right\Vert _{\mathcal{E}%
}+\left\Vert \mathrm{L}_{a-u}\mathrm{R}_{b-v}\right\Vert _{\mathcal{E}}\\
&  \leq\Vert u\Vert(d+\varepsilon)+\Vert v\Vert(d+\varepsilon)+(d+\varepsilon
)^{2}\\
&  \leq(\Vert a\Vert+d+\varepsilon)(d+\varepsilon)+(\Vert b\Vert
+d+\varepsilon)(d+\varepsilon)+(d+\varepsilon)^{2}.
\end{align*}
Taking $\varepsilon\rightarrow0$ we get that $\left\Vert \mathrm{L}%
_{a}\mathrm{R}_{b}\right\Vert _{\mathcal{E}}\leq d(\Vert a\Vert+\Vert
b\Vert+3d)\leq5\Vert M/K\Vert\Vert M\Vert$. Therefore
\begin{equation}
\left\Vert \mathrm{W}_{M}\right\Vert _{\mathcal{E}}\leq5\Vert M/K\Vert\Vert
M\Vert. \label{eFq}%
\end{equation}
Changing $M$ by $M^{n}$ in (\ref{eFq}) we have $\left\Vert \mathrm{W}_{M^{n}%
}\right\Vert _{\mathcal{E}}\leq5\Vert M^{n}/K\Vert\Vert M^{n}\Vert$.
Therefore, there is $J\in\mathcal{E}$ with%
\[
\Vert\mathrm{W}_{M^{n}}|_{J}\Vert\leq6\Vert M^{n}/K\Vert\Vert M^{n}\Vert
\]
whence $\rho(\mathrm{W}_{M}|_{J})\leq\Vert\mathrm{W}_{M^{n}}|_{J}\Vert
^{1/n}\leq6^{1/n}\Vert M^{n}/K\Vert^{1/n}\Vert M^{n}\Vert^{1/n}$. Thus
\[
\inf_{J\in\mathcal{E}}\rho(\mathrm{W}_{M}|_{J})\leq\inf_{n}\left(
6^{1/n}\Vert M^{n}/K\Vert^{1/n}\Vert M^{n}\Vert^{1/n}\right)  =\rho
(M/K)\rho(M)
\]
for any $K\in\mathbb{K}\left(  A\right)  .$ This is what we need.
\end{proof}

\begin{theorem}
\label{inv-rho0} Let $A$ be a Banach algebra, let $\mathcal{F}$ be a family of
closed ideals of $A$ such that $\cap_{I\in\mathcal{F}}I=0$, and let
$\mathcal{E}$ be the set of all finite intersections of ideals from
$\mathcal{F}$. Then, for any bounded subset $M$ of $A$,
\begin{equation}
\rho(M)=\max\left\{  \inf_{K\in\mathbb{K}_{\mathcal{E}}\left(  A\right)  }%
\rho(M/K),\max_{I\in\mathcal{F}}{\rho(M/I)}\right\}  . \label{genRho}%
\end{equation}

\end{theorem}

\begin{proof}
Using the result of Lemma \ref{bik}, we get from Proposition \ref{finInt}
that
\begin{align*}
\rho(M)^{2}  &  \leq\max\left\{  \rho(M)\inf_{K\in\mathbb{K}_{\mathcal{E}%
}\left(  A\right)  }\rho(M/K),\max_{I\in\mathcal{F}}{\rho(M/I)}^{2}\right\} \\
&  \leq\rho(M)\max\left\{  \inf_{K\in\mathbb{K}_{\mathcal{E}}\left(  A\right)
}\rho(M/K),\max_{I\in\mathcal{F}}{\rho(M/I)}\right\}
\end{align*}
whence we obtain the inequality $\leq$. The opposite inequality is trivial.
\end{proof}

\subsection{The joint spectral radius and primitive ideals}

We already mentioned and used the fact that the equality (\ref{rho-one}) holds
for $\mathcal{F}=\operatorname{Prim}(A)$. Since $\cap\{I:I\in
\operatorname{Prim}(A)\}=\operatorname{Rad}(A)$, it follows from Proposition
\ref{r-cond} that
\begin{equation}
r(M)=\sup_{I\in\operatorname{Prim}(A)}r(M/I)\text{ for each bounded }M\subset
A. \label{r-prim}%
\end{equation}
Since $\rho(\pi(a))\leq\rho(a/\ker\pi)$ the following well known equality
extends (\ref{rho-one}), for the case $\mathcal{F}=\operatorname{Prim}(A)$:
\begin{equation}
\rho(a)=\sup_{\pi\in\operatorname{Irr}(A)}\rho(\pi(a)),\text{ for each }a\in
A. \label{rho-one-pi}%
\end{equation}
Arguing as in the proof of Proposition \ref{r-cond} we obtain a more strong
version of (\ref{r-prim}):

\begin{proposition}
\label{3.1}Let $A$ be a Banach algebra. Then
\[
r(M)=\sup_{\pi\in\operatorname{Irr}(A)}r(\pi(M))
\]
for each bounded subset $M$ of $A$.
\end{proposition}

We are looking for the conditions that provide the validity of similar
statements for the joint spectral radius:
\begin{align}
\rho(M)  &  =\sup_{I\in\operatorname{Prim}(A)}\rho(M/I)\text{ for each
precompact }M\subset A,\label{rhorepr}\\
\rho(M)  &  =\sup_{\pi\in\operatorname{Irr}(A)}\rho(\pi(M))\text{ for each
precompact }M\subset A. \label{rhoreprPi}%
\end{align}


Recall that a bounded subset $M$ of a normed algebra $A$ is a \textit{point of
continuity} of the joint spectral radius if $\rho(M_{n})\rightarrow\rho(M)$,
for each sequence of bounded subsets $M_{n}\subset A$ that tends to $M$ in the
Hausdorff metric.

\begin{lemma}
\label{sirano} Let $A$ be a normed algebra, and let $M$ be a bounded set of
$A$ such that $\rho(M)=r(M)$. If $M$ consists of the points of continuity of
the (usual) spectral radius then $M$ is a point of continuity of the joint
spectral radius.
\end{lemma}

\begin{proof}
Let $M_{n}\to M$. Since the joint spectral radius is upper semicontinuous
\cite[Proposition 3.1]{ST00}, we have only to show that $\liminf_{n\to\infty
}\rho(M_{n}) \ge\rho(M)$.

Assume, to the contrary, that
\[
\lim_{n\rightarrow\infty}\rho(M_{n})<1<\rho(M).
\]
Since $\rho(M)=r(M)$, there are $k\in\mathbb{N}$ and $T\in M^{k}$ such that
$\rho(T)>1$. Clearly there are $T_{n}\in M_{n}^{k}$ with $T_{n}\rightarrow T$;
since $\rho$ is continuous at $T$ then $\rho(T_{n})\rightarrow\rho(T)$. But
\[
\rho(T_{n})\leq\rho(M_{n}^{k})=\rho(M_{n})^{k}<1
\]
for sufficiently big $n$, a contradiction.
\end{proof}

Recall that a normed algebra $A$ is a \textit{Berger-Wang algebra} if
$\rho(M)=r(M)$, for each precompact subset $M$ of $A$. It follows immediately
from Proposition \ref{3.1} that (\ref{rhorepr}) and (\ref{rhoreprPi}) hold for
every Berger-Wang algebra.

\begin{theorem}
\label{rhoJ} Let $A$ be a Berger-Wang Banach algebra. Then every precompact
subset of $\mathcal{R}_{s}^{p\ast}(A)$ is a point of continuity of the joint
spectral radius.
\end{theorem}

\begin{proof}
Follows from Theorem \ref{mod-scat-r0} and Lemma \ref{sirano}.
\end{proof}


Recall that a non-necessarily Hausdorff topological space $T$ is called
\textit{quasicompact} if each its open covering contains a finite subcovering.
A function $\phi:T\longrightarrow\mathbb{R}$ is \textit{lower (upper)
semicontinuous} if for each $\lambda\in\mathbb{R}$, the set $\{t\in
T:\phi(t)\leq\lambda\}$ (respectively $\{t\in T:\phi(t)\geq\lambda\}$) is closed.

The following result is a variation of the classical Dini Theorem:

\begin{theorem}
\label{Dini} Let $f_{n}$ be a decreasing sequence of functions on a
quasicompact space $T$ pointwise converging to a function $f$. If all $f_{n}$
are upper semicontinuous then
\[
\sup_{t\in T}f_{n}(t)\rightarrow\sup_{t\in T}f(t)\text{ as }n\rightarrow
\infty.
\]

\end{theorem}

\begin{proof}
Choose a number $d>\sup_{t\in T}f(t)$. If $\sup_{t\in T}f_{n}(t)>d$, for all
$n$, then the sets $E_{n}=\{t\in T:f_{n}(t)\geq d\}$ are non-empty and closed.
Since $E_{n}\subset E_{n-1}$ and $T$ is quasicompact, there is a point
$t_{0}\in\cap_{n}E_{n}$. Thus $f_{n}(t_{0})\geq d$ for all $n$, and therefore
\[
f(t_{0})=\lim_{n\rightarrow\infty}f_{n}(t_{0})\geq d>\sup_{t\in T}f(t),
\]
a contradiction.

It follows that
\[
\lim_{n\rightarrow\infty}\sup_{t\in T}f_{n}(t)=\inf_{n\rightarrow\infty}%
\sup_{t\in T}f_{n}(t)\leq\sup_{t\in T}f(t).
\]
On the other hand,
\[
\lim_{n\rightarrow\infty}\sup_{t\in T}f_{n}(t)\geq\sup_{t\in T}f(t)
\]
because $f_{n}(t)\geq f(t)$ for all $t$.
\end{proof}

We consider the following properties which a Banach algebra $A$ can have:

\begin{enumerate}
\item[$\left(  1_{c}\right)  $] For each $a\in A$, $\Vert a\Vert=\sup\{\Vert
a/I\Vert:I\in\operatorname{Prim}(A)\}$;

\item[$\left(  2_{c}\right)  $] For each $a\in A$, the map $I\longmapsto\Vert
a/I\Vert$ is upper semicontinuous on $\operatorname{Prim}(A)$.
\end{enumerate}

\begin{lemma}
\label{UpSem} Let $A$ be a Banach algebra, and let $M$ be a precompact subset
of $A$. Then

\begin{enumerate}
\item If $A$ has the property $\left(  1_{c}\right)  $ then $\Vert M\Vert
=\sup_{I\in\operatorname{Prim}(A)}\Vert M/I\Vert$;

\item If $A$ has the property $\left(  2_{c}\right)  $ then the map
$I\longmapsto\Vert M/I\Vert$ is upper semicontinuous on $\operatorname{Prim}%
(A)$.
\end{enumerate}
\end{lemma}

\begin{proof}
(1) For each $\varepsilon>0$, choose $a\in M$ with $\Vert a\Vert>\Vert
M\Vert-\varepsilon$. Using $\left(  1_{c}\right)  $, choose $I_{0}%
\in\operatorname{Prim}(A)$ with $\Vert a/I_{0}\Vert>\Vert a\Vert-\varepsilon$.
Then
\[
\Vert M\Vert<\Vert a/I_{0}\Vert+2\varepsilon\leq\Vert M/I_{0}\Vert
+2\varepsilon\leq\sup_{I}\Vert M/I\Vert+2\varepsilon.
\]
Taking $\varepsilon\rightarrow0$ we get that $\Vert M\Vert\leq\sup_{I}\Vert
M/I\Vert$, the converse is evident.

(2) If $N\subset M$ is finite then $I\longmapsto\Vert N/I\Vert$ is upper
semicontinuous because the set
\[
\{I\in\operatorname{Prim}(A):\Vert N/I\Vert\geq\lambda\}=\cup_{a\in N}%
\{I\in\operatorname{Prim}(A):\Vert a/I\Vert\geq\lambda\}
\]
is closed being a finite union of closed sets. Now, for each $\varepsilon>0$,
let $N_{\varepsilon}$ be a finite $\varepsilon$-net in $M$. Then
$E_{\varepsilon}:=\{I\in\operatorname{Prim}(A):\Vert N_{\varepsilon}%
/I\Vert\geq\lambda-\varepsilon\}$ is closed and therefore
\[
\{I\in\operatorname{Prim}(A):\Vert M/I\Vert\geq\lambda\}=\cap_{\varepsilon
>0}E_{\varepsilon}%
\]
is closed.
\end{proof}

\begin{theorem}
\label{coinc} Let $A$ be a Banach algebra satisfying $\left(  1_{c}\right)  $
and $\left(  2_{c}\right)  $, and let $M$ be a precompact subset of $A$. If
$\rho(M/I)=r(M/I)$ for all $I\in\operatorname{Prim}(A)$, then
\begin{equation}
\rho(M)=\sup_{I\in\operatorname{Prim}(A)}\rho(M/I)=r(M). \label{both}%
\end{equation}

\end{theorem}

\begin{proof}
The functions $I\longmapsto f_{n}(I):=\Vert M^{2^{n}}/I\Vert^{1/{2^{n}}}$ are
upper semicontinuous by Lemma \ref{UpSem}(2). Moreover, they decrease
($f_{n+1}(I)\leq f_{n}(I)$) and
\[
\lim_{n}f_{n}(I)=\rho(M/I)=r(M/I).
\]
Thus by Theorem \ref{Dini}, $\sup_{I}f_{n}(I)\rightarrow\sup_{I}\rho(M/I)$.
Using Lemma \ref{UpSem}(1), one gets
\[
\sup_{I}f_{n}(I)=\sup_{I}\Vert M^{2^{n}}/I\Vert^{1/{2^{n}}}=\Vert M^{2^{n}%
}\Vert^{1/{2^{n}}}\rightarrow\rho(M).
\]
We proved that
\[
\rho(M)=\sup_{I\in\operatorname{Prim}(A)}\rho(M/I);
\]
the second equality follows from Theorem \ref{3.1}.
\end{proof}

It is convenient to formulate an analogue of Theorem \ref{coinc} for algebras
of vector-valued functions on arbitrary compacts.

\begin{theorem}
\label{functions} Let $B$ be a Berger-Wang Banach algebra, and $A=C(T,B)$, the
algebra of all continuous $B$-valued functions from a quasicompact space $T$
to $B$, supplied with the sup-norm. Then $A$ is a Berger-Wang algebra.

If the spectral radius function $x\longmapsto\rho(x)$ is continuous on $B$
then the same is true for $A$, and moreover, the joint spectral radius is
continuous on $B$.
\end{theorem}

\begin{proof}
We argue as in the proof of Theorem \ref{coinc} with the change of
$\operatorname{Prim}(A)$ by $T$. The analogue of (1$_{c}$) follows from the
definition of the norm in $A$. The analogue of (2$_{c}$) holds by definition
(functions are continuous, so the norms are continuous).

Let $M$ be a precompact subset of $A$; for each $t\in T$, let
$M(t)=\{a(t):a\in M\}$. Arguing as above, we prove that
\begin{equation}
\rho(M)=\sup_{t\in T}\rho(M(t))=r(M). \label{rhocont}%
\end{equation}

Let $x\longmapsto\rho(x)$ be continuous on $B$. To check the continuity of
$\rho$ on $A$, note that for $a\in A$
\[
\rho(a)=\sup_{t\in T}\rho(a(t))
\]
(a special case of (\ref{rhocont})). Now if $a_{\lambda}\rightarrow a$ in $A$
then $a_{\lambda}(t)\rightarrow a(t)$ in $B$, for each $t\in T$, whence
\[
\rho(a(t))=\lim_{\lambda}\rho(a_{\lambda}(t))\leq\underset{\lambda}{\liminf
}\rho(a_{\lambda}).
\]
Therefore $\rho(a)\leq\underset{\lambda}{\liminf}\rho(a_{\lambda})$. Since
upper semicontinuity holds in general, we get the continuity of $a\longmapsto
\rho(a)$ on $A$.

The continuity of the joint spectral radius follows from Lemma \ref{sirano}.
\end{proof}

One can take for $B$ an arbitrary algebra of compact operators on a Banach
space $X$ (for this case, the equality $\rho(M) = r(M)$ was established in
\cite{ST00}).

\subsection{A C*-algebra version of the joint spectral radius formula}

In \cite{TR3} it was shown that for each Banach algebra $A$ and each
precompact set $M\subset A$,
\begin{equation}
\rho(M)=\max\{r(M),\rho(M/\mathcal{R}_{\mathrm{hc}}(A))\}. \label{GBWF}%
\end{equation}
It is interesting and important question if dealing with C*-algebras one can
change in (\ref{GBWF}) the ideal $\mathcal{R}_{\mathrm{hc}}(A)$ by the much
larger ideal $\mathcal{R}_{\mathfrak{gcr}}(A)$. In other words, we study the
validity of the equality
\begin{equation}
\rho(M)=\max\{r(M),\rho(M/\mathcal{R}_{\mathfrak{gcr}}(A))\}. \label{GBWFC}%
\end{equation}
Our approach is based on the consideration of primitive ideals; in particular
we study C*-algebras for which the equality (\ref{both}) holds.

The results of previous sections give some valuable information on this
question. Note that the condition (1$_{c}$) holds for each C*-algebra. The
condition (2$_{c}$) holds for C*-algebras with Hausdorff spectra. Indeed, it
is proved in \cite[Proposition 3.3.7]{Dixmier} that for each element $a\in A$,
the set $\{I\in\operatorname{Prim}(A):\Vert a/I\Vert\geq\lambda\}$ is
quasicompact; thus if $\operatorname{Prim}(A)$ is Hausdorff then this set is
closed, whence the function $I\longmapsto\Vert a/I\Vert$ is upper
semicontinuous (in fact, continuous because it is always lower semicontinuous
\cite[Proposition 3.3.2]{Dixmier}). Thus we obtain from Theorem \ref{coinc}
that the equality (\ref{both}) holds for each CCR C*-algebra $A$ with
Hausdorff $\operatorname{Prim}(A)$.

We will obtain much more general results here. Let us say that a topological
space has the\textit{ property }$(QC)$ if the intersection of any
down-directed net of non-empty quasicompact subsets is non-empty. Of course,
each Hausdorff space has this property.

\begin{lemma}
\label{QC} Let $T$ be a topological space. Then

\begin{enumerate}
\item If $T=\cup_{\lambda\in\Lambda}U_{\lambda}$, where the net $\{U_{\lambda
}:{\lambda\in\Lambda}\}$ is up-directed, all $U_{\lambda}$ are open and have
property $(QC)$, then $T$ has property $(QC)$;

\item If $T=F\cup U$, where $F$ is closed and Hausdorff, $U$ is open and has
property $(QC)$, then $T$ has property $(QC)$;

\item If $T$ has the property $(QC)$ then each subset of $T$ has this property.
\end{enumerate}
\end{lemma}

\begin{proof}
Let $\{E_{\omega}: \omega\in\Omega\}$ be a down-directed set of quasicompact subsets.

(1) For some $\omega_{1}\in\Omega$, let $\lambda_{1},...,\lambda_{n}\in
\Lambda$ be such that $E_{\omega_{1}}\subset\cup_{k=1}^{n}U_{\lambda_{k}}$.
Since the net $\{U_{\lambda}:{\lambda\in\Lambda}\}$ is up-directed,%
\[
E_{\omega_{1}}\subset U_{\lambda_{0}}%
\]
for some $\lambda_{0}\in\Lambda$. Then all $E_{\omega}$ with $\omega
>\omega_{1}$ are contained in the space $U_{\lambda_{0}}$ which has the
property $(QC)$. By definition, their intersection is non-empty whence
$\cap_{\omega\in\Omega}E_{\omega}$ is non-empty.

(2) Let us check firstly that the intersection of a quasicompact set $K$ with
a closed set $W$ is quasicompact. Indeed, if $\{U_{\lambda}:{\lambda\in
\Lambda}\}$ is a family of open subsets in $T$ with $K\cap W\subset
\cup_{\lambda\in\Lambda}U_{\lambda}$ then
\[
K\subset(\cup_{\lambda\in\Lambda}U_{\lambda})\cup(T\setminus W);
\]
choosing a finite subcovering and removing $T\setminus W$ we obtain a finite
subcovering of $K\cap W$.

Now in assumptions of (2) all sets $E_{\omega}\cap F$ are quasicompact subsets
of a Hausdorff space, so they are compact. Hence if all $E_{\omega}\cap F$ are
non-empty then they have non-empty intersection.

On the other hand, if $E_{\omega_{0}}\cap F=\varnothing$, for some $\omega
_{0}\in\Omega$, then $E_{\omega_{0}}\subset U$, and the same is true for all
$\omega>\omega_{0}$, so the intersection is non-empty because $U$ has property
$(QC)$.

(3) Let $W$ be an arbitrary subset of $T$. It follows easily from the
definition that if $E\subset W$ is quasicompact in $W$ then it is quasicompact
in $T$. Thus a down-directed net of quasicompact subsets of $W$ is a
down-directed net of quasicompact subsets of $T$; by the assumptions, the
intersection is non-void.
\end{proof}

Let $\mathfrak{C}_{\mathfrak{qc}}$ denote the class of all C*-algebras $A$ for
which $\operatorname{Prim}(A)$ has the property $(QC)$.

\begin{corollary}
\label{C-QC} Let $A$ be a C*-algebra. Then

\begin{enumerate}
\item If $\left(  J_{\lambda}\right)  _{{\lambda\in\Lambda}}$ is an
up-directed net of closed ideals of $A$ such that $\cup_{\lambda\in\Lambda
}J_{\lambda}$ is dense in $A$ and $J_{\lambda}\in\mathfrak{C}_{\mathfrak{qc}}$
for all $\lambda$, then $A\in$ $\mathfrak{C}_{\mathfrak{qc}}$;

\item If $A$ has an ideal $J\in\mathfrak{C}_{\mathfrak{qc}}$ such that
$\operatorname{Prim}(A/J)$ is Hausdorff, then $A\in$ $\mathfrak{C}%
_{\mathfrak{qc}}$;

\item If $A\in\mathfrak{C}_{\mathfrak{qc}}$ then all closed ideals and all
quotients of $A$ belong to $\mathfrak{C}_{\mathfrak{qc}}$.
\end{enumerate}
\end{corollary}

\begin{proof}
(1) One has that
\[
\operatorname{Prim}(A)\cong\overline{\cup_{\lambda\in\Lambda}%
\operatorname{Prim}(J_{\lambda})},
\]
if $\operatorname{Prim}(J_{\lambda})$ is identified with the set of all
$I\in\operatorname{Prim}(A)$ that do not contain $J_{\lambda}$. All
$\operatorname{Prim}(J_{\lambda})$ are open and have the property $(QC)$.
Using Lemma \ref{QC}$\left(  1\right)  $ we get that $\operatorname{Prim}(A)$
has property $(QC)$. So $A\in\mathfrak{C}_{\mathfrak{qc}}$.

(2) In this case $\operatorname{Prim}(A)\cong\operatorname{Prim}%
(J)\cup\operatorname{Prim}(A/J)$ and it remains to apply Lemma \ref{QC}%
$\left(  2\right)  $.

(3) Follows from Lemma \ref{QC}$\left(  3\right)  $ because, for each closed
ideal $J$ of $A$, the space $\operatorname{Prim}(J)$ (respectively,
$\operatorname{Prim}(A/J)$) is homeomorphic to an open (respectively, closed)
subset of $\operatorname{Prim}(A)$.
\end{proof}

\begin{corollary}
\label{GCR-QC} All GCR C*-algebras belong to $\mathfrak{C}_{\mathfrak{qc}}$.
\end{corollary}

\begin{proof}
It is known \cite[Proposition 4.5.3 and Theorem 4.5.5]{Dixmier} that if $A$ is
a GCR algebra then there is an increasing transfinite chain $\left(
J_{\alpha}\right)  _{\alpha\leq\delta}$ of closed ideals such that $J_{0}=0$,
$J_{\delta}=A$ and all gap-quotients $J_{\alpha+1}/J_{\alpha}$ have Hausdorff
space of primitive ideals.

Assume, to the contrary, that there is a smallest ordinal $\gamma$ for which
$J_{\gamma}\notin\mathfrak{C}_{\mathfrak{qc}}$. If $\gamma$ is limit then
$J_{\gamma}\in\mathfrak{C}_{\mathfrak{qc}}$ by Lemma \ref{C-QC}$\left(
1\right)  $; if not, then $\gamma=\alpha+1$ for some ordinal $\alpha$, whence
$J_{\gamma}\in\mathfrak{C}_{\mathfrak{qc}}$ by Lemma \ref{C-QC}$\left(
2\right)  $, a contradiction.
\end{proof}

At the moment we do not know an example of a C*-algebra which does not belong
to $\mathfrak{C}_{\mathfrak{qc}}$.

\begin{lemma}
\label{norm-qc}Let $A$ be a C*-algebra, and let $M$ be a precompact subset of
$A$. Then, for any $t>0$, the set $\{I\in\operatorname{Prim}(A):\Vert
M/I\Vert\geq t\}$ is quasicompact.
\end{lemma}

\begin{proof}
We may argue as in \cite[Proposition 3.3.7]{Dixmier}, where the statement is
proved for one-element sets, if we establish that for each closed ideal $J$ of
$A$, there is a primitive ideal $I\supset J$ with $\Vert M/I\Vert=\Vert
M/J\Vert$. To show the existence of $I$, let $x_{n}\in M$ be such that
\[
\Vert M/J\Vert=\lim_{n\rightarrow\infty}\Vert x_{n}/J\Vert.
\]
By precompactness of $M$ we may assume that $x_{n}\rightarrow x\in A$, and it
remains to take $I$ with $\Vert x/I\Vert=\Vert x/J\Vert$ (\cite[Lemma
3.3.6]{Dixmier}).
\end{proof}

\begin{theorem}
\label{QC-rho}Let $A$ be a C*-algebra. If $A\in\mathfrak{C}_{\mathfrak{qc}}$
then%
\[
\rho(M)=\sup_{I\in\operatorname{Prim}(A)}\rho(M/I)
\]
for each precompact subset $M$ of $A$.
\end{theorem}

\begin{proof}
For each $n$, let $K_{n}=\left\{  I\in\operatorname{Prim}(A):\Vert M^{2^{n}%
}/I\Vert\geq\rho(M)^{2^{n}}\right\}  $. Since
\[
\Vert M^{2^{n}}\Vert\geq\rho(M)^{2^{n}},
\]
all $K_{n}$ are non-zero, and they are quasicompact by Lemma \ref{norm-qc}.
Since the sequence $\Vert M^{2^{n}}/I\Vert^{1/2^{n}}$ decreases, the inclusion%
\[
K_{n}\supset K_{n+1}%
\]
holds for each $n$. It follows that there is a primitive ideal $I\in\cap
_{n=1}^{\infty}K_{n}$. Thus
\[
\Vert M^{2^{n}}/I\Vert^{1/2^{n}}\geq\rho(M)
\]
for all $n$. Taking the limit we get that $\rho(M/I)\geq\rho(M)$.
\end{proof}

\begin{corollary}
\label{QC-quot}Let $A$ be a C*-algebra, and let $M$ be a precompact subset of
$A$. If $A\in\mathfrak{C}_{\mathfrak{qc}}$ and $\rho(M/I)=r(M/I)$ for all
$I\in\operatorname{Prim}(A)$ then
\[
\rho(M)=r(M).
\]

\end{corollary}

\begin{proof}
By Theorem \ref{QC-rho}, we obtain that
\[
\rho(M)=\sup_{I\in\operatorname{Prim}(A)}\rho(M/I)=\sup_{I\in
\operatorname{Prim}(A)}r(M/I)\leq r(M).
\]

\end{proof}

The following result establishes a C*-version of the joint spectral radius
formula for algebras satisfying the condition (\ref{rhorepr}) with all their quotients.

\begin{proposition}
\label{if1} Let $A$ be a C*-algebra. If, for any closed ideal $K$ of $A$,%
\[
\rho(N)=\sup_{I\in\operatorname{Prim}(A/K)}\rho(N/I)
\]
for every precompact subset $N$ of $A/K$ then
\[
\rho(M)=\max\{r(M),\rho(M/\mathcal{R}_{\mathfrak{gcr}}(A))\}
\]
for every precompact subset $M$ of $A$.
\end{proposition}

\begin{proof}
We show that
\begin{equation}
\rho(M)=\max\{r(M),{\rho}(M/J)\} \label{GBWFstar}%
\end{equation}
for each closed GCR ideal $J$ of $A$.

Let us firstly prove (\ref{GBWFstar}) for the case that $J$ is a CCR-ideal
(assuming only that (\ref{rhorepr}) is true for $A$). By (\ref{rhorepr}), it
suffices to show that
\[
{\rho}(M/I)\leq\max\{{\rho}(M/J),r(M)\}
\]
for each $I\in\operatorname{Prim}(A)$. Let $\pi$ be an irreducible
representation of $A$ on a Hilbert space $H$, with $\ker\pi=I$. Then
\[
\rho(M/I)=\rho(\pi(M)).
\]
If ${\pi}(J)=0$ then ${\pi}$ is a representation of $A/J$ and the inequality
\[
{\rho}({\pi}(M))\leq{\rho}(M/J)
\]
is evident. Otherwise ${\pi}(J)=\mathcal{K}(H)$ and%
\[
{\rho}_{e}({\pi}(M))={\rho}({\pi}(M)/{\pi}(J))\leq{\rho}(M/J).
\]
Applying the operator version of the joint spectral radius formula $\left(
\ref{afo}\right)  $ to ${\pi}(M)$ we obtain the needed inequality.

In general, we have an increasing transfinite chain $\left(  J_{\alpha
}\right)  _{{{\alpha}\leq\delta}}$ such that $J_{0}=0$, $J_{\delta}=J$ and all
gap-quotients $J_{{\alpha}+1}/J_{{\alpha}}$ are CCR-algebras.

Assume, to the contrary, that there is an ordinal $\gamma$ which is the
smallest one among ordinals ${\alpha}$ for which (\ref{GBWFstar}) is not true
with $J=J_{\alpha}$. If ${\gamma}$ is not limit then ${\gamma}={\alpha}+1$ for
some $\alpha$. As $M/J_{\gamma}\cong(M/J_{\alpha})/(J_{\gamma}/J_{\alpha})$
then
\[
\rho(M/J_{\alpha})=\max\{r(M/J_{\alpha}),{\rho}(M/J_{\gamma})\}.
\]
Since also $\rho(M)=\max\{r(M),{\rho}(M/J_{\alpha})\}$ then
\[
\rho(M)=\max\{r(M),{\rho}(M/J_{\gamma})\},
\]
a contradiction.

Let now ${\gamma}$ be a limit ordinal. By Lemma \ref{410},
\begin{align*}
\max\{r(M),{\rho}(M/J_{\gamma})\}  &  =\max\{r(M),\inf_{\alpha^{\prime}%
<\gamma}{\rho}(M/J_{\alpha^{\prime}})\}\\
&  =\inf_{\alpha^{\prime}<\gamma}\max\{r(M),{\rho}(M/J_{\alpha^{\prime}%
})\}={\rho}(M),
\end{align*}
a contradiction, so the equality holds for all ordinals.
\end{proof}

\begin{corollary}
\label{C(QC)-BW}Let $A$ be a C*-algebra. If $A\in\mathfrak{C}_{\mathfrak{qc}}$
then
\[
\rho(M)=\max\{r(M),\rho(M/\mathcal{R}_{\mathfrak{gcr}}(A))\}
\]
for every precompact subset $M$ of $A$.
\end{corollary}

\begin{proof}
Indeed, any quotient of a C*-algebra from $\mathfrak{C}_{\mathfrak{qc}}$
belongs to $\mathfrak{C}_{\mathfrak{qc}}$ (see Corollary \ref{C-QC}$\left(
3\right)  $). So the result follows from Theorems \ref{if1} and \ref{QC-rho}.
\end{proof}

\begin{corollary}
\label{GCR-BW} Each GCR C*-algebra is a Berger-Wang algebra.
\end{corollary}

\begin{proof}
Follows from Corollary \ref{C(QC)-BW}.
\end{proof}

\begin{corollary}
\label{finiteW} Let $A$ be a C*-algebra. Then any precompact subset of
$\mathcal{R}_{\mathfrak{gcr}}\left(  A\right)  $ is a point of continuity of
the joint spectral radius.
\end{corollary}

\begin{proof}
Note that $\mathcal{R}_{\mathfrak{gcr}}\left(  A\right)  $ is the largest
GCR-ideal of $A$ by Theorem \ref{special}. So the result follows from
Corollary \ref{GCR-BW} and Theorems \ref{CCR-R} and \ref{rhoJ}.
\end{proof}

\end{document}